\documentclass[leqno,11pt]{amsart}

\usepackage{amsfonts}
\usepackage{color}
\usepackage{amssymb,amsthm,amsmath,hyperref,txfonts}
\usepackage{epsfig}
\usepackage{graphicx,float}
\usepackage{CJK}
\usepackage{setspace}
\usepackage{cancel}

 \newtheorem{thm}{Theorem}[section]
 \newtheorem{coro}[thm]{Corollary}
 \newtheorem{lem}[thm]{Lemma}
 \newtheorem{prop}[thm]{Proposition}
 \theoremstyle{definition}
 
 \newtheorem{rem}[thm]{Remark}
 \numberwithin{equation}{section}

\def\dl{\delta}
\def\tl{\tilde}

\def\Dl{\Delta}
\def\gt{\gtrsim}

\def\nn{\nonumber}

\def\sq{\sqrt}
\def\eps{\epsilon}

\def\fr{\frac}
\def\al{\alpha}

\def\la{\langle}
\def\ra{\rangle}

\def\gtr{\gtrsim}

\def\les{\lesssim}
\def\lm{\lambda}

\def\Om{\Omega}

\def\om{\omega}

\def\l|{\left\|}
\def\r|{\right\|}

\newcommand{\beq}{\begin{eqnarray}}
\newcommand{\eeq}{\end{eqnarray}}
\newcommand{\beqno}{\begin{eqnarray*}}
\newcommand{\eeqno}{\end{eqnarray*}}
\newcommand{\be}{\begin{equation}}
\newcommand{\ee}{\end{equation}}
\newcommand{\beno}{\begin{equation*}}
\newcommand{\eeno}{\end{equation*}}
\newtheorem{theorem}{Theorem}[section]

\newtheorem{Lemma A.1}{Lemma A.1}
\theoremstyle{definition}

\theoremstyle{remark}

\allowdisplaybreaks[4]

\begin{document}
\title[2d MHD]{Asymptotic stability of Couette flow in a strong uniform magnetic field for the Euler-MHD system}

\author{Weiren Zhao}
\address[W. Zhao]{Department of Mathematics, New York University Abu Dhabi, Saadiyat Island, P.O. Box 129188, Abu Dhabi, United Arab Emirates.}
\email{zjzjzwr@126.com, wz19@nyu.edu}

\author{Ruizhao Zi}

\address[R. Zi]{School of Mathematics and Statistics, and Key Laboratory of Nonlinear Analysis \& Applications (Ministry of Education), Central China Normal University, Wuhan,  430079,  P. R. China.}
\email{rzz@mail.ccnu.edu.cn}

\date{\today}
\maketitle
\begin{abstract}
In this paper, we prove the asymptotic stability of Couette flow in a strong uniform magnetic field for the Euler-MHD system, when the perturbations are in Gevrey-$\frac{1}{s}$, $(\frac12<s\leq 1)$ and of size smaller than the resistivity coefficient $\mu$. More precisely, we prove 
\begin{enumerate}
\item the $\mu^{-\frac13}$-amplification of the perturbed vorticity, namely, the size of the vorticity grows from $\|\omega_{\mathrm{in}}\|_{\mathcal{G}^{\lambda_{0}}}\lesssim \mu$ to $\|\omega_{\infty}\|_{\mathcal{G}^{\lambda'}}\lesssim \mu^{\frac23}$;
\item the polynomial decay of the perturbed current density, namely, $\left\|j_{\neq}\right\|_{L^2}\lesssim \frac{c_0 }{\langle t\rangle^2 }\min\left\{\mu^{-\fr13},\langle t \rangle\right\}$;
\item and the damping for the perturbed velocity and magnetic field, namely, 
\[
\left\|(u^1_{\neq},b^1_{\neq})\right\|_{L^2}\lesssim \frac{c_0\mu }{\langle t\rangle }\min\left\{\mu^{-\fr13},\langle t \rangle\right\}, \quad 
\left\|(u^2,b^2)\right\|_{L^2}\lesssim \frac{c_0\mu }{\langle t\rangle^2 }\min\left\{\mu^{-\fr13},\langle t \rangle\right\}. 
\]
\end{enumerate} 
We also confirm that the strong uniform magnetic field stabilizes the Euler-MHD system near Couette flow. 
\end{abstract}
\setcounter{tocdepth}{1}
{\small\tableofcontents}

\section{Introduction}
In this paper, we consider the 2D incompressible Euler-MHD equations with small resistivity coefficient $\mu>0$ on $\mathbb{T}\times\mathbb{R}$:
\be\label{MHD}
\begin{cases}
\partial_t\tl{u}+\tl{u}\cdot\nabla\tl{u}-\tl{b}\cdot\nabla\tl{b}=-\nabla\tl{p}\\
\partial_t\tl{b}+\tl{u}\cdot\nabla\tl{b}-\tl{b}\cdot\nabla\tl{u}=\mu\Dl\tl{b},\\
\mathrm{div}\, \tl{u}=\mathrm{div}\, \tl{b}=0.
\end{cases}
\ee
Here $\mathbb{T}$ is normalized to be $2\pi$-periodic. The vector-value functions $\tl{u}$ and $\tl{b}$ are the velocity and the magnetic field respectively. 

A fundamental problem in plasma physics is understanding the long-time asymptotic stability of an equilibrium solution $(u_s, b_s)$. For the fluid equation (taking $\tl{b}=0$ in \eqref{MHD}),  the Couette flow $u_s=(y,0)$ is perhaps the simplest nontrivial stationary solution, which is asymptotically stable, proved  in \cite{BM15}. However, the magnetic field may destabilize the system even with shear flows (including Couette flow) that are asymptotically stable \cite{ChenMorrison1991}. For example, in \cite{HirotaTatsunoYoshida2005}, the authors consider the linearized ideal MHD equation around Couette flow $(k_fy, 0)$ and linear magnetic field $(k_my,0)$, and predicted that if $|k_f|<|k_m|$, then the magnetic island appears in the final state, namely the linear asymptotic stability fails, and if $|k_m|<|k_f|$, then linear damping holds and the magnetic island will be destructed. In the first case, the equilibrium is linear asymptotically unstable which was proved in \cite{ZhaiZhangZhao2021} mathematically rigorously, while in the second case, the equilibrium is the linear asymptotically stable, see \cite{RenWeiZhang2021} for the rigorously mathematical proof. The presence of the magnetic field has a dual role in the stability of the shear flow. The magnetic field exerts tension on the fluid, which usually acts as a restoring force on a disturbance. Also, it is observed in physical experiments and numerical simulations that a strong background magnetic field is a stabilizing effect on the electrically conducting fluids (see, e.g., \cite{AlemanyMoreauSulemFrisch1979, Alexakis2011, GalletBerhanuMordant2009}). 
In this paper, we study the asymptotic stability of the Couette flow in some strong uniform magnetic field, namely, we are interested in the following equilibrium:
\be\label{eq:equilibrium}
u_s=(y,0)^\top, \quad b_s=(\al, 0)^\top
\ee
with $|\al|$ large enough. 

It is natural to introduce the perturbations $u$ and $b$ defined by $\tl{u}=u+u_s$ and $\tl{b}=b+b_s.$
Then the equations of $(u, b)$ take the following form:
\be\label{pMHD}
\begin{cases}
\partial_t{u}+y\partial_xu+\left(\begin{array}{c}u^2\\0\end{array}\right)-\al\partial_xb+{u}\cdot\nabla{u}-{b}\cdot\nabla{b}=-\nabla{p},\\
\partial_t{b}+y\partial_xb-\left(\begin{array}{c}b^2\\0\end{array}\right)-\al\partial_xu+{u}\cdot\nabla{b}-{b}\cdot\nabla{u}=\mu\Dl{b},\\
\mathrm{div}{u}=\mathrm{div}{b}=0,
\end{cases}
\ee

We also introduce the system of the vorticity $\om=\partial_xu^2-\partial_yu^1$, current density $j=\partial_xb^2-\partial_yb^1$, the stream function $\psi=\Dl^{-1}\om$, and the magnetic potential function $\phi=\Dl^{-1}j$: 
\be\label{omj}
\begin{cases}
\partial_t\om+y\partial_x\om-\al\partial_xj=\mathrm{NL}[\om],\\
\partial_tj+y\partial_xj-\al\partial_x\om-\mu\Dl j+2\partial_{xy}\phi=\mathrm{NL}[j],\\
u=(-\partial_y\psi,\partial_x\psi)^\top=\nabla^\bot\psi,\quad b=(-\partial_y\phi,\partial_x\phi)^\top=\nabla^\bot\phi\\
\Dl\psi=\om,\quad \Dl\phi=j
\end{cases}
\ee
where
\beqno
\mathrm{NL}[\om]&=&-u\cdot\nabla\om+b\cdot\nabla j,\\
\mathrm{NL}[j]&=&-u\cdot\nabla j+b\cdot\nabla\om-2\partial_{xy}\phi(-\om+2\partial_{xx}\psi)+2\partial_{xy}\psi(-j+2\partial_{xx}\phi).
\eeqno
In this paper, we study the system \eqref{omj} and prove the global asymptotic stability. Moreover, we confirm the stabilizing effect of the background uniform magnetic field. Our main result states as follows:
\begin{theorem}\label{thm}
For all  $1\geq s>1/2$ and $\lambda_{0}>\lm'>0$, there exist a universal constant $\al_0>0$ and  a sufficiently small constant $c_0=c_0(\lambda_{0},\lambda',s, \al)>0$,  such that if $|\al|\ge \al_0$ and the initial data $(u_{\mathrm{in}},b_{\mathrm{in}},\omega_{\mathrm{in}},j_{\mathrm{in}})$ satisfies $\int_{\mathbb{R}}|y\omega_{\mathrm{in}}(x,y)|dy+\int_{\mathbb{R}}|yj_{\mathrm{in}}(x,y)|dy<\infty$, and 
\begin{align}\label{initial}
\|(u_{\mathrm{in}},b_{\mathrm{in}})\|_{L^2}^2+\sum_{k\in\mathbb{Z}}\int (|\widehat{\omega_{\mathrm{in}}}|^2+|\widehat{j_{\mathrm{in}}}|^2)e^{2\lambda_{0}|k,\eta|^{s}}d\eta \le \left(c_0\mu\right)^2,
\end{align}
then the solution $(u,b,\omega,j)$ to \eqref{omj} satisfies 
\begin{itemize}
\item (Scattering) There exists $\om_{\infty}\in \mathcal{G}^{\lambda'}$  such that 
\be\label{om-infty}
\|\om_{\infty}\|_{\mathcal{G}^{\lambda'}}\lesssim c_0\mu^{\frac23}, 
\ee
and for all $t\gtr \mu^{-\fr13}$, 
\begin{equation}\label{eq: Scattering} 
\left\|\omega(t,x+ty+\Xi(t,y),y)-\om_{\infty}(x,y)\right\|_{\mathcal{G}^{\lambda'}}\lesssim \frac{c_0\mu^{-\fr{1}{3}}}{\langle t\rangle},
\end{equation}
where $\Xi(t,y)$ is given explicitly by 
\begin{equation}\label{Xi}
\Xi(t,y)=\frac{1}{2\pi}\int_0^t\int_{\mathbb{T}}u^1(\tau,x,y)dxd\tau=u_{\infty}(y)t+O(c_0^2\mu^{\fr34}),
\end{equation}
with $u_{\infty}=-\fr{1}{2\pi}\partial_y\int_{\mathbb{T}}\Dl^{-1}\om_{\infty}(x,y)dx$, and 
\be\label{-u_infty}
\left\|\frac{1}{2\pi}\int_{\mathbb{T}}u^1(t,x,y)dx-u_\infty\right\|_{\mathcal{G}^{\lm'}}\les (c_0\mu)^2\fr{\mu^{-\fr23}}{\la t\ra^2}.
\ee
\item (Decay) It holds that
\begin{align}
\label{damping-u}&\left\|u^1-\frac{1}{2\pi}\int_{\mathbb{T}}u^1(t,x,y)dx\right\|_{L_{x,y}^2}
+\langle t\rangle\left\|u^2\right\|_{L_{x,y}^2}\les \frac{c_0\mu }{\langle t\rangle }\min\left\{\mu^{-\fr13},\la t \ra\right\},\\
\label{damping-b}&\left\|b^1-\frac{1}{2\pi}\int_{\mathbb{T}}b^1(t,x,y)dx\right\|_{L_{x,y}^2}
+\langle t\rangle\left\|b^2\right\|_{L_{x,y}^2}\les \min\left\{ \frac{c_0\mu }{\langle t\rangle }, \frac{c_0}{\langle t\rangle^3 }\right\}\min\left\{\mu^{-\fr13},\la t \ra\right\},\\
\label{damping-j}&\left\|j-\frac{1}{2\pi}\int_{\mathbb{T}}j(t,x,y)dx\right\|_{L_{x,y}^2}\les \min\left\{\frac{c_0 }{\langle t\rangle^2 }, c_0\mu\right\}\min\left\{\mu^{-\fr13},\la t \ra\right\}.
\end{align}
\end{itemize}

\end{theorem}
A few remarks are in order.
\begin{rem}\label{Rmk:1}
The factor $\min\left\{\mu^{-\fr13},\la t \ra\right\}$ on the right-hand side of \eqref{damping-u}--\eqref{damping-j} stemming from the use of the multiplier $m$ (see \eqref{m1}) exhibits the balance between the growth driven by the linear stretch term $2\partial_{xy}\phi$ in \eqref{omj} and the magnetic diffusion.
\end{rem}

\begin{rem}\label{Rmk:2}
The Gevrey-$2$ regularity seems optimal. In \cite{KZ23}, the authors study a linear toy model, which has a similar transient growth as that in the nonlinear system of this paper. They obtain the optimality of the required regularity for their toy model. 
\end{rem}

\begin{rem}\label{Rmk:3}
At the linear level, the strong background magnetic field stabilizes the system in the following sense: 
\begin{itemize}
\item If $\al=0$, then the transient growth of the non-zero mode of the perturbed current density will lead to the $\mu^{-\frac{2}{3}}$-amplification, see \eqref{eq: J-alpha=0}. 
\item If $|\al|$ is large enough, then the transient growth of the non-zero mode of the perturbed current density will lead to the $\mu^{-\frac{1}{3}}$-amplification. 
\end{itemize}

\end{rem}

\begin{rem}\label{Rmk:4}
The $\frac{1}{\langle t\rangle^2}$ decay of the non-zero mode of the perturbed current density is optimal. It can be easily deduced from the definition of the good unknown \eqref{def-f}. 
\end{rem}
\subsection{Known results and the stabilizing mechanism}
In this section, we present some known results related to our results and discuss the stabilization mechanisms. 
\subsubsection{Stability of shear flows} The study of the stability and instability of laminar shear flows goes back to the early experiments of Reynolds \cite{Rey83}. In \cite{Orr07}, Orr predicted the algebraic decay of the velocity for the inviscid fluids. This phenomenon is referred to as the {\em inviscid damping}, which is the hydrodynamic analog of Landau damping in the Vlasov equations \cite{Lan46, MV11}. It is challenging to prove the inviscid damping for general shear flows even at a linear level due to the presence of the nonlocal term.  We refer to \cite{Zi17, WZZ18, Jia20-1, Jia20-2} for the linear inviscid damping results for general {\em monotone} shear flows. The results for {\em non-monotone shear flows} such as the Poiseuille flow and the Kolmogorov flow can be found in \cite{WZZ19, WZZ20, IIJ23}. See \cite{BCV19, LMZZ22} and references therein for more related depletion results. At the nonlinear level, the scenario is much more complicated even for the Couette flow. In \cite{LinZeng11}, Lin and Zeng proved that the nonlinear inviscid damping fails for perturbations
around the Couette flow in $H^s$ with $s<\fr32$. The nonlinear inviscid damping was confirmed by Bedrossian and Masmoudi \cite{BM15} with perturbations in Gevrey class $m$ with $1\leq m<2$ on the domain $\mathbb{T}\times\mathbb{R}$. Then Deng and Masmoudi showed that the instability occurs when the perturbations are in Gevrey class $m$ with $m>2$. If the initial perturbation is compactly supported in the periodic finite channel  $\mathbb{T}\times[0,1]$, the nonlinear stability of the Couette flow was obtained by Ionescu and Jia \cite{IJ20-1}. Furthermore, one can find the nonlinear inviscid damping results for stable monotone shear flows in \cite{IJ20-2} and  \cite{MZ20}. Recently, In \cite{CWZZ23, Zhao2023}, the authors prove the nonlinear stability of shear flows for 2D  inhomogeneous ideal fluids.

For the viscous fluids around the shear flows, another stability mechanism known
as {\em enhanced dissipation} should be taken into consideration. This dynamic phenomenon is manifested as an acceleration of decay due to the transmissions of enstrophy to very high frequencies via advection.
It is worth pointing out that the enhanced dissipation is closely related to the stability threshold problem. That is, {\em how small should initial perturbations be in terms of the viscous coefficient to ensure nonlinear stability}? 
 For the 2D  the domain $\mathbb{T}\times\mathbb{R}$, a series of results can be found in \cite{BedrossianMasmoudiVicol2016, BVW18, LiMasmoudiZhao2022, MZ19}. For the periodic finite channels, a delicate resolvent estimate method was developed in \cite{CLWZ20, CWZ20}. We also refer to \cite{BGM17, BGM20, BGM22,  WZ21} for the stability threshold of the 3D Couette flow.

\subsubsection{Stabilizing effect of the uniform magnetic field}
Let us first review some global stability results on MHD equations in the absence of shear flows. In 1988, Bardos et al. \cite{BardosSulemSulem1988} studied the behavior of an ideal conducting fluid subject to a strong magnetic field. They proved that for a strong background magnetic field and small enough initial displacement, the solution to the ideal MHD equations remains smooth for all time and that the dynamics may reduce to linear Alfv\'en waves. The nonlinear stability of Alfv\'en waves even holds for viscous MHD equations (with nonzero viscosity and resistivity coefficient), see \cite{CaiLei2018, HeXuYu2018, WeiZhang2017}.  We also refer to \cite{AZ17, CMRR16, FMRR14, DZ18, RWXZ14, Zhang16} for the well-posedness and large time behavior of solutions to the non-resistive Navier-Stokes-MHD equations ($\nu>0, \mu=0$). For the Euler-MHD system with the nonzero resistivity coefficient, the global well-posedness results are much fewer, see \cite{WZ20, ZZ18}.

In the presence of the shear flows in the magnetic fields, the dynamics of the fluids become more complicated as mentioned above. For the 2D inhomogeneous equilibrium $u_s=(u(y),0)^\top$, and $b_s=(b(y),0)^\top$, there are few rigorous mathematical results on the (in)stability of $(u_s, b_s)$. We refer to \cite{LMZZ22, RenWeiZhang2021, RZ17, ZhaiZhangZhao2021} for long time behaviors (including inviscid damping, the generation/destruction of the magnetic island, and the depletion phenomenon) of solutions to the linearized ideal MHD equations around $(u_s, b_s)$. 
To the best of our knowledge, so far the only nonlinear stability result was given by Liss \cite{Liss20} for the Couette flow $(y,0,0)^\top$ in a strong and suitably oriented background magnetic field on $\mathbb{T}\times\mathbb{R}\times\mathbb{T}$.

\subsection{Ideas of the proof}
In this section, we highlight the main difficulties and new ideas in this paper.

First, from Remark \ref{Rmk:3} we can see that a sufficiently strong magnetic field has stabilizing effect on the Couette flow. The usual way of obtaining such a good stabilizing effect is to introduce new unknowns $z_{\pm}=\tilde{u}\pm \tilde{b}$ which is the so-called Els\"asser variables. To take advantage of these good variables, one may require some `symmetry' of both equations, for example, both momentum equation and magnetic equation have the same diffusion \cite{CaiLei2018, HeXuYu2018, Liss20} or no diffusion \cite{BardosSulemSulem1988}, or a slight difference between the viscosity and the resistivity coefficient \cite{WeiZhang2017}. However, in this paper, there is a significant difference between the momentum equation and the magnetic equation. Therefore, we do not use the Els\"asser variables. Instead, we carefully study the linearized equation \eqref{L-Omj} and introduce an energy functional of the form:
\be\label{linear-functional}
\left\|\mathcal{M}{\Om}^*\right\|^2_{H^\sigma}+\left\|\mathcal{M}{J}^*\right\|^2_{H^\sigma}-\fr{2}{\al}\left\la\fr{\partial_t(m^\fr12)}{m^\fr12}\partial_X^{-1}\la\nabla\ra^\sigma\mathcal{M}{\Om}^*, \la\nabla\ra^\sigma\mathcal{M}J^*\right\ra, \quad \sigma\ge0.
\ee
The definitions of the multipliers $m$ and $\mathcal{M}$ can be found in Section \ref{sec-linear}.
A similar approach can also be found in the study of the stability problem for the compressible Navier-Stokes equations \cite{AntonelliDolceMarcati2021, ZengZhangZi2022}, and the Navier-Stokes(Euler)-Boussinesq equation \cite{BBCD21, ZhaiZhao22}. This Lyapunov-type functional fails to work when there is no background magnetic field ($\al=0$). Moreover, the linear energy functional \eqref{linear-functional} is well adapted to the nonlinear system, so we also use the same time-dependent multiplier $\mathcal{M}$ constructed in the linear estimate to capture the growth from linear terms in the nonlinear system. See \eqref{En} for the definition of the Lyapunov-type functional for the nonlinear system \eqref{OMJ}.

To tackle the nonlinear problem \eqref{omj}, we make a suitable nonlinear change of variables. The basic idea of this change of variables is to get rid of the non-decaying zero mode of the velocity. There are two different types of change of coordinates: one is the inviscid one for the Euler equation, see \cite{BM15}, and the other one is the viscous one for the Navier-Stokes equation, see \cite{BedrossianMasmoudiVicol2016, LiMasmoudiZhao2022}. The coordinate systems were chosen in a very natural way in both cases. Nevertheless, in this paper, the diffusion term only appears in the current density equation, and the vorticity equation is a transport equation. 
We use the inviscid change of coordinates \eqref{nl-coor} in our problem as that in \cite{BM15}, which also seems necessary. Then we arrive at a new system \eqref{OMJ}. 
We would like to emphasize that the change of coordinates brings us a new problematic term $(Y'')_{\text{high}}\partial^{L}_{Y}(J_{\neq})_{\text{low}}$ which is the high-low interaction in the dissipation error term \eqref{D-HL}. A similar problem also appears in the Navier-Stokes-Boussinesq problem \cite{MasmoudiSaid-HouariZhao2022}, the authors there use the dissipation effect of the zero mode of the velocity to absorb the derivative loss in $Y''$. However, in our problem, there is no dissipation effect for the zero mode of the velocity. The new idea is to use time decay to absorb the derivative loss, see \eqref{e-ReD} for more details. This requires a better understanding of the long-time behavior of the non-zero modes of the current density $J_{\neq}$. 
The key estimate is to get $\mu^{-1}\langle t\rangle^{-2}$ decay of the current density $J_{\neq}$. 
It is based on an observation of a new {\em good unknown} 
\be\label{def-f}
f=\mu\Dl j+(\al+b_0^1)\partial_x\om.
\ee
The good unknown $f$ can be regarded as the cancellation of the growth between $(\alpha+b_0^1)\partial_{x}\omega_{\neq}$ and $\mu\Delta j_{\neq}$. A simple heuristic behind it is that one may think $j_{\neq}\to 0$ and $\partial_tj_{\neq}\to 0$ as $t\to \infty$, then from \eqref{omj}, one may expect $f$ is under control. In fact,
from \eqref{omj} and the equation of $b_0^1$ (see second equation of \eqref{ub01}), one derives the equation of $f$ as follows:
\beq\label{e-f}
\partial_tf+y\partial_xf-\mu\Dl f+4\mu\partial_{xy}j-\al^2\partial_{xx}j=\mathrm{NL}[f],
\eeq
\beq\label{NLf}
\nn\mathrm{NL}[f]&=&-u\cdot\nabla f+\al b^1_0\partial_{xx}j+\left[\mu\partial_{yy}b^1_0-\left(u_\ne\cdot\nabla b^1_\ne\right)_0+\left(b_\ne\cdot\nabla u^1_\ne\right)_0\right]\partial_x\om\\
\nn&&-[b_0^1, u_\ne\cdot\nabla]\partial_x\om-(\al+b_0^1)\partial_xu\cdot\nabla\om+(\al+b_0^1)\partial_x\left(b\cdot\nabla j\right)\\
\nn&&-\mu\nabla^\bot\om\cdot \nabla j-2\mu\partial_xu\cdot\nabla\partial_{x}j-2\mu\partial_yu\cdot\nabla\partial_{y}j+\mu\Dl(\nabla^\bot\phi_\ne\cdot\nabla\om)\\
&&-2\mu\Dl\left(\partial_{xy}\phi(-\om+2\partial_{xx}\psi)\right)+2\mu\Dl\left(\partial_{xy}\psi(-j+2\partial_{xx}\phi)\right).
\eeq
Under the change of coordinates \eqref{nl-coor}, the good unknown $f$ can be rewritten as 
\beno
F=\mu\Delta_t J+(\alpha+B_0^1)\partial_{X}\Omega.
\eeno
 After proving the boundedness of the good unknown, we can obtain the polynomial decay of $J_{\ne}$ by using a standard loss elliptic estimate. 

To capture the growth of the full system, we use the multiplier $\mathcal{M}$ to control the growth from linear terms and use the well-designed time-depend multiplier $\mathcal{J}$ from \cite{BM15} to control the growth from nonlinear interactions. We will use some basic properties of $\mathcal{J}$ and some similar estimates in \cite{BM15} as a black box in the proof to reduce the length of our paper. Even so, different from the Euler equation investigated in \cite{BM15}, the use of the nonlinear coordinate transform \eqref{nl-coor} in this paper cannot get rid of the zero mode $b_0^1$ of the magnetic field. Accordingly,  we give a new commutator estimate on $\mathcal{J}$ in Lemma \ref{lem-com-J}  to deal with the interactions between the zero mode $b_0^1$ and the non-zero modes $\partial_x\om$ and $\partial_xb$ (see Section \ref{sec-zero-mode}). In addition, to adapt  the the linear multiplier $\mathcal{M}$ to the nonlinear problem, the delicate commutator estimates on $\mathcal{M}$   (see Lemmas \ref{lem-com-sqm}, \ref{lem-com-mu} and Corollary \ref{coro-com-1})are inevitable in the treatment of the transport nonlinearities in Section \ref{sec-transport}.

\subsection{Notations and conventions}
A convention we generally use is to denote the discrete $x$
(or $X$) frequencies as subscripts. By convention, we always use Greek letters such as $\eta$ and $\xi$ to denote frequencies in the $y$ or $Y$ direction and lowercase Latin characters commonly used
as indices such as $k, l$ to denote frequencies in the $x$ or $X$ direction (which are discrete). Another convention we use is to denote $M, N$ as dyadic integers $M, N\in\mathbb{D}$ where
\beno
\mathbb{D}=\left\{\fr12, 1, 2, \cdots, 2^k, \cdots \right\}.
\eeno
This will be used when defining the Littlewood-Paley decomposition and paraproducts, see Section \ref{subsec-LP} in the Appendix. A few notations are listed below:
\begin{enumerate}
\item We use the notation $f\les g$ when there exists a constant $C>0$ independent of  the parameters of interest such that $f\le C g$ (the  notation $g\gt f$ is defined analogously).  Similarly, we
use the notation $f\approx g$  when there exists $C>0$  such that $C^{-1}g\le f\le Cg$.

\item The Fourier transform $\hat{f}(k,\eta)$ of a function $f(x,y)$ is defined by
\beno
\hat{f}(k,\eta)=\fr{1}{2\pi}\int_\mathbb{T}\int_{\mathbb{R}}f(x,y)e^{-i(kx+\eta y)}dxdy.
\eeno
Then 
$$
f(x,y)=\sum_{k\in\mathbb{Z}}\int_\mathbb{R}\hat{f}(k,\eta)e^{i\eta y}d\eta e^{ikx}.
$$
We also use $\mathcal{F}[f](k,\eta)$ to denote the Fourier transform of $f$. Given a function $m\in L^\infty$, we define the Fourier multiplier $m(\nabla) f$ by
\[
\mathcal{F}[m(\nabla)f](k,\eta)=m\big((ik,i\eta)\big)\hat{f}(k,\eta).
\]

\item For a vector $x=(x_1, x_2)$, we use $|x|$ to denote the $\ell^1$ norm of $x$. We denote 
$\la x\ra=\left(1+|x|^2\right)^\fr12.$
In particular, for $(k,\eta)\in\mathbb{Z}\times\mathbb{R}$, we use the shorthand
$|k,\eta|=|(k,\eta)|, \la k,\eta\ra=\la (k,\eta)\ra$.

\item We use the notation $f_0$ to denote the $x$ (or $X$) average of a function $f$:
$f_0=\fr{1}{2\pi}\int_{\mathbb{T}}f(x,y)dx$. The non-zero mode of $f$ is denoted by $f_{\neq}=f-f_0$.

\item The Gevrey-$\fr{1}{s}$ norm with Sobolev correction is defined by
\[
\|f\|_{\mathcal{G}^{\lm,\sigma;s}}^2=\sum_{k\in\mathbb{Z}}\int_{\mathbb{R}}|\hat{f}_k(\eta)|^2e^{2\lm|k,\eta|^s}\la k,\eta\ra^{2\sigma}d\eta.
\]
We use  the shorthand $\mathcal{G}^{\lm, \sigma}=\mathcal{G}^{\lm, \sigma;s}$ frequently throughout the paper.
\item Given a Banach space $X$,  we define the Banach space $L^p(a,b; X)$  for $1\le p\le \infty$ by the norm
\[
\|\cdot\|_{L^p(a, b; X)}=\big\|\|\cdot\|_{X}\big\|_{L^p(a, b)}.
\]
We often use the shorthand notation $\|\cdot\|_{L^p X}=\|\cdot\|_{L^p(a, b; X)}$ when the time interval is clear from context.
\item For any two real functions $f, g\in L^2(\mathbb{T}\times\mathbb{R})$, we denote the inner product of $f$ and $g$ in $ L^2(\mathbb{T}\times\mathbb{R})$ by $\la f, g\ra$, i.e.,
\[
\la f, g\ra=\int_{\mathbb{T}\times\mathbb{R}}fgdXdY.
\]

\end{enumerate}

\section{Change of  coordinates and the linearized dynamics}
\subsection{Nonlinear coordinate transform} In this paper, we will use the nonlinear coordinate transform introduced first in \cite{BM15}. Define
\be\label{nl-coor}
X=x-tY, \quad Y=y+\fr{1}{t}\int_0^tu^1_0(s, y)ds.
\ee
Denote
\begin{gather}
\nn\Om(t, X, Y)=\om(t, x, y),\ J(t, X, Y)=j(t, x,y),\ 
\Psi (t, X, Y)=\psi(t, x,y),\ \Phi(t, X, Y)=\phi(t, x, y),\\
\label{Y'}Y'(t,Y)=\partial_yY(t, y)=1-\fr{1}{t}\int_0^t\om_0(s,y)ds,\\
\label{Y''}Y''(t, Y)=\partial_{yy}Y(t, y)=-\fr{1}{t}\int_0^t\partial_y\om_0(s,y)ds,\\
\label{Y_t}\dot{Y}(t,Y)=\partial_tY(t, y)=\fr{1}{t}\left(u^1_0(t,y)-\fr{1}{t}\int_0^tu^1_0(s,y)ds\right).
\end{gather}
By the chain rule, we have
\be\label{chain}
Y''=Y'\partial_YY'.
\ee
Then for $g(t, x,y)=G(t, X,Y)$,
\beno
\nabla g(t,x,y)=(\partial_xg, \partial_yg)=(\partial_XG, Y'(\partial_Y-t\partial_X)G)=(\partial_X^tG,\partial_Y^tG)=\nabla_tG(t,X,Y),
\eeno
and
\beno
\Dl g(t,x, y)=\Dl_tG(t,X,Y),
\eeno
where
\beq\label{Dlt}
\nn\Dl_t&:=&\partial_{XX}+(Y')^2(\partial_Y-t\partial_X)^2+Y''(\partial_Y-t\partial_X)\\
&=&\Dl_L+\left((Y')^2-1\right)\partial_{YY}^L+Y''\partial_Y^L.
\eeq

In the new variables \eqref{nl-coor}, the system \eqref{omj} can be rewritten as
\be\label{OMJ}
\begin{cases}
\partial_t\Om-\al \partial_X J=\mathrm{NL}[\Om],\\[3mm]
\partial_t J-\al \partial_X\Om+2\partial^t_{XY}\Phi-\mu \Dl_t J=\mathrm{NL}[J],\\[3mm]
V=Y'\nabla^\bot\Psi_\ne+\begin{pmatrix} 0\\ \dot{Y} \end{pmatrix},\\[3mm]
\Dl_t\Psi=\Om,\quad \Dl_t\Phi=J,
\end{cases}
\ee
with
\be\label{NL}
\begin{split}
\mathrm{NL}[\Om]&=-V\cdot\nabla \Om+Y'\nabla^\bot\Phi\cdot \nabla J,\\
\mathrm{NL}[J]&=-V\cdot\nabla J+Y'\nabla^\bot\Phi\cdot \nabla \Om-2\partial_{XY}^t\Phi\left(-\Om+2\partial_{XX}\Psi\right)+2\partial_{XY}^t\Psi\left(-J+2\partial_{XX}\Phi\right).
\end{split}
\ee
Moreover, the equation \eqref{e-f} of the good unknown $f$ now takes the following form:
\beq\label{e-F}
\partial_tF-\mu\Dl_t F+4\mu\partial_{XY}^tJ-\al^2\partial_{XX}J=\mathrm{NL}[F],
\eeq
where
\beq\label{NLF}
\nn\mathrm{NL}[F]&=&-V\cdot\nabla F+\al B^1_0\partial_{XX}J+\left[\mu Y'\partial_Y\left(Y'\partial_YB^1_0\right)-Y'\left(\nabla^\bot\Psi_\ne\cdot\nabla B^1_\ne\right)_0+Y'\left(\nabla^{\bot}\Phi_\ne\cdot\nabla U^1_\ne\right)_0\right]\partial_X\Om\\
\nn&&-[B_0^1, Y'\nabla^\bot\Psi_\ne\cdot\nabla]\partial_X\Om-(\al+B_0^1)Y'\nabla^\bot\partial_X\Psi\cdot\nabla\Om+(\al+B_0^1)\partial_X\left(Y'\nabla\Phi\cdot\nabla J\right)\\
\nn&&-\mu Y'\nabla^\bot\Om\cdot \nabla J-2\mu Y'\nabla^\bot\partial_X\Psi\cdot\nabla\partial_{X}J-2\mu Y'\nabla^\bot\left(Y'\partial_Y^L\Psi\right)\cdot\nabla\left(Y'\partial_{Y}^LJ\right)\\
&&+\mu\Dl_t(Y'\nabla^\bot\Phi_\ne\cdot\nabla\Om)-2\mu\Dl_t\left(\partial_{XY}^t\Phi\left(-\Om+2\partial_{XX}\Psi\right)\right)+2\mu\Dl_t\left(\partial_{XY}^t\Psi\left(-J+2\partial_{XX}\Phi\right)\right).
\eeq

Next, we derive expressions for the coordinate system that are amenable to energy estimates. Note that the zero-mode $(u_0,b_0)$ solves
\be\label{ub01}
\begin{cases}
\partial_tu^1_0=-(u_{\ne}\cdot \nabla u^1_{\ne})_0+(b_{\ne}\cdot\nabla b^1_{\ne})_0,\\[2mm]
\partial_tb^1_0=\mu\partial_{yy}b^1_0-\left(u_\ne\cdot\nabla b^1_\ne\right)_0+\left(b_\ne\cdot\nabla u^1_\ne\right)_0.
\end{cases}
\ee
Starting from \eqref{Y'}, we see that $\displaystyle t(Y'-1)(t, Y(t, y))=-\int_0^t\om_0(s,y)ds$. Then taking the $t$ derivative of this equation, we are led to
\be\label{Y'-1}
\partial_t(Y'-1)+\dot{Y}\partial_Y(Y'-1)=H,
\ee
where
\be\label{def-H}
H:=-\fr{Y'-1+\Om_0}{t}.
\ee
From \eqref{Y'}  and \eqref{Y_t},  by the chain rule, one derives the relation between $H$ and $\dot{Y}$:
\be\label{HdotY}
H=Y'\partial_Y\dot{Y}.
\ee
To  derive the evolution of $H$, we first write the equation of $\Om_0$:
\beno
\partial_t\Om_0+\dot{Y}\partial_Y\Om_0+Y'\left(\nabla^\bot\Psi_\ne\cdot\nabla\Om_{\ne}\right)_0-Y'\left(\nabla^\bot\Phi_\ne\cdot \nabla J_{\ne}\right)_0=0.
\eeno
Combining this with \eqref{Y'-1} and the definition of $H$ in \eqref{def-H} yields
\beno
\partial_tH+\fr{2}{t}H+\dot{Y}\partial_YH=\fr{Y'}{t}\left(\nabla^\bot\Psi_\ne\cdot\nabla\Om_{\ne}\right)_0-\fr{Y'}{t}\left(\nabla^\bot\Phi_\ne\cdot \nabla J_{\ne}\right)_0.
\eeno
For the evolution of $\dot{Y}$,  we derive from \eqref{Y_t} that $\displaystyle t\dot{Y}(t, Y(t, y))=u^1_0(t,y)-\fr{1}{t}\int_0^tu^1_0(s,y)ds$. After taking the $t$ derivative of this equation, and using the equation of $u_0^1$ (see the first equation of \eqref{ub01}), we arrive at
\beno
\partial_t\dot{Y}+\fr{2}{t}\dot{Y}+\dot{Y}\partial_Y\dot{Y}=-\fr{Y'}{t}\left(\nabla^\bot\Psi_{\ne}\cdot \nabla U^1_{\ne}\right)_0+\fr{Y'}{t}\left(\nabla^\bot\Phi_{\ne}\cdot \nabla B^1_{\ne}\right)_0.
\eeno
In summary, the evolution of the coordinate system can be described by
\be\label{sys-coor}
\begin{cases}
\partial_t(Y'-1)+\dot{Y}\partial_Y(Y'-1)=H,\\[3mm]
\partial_t\dot{Y}+\fr{2}{t}\dot{Y}+\dot{Y}\partial_Y\dot{Y}=-\fr{Y'}{t}\left(\nabla^\bot\Psi_\ne\cdot \nabla U^1_{\ne}\right)_0+\fr{Y'}{t}\left(\nabla^\bot\Phi_\ne\cdot \nabla B^1_{\ne}\right)_0,\\[3mm]
\partial_tH+\fr{2}{t}H+\dot{Y}\partial_YH=\fr{Y'}{t}\left(\nabla^\bot\Psi_\ne\cdot\nabla\Om_{\ne}\right)_0-\fr{Y'}{t}\left(\nabla^\bot\Phi_\ne\cdot \nabla J_{\ne}\right)_0,\\[3mm]
U^1=-Y'\partial_Y^L\Psi,\quad B^1=-Y'\partial_Y^L\Phi.
\end{cases}
\ee

\subsection{The linearized dynamics}\label{sec-linear}
 Before handling the nonlinear system \eqref{OMJ}, we study the linearized behavior in detail.  To this end, let us denote the linear part of $\nabla_t$ and $\Dl_t$ by $\nabla_L$ and $\Dl_L$, respectively:
\be
\nabla_L=(\partial_X, \partial_Y-t\partial_X)=(\partial_X, \partial_Y^L),\quad \Dl_L=\partial_{XX}+(\partial_Y-t\partial_X)^2=\partial_{XX}+\partial_{YY}^L.
\ee
Then the linearized system of \eqref{OMJ} takes the form of
\be\label{L-Omj}
\begin{cases}
\partial_t\Om^*-\al\partial_XJ^*=0,\\
\partial_tJ^*-\al\partial_X\Om^*+2\partial_{XY}^L\Dl^{-1}_LJ^*-\mu\Dl_L J^*=0.
\end{cases}
\ee
In the case $\alpha=0$, a direct calculation gives that
\be\label{eq: J-alpha=0}
\|J^*_{\ne}\|_{H^{\sigma}}\lesssim \langle t\rangle^2e^{-c\mu t^3}\left\|J^{*}_{\mathrm{in}}\right\|_{H^{\sigma+2}}. 
\ee
The linear behavior of $(\Om^*,J^*)$ is non-trivial  when the background magnetic field $(\al, 0)^{\top}$ is strong enough ($|\al|$ sufficiently large). We will construct a Lyapunov functional which is compatible with the nonlinear problem via the Fourier multiplier method.
\subsubsection{The linear multipliers}
Define
\be\label{M1} 
\frac{\partial_t{\mathit{M}}^1}{\mathit{M}^1}=\frac{k^2}{\emph{k}^2+(\eta-kt)^2}\quad \mathrm{and}\quad  \mathit{M}^1_k(0,\eta)=1;
\ee
and
\be\label{Mmu}
\begin{cases}
\displaystyle\frac{\partial_t{\mathit{M}}^\mu}{\mathit{M^\mu}} = \frac{\mu^{\frac{1}{3}}}{1+ \mu^{\frac{2}{3}}(\fr{\eta}{k}-t)^2  }\quad \mathrm{with}\quad  \mathit{M}^\mu_{k}(0,  \eta)=1\quad \mathrm{for}\quad k\ne0;\\[3mm]
M^\mu=1,\quad\mathrm{for}\quad k=0.
\end{cases}
\ee 
The above two multipliers are frequently used to study the stability of Couette flow, see \cite{BGM17, BVW18} for instance. In this paper, $M^1$ is used to derive the weighted $L^2$ time integrability of the vorticity {\em independent of the magnetic diffusion coefficient $\mu$}, and $M^\mu$ is used to capture the {\em weak enhanced dissipation} of the current density, see \eqref{Lbound} below. By the definition of $M^{\mu}$ in \eqref{Mmu},  it is straightforward to deduce the following lemma, the proof of which can be found in \cite{BVW18}.
\begin{lem}\label{lem-enh}
If $k\ne0$, then there holds
\be\label{enh-dis}
 \mu(\eta-kt)^2+\fr{\partial_tM^\mu_k(t,  \eta)}{M^\mu_k(t,  \eta)}\ge\fr12\mu^\fr13.
\ee
\end{lem}
The possible growth of $J^*$ caused by the linear stretch term $2\partial_{XY}^LJ^*$ in the second equation of \eqref{L-Omj} may be balanced by the dissipative term $-\mu\Dl_L J^*$. There holds
\be\label{dis-domi}
\fr{2|k(\eta-kt)|}{k^2+\left(\eta-kt\right)^2}<\fr{\mu}{4}(\eta-kt)^2,\quad\mathrm{if}\quad \left|\fr{\eta}{k}-t\right|\ge 2\mu^{-\fr13}.
\ee
On the other hand, if $\left|\fr{\eta}{k}-t\right|< 2\mu^{-\fr13}$, we need an extra multiplier to mimic the growth of $J^*$  caused by $2\partial_{XY}^LJ^*$. The multiplier $m_k(t, \eta)$ below is a 2D version of $m(t,k,\eta,l)$ in \cite{BGM17}:
\begin{enumerate}
\item if $k = 0:  m_k(t,\eta) = 1$;

\item if $k \ne 0, \frac{\eta}{k} < -2\mu^{-\frac{1}{3}}: m_k(t,\eta) = 1$;

\item if $k \ne 0, -2\mu^{-\frac{1}{3}} < \frac{\eta}{k} < 0$, $m_k(t,\eta)=\left\{\begin{aligned}
&\frac{k^2+(\eta-kt)^2}{k^2+\eta^2}\quad \text{if}\quad 0 < t < \frac{\eta}{k}+2\mu^{-\frac{1}{3}}\\
&\frac{k^2+(2k\mu^{-\frac{1}{3}})^2}{k^2+\eta^2}\quad \text{if}\quad t > \frac{\eta}{k}+2\mu^{-\frac{1}{3}};
\end{aligned}\right.
$
\item if $k \ne 0, \frac{\eta}{k} > 0$, $m_k(t,\eta)=\left\{\begin{aligned}
&1\quad \text{if}\quad t < \frac{\eta}{k},\\
&\frac{k^2+(\eta-kt)^2}{k^2}\quad \text{if}\quad \frac{\eta}{k} < t < \frac{\eta}{k} + 2\mu^{-\frac{1}{3}}\\
&{1+(2\mu^{-\frac{1}{3}})^2}\quad \text{if}\quad t > \frac{\eta}{k} + 2\mu^{-\frac{1}{3}}.
\end{aligned}\right.$
\end{enumerate}
Notice, in particular, that $m$ is nondecreasing in $t$, and has the following lower and upper bounds: 
\be\label{m1}
1\le m_k(t,  \eta)\les \min\left\{\mu^{-\fr23}, \la t\ra^2\right\}.
\ee
Denote 
\begin{align*}
&D_1:=\left\{(t, k,\eta)\in [0,\infty)\times\mathbb{N}\times\mathbb{R}: k\ne0, -2\mu^{-\frac{1}{3}} < \frac{\eta}{k} < 0, t < \frac{\eta}{k}+2\mu^{-\frac{1}{3}}\right\},\\
&D_2:=\left\{(t, k,\eta)\in [0,\infty)\times\mathbb{N}\times\mathbb{R}: k\ne0, \frac{\eta}{k} > 0, \fr{\eta}{k}<t < \frac{\eta}{k}+2\mu^{-\frac{1}{3}}\right\}.
\end{align*}
Note that
\be\label{pt-m12}
\fr{\partial_t(m^\fr12)}{m^\fr12}=\fr12\fr{\partial_tm}{m}={\bf 1}_{D_1\cup D_2}\fr{-k({\eta}-kt)}{k^2+({\eta}-kt)^2}\ge0.
\ee
Thus, we have
\be\label{e215}
\fr{\partial_t(m^\fr12)}{m^\fr12}+\fr{2k(\eta-kt)}{k^2+(\eta-kt)^2}=-\fr{\partial_t(m^\fr12)}{m^\fr12}+{\bf 1}_{(D_1\cup D_2)^c}\fr{2k({\eta}-kt)}{k^2+({\eta}-kt)^2},
\ee
and
\be\label{ptpt-m12}
\left|\partial_t\left(\fr{\partial_t(m^\fr12)}{m^\fr12}\right)\right|={\bf 1}_{D_1\cup D_2}\left|\fr{k^2\left(k^2-(\eta-kt)^2\right)}{\left(k^2+(\eta-kt)^2\right)^2}\right|\le \fr{k^2}{k^2+(\eta-kt)^2}.
\ee

\subsubsection{The linear stability} To study the linear behavior of $(\Om^*, J^*)$, we introduce the multiplier
\be\label{M}
 \mathcal{M}_k(t, \eta)=\left(M^1M^\mu m^{\fr12}\right)^{-1}_k(t, \eta).
\ee
The aim of this section is to establish the following proposition.
\begin{prop}\label{prop-L}
Let $\sigma\in[0,\infty)$. Then there exists a sufficiently large $\al_0>0$, such that if $|\al|\ge \al_0$,  the solution $(\Om^*, J^*)$ to \eqref{L-Omj} satisfies for all $t\ge0$
\beq\label{Lbound}
\nn&&\|\mathcal{M}\Om^*(t)\|_{H^\sigma}^2+\|\mathcal{M}J^*(t)\|_{H^\sigma}^2+\sum_{\theta\in\{M^1,M^2\}}\sum_{g\in\{\Om^*, J^*\}}\left\|\sqrt{\fr{\partial_t\theta}{\theta}}\mathcal{M}g \right\|^2_{L^2H^\sigma}\\
&&+\mu\left\|\nabla_L\mathcal{M}J^*\right\|^2_{L^2H^\sigma}+\mu^\fr13\|\mathcal{M}J^*_{\ne}\|^2_{L^2H^\sigma}\les \|\om_{\mathrm{in}}\|_{H^\sigma}^2+\|j_{\mathrm{in}}\|_{H^\sigma}^2. 
\eeq
\end{prop}
\begin{proof}
Taking the Fourier transform of \eqref{L-Omj}, and using \eqref{e215}, it is not difficult to derive the evolution of $|\mathcal{M}\hat{\Om}^*|^2+|\mathcal{M}\hat{J}^*|^2$:
\beq\label{MOmJ}
\nn&&\fr12\partial_t\left(|\mathcal{M}\hat{\Om}^*|^2+|\mathcal{M}\hat{J}^*|^2\right)+\sum_{\theta\in\{M^1, M^{\mu}\}}\fr{\partial_t\theta}{\theta}\left(|\mathcal{M}\hat{\Om}^*|^2+|\mathcal{M}\hat{J}^*|^2\right)+\mu \left(k^2+(\eta-kt)^2\right)|\mathcal{M}\hat{J}^*|^2\\
\nn&=&-\fr{\partial_t(m^\fr12)}{m^\fr12}|\mathcal{M}\hat{\Om}^*|^2-\left(\fr{\partial_t(m^\fr12)}{m^\fr12}+\fr{2k(\eta-kt)}{k^2+(\eta-kt)^2}\right)|\mathcal{M}\hat{J}^*|^2\\
&=&\fr{\partial_t(m^\fr12)}{m^\fr12}\left(|\mathcal{M}\hat {J}^*|^2-|\mathcal{M}\hat{\Om}^*|^2\right)-{\bf 1}_{(D_1\cup D_2)^c}\fr{2k(\eta-kt)}{k^2+(\eta-kt)^2}|\mathcal{M}\hat{J}^*|^2.
\eeq
To deal with the first term on the right-hand side of  \eqref{MOmJ}, our strategy is to investigate the evolution of $(\mathcal{M}\hat{\Om}^*)(\mathcal{M}\bar{\hat{J}}^*)$. In fact, we infer   from \eqref{L-Omj} and \eqref{e215} that
\beq\label{inner-OmJ}
\partial_t\left((\mathcal{M}\hat{\Om}^*)(\mathcal{M}\bar{\hat{J}}^*)\right)
\nn&=&-{\bf 1}_{(D_1\cup D_2)^c}\fr{2k({\eta}-{k}t)}{k^2+({\eta}-{k}t)^2}(\mathcal{M}\hat{\Om}^*)(\mathcal{M}\bar{\hat{J}}^*)-2\sum_{\theta\in\{M^1, M^{\mu}\}}\fr{\partial_t\theta}{\theta}(\mathcal{M}\hat{\Om}^*)(\mathcal{M}\bar{\hat{J}}^*)\\
&&+i\al k\left(|\mathcal{M}{\hat{J}}|^2-|\mathcal{M}\hat{\Om}^*|^2\right)-\mu \left(k^2+(\eta-kt)^2\right)(\mathcal{M}\hat{\Om}^*)(\mathcal{M}\bar{\hat{J}}^*).
\eeq
Consequently,  for $k\ne0$, there holds
\beq\label{inner}
\fr{1}{i\al k}\partial_t\left(\fr{\partial_t(m^\fr12)}{m^\fr12}(\mathcal{M}\hat{\Om}^*)(\mathcal{M}\bar{\hat{J}}^*)\right)
\nn&=&-\fr{2}{i\al k}\fr{\partial_t(m^\fr12)}{m^\fr12}\sum_{\theta\in\{M^1, M^{\mu}\}}\fr{\partial_t\theta}{\theta}(\mathcal{M}\hat{\Om}^*)(\mathcal{M}\bar{\hat{J}}^*)\\
\nn&&+\fr{\partial_t(m^\fr12)}{m^\fr12}\left(|\mathcal{M}{\hat{J}}^*|^2-|\mathcal{M}\hat{\Om}^*|^2\right)-\fr{\mu}{i\al k}\fr{\partial_t(m^\fr12)}{m^\fr12}\left(k^2+(\eta-kt)^2\right)(\mathcal{M}\hat{\Om}^*)(\mathcal{M}\bar{\hat{J}}^*)\\
&&+\fr{1}{i\al k}\partial_t\left(\fr{\partial_t(m^\fr12)}{m^\fr12}\right)(\mathcal{M}\hat{\Om}^*)(\mathcal{M}\bar{\hat{J}}^*),
\eeq
where we have used \eqref{pt-m12} to deduce that
\beno
\fr{\partial_t(m^\fr12)}{m^\fr12}\cdot\fr{{\bf 1}_{(D_1\cup D_2)^c}2k({\eta}-kt)}{k^2+({\eta}-kt)^2}=0.
\eeno
Combining \eqref{MOmJ} with \eqref{inner} yields
\beq\label{223}
\nn&&\fr12\partial_t\left(|\mathcal{M}\hat{\Om}^*|^2+|\mathcal{M}\hat{J}^*|^2-\fr{2}{i\al k}\fr{\partial_t(m^\fr12)}{m^\fr12}(\mathcal{M}\hat{\Om}^*)(\mathcal{M}\bar{\hat{J}}^*)\right)\\
\nn&&+\sum_{\theta\in\{M^1, M^{\mu}\}}\fr{\partial_t\theta}{\theta}\left(|\mathcal{M}\hat{\Om}^*|^2+|\mathcal{M}\hat{J}^*|^2\right)+\mu \left(k^2+(\eta-kt)^2\right)|\mathcal{M}\hat{J}^*|^2\\
\nn&=&-{\bf 1}_{(D_1\cup D_2)^c}\fr{2k({\eta}-kt)}{k^2+({\eta}-kt)^2}|\mathcal{M}\hat{J}^*|^2+\fr{2}{i\al k}\fr{\partial_t(m^\fr12)}{m^\fr12}\sum_{\theta\in\{M^1, M^{\mu}\}}\fr{\partial_t\theta}{\theta}(\mathcal{M}\hat{\Om}^*)(\mathcal{M}\bar{\hat{J}}^*)\\
&&+\fr{\mu}{i\al k}\fr{\partial_t(m^\fr12)}{m^\fr12}\left(k^2+(\eta-kt)^2\right)(\mathcal{M}\hat{\Om}^*)(\mathcal{M}\bar{\hat{J}}^*)-\fr{1}{i\al k}\partial_t\left(\fr{\partial_t(m^\fr12)}{m^\fr12}\right)(\mathcal{M}\hat{\Om}^*)(\mathcal{M}\bar{\hat{J}}^*).
\eeq
It is worth pointing out that, for $(t, k, \eta)\in D_1\cup D_2$, there holds
\be\label{D1+D2}
0<t-\fr{\eta}{k}<2\mu^{-\fr13}.
\ee
Therefore, for $k\ne0$, we have
\be\label{low-M1}
\frac{k^2}{\emph{k}^2+(\eta-kt)^2}\ge\fr{{\bf 1}_{D_1\cup D_2}}{1+\left(t-\fr{\eta}{k}\right)^2}
\ge\fr{{\bf 1}_{D_1\cup D_2}}{1+4\mu^{-\fr23}}\ge \fr{\mu^{\fr23}}{5}{\bf 1}_{D_1\cup D_2}.
\ee
From \eqref{pt-m12}, \eqref{D1+D2} and \eqref{low-M1}, we find that
\beq\label{cro-dis}
\nn&&\fr{\mu}{i\al k}\fr{\partial_t(m^\fr12)}{m^\fr12}\left(k^2+(\eta-kt)^2\right)(\mathcal{M}\hat{\Om}^*)(\mathcal{M}\bar{\hat{J}}^*)\\
&\le&\fr{\mu}{4|\al|}\left(k^2+(\eta-kt)^2\right)|\mathcal{M}\hat{J}^*|^2+\fr{5}{|\al|}\fr{\partial_t{M}^1}{M^1}|\mathcal{M}\hat{\Om}^*|^2.
\eeq
Using \eqref{pt-m12} again, it is easy to see that
\be
\fr{2}{i\al k}\fr{\partial_t(m^\fr12)}{m^\fr12}\sum_{\theta\in\{M^1, M^{\mu}\}}\fr{\partial_t\theta}{\theta}(\mathcal{M}\hat{\Om}^*)(\mathcal{M}\bar{\hat{J}}^*)\le\fr{1}{2|\al|}\sum_{\theta\in\{M^1, M^{\mu}\}}\fr{\partial_t\theta}{\theta}\left(|\mathcal{M}\hat{\Om}^*|^2+|\mathcal{M}\hat{J}^*|^2\right).
\ee
In view of \eqref{ptpt-m12}, we are led to
\beq
\fr{1}{i\al k}\partial_t\left(\fr{\partial_t(m^\fr12)}{m^\fr12}\right)(\mathcal{M}\hat{\Om}^*)(\mathcal{M}\bar{\hat{J}}^*)\le \fr{1}{2|\al|}\fr{\partial_tM^1}{M^1}\left(|\mathcal{M}\hat{\Om}^*|^2+|\mathcal{M}\hat{J}^*|^2\right).
\eeq
For $(t,k,\eta)\in \left(D_1\cup D_2\right)^c$, either $\fr{-2k(\eta-kt)}{k^2+(\eta-kt)^2}<0$ or \eqref{dis-domi} holds. Then
\be\label{LSerr}
-{\bf 1}_{(D_1\cup D_2)^c}\fr{2k({\eta}-kt)}{k^2+({\eta}-kt)^2}|\mathcal{M}\hat{J}^*|^2\le\fr{\mu}{4}\left(k^2+(\eta-kt)^2\right)|\mathcal{M}\hat{J}^*|^2.
\ee
Substituting \eqref{cro-dis}--\eqref{LSerr} into \eqref{223}, and using \eqref{enh-dis}, for $|\al|$ sufficiently large, say $|\al|\ge8$, we obtain
\beq\label{231}
\nn&&\partial_t\left(|\mathcal{M}\hat{\Om}^*|^2+|\mathcal{M}\hat{J}^*|^2-\fr{2}{i\al k}\fr{\partial_t(m^\fr12)}{m^\fr12}(\mathcal{M}\hat{\Om}^*)(\mathcal{M}\bar{\hat{J}}^*)\right)\\
\nn&&+\fr12\sum_{\theta\in\{M^1, M^{\mu}\}}\fr{\partial_t\theta}{\theta}\left(|\mathcal{M}\hat{\Om}^*|^2+|\mathcal{M}\hat{J}^*|^2\right)+\mu \left(k^2+(\eta-kt)^2\right)|\mathcal{M}\hat{J}^*|^2+\fr{\mu^\fr13}{2}|\mathcal{M}\hat {J}^*_\ne|^2\le0,
\eeq
and
\be\label{approx}
|\mathcal{M}\hat{\Om}^*|^2+|\mathcal{M}\hat{J}^*|^2-\fr{2}{i\al k}\fr{\partial_t(m^\fr12)}{m^\fr12}(\mathcal{M}\hat{\Om}^*)(\mathcal{M}\bar{\hat{J}}^*)\approx|\mathcal{M}\hat{\Om}^*|^2+|\mathcal{M}\hat{J}^*|^2.
\ee
 Then \eqref{Lbound} follows from \eqref{231} immediately.
\end{proof}

\section{Outline of the proof}
To begin with, let us define the time-dependent multipliers that are used in the energy estimates of this paper:
\begin{gather}
\label{A}    A_k(t,\eta)=e^{\lm(t)|k,\eta|^s}\la k,\eta\ra^\sigma \mathcal{M}_k(t,\eta)\mathcal{J}_k(t,\eta),\\
\label{tlA}  \tl A_k(t,\eta)=e^{\lm(t)|k,\eta|^s}\la k,\eta\ra^\sigma \mathcal{M}_k(t,\eta)\tl{\mathcal{J}}_k(t,\eta),
\end{gather}
where
\be\label{def-J}
\mathcal{J}_k(t,\eta)=\fr{e^{2r\la\eta\ra^\fr12}}{w_k(t,\eta)}+e^{2r |k|^\fr12},\quad\mathrm{and}\quad \tl{\mathcal{J}}_k(t,\eta)=\fr{e^{2r\la\eta\ra^\fr12}}{w_k(t,\eta)}.
\ee 
The weighted function $w_k(t,\eta)$ is designed to mimic the possible nonlinear growth driven by hydrodynamics echoes as in \cite{BM15}. We sketch the construction of $w_k(t,\eta)$  in Section \ref{subsec-con-w} of the Appendix. The positive constant  $r=4(1+\mathrm{C}\kappa)$, with $\mathrm{C}$ and $\kappa$ fixed also in Section \ref{subsec-con-w}. The function $\lm(t)$ is  a radius of the Gevrey-$\fr{1}{s}$ regularity and will be chosen  also  as that in \cite{BM15}:
\beno
\lm(t)=\fr34\lm_0+\fr14\lm', \quad t\le1,
\eeno
and
\be\label{dot-lam}
\dot{\lm}(t)=\fr{d}{dt}\lm(t)=-\fr{\dl_\lm}{\la t\ra^{2\tl{q}}}\left(1+\lm(t)\right),\quad t>1,
\ee
where $\dl_\lm\approx\lm_0-\lm'$ is a small parameter that ensures $\lm(t)>\fr{\lm_0}{2}+\fr{\lm'}{2}$ and $\fr12<\tl{q}\le\fr{s}{8}+\fr{7}{16}$ is a parameter chosen by the proof. Clearly, $A\les \tl A$ if $|k|\le|\eta|$. 

To control the coefficients $Y'$ and $Y''$, we will also use a multiplier $A^R(t,\eta)$ that is slightly stronger  than $A_0(t,\eta)$:
\begin{align*}
A^R(t,\eta)=e^{\lm(t)|\eta|^s}\la\eta\ra^\sigma\mathcal{J}^R(t,\eta), \quad\mathrm{with}\quad \mathcal{J}^R(t,\eta)=\fr{e^{2r\la\eta\ra^\fr12}}{w_R(t,\eta)}.
\end{align*}
In addition, we use the following multiplier
\be\label{varA}
\mathcal{A}_k(t, \eta)=\fr{e^{\lm |k,\eta|^s}\la k,\eta\ra^{\sigma-6}}{M^{\mu}_k(t,\eta)}.
\ee
to bound the good unknown $F$ with lower regularity.
\subsection{The bootstrap hypotheses}
The main energy functional in this paper is defined by
\be
\mathcal{E}(t)=E(t)+E_Y(t),
\ee
where
\be\label{En}
{E}(t)=\|A\Om(t)\|^2_{L^2}+\|AJ(t)\|^2_{L^2}-\fr{2}{\al }\left\la\fr{\partial_t(m^\fr12)}{m^{\fr12}}\partial_X^{-1}A{\Om}_{\ne} , AJ_{\ne} \right\ra,
\ee
and
\be
E_Y(t)=\la t\ra^{2+2s}\left\|\fr{A}{\la\partial_Y\ra^s}H(t) \right\|^2_{L^2}+\la t\ra^{4-K_D\mu^{-\fr23}\eps}\left\|\dot{Y}(t)\right\|^2_{\mathcal{G}^{\lm(t),\sigma-6}}+\fr{1}{K_Y}\left\|A^R\left(Y'-1\right)(t)\right\|^2_{L^2},
\ee
with positive constants $K_Y$ and $K_D$ depending  on $s, \lm, \lm'$ and $\al$ fixed  by  the proof.

The local well-posedness theory for the 2D Euler-MHD equations \eqref{MHD} in Gevrey spaces is analogous to that of Euler equations given in \cite{BM15}. We just state the result and the proof is omitted.
\begin{lem}
For $\eps>0, s\in(\fr12,1)$ and $\lm_{0}>\lm'>0$, there exists an $\eps'>0$, such that if $\|(u_{\mathrm{in}},b_{\mathrm{in}})\|_{L^2}+\|(\om_{\mathrm{in}},j_{\mathrm{in}})\|_{\mathcal{G}^{\lm_{0}}}<\eps'$, then $\mathcal{E}(1)<\eps^2$, and $\sup_{t\in[0,1]}\|Y'(t)-1\|_{L^\infty}<\fr12$.
\end{lem}
The main theorem of this paper is going to be proven by a continuity argument. We define
the following controls referred to in the sequel as the bootstrap hypotheses for $t\ge1$.  Throughout this paper, we denote by  $\eps$ the small positive constant which may depend on $\mu$. The explicit dependence on $\mu$ will be determined by the proof, and concluded in Proposition \ref{prop-main}, see \eqref{initial-1} below.

\noindent{\bf Main system:}
\be\label{hypo-main}
E(t)+\int_1^t\sum_{\theta\in\{M^1, M^\mu\}}\sum_{g\in\{\Om, J\}}\mathrm{CK}_{\theta, g}+\sum_{g\in\{\Om, J\}}\mathrm{CK}_{\lambda, g}+\mu\|\nabla_LAJ\|_{L^2}^2dt'\le4\eps^2,
\ee
where
\beno
\mathrm{CK}_{\lm, g}=-\dot{\lm}(t)\left\| |\nabla|^{\fr{s}{2}}Ag\right\|^2_{L^2},\quad \mathrm{CK}_{w,g}=\left\|\sqrt{\fr{\partial_tw}{w}}\sqrt{\tl{A}A}g\right\|^2_{L^2}, \quad \mathrm{CK}_{\theta, g}=\left\| \sqrt{\fr{\partial_t\theta}{\theta}}Ag\right\|^2_{L^2}.
\eeno
\noindent{\bf Coordinate system:} 
\be\label{coor1}
\left\|A^R(Y'-1)(t)\right\|_{L^2}^2+\fr12\int_1^t\mathrm{CK}_{w}^R(t')+\mathrm{CK}_{\lm}^R(t')dt'\le 4K_Y\eps^2,
\ee

\be\label{coor2}
\la t\ra^{2+2s}\left\|\fr{A}{\la \partial_Y\ra^s}H(t) \right\|^2_{L^2}+\fr12\int_1^t\mathrm{CK}_{w}^{Y,2}(t')+\mathrm{CK}_{\lm}^{Y,2}(t')dt'\le4\eps^2,
\ee

\be\label{coor3}
\la t\ra^{4-K_D\mu^{-\fr23}\eps}\left\|\dot{Y}\right\|^2_{\mathcal{G}^{\lm(t),\sigma-6}}\le 4\eps^2,
\ee
where 
\beno
\mathrm{CK}_{w}^R=\left\|\sqrt{\fr{\partial_tw_R}{w_R}}A^R(Y'-1) \right\|_{L^2}^2,\quad \mathrm{CK}_{\lm}^R=-\dot{\lm}\left\||\partial_Y|^{\fr{s}{2}}A^R(Y'-1)\right\|_{L^2}^2,
\eeno
and
\beno
\mathrm{CK}_{w}^{Y,2}(t)=\la t\ra^{2+2s}\left\|\sqrt{\fr{\partial_tw}{w}}\fr{A}{\la \partial_Y\ra^s}H(t) \right\|_{L^2}^2,\quad \mathrm{CK}_{\lm}^{Y,2}(t)=\la t\ra^{2+2s}\left(-\dot{\lm}(t)\right)\left\| |\partial_Y|^{\fr{s}{2}}\fr{A}{\la\partial_Y\ra^s}H(t)\right\|_{L^2}^2.
\eeno
\noindent{\bf The zero mode $u_0^1$ and $b_0^1$:}
\be\label{hy-ub0}
\|u^1_0(t)\|^2_{L^2}+\|b^1_0(t)\|^2_{L^2}+\mu\|\partial_yb_0^1\|_{L^2L^2}^2\le 4\eps^2.
\ee
\noindent{\bf The good unknown $F$:}
\be\label{hy-F}
\left\|\mathcal{A}F(t)\right\|_{L^2}^2+\mu\left\|\nabla_L\mathcal{A}F\right\|_{L^2L^2}^2+\int_1^t\left(-\dot{\lm}(t')\right)\left\||\nabla|^{\fr{s}{2}}\mathcal{A}F\right\|^2_{L^2}dt'+\left\|\sqrt{\fr{\partial_tM^{\mu}}{M^\mu}}\mathcal{A}F\right\|^2_{L^2L^2}\le 4\left(C_0\eps\mu^{-\fr23}\right)^2,
\ee
where $C_0$ is a positive constant depending on $s, \lm, \lm'$ and  $\al$ determined in the proof.

Let $T^*$ be the maximal time $T>1$ such  that the hypotheses \eqref{hypo-main}--\eqref{hy-F} hold on $[1, T]$. The goal is then to prove that $T^*=+\infty$. By continuity, it suffices to improve the bounds on the right-hand side of \eqref{hypo-main}--\eqref{hy-F}, which is the purpose of the following proposition.
\begin{prop}\label{prop-main} 
Let $\sigma>9$.
Assume that \eqref{hypo-main}--\eqref{hy-F} hold on $[1, T^*]$. Then there exist a universal constant $\al_0>0$  and a sufficiently small constant $c_0$ depending on $\lm_{0}, \lm', s$ and $\al$, such that if $|\al|\ge\al_0$ and the initial data $(u_{\mathrm{in}}, b_{\mathrm{in}}, \om_{\mathrm{in}}, j_{\mathrm{in}})$ satisfies 
\be\label{initial-1}
\|(u_{\mathrm{in}},b_{\mathrm{in}})\|_{L^2}^2+\sum_{k\in\mathbb{Z}}\int (|\widehat{\omega_{\mathrm{in}}}|^2+|\widehat{j_{\mathrm{in}}}|^2)\la k,\eta\ra^{2\sigma}e^{2\lambda_0|k,\eta|^{s}}d\eta \le\eps^2=c_0^2\mu^2,
\ee
then the estimates in \eqref{hypo-main}--\eqref{hy-F} hold on $[1, T^*]$ with all the occurrences of ``4'' replaced by ``3''.
\end{prop}

The remainder of the paper is devoted to proving Proposition \ref{prop-main}.
\begin{itemize}
\item The improvement of \eqref{hypo-main} will be achieved by combining  the estimates in Sections \ref{energy-main}--\ref{sec-NLS}.
\item The improvement of \eqref{hy-ub0} is given in Section \ref{sec-ub0}.
\item The improvement of \eqref{hy-F} can be found in Section \ref{sec-F}.
\item The improvements of \eqref{coor1}--\eqref{coor3} can be found in Section \ref{sec-coor}.
\end{itemize}

\subsection{Estimates following from the bootstrap hypotheses} In this section, we give some estimates that follow from the bootstrap hypotheses \eqref{hypo-main}--\eqref{hy-F} immediately.

The first lemma below shows the uniform estimates of $\Om$ and $J$ implied by $E(t)$  and the enhanced dissipation of $J_{\ne}$ and $F_{\ne}$ in the form of the space-time $L^2$ integration.
\begin{lem}
Under the bootstrap hypotheses \eqref{hypo-main}--\eqref{hy-F}, the following estimates hold on $[1,T^*]$:
\be\label{enhan-JF}
\|A\Om(t)\|^2_{L^2}+\|AJ(t)\|^2_{L^2}+\mu^\fr13\|AJ_{\ne}\|_{L^2L^2}^2+\mu^{\fr53}\|\mathcal{A}F_{\ne}\|_{L^2L^2}^2\les\eps^2.
\ee
\end{lem}
\begin{proof}
Similar to \eqref{approx}, for $|\al|\ge8$ we have
\beno
E(t)\approx \|A\Om(t)\|^2_{L^2}+\|AJ(t)\|^2_{L^2}.
\eeno
The rest two estimates in \eqref{enhan-JF} follow from \eqref{enh-dis}, \eqref{hypo-main}, and \eqref{hy-F} immediately.
\end{proof}
 We will use \eqref{enhan-JF} so frequently throughout the proof that we will typically do so without any remark.
The next lemma concerns the zero-mode estimates of $U_0^1$, $B_0^1$, and $J_0$ under the change of coordinate \eqref{nl-coor}.
\begin{lem}
Under the bootstrap hypotheses \eqref{hypo-main}--\eqref{hy-F}, the following estimates hold on $[1,T^*]$:
\be\label{U0B0J0}
\|AU_0^1\|_{L^\infty L^2}+\|AB_0^1\|_{L^\infty L^2}+\mu^\fr12\left\|A\partial_YB_0^1\right\|_{L^2L^2}+\mu^{\fr12}\|AJ_0\|_{L^2L^2}\les\eps.
\ee
\end{lem}
\begin{proof}
Dividing $J_0$ into high and low frequency, one easily deduces that
\[
\|AJ_0\|_{L^2}\les\left\|{\bf 1}_{|\eta|\le1}A\hat{J}_0\right\|_{L^2}+\left\|{\bf 1}_{|\eta|>1}A\hat{J}_0\right\|_{L^2}\les\|J_0\|_{L^2}+\|\partial_YAJ_0\|_{L^2}.
\]
Thanks to the fact $\|J_0\|_{L^2}\approx\|j_0\|_{L^2}$, and recalling that  $j_0=-\partial_yb^1_0$, we obtain
\[
\|AJ_0\|_{L^2}\les\|\partial_yb^1_0\|_{L^2}+\|\partial_YAJ_0\|_{L^2}.
\]
To  bound $B_0^1$, we write $\partial_YB^1_0$ in terms of $J_0=-Y'\partial_YB_0^1$, and use  the algebra property of the operator $A_0$ in $L^2$ (see {\bf(3.40)} of \cite{BM15}) to deduce that
\begin{align*}
\left\|A\partial_YB_0^1\right\|_{L^2}
&=\left\|A\left[\left(1+\fr{1-Y'}{Y'}\right)Y'\partial_YB_0^1\right]\right\|_{L^2}
\lesssim \left\|AJ_0\right\|_{L^2}+\left\|A_0\sum_{n=1}^\infty\left(Y'-1\right)^n\right\|_{ L^2}\|AJ_0\|_{L^2L^2}\\
&\lesssim \left(1+\sum_{n=1}^\infty\left\|A^R(Y'-1)\right\|_{L^\infty L^2}^n\right)\|AJ_0\|_{L^2}
\les \|AJ_0\|_{L^2},
\end{align*}
where we have used the fact $A_0\le A^R$ above. Similarly, there hold
\[
\|AB_0^1\|_{L^\infty L^2}\les\|B_0^1\|_{L^\infty L^2}+\|A\partial_YB_0^1\|_{L^\infty L^2}\les \|b_0^1\|_{L^\infty L^2}+\|AJ_0\|_{L^\infty L^2},
\]
and
\[
\|AU_0^1\|_{L^\infty L^2}\les\|U_0^1\|_{L^\infty L^2}+\|A\partial_YU_0^1\|_{L^\infty L^2}\les \|u_0^1\|_{L^\infty L^2}+\|A\Om_0\|_{L^\infty L^2}.
\]
Then by the bootstrap hypotheses, \eqref{U0B0J0} follows immediately.
\end{proof}
Recalling the relation \eqref{HdotY} between $H$, and $\dot{Y}$, using an interpolation technique, one can derive some estimates of $\dot{Y}$ in terms of $H$. More precisely, let us denote 
\beno
\mathrm{CK}_{w}^{Y,1}(t)=\la t\ra^{2+2s}\left\|\sqrt{\fr{\partial_tw}{w}}\fr{A}{\la \partial_Y\ra^s}\dot{Y}(t) \right\|_{L^2}^2,\quad \mathrm{CK}_{\lm}^{Y,1}(t)=\la t\ra^{2+2s}\left(-\dot{\lm}(t)\right)\left\| |\partial_Y|^{\fr{s}{2}}\fr{A}{\la\partial_Y\ra^s}\dot{Y}(t)\right\|_{L^2}^2.
\eeno
Then we have the following lemma, the proof of which can be found in Section {\bf 8.2} of \cite{BM15}, and is thus omitted.
\begin{lem}
Assume that the hypotheses \eqref{hypo-main}--\eqref{hy-F} hold on $[1, T^*]$. Then for sufficiently small $\eps$  such that $K_D\mu^{-\fr23}\eps<2\tl{q}-1$, it holds that
\be\label{CKY1}
\int_1^{T^*}\mathrm{CK}_{w}^{Y,1}(t')+\mathrm{CK}_{\lm}^{Y,1}(t')dt'\les\eps^{2}.
\ee
\end{lem}

\subsection{Conclusion of the proof} In this section, we complete the proof of Theorem \ref{thm}.
By proposition \ref{prop-main}, we have a global uniform bound on $\mathcal{E}(t)$, which together with  \eqref{m1} implies that
\be\label{uni-1}
\|\Om(t)\|_{\mathcal{G}^{\lm(t),\sigma}}+\|J(t)\|_{\mathcal{G}^{\lm(t),\sigma}}+\la t\ra^2\left[\left(\|\Psi_{\ne}\|_{\mathcal{G}^{\lm(t),\sigma-3}}+\|\Phi_{\ne}\|_{\mathcal{G}^{\lm(t),\sigma-3}}\right)+\mu\|J_{\ne}\|_{\mathcal{G}^{\lm(t),\sigma-8}}\right]\les\min\left\{\mu^{-\fr13}, \la t\ra\right\}\eps,
\ee
and
\be
\la t\ra^{4-K_D\mu^{-\fr23}\eps}\left\|\dot{Y}\right\|^2_{\mathcal{G}^{\lm(t),\sigma-6}}+K_Y^{-1}\|Y'-1\|^2_{\mathcal{G}^{\lm(t),\sigma}}\les \eps^2.
\ee
On the other hand, we infer from \eqref{U0B0J0} that
\be\label{uni-3}
\|U_0^1(t)\|_{\mathcal{G}^{\lm(t),\sigma}}+\|B_0^1(t)\|_{\mathcal{G}^{\lm(t),\sigma}}\les\eps.
\ee
Next, we undo the change of coordinates in $Y$, switching to the coordinates $(X, y)$. Writing
\begin{gather*}
\tl{\om}(t,X, y)=\Om(t,X,Y)=\om(t,x,y), \quad \tl\psi(t,X, y)=\Psi(t,X,Y)=\psi(t,x,y),\\
\tl{j}(t, X,y)=J(t,X, Y)=j(t,x,y),\quad \tl\phi(t, X,y)=\Phi(t,X,Y)=\phi(t,x,y).
\end{gather*}
Then we drive from the first equation of \eqref{omj} that the equation of $\tl{\om}$ takes the form of
\be\label{eq-tlom}
\partial_t\tl\om-(\al+{b}_0^1)\partial_X\tl j=-\nabla^\bot_{X,y}\tl{\psi}_{\ne}\cdot\nabla_{X,y}\tl\om+\nabla^\bot_{X,y}\tl{\phi}_{\ne}\cdot\nabla_{X,y}\tl j.
\ee
Define $\lm_{\infty}=\lim_{t\rightarrow\infty}\lm(t)$. By the definition of $\lm(t)$ (see \eqref{dot-lam}), we have $\lm_{\infty}>\lm'$. Then following the argument in \cite{BM15}, one can show that the uniform bounds in \eqref{uni-1} and \eqref{uni-3} can be preserved in the $(X,y)$ coordinate system at the price of losing some regularities. The main ingredient is to use the composition inequality and the inverse function theorem in Gevrey spaces, see Lemmas {\bf A.4}  and {\bf A.5} of \cite{BM15}. We  omit the details for brevity and conclude that (adjusting $c_0$  if necessary) there exists some $\lm''\in(\lm',\lm_{\infty})$, such that
\beq\label{uni-Xy}
\nn&&\min\left\{\mu^{-\fr13}, \la t\ra\right\}^{-1}\Bigg[\left\|\tl{\om}(t)\right\|_{\mathcal{G}^{\lm''}}+\left\|\tl{j}(t)\right\|_{\mathcal{G}^{\lm''}}\\
&&+\la t\ra^2\left(\left\|\tl{\psi}_{\ne}(t)\right\|_{\mathcal{G}^{\lm''}}+\left\|\tl{\phi}_{\ne}(t)\right\|_{\mathcal{G}^{\lm''}}+\mu\|\tl j_{\ne}(t)\|_{\mathcal{G}^{\lm''}}\right)\Bigg]+\left\|b_0^1\right\|_{\mathcal{G}^{\lm''}}\les\eps.
\eeq
This enables us to define $\om_{\infty}$ by integrating \eqref{eq-tlom} with respect to $t$ over $[1,\infty]$:
\be\label{om_infty}
\om_\infty=\tl\om(1)+\int_1^\infty(\al+b_0^1(s))\partial_X\tl j(s)ds-\int_1^\infty\nabla^\bot_{X,y}\tl{\psi}_{\ne}(s)\cdot\nabla_{X,y}\tl\om(s)ds+\int_1^\infty\nabla^\bot_{X,y}\tl{\phi}_{\ne}(s)\cdot\nabla_{X,y}\tl j(s)ds.
\ee
It is worth pointing out that the integral over $[1, \infty)$ should be divided into the short time part over $[1,\mu^{-\fr13}]$ and the long time part over $[\mu^{-\fr13}, \infty)$. More precisely, to ensure the absolute convergence of the integral in \eqref{om_infty} uniformly in $\mu$, we use the bounds $\left\|\tl j(t)\right\|_{\mathcal{G}^{\lm''}}\les\eps\la t\ra$ and $\left\|\tl j(t)\right\|_{\mathcal{G}^{\lm''}}\les\mu^{-\fr{4}{3}}\la t\ra^{-2}\eps$ for the short time part and the long time part, respectively. Once $\om_\infty$ is well defined by \eqref{om_infty}, then one can get \eqref{om-infty} via \eqref{uni-Xy} since $\om_\infty$ is the limit of $\tl{\om}$ as $t\rightarrow\infty$.

Now integrating \eqref{eq-tlom} with respect to $t$ over $[t,\infty)$ for $t>\mu^{-\fr13}$, using the fact $\lm'<\lm''$, the algebra property of $\mathcal{G}^{\lm'}$, and the elementary inequality \eqref{ap-6}, we infer from \eqref{uni-Xy} that
\beq
\nn\left\|\tl\om(t)-\om_\infty\right\|_{\mathcal{G}^{\lm'}}&\le&\left\|\int_t^\infty(\al+b_0^1(s))\partial_X\tl j(s)ds\right\|_{\mathcal{G}^{\lm'}}+\left\|\int_t^\infty\nabla^\bot_{X,y}\tl{\psi}_{\ne}(s)\cdot\nabla_{X,y}\tl\om(s)ds\right\|_{\mathcal{G}^{\lm'}}\\
&&+\left\|\int_1^\infty\nabla^\bot_{X,y}\tl{\phi}_{\ne}(s)\cdot\nabla_{X,y}\tl j(s)ds\right\|_{\mathcal{G}^{\lm'}}\les\fr{\mu^{-\fr43}\eps}{\la t\ra}.
\eeq
Applying a similar argument to the equation of $u_0^1$ (see the first equation \eqref{ub01}) yields \eqref{Xi} and \eqref{-u_infty}. The decay rates \eqref{damping-u}--\eqref{damping-j} can be obtained via elliptic estimates. We complete the proof of Theorem \ref{thm}.

\section{Elliptic estimates}
\subsection{Lossy elliptic estimate}
\begin{lem}\label{lem-loss}
Under the bootstrap hypotheses, for $\eps$ sufficiently small, there hold
\be\label{loss1}
\left\|\Psi_{\ne}(t)\right\|_{\mathcal{G}^{\lm,\sigma-3}}\les\fr{\left\|\Om_{\ne}(t)\right\|_{\mathcal{G}^{\lm,\sigma-1}}}{\la t\ra^2},
\ee
\be\label{loss2}
\left\|\Phi_{\ne}(t)\right\|_{\mathcal{G}^{\lm,\sigma-3}}\les\fr{\left\|J_{\ne}(t)\right\|_{\mathcal{G}^{\lm,\sigma-1}}}{\la t\ra^2},
\ee
and
\be\label{loss3}
\|J_\ne\|_{\mathcal{G}^{\lm,\sigma-8}}\les\fr{\|F_\ne\|_{\mathcal{G}^{\lm,\sigma-6}}+\|\Om_{\ne}\|_{\mathcal{G}^{\lm,\sigma-5}}}{\mu\la t\ra^2}.
\ee
\end{lem}
\begin{proof}
We only prove \eqref{loss3} because the proof of \eqref{loss1} and \eqref{loss2} are the same and can be found in \cite{BM15}. Recalling the definition of $f$ in \eqref{def-f}, and using \eqref{Dlt}, we write
\be
\Dl_LJ_\ne=\mu^{-1}\left(F_\ne-(\al+B^1_0)\partial_X\Om\right)-\left((Y')^2-1\right)\partial_{YY}^LJ_\ne-Y''\partial_Y^LJ_\ne.
\ee
 Clearly,
\begin{gather}
\label{ref-Y'}(Y')^2-1=(Y'-1)^2+2(Y'-1),\\
\label{ref-Y''}Y''=Y'\partial_YY'=\fr12\partial_Y(Y'-1)^2+\partial_Y(Y'-1),
\end{gather}
which, together with  the algebra property \eqref{alge} in the Appendix, imply that
\beq
\nn\|\Dl_LJ_\ne\|_{\mathcal{G}^{\lm,\sigma-6}}&\les&\mu^{-1}\left(\|F_\ne\|_{\mathcal{G}^{\lm,\sigma-6}}+(1+\|B^1_0\|_{\mathcal{G}^{\lm,\sigma-6}})\|\Om_{\ne}\|_{\mathcal{G}^{\lm,\sigma-5}}\right)\\
\nn&&+\left(\|Y'-1\|^2_{\mathcal{G}^{\lm,\sigma-6}}+\|Y'-1\|_{\mathcal{G}^{\lm,\sigma-6}}\right)\|\Dl_LJ_\ne\|_{\mathcal{G}^{\lm,\sigma-6}}\\
\nn&&+\left(\|Y'-1\|^2_{\mathcal{G}^{\lm,\sigma-5}}+\|Y'-1\|_{\mathcal{G}^{\lm,\sigma-5}}\right)\|\Dl_LJ_\ne\|_{\mathcal{G}^{\lm,\sigma-6}}.
\eeq
Combing this with the hypothesis \eqref{coor1} and \eqref{U0B0J0} yields
\[
\|\Dl_LJ_\ne\|_{\mathcal{G}^{\lm,\sigma-6}}
\les\mu^{-1}\left(\|F_\ne\|_{\mathcal{G}^{\lm,\sigma-6}}+\|\Om_{\ne}\|_{\mathcal{G}^{\lm,\sigma-5}}\right).
\]
Thanks to the elementary inequality $\fr{{\bf 1}_{k\ne0}}{k^2+(\eta-kt)^2}\les\fr{{\bf 1}_{k\ne0}\left\la\fr{\eta}{k}\right\ra^2}{\la t\ra^2}$, there holds
\[
\|J_\ne\|_{\mathcal{G}^{\lm,\sigma-8}}\les\fr{\|\Dl_LJ_{\ne}\|_{\mathcal{G}^{\lm,\sigma-6}}}{\la t\ra^2}.
\]
Then \eqref{loss3} follows immediately. This completes the proof of Lemma \ref{lem-loss}.
\end{proof}

\subsection{Precision elliptic control}
Let us denote
\be\label{M12}
\mathcal{M}_1(t,k,\eta)=\left\la \fr{\eta}{kt}\right\ra^{-1}\fr{|k,\eta|^{\fr{s}{2}}}{\la t\ra^s}A\mathbb{P}_{\ne0},\quad \mathcal{M}_2(t,k,\eta)=\left\la \fr{\eta}{kt}\right\ra^{-1}\sqrt{\fr{\partial_tw}{w}}\tl A\mathbb{P}_{\ne0},
\ee
and
\be\label{T12}
\mathrm{T}^1=\left((Y')^2-1\right)\partial_{YY}^L\Phi, \quad \mathrm{T}^2=Y''\partial_Y^L\Phi.
\ee
\begin{prop}\label{prop-pre-ell1}
Under the bootstrap hypotheses, there hold
\be\label{pre-ell1}
\left\|\mathcal{M}_1\Dl_L\Psi\right\|_{L^2}^2+\left\|\mathcal{M}_2\Dl_L\Psi\right\|_{L^2}^2
\les\mathrm{CK}_{\lm, \Om}+\mathrm{CK}_{w,\Om}+\mu^{-\fr23}\eps^2\sum_{i=1}^2\left(\mathrm{CCK}^i_\lm+\mathrm{CCK}^i_w\right),
\ee

\be\label{pre-ell2}
\left\|\mathcal{M}_1\Dl_L\Phi\right\|_{L^2}^2+\left\|\mathcal{M}_2\Dl_L\Phi\right\|_{L^2}^2
\les\mathrm{CK}_{\lm, J}+\mathrm{CK}_{w,J}+\mu^{-\fr23}\eps^2\sum_{i=1}^2\left(\mathrm{CCK}^i_\lm+\mathrm{CCK}^i_w\right),
\ee
and
\be\label{pre-ell3}
\sum_{i=1}^2\left(\left\|\mathcal{M}_1\mathrm{T}^i\right\|_{L^2}^2+\left\|\mathcal{M}_2\mathrm{T}^i\right\|_{L^2}^2\right)\\
\les\eps^2\left(\left\|\mathcal{M}_1\Dl_L\Phi\right\|_{L^2}^2+\left\|\mathcal{M}_2\Dl_L\Phi\right\|_{L^2}^2\right)+\mu^{-\fr23}\eps^2\sum_{i=1}^2\left(\mathrm{CCK}^i_\lm+\mathrm{CCK}^i_w\right),
\ee
where the `coefficient Cauchy-Kovalevskaya' terms are given by
\beq
\mathrm{CCK}^1_\lm&=&-\dot{\lm}(t)\left\||\partial_Y|^\fr{s}{2}A^R\left((Y')^2-1\right)\right\|_{L^2}^2,\\
\mathrm{CCK}^1_w&=&\left\|\sqrt{\fr{\partial_tw}{w}}A^R\left((Y')^2-1\right)\right\|_{L^2}^2,\\
\mathrm{CCK}^2_\lm&=&-\dot{\lm}(t)\left\| |\partial_Y|^\fr{s}{2}\fr{A^R}{\la \partial_Y\ra}Y''\right\|_{L^2}^2,\\
\mathrm{CCK}^2_w&=&\left\|\sqrt{\fr{\partial_tw}{w}}\fr{A^R}{\la \partial_Y\ra}Y''\right\|_{L^2}^2.
\eeq
\end{prop}
\begin{proof}
The first precision elliptic estimate \eqref{pre-ell1} can be obtained by following the proof  of {\bf(2.28)} in \cite{BM15}, and the loss $\mu^{-\fr23}$ on the right hand side of \eqref{pre-ell1} stems from the using of \eqref{m1} to recover the estimate of $m^{-\fr12}\Psi$ when $\Psi$ is at low frequency. Replacing $\Psi$ and $\Om$ by $\Phi$ and $J$ in \eqref{pre-ell1}, respectively, we get \eqref{pre-ell2}. The third precision elliptic estimate \eqref{pre-ell3} can be found in the proving process of {\bf (2.28)} in \cite{BM15} (see Section 4.2 of \cite{BM15}). We would like to point out that the first term on the right-hand side of \eqref{pre-ell3} originates from the contributions when $\Phi$ is at high frequency. In this case, one can use \eqref{com-m} to exchange the frequencies in $m_k(t,\cdot)$ without losing any power of $\mu$.
\end{proof}
\begin{rem}\label{rem-s}
Recalling the definition of $\dot{\lm}(t)$ in \eqref{dot-lam}, in this paper $\tl q$ is chosen to satisfy $\fr12<\tl q\le\fr{s}{8}+\fr{7}{16}$.  Then the fact $s>\fr12$ ensures that $\tl q<\fr{s}{2}+\fr14$. Therefore, $\fr{1}{\la t\ra^{s+\fr12}}\left\||\nabla|^{\fr{s}{2}}Ag\right\|_{L^2}^2\les\mathrm{CK}_{\lm,g}$ with $g\in\{\Om,J\}$. As a result, the estimates \eqref{pre-ell1}--\eqref{pre-ell3} in Proposition \ref{prop-pre-ell1} still hold if $\fr{1}{\la t\ra^s}$ is replaced by $\fr{1}{\la t\ra^{\fr{s}{2}+\fr14}}$ in the definition of $\mathcal{M}_1(t, k,\eta)$ in \eqref{M12}. We will use the time weight $\fr{1}{\la t\ra^{\fr{s}{2}+\fr14}}$ instead of $\fr{1}{\la t\ra^s}$ in Section \ref{sec-eNLS}, see \eqref{need-ti1} for more details.
\end{rem}

\begin{rem} Given the reformulations of $(Y')^2-1$ and $Y''$ in \eqref{ref-Y'} and \eqref{ref-Y''}, applying the product rules obeyed by the multipliers appearing in the definitions of $\mathrm{CCK_{\lm}}^i$ and $\mathrm{CCK}_{w}^i$, $i=1,2$, it turns out that the sum of $\mathrm{CCK_{\lm}}^i$ and $\mathrm{CCK}_{w}^i$ can be bounded by $\mathrm{CK}_{\lm}^R+\mathrm{CK}_{w}^R$. More precisely,
\be\label{e-CCK}
\sum_{i=1}^2\left(\mathrm{CCK}^i_\lm+\mathrm{CCK}^i_w\right)\les\mathrm{CK}_{\lm}^R+\mathrm{CK}_{w}^R+\eps^2\left(\mathrm{CK}_{\lm}^R+\mathrm{CK}_{w}^R\right).
\ee
Please refer to Lemma {\bf 3.8}, {\bf (8.28)} and {\bf (8.29)} of \cite{BM15} for a proof.
\end{rem}

The following precision elliptic estimate without involving the multipliers $\fr{|\nabla|^{\fr{s}{2}}}{\la t\ra^s}$ and $\sqrt{\fr{\partial_tw}{w}}$ is needed to deal with the nonlinear term $-2\partial_{XY}^t\Phi\left(-\Om+2\partial_{XX}\Psi\right)$ in $\mathrm{NL}[J]$, see \eqref{need-ti1} for instance.
\begin{prop}\label{prop-pre-ell2}
Under the bootstrap hypotheses, there holds
\be\label{pre-ell4}
\left\|\left\la\fr{\partial_Y}{t\partial_X} \right\ra^{-1}A\Dl_L\Phi_\ne\right\|_{L^2}\les \left\| AJ{_\ne}\right\|_{L^2}+\mu^{-\fr13}\eps\left\| AJ{_\ne}\right\|_{L^2}.
\ee
\end{prop}
\begin{proof}
Recalling that $J=\Dl_t\Phi$ and using  \eqref{Dlt}, we have
\be\label{ePhi}
\Dl_L\Phi_{\ne}=J_{\ne}-\left((Y')^2-1\right)\partial_{YY}^L\Phi_{\ne}-Y''\partial_Y^L\Phi_{\ne}.
\ee
Then
\be\label{517}
\left\|\left\la\fr{\partial_Y}{t\partial_X} \right\ra^{-1}A\Dl_L\Phi_{\ne}\right\|_{L^2}\les\left\| AJ{_\ne}\right\|_{L^2}+\left\|\left\la\fr{\partial_Y}{t\partial_X} \right\ra^{-1}A\mathrm{T}^1{_\ne}\right\|_{L^2}+\left\|\left\la\fr{\partial_Y}{t\partial_X} \right\ra^{-1}A\mathrm{T}^2_{\ne}\right\|_{L^2}.
\ee
To proceed with the proof, for any $h(X,Y)\in L^2$, we split $\left\la\left\la\fr{\partial_Y}{t\partial_X} \right\ra^{-1}A\mathrm{T}^2_{\ne}, h\right\ra$
by paraproduct decomposition: 
\beq
\nn\left\la\left\la\fr{\partial_Y}{t\partial_X} \right\ra^{-1}A\mathrm{T}^2_{\ne}, h\right\ra
&=&i\sum_{N\ge8}\sum_{k\ne0}\int_{\eta,\xi}\left\la\fr{\eta}{tk} \right\ra^{-1}A_k(\eta)\widehat{Y''}(\xi)_N\left((\eta-\xi)-kt\right)\hat{\Phi}_k(\eta-\xi)_{<\fr{N}{8}} \bar{\hat{h}}_{k}(\eta)d\xi d\eta\\
\nn&&+i\sum_{N\in\mathbb{D}}\sum_{k\ne0}\int_{\eta,\xi}\left\la\fr{\eta}{tk} \right\ra^{-1}A_k(\eta)\widehat{Y''}(\eta-\xi)_{<16N}\left(\xi-kt\right)\hat{\Phi}_k(\xi)_{N} \bar{\hat{h}}_{k}(\eta)d\xi d\eta\\
\nn&=&\mathrm{T}^2_{\mathrm{HL}}+\mathrm{T}^2_{\mathrm{LH}}.
\eeq

On the support of the integrand of $\mathrm{T}^2_{\mathrm{HL}}$,  there hold
\be\label{sptR1}
\fr{N}{2}\le|\xi|\le\fr{3N}{2},\quad\mathrm{and}\quad |k,\eta-\xi|\le\fr{3}{4}\cdot\fr{N}{8}.
\ee
Hence,
\be\label{sptR2}
\big| |\xi|-|k,\eta|\big|\le|k,\eta-\xi|\le\fr{3}{16}|\xi|,
\ee
which implies that
\be\label{sptR3}
\fr{13}{16}|\xi|\le|k,\eta|\le\fr{19}{16}|\xi|,\quad\mathrm{and}\quad \fr{13}{16}|\xi|\le|\eta|\le\fr{19}{16}|\xi|.
\ee
Accordingly,  from \eqref{wNR-ratio}, the relation between $w_R(\cdot)$ and $w_{NR}(\cdot)$ defined in \eqref{wRwNR} and  the elementary inequalities \eqref{ap2}, \eqref{ap3},  \eqref{ap-5} and \eqref{ap-6}, we infer  that
\beq\label{AkR1}
\nn{\bf 1}_{t\in\mathbb{I}_{k,\eta}\cap\mathbb{I}_{k,\xi}}\fr{A_k(\eta)}{A^R(\xi)}&=&{\bf 1}_{t\in\mathbb{I}_{k,\eta}\cap\mathbb{I}_{k,\xi}}\fr{e^{\lm|k,\eta|^s}\la k,\eta\ra^\sigma\mathcal{M}_k(t,\eta)}{e^{\lm|\xi|^s}\la\xi\ra^\sigma}\fr{\fr{e^{2r\la\eta\ra^\fr12}}{w_k(\eta)}+e^{2r|k|^\fr12}}{\fr{e^{2r\la\xi\ra^\fr12}}{w_R(\xi)}}\\
\nn&\les&{\bf 1}_{t\in\mathbb{I}_{k,\eta}\cap\mathbb{I}_{k,\xi}}e^{c'\lm|k,\eta-\xi|^s}\left(e^{2r\left(\la\eta\ra^\fr12-\la\xi\ra^\fr12\right)}\fr{w_R(\xi)}{w_R(\eta)}+e^{2r\left(|k|^\fr12-\la\xi\ra^\fr12\right)}w_R(\xi)\right)\\
&\les&e^{c\lm|k,\eta-\xi|^s},
\eeq
for some $0<c'<c<1$. On the other hand, noting that $|\eta|\approx|\xi|$, similar to the proof of \eqref{J-ratio}, one deduces that there exists a constant $c\in(0,1)$ which may be larger than the one appearing in \eqref{AkR1}, such that
\be\label{AkR2}
{\bf 1}_{t\notin\mathbb{I}_{k,\eta}\cap\mathbb{I}_{k,\xi}}\fr{A_k(\eta)}{A^R(\xi)}\les e^{c\lm|k,\eta-\xi|^s}. \quad\mathrm{for \ \ some}\ \ c\in(0,1). 
\ee
It follows from \eqref{AkR1},  \eqref{AkR2}, and the fact $|\eta|\approx|\xi|$  that
\beq\label{522}
\mathrm{T}^2_{\mathrm{HL}}\nn&\les&\sum_{N\ge8}\sum_{k\ne0}\int_{\eta,\xi}\fr{1}{\left\la\fr{\xi}{tk} \right\ra |kt|}A^R(\xi)\left|\widehat{Y''}(\xi)_N\right|\left|(\eta-\xi)-kt\right||kt|\\
\nn&&\quad\quad\quad\times e^{c\lm|k,\eta-\xi|^s}\left|\hat{\Phi}_k(\eta-\xi)_{<\fr{N}{8}}\right| \left|\hat{h}_{k}(\eta)\right|d\xi d\eta\\
\nn&\les&\sum_{k\ne0}\int_{\eta,\xi}\fr{A^R(\xi)}{\left\la\xi \right\ra }\left|\widehat{Y''}(\xi)\right| \left(t^2e^{\lm|k,\eta-\xi|^s}\left|\hat{\Phi}_k(\eta-\xi)\right|\right) \left|\hat{h}_{k}(\eta)\right|d\xi d\eta\\
&\les&\left\|\fr{A^R}{\la\partial_Y\ra}Y''\right\|_{L^2}\left\|t^2\Phi_{\ne}\right\|_{\mathcal{G}^{\lm, 1}}\|h\|_{L^2}.
\eeq
As for $\mathrm{T}^2_{\mathrm{LH}}$, on the support of the integrand of  $\mathrm{T}^2_{\mathrm{LH}}$, there hold
\be\label{spt<16N}
\fr{N}{2}\le|k,\xi|\le\fr{3}{2}N, \quad\mathrm{and}\quad |\eta-\xi|\le\fr{3}{4}\cdot16N=12N\le24|k,\xi|.
\ee
Then  use \eqref{ap3} when $|\eta-\xi|\le\fr{1}{24}|k,\xi|$ and use \eqref{ap4} when $\fr{1}{24}|k,\xi|\le|\eta-\xi|\le24|k,\xi|$ to deduce that
\be\label{up-e}
e^{\lm(|k,\eta|^s-|k,\xi|^s)}\les e^{c'\lm|\eta-\xi|^s}, \quad \mathrm{for\ \ some\ \ } c'\in(0,1).
\ee
Combining this with \eqref{J-ratio} and \eqref{com-m} yields
\be\label{AkAk}
\fr{A_k(\eta)}{A_k(\xi)}=\fr{e^{\lm|k,\eta|^s}\la k,\eta\ra^\sigma\mathcal{M}_k(t,\eta)}{e^{\lm|k,\xi|^s}\la k,\xi\ra^\sigma\mathcal{M}_k(t,\xi)}\fr{\mathcal{J}_k(\eta)}{\mathcal{J}_k(\xi)}\les e^{c\lm|\eta-\xi|^s},
\ee
for some $c\in(c',1)$. On the other hand, it is easy to see that
\be\label{ex-etaxi}
\left\la\fr{\eta}{tk}\right\ra^{-1} =\left\la\fr{\eta}{tk}\right\ra^{-1}\left\la\fr{\xi}{tk} \right\ra\left\la\fr{\xi}{tk} \right\ra^{-1}\les\la\eta-\xi\ra\left\la\fr{\xi}{tk} \right\ra^{-1}.
\ee
Consequently,
\beq\label{527}
\mathrm{T}^2_{\mathrm{LH}}\nn&\les&\sum_{k\ne0}\int_{\eta,\xi}\la\eta-\xi\ra e^{c\lm|\eta-\xi|^s} \left|\widehat{Y''}(\eta-\xi)\right|\left\la\fr{\xi}{tk} \right\ra^{-1}A_k(\xi)\left|\Dl_L\hat{\Phi}_k(\xi)\right| \left|\hat{h}_{k}(\eta)\right|d\xi d\eta\\
&\les&\|Y''\|_{\mathcal{G}^{\lm, 2}}\left\|\left\la\fr{\partial_Y}{t\partial_X} \right\ra^{-1}A\Dl_L\Phi_{\ne}\right\|_{L^2}\|h\|_{L^2}.
\eeq
In view of \eqref{522}, \eqref{527} and the lossy elliptic estimate \eqref{loss2}, we have
\be\label{e-AT2}
\left\|\left\la\fr{\partial_Y}{t\partial_X} \right\ra^{-1}A\mathrm{T}^2_{\ne}\right\|_{L^2}\les\left\|\fr{A^R}{\la\partial_Y\ra}Y''\right\|_{L^2}\|J_{\ne}\|_{\mathcal{G}^{\lm,3}}+\|Y''\|_{\mathcal{G}^{\lm, 1}}\left\|\left\la\fr{\partial_Y}{t\partial_X} \right\ra^{-1}A\Dl_L\Phi_{\ne}\right\|_{L^2}.
\ee
Note that it is unnecessary and impossible to recover one derivative on the coefficient $(Y')^2-1$. With slight modifications, the above argument applies to $\left\la\fr{\partial_Y}{t\partial_X} \right\ra^{-1}A\mathrm{T}^1_{\ne}$:
\be\label{e-AT1}
\left\|\left\la\fr{\partial_Y}{t\partial_X} \right\ra^{-1}A\mathrm{T}^1_{\ne}\right\|_{L^2}\les\left\|{A^R}\left((Y')^2-1\right)\right\|_{L^2}\|J_{\ne}\|_{\mathcal{G}^{\lm,3}}+\left\|(Y')^2-1\right\|_{\mathcal{G}^{\lm, 1}}\left\|\left\la\fr{\partial_Y}{t\partial_X} \right\ra^{-1}A\Dl_L\Phi_{\ne}\right\|_{L^2}.
\ee
Substituting the above two inequalities into \eqref{517}, and absorbing the terms with lower order coefficients, we have
\[
\left\|\left\la\fr{\partial_Y}{t\partial_X} \right\ra^{-1}A\Dl_L\Phi_\ne\right\|_{L^2}\les \left\| AJ{_\ne}\right\|_{L^2}+\left(\left\|\fr{A^R}{\la\partial_Y\ra}Y''\right\|_{L^2}+\left\|{A^R}\left((Y')^2-1\right)\right\|_{L^2}\right)\left\|J_{\ne}\right\|_{\mathcal{G}^{\lm,3}}.
\]
Recalling \eqref{ref-Y'} and \eqref{ref-Y''}, using the algebra property, both $\left\|\fr{A^R}{\la\partial_Y\ra}Y''\right\|_{L^2}\|J_{\ne}\|_{\mathcal{G}^{\lm,3}}$ and $\left\|{A^R}\left((Y')^2-1\right)\right\|_{L^2}$ can be bounded by $\left\|{A^R}\left(Y'-1\right)\right\|_{L^2}+\left\|{A^R}\left((Y')^2-1\right)\right\|_{L^2}^2$, see {\bf (3.41)} in \cite{BM15}. Then  by \eqref{m1} and the hypothesis \eqref{coor1}, we obtain \eqref{pre-ell4}. This completes the proof of Proposition \ref{prop-pre-ell2}.
\end{proof}

\begin{rem}\label{rem-preci-Psi}
Obviously, similar  precision elliptic estimate holds for $\Psi_{\ne}$ with $J_{\ne}$ replaced by $\Om_{\ne}$ on the righthand side of \eqref{pre-ell4}.
\end{rem}

\section{Main energy estimates}\label{energy-main}

By the definitions of the multipliers $A$ and $\tl A$, we have
\be\label{A_t}
\partial_tA=\dot{\lm}(t)|k,\eta|^sA-\fr{\partial_tw}{w}\tl{A}-\sum_{\theta\in\left\{m^\fr12, M^1, M^\mu\right\}}\fr{\partial_t\theta}{\theta}A.
\ee
From the first two equations of $\eqref{OMJ}$ and \eqref{A_t},  integrating by parts and using \eqref{Dlt},  one deduces the evolution of $\|A\Om\|^2_{L^2}+\|AJ\|^2_{L^2}$:
\beq\label{E1}
\nn&&\fr12\fr{d}{dt}\left(\|A\Om\|^2_{L^2}+\|AJ\|^2_{L^2}\right)+\sum_{g\in\{\Om, J\}}\left(\mathrm{CK}_{\lm, g}+\mathrm{CK}_{w,g}+\sum_{\theta\in\left\{ M^1, M^\mu\right\}}\mathrm{CK}_{\theta, g}\right)+\mu\|\nabla_LAJ\|^2_{L^2}\\
\nn&=&-\sum_{g\in\{\Om, J\}}\mathrm{CK}_{m^\fr12, g}-2\left\la AJ, A\partial_{XY}^t\Phi \right\ra+\mu\left\la AJ, A\left[\left((Y')^2-1\right)\partial_{YY}^L+Y''\partial_Y^L\right]J\right\ra\\
&&+\left\la A\Om, A\mathrm{NL}[\Om]\right\ra+\left\la AJ, A\mathrm{NL}[J]\right\ra.
\eeq
Combining \eqref{ePhi} with \eqref{e215} yields
\beq\label{E3}
\nn&&-\sum_{g\in\{\Om, J\}}\mathrm{CK}_{m^\fr12, g}-2\left\la AJ, A\partial_{XY}^t\Phi \right\ra\\
\nn&=&\mathrm{CK}_{m^\fr12, J}-\mathrm{CK}_{m^\fr12, \Om}-\sum_{k\in\mathbb{Z}}\int {\bf 1}_{(D_1\cup D_2)^c}\fr{2k(\eta-kt)}{k^2+(\eta-kt)^2}|A\hat{J}|^2d\eta\\
&&+2\left\la AJ, A\partial_{XY}^L\Dl_L^{-1}\left[\left((Y')^2-1\right)\partial_{YY}^L\Phi+Y''\partial_Y^L\Phi \right]\right\ra-2\left\la AJ, A\left((Y'-1)\partial_{XY}^L\Phi\right)\right\ra.
\eeq
To obtain the evolution of the cross term in \eqref{En}, 
using \eqref{A_t},  \eqref{ePhi} and \eqref{e215}, we have
\beq
\nn&&-\partial_tA\hat{\Om}A\bar{\hat{J}}-A\hat{\Om}\partial_tA\bar{\hat{J}}+2A\hat{\Om}A\overline{\widehat{\partial_{XY}^t\Phi}}\\
\nn&=&-2\dot{\lm}(t)|k,\eta|^sA\hat{\Om}A\bar{\hat{J}}+2\fr{\partial_tw}{w}\sqrt{\tl{A}A}\hat{\Om}\sqrt{\tl{A}A}\bar{\hat{J}}\\
\nn&&+2\sum_{\theta\in\left\{ M^1, \, M^\mu \right\}}\fr{\partial_t\theta}{\theta}A\hat{\Om}A\bar{\hat{J}}
+{\bf 1}_{(D_1\cup D_2)^c}\fr{2k(\eta-kt)}{k^2+(\eta-kt)^2}A\hat{\Om}A\bar{\hat{J}}
\\
\nn&&-2A\hat{\Om}A\overline{\mathcal{F}\left[\partial_{XY}^L\Dl_L^{-1}\left(\big((Y')^2-1\big)\partial_{YY}^L\Phi+Y''\partial_Y^L\Phi\right)\right]}+2A\hat{\Om}A\overline{\mathcal{F}\left[(Y'-1)\partial_{XY}^L\Phi\right]}.
\eeq
Then the analog of  \eqref{inner-OmJ} follows
\beq
\nn\partial_t\left(A\hat{\Om} A\bar{\hat{J}} \right)&=&i\al k\left(|A\hat{J}|^2-|A\hat{\Om}|^2\right)-{\bf 1}_{(D_1\cup D_2)^c}\fr{2k(\eta-kt)}{k^2+(\eta-kt)^2}A\hat{\Om}A\bar{\hat{J}}\\
\nn&&+2\dot{\lm}(t)|k,\eta|^sA\hat{\Om}A\bar{\hat{J}}-2\fr{\partial_tw}{w}\sqrt{\tl{A}A}\hat{\Om}\sqrt{\tl{A}A}\bar{\hat{J}}-2\sum_{\theta\in\left\{ M^1, \, M^\mu \right\}}\fr{\partial_t\theta}{\theta}A\hat{\Om}A\bar{\hat{J}}\\
\nn&&+2A\hat{\Om}A\overline{\mathcal{F}\left[\partial_{XY}^L\Dl_L^{-1}\left(\big((Y')^2-1\big)\partial_{YY}^L\Phi+Y''\partial_Y^L\Phi\right)\right]}\\
\nn&&-2A\hat{\Om}A\overline{\mathcal{F}\left[(Y'-1)\partial_{XY}^L\Phi\right]}+\mu A\hat{\Om}A\overline{\widehat{\Dl_tJ}}+A\widehat{\mathrm{NL}[\Om]}A\bar{\hat{J}}+A\hat{\Om}A\overline{\widehat{\mathrm{NL}[J]}}.
\eeq
This enables us to get the evolution of the cross term in \eqref{En}. Indeed, similar to \eqref{inner}, we arrive at
\beq\label{E4}
\nn \fr{1}{\al }\fr{d}{dt}\left\la\fr{\partial_t(m^\fr12)}{m^{\fr12}}\partial_X^{-1}A{\Om}_{\ne} , AJ_{\ne} \right\ra&=&\mathrm{CK}_{m^\fr12, J}-\mathrm{CK}_{m^\fr12, \Om}+\fr{2}{\al }\dot{\lm}(t)\left\la\fr{\partial_t(m^\fr12)}{m^{\fr12}}\partial_X^{-1}|\nabla|^\fr{s}{2}A\Om_{\ne}, |\nabla|^\fr{s}{2}AJ_{\ne}\right\ra\\
\nn&&-\fr{2}{\al }\left\la\fr{\partial_t(m^\fr12)}{m^{\fr12}}\partial_{X}^{-1}\sqrt{\fr{\partial_tw}{w}}\sqrt{\tl{A}A}\Om_{\ne}, \sqrt{\fr{\partial_tw}{w}}\sqrt{\tl{A}A}J_\ne\right\ra\\
\nn&&-\fr{2}{\al }\sum_{\theta\in\left\{ M^1, \, M^\mu \right\}}\left\la\fr{\partial_t(m^\fr12)}{m^{\fr12}}\partial_X^{-1}\sqrt{\fr{\partial_t\theta}{\theta}}A\Om_{\ne}, \sqrt{\fr{\partial_t\theta}{\theta}}AJ_{\ne}\right\ra\\
\nn&&+\fr{2}{\al }\left\la\fr{\partial_t(m^\fr12)}{m^{\fr12}}\partial_X^{-1}A\Om_{\ne}, A\partial_{XY}^L\Dl_L^{-1}\left[\left((Y')^2-1\right)\partial_{YY}^L\Phi_{\ne}+Y''\partial_Y^L\Phi_{\ne}\right]\right\ra\\
\nn&&-\fr{2}{\al }\left\la\fr{\partial_t(m^\fr12)}{m^{\fr12}}\partial_X^{-1}A\Om_{\ne}, A\left((Y'-1)\partial_{XY}^L\Phi_{\ne}\right)\right\ra+\fr{\mu}{\al }\left\la\fr{\partial_t(m^\fr12)}{m^{\fr12}} \partial_X^{-1}A\Om_{\ne}, A\Dl_tJ_{\ne}\right\ra\\
\nn&&+\fr{1}{\al }\left\la\fr{\partial_t(m^\fr12)}{m^{\fr12}}\partial_X^{-1}A\mathrm{NL}[\Om]_{\ne}, AJ_{\ne}\right\ra+\fr{1}{\al }\left\la\fr{\partial_t(m^\fr12)}{m^{\fr12}}\partial_X^{-1}A\Om_{\ne}, A\mathrm{NL}[J]_{\ne}\right\ra\\
&&+\fr{1}{\al }\left\la\partial_t\left(\fr{\partial_t(m)^\fr12}{m^\fr12}\right)\partial_X^{-1}A\Om_{\ne}, AJ_{\ne}\right\ra.
\eeq
Collecting \eqref{E1},  \eqref{E3} and \eqref{E4},
and using the fact that
\beq
\nn&&-\fr{\mu}{\al }\left\la\fr{\partial_t(m^\fr12)}{m^{\fr12}} \partial_X^{-1}A\Om_{\ne}, A\Dl_tJ_{\ne}\right\ra\\
\nn&=&-\fr{\mu}{\al }\left\la\fr{\partial_t(m^\fr12)}{m^{\fr12}} \partial_X^{-1}A\Om_{\ne}, A\Dl_LJ_{\ne}\right\ra-\fr{\mu}{\al }\left\la\fr{\partial_t(m^\fr12)}{m^{\fr12}} \partial_X^{-1}A\Om_{\ne}, A\left[\left((Y')^2-1\right)\partial_{YY}^L+Y''\partial_Y^L\right]J_{\ne}\right\ra,
\eeq
we are led to
\beq
\nn&&\fr12\fr{d}{dt}{E}(t)+\sum_{g\in\{\Om, J\}}\left(\mathrm{CK}_{\lm, g}+\mathrm{CK}_{w,g}+\sum_{\theta\in\left\{ M^1, M^\mu\right\}}\mathrm{CK}_{\theta, g}\right)+\mu\|\nabla_LAJ\|^2_{L^2}\\
&=&{\bf LE}+\mathbf{NL}+\mathbf{CNL}+\mathbf{DE}+\mathbf{CDE}+\mathbf{LSE}+\mathbf{CLSE},
\eeq
where
\beqno
{\bf LE}\nn&=&-\sum_{k\in\mathbb{Z}}\int {\bf 1}_{(D_1\cup D_2)^c}\fr{2k(\eta-kt)}{k^2+(\eta-kt)^2}|A\hat{J}|^2d\eta-\fr{\mu}{\al }\left\la\fr{\partial_t(m^\fr12)}{m^{\fr12}} \partial_X^{-1}A\Om_{\ne}, A\Dl_LJ_{\ne}\right\ra\\
\nn&&-\fr{1}{\al }\left\la\partial_t\left(\fr{\partial_t(m)^\fr12}{m^\fr12}\right)\partial_X^{-1}A\Om_{\ne}, AJ_{\ne}\right\ra+\fr{2}{\al }\left\la\fr{\partial_t(m^\fr12)}{m^{\fr12}}\partial_{X}^{-1}\sqrt{\fr{\partial_tw}{w}}\sqrt{\tl{A}A}\Om_{\ne}, \sqrt{\fr{\partial_tw}{w}}\sqrt{\tl{A}A}J_\ne\right\ra\\
&&+\fr{2}{\al }\sum_{\theta\in\left\{ M^1, \, M^\mu \right\}}\left\la\fr{\partial_t(m^\fr12)}{m^{\fr12}}\partial_X^{-1}\sqrt{\fr{\partial_t\theta}{\theta}}A\Om_{\ne}, \sqrt{\fr{\partial_t\theta}{\theta}}AJ_{\ne}\right\ra-\fr{2}{\al }\dot{\lm}(t)\left\la\fr{\partial_t(m^\fr12)}{m^{\fr12}}\partial_X^{-1}|\nabla|^\fr{s}{2}A\Om_{\ne}, |\nabla|^\fr{s}{2}AJ_{\ne}\right\ra,
\eeqno
and
\beqno
\mathbf{NL}&=&\left\la A\Om, A\mathrm{NL}[\Om]\right\ra+\left\la AJ, A\mathrm{NL}[J]\right\ra,\\
\mathbf{CNL}&=&-\fr{1}{\al }\left\la\fr{\partial_t(m^\fr12)}{m^{\fr12}}\partial_X^{-1}A\mathrm{NL}[\Om]_{\ne}, AJ_{\ne}\right\ra-\fr{1}{\al }\left\la\fr{\partial_t(m^\fr12)}{m^{\fr12}}\partial_X^{-1}A\Om_{\ne}, A\mathrm{NL}[J]_{\ne}\right\ra,\\
\mathbf{DE}&=&\mu\left\la AJ, A\left[\left((Y')^2-1\right)\partial_{YY}^L+Y''\partial_Y^L\right]J\right\ra,\\
\mathbf{CDE}&=&-\fr{\mu}{\al }\left\la\fr{\partial_t(m^\fr12)}{m^{\fr12}} \partial_X^{-1}A\Om_{\ne}, A\left[\left((Y')^2-1\right)\partial_{YY}^L+Y''\partial_Y^L\right]J_{\ne}\right\ra,\\
\mathbf{LSE}&=&2\left\la AJ, A\partial_{XY}^L\Dl_L^{-1}\left[\left((Y')^2-1\right)\partial_{YY}^L\Phi+Y''\partial_Y^L\Phi \right]\right\ra-2\left\la AJ, A\left((Y'-1)\partial_{XY}^L\Phi\right)\right\ra,\\
{\bf CLSE}\nn&=&-\fr{2}{\al }\left\la\fr{\partial_t(m^\fr12)}{m^{\fr12}}\partial_X^{-1}A\Om_{\ne}, A\partial_{XY}^L\Dl_L^{-1}\left[\left((Y')^2-1\right)\partial_{YY}^L\Phi_{\ne}+Y''\partial_Y^L\Phi_{\ne}\right]\right\ra\\
&&+\fr{2}{\al }\left\la\fr{\partial_t(m^\fr12)}{m^{\fr12}}\partial_X^{-1}A\Om_{\ne}, A\left((Y'-1)\partial_{XY}^L\Phi_{\ne}\right)\right\ra,
\eeqno
where `{\bf C}' stands for `cross terms'.
Recalling the definitions of $\mathrm{NL}[\Om]$ and $\mathrm{NL}[J]$ in \eqref{NL}, noting that $B^1_0=-Y'\partial_Y\Phi_0$,   integrating by parts, and using the facts
\beno
\la A\Om_{\ne}, B^1_0\partial_XAJ\ra+\la AJ_{\ne},  B^1_0\partial_XA\Om\ra=0,
\eeno
and
\beno
\left\la A\Om, Y'\nabla^\bot\Phi_{\ne}\cdot\nabla AJ\right\ra+\left\la AJ, Y'\nabla^\bot\Phi_{\ne}\cdot\nabla A\Om\right\ra=-\int\nabla\cdot\left(Y'\nabla^\bot\Phi_{\ne}\right)A\Om AJdXdY,
\eeno
we are led to
\beq\label{NL1}
\mathbf{NL}\nn&=&\fr12\sum_{g\in\{\Om, J\}}\int\nabla\cdot V|Ag|^2dXdY-\int\nabla\cdot\left(Y'\nabla^\bot\Phi_\ne\right)A\Om AJdXdY\\
&&+\sum_{g\in\{\Om, J\}}\left({\bf Com}[B_0]_g+{\bf Com}[B_{\ne}]_g-{\bf Com}[V]_g\right)+\mathbf{NLS},
\eeq
where
\beq
{\bf Com}[V]_g&=&\left\la Ag, A(V\cdot\nabla g)-V\cdot\nabla Ag\right\ra,\\
\label{com-J}{\bf Com}[B_{\ne}]_J&=&\left\la A\Om, A\left(Y'\nabla^\bot\Phi_{\ne}\cdot\nabla J\right)-Y'\nabla^\bot\Phi_{\ne}\cdot\nabla AJ\right\ra,\\
\label{com-Om}{\bf Com}[B_{\ne}]_{\Om}&=&\left\la AJ, A\left(Y'\nabla^\bot\Phi_{\ne}\cdot\nabla\Om\right)-Y'\nabla^\bot\Phi_{\ne}\cdot\nabla A\Om\right\ra,\\
{\bf Com}[B_0]_J&=&\left\la A\Om_{\ne}, A\left(B^1_0\partial_X J\right)-B^1_0\partial_X AJ\right\ra,\\
{\bf Com}[B_0]_{\Om}&=&\left\la AJ_{\ne}, A\left(B^1_0\partial_X\Om\right)-B^1_0\partial_XA\Om\right\ra,\\
\mathbf{NLS}\nn&=&-2\left\la AJ, A\left(\partial_{XY}^t\Phi(-\Om+2\partial_{XX}\Psi)\right)\right\ra\\
\label{NLS}&&+2\left\la AJ, A\left(\partial_{XY}^t\Psi(-J+2\partial_{XX}\Phi)\right)\right\ra.
\eeq

Next, we rewrite $\mathbf{CNL}$ more explicitly. Using again $B_0^1=-Y'\partial_Y\Phi_0$, we have
\beq
\nn&&-\fr{1}{\al }\left\la\fr{\partial_t(m^\fr12)}{m^{\fr12}}\partial_X^{-1}A\left(Y'\nabla^\bot\Phi_0\cdot\nabla J_{\ne}\right), AJ_{\ne}\right\ra\\
\nn&=&-\fr{1}{\al }\left\la\fr{\partial_t(m^\fr12)}{m^{\fr12}}\left[A\left(B^1_0 J_{\ne}\right)-B^1_0AJ_{\ne}\right], AJ_{\ne}\right\ra-\fr{1}{\al }\left\la\fr{\partial_t(m^\fr12)}{m^{\fr12}}\left(B^1_0AJ_{\ne}\right), AJ_{\ne}\right\ra,
\eeq
and
\beq
\nn&&-\fr{1}{\al }\left\la\fr{\partial_t(m^\fr12)}{m^{\fr12}}\partial_X^{-1}A\Om_{\ne}, A\left(Y'\nabla^\bot\Phi_0\cdot\nabla \Om_{\ne}\right)\right\ra\\
\nn&=&\fr{1}{\al }\left\la\fr{\partial_t(m^\fr12)}{m^{\fr12}}A\Om_{\ne}, \left[A\left(B^1_0 \Om_{\ne}\right)-B^1_0A\Om_{\ne}\right]\right\ra+\fr{1}{\al }\left\la\fr{\partial_t(m^\fr12)}{m^{\fr12}}A\Om_{\ne}, B^1_0A\Om_{\ne}\right\ra.
\eeq
Then
\be
\mathbf{CNL}=\mathbf{CNLT}+\sum_{g\in\{\Om, J\}}{\bf CCom}[B_0]_g+\mathbf{CComE}+\mathbf{CNLS},
\ee
where
\beq
\mathbf{CNLT}\nn&=&\fr{1}{\al}\left\la\fr{\partial_t(m^\fr12)}{m^{\fr12}}\partial_X^{-1}A\left(V\cdot\nabla\Om\right)_{\ne}, AJ_{\ne} \right\ra+\fr{1}{\al}\left\la\fr{\partial_t(m^\fr12)}{m^{\fr12}}\partial_X^{-1}A\Om_\ne, A\left(V\cdot\nabla J\right)_\ne \right\ra\\
\nn&&-\fr{1}{\al}\left\la\fr{\partial_t(m^\fr12)}{m^{\fr12}}\partial_X^{-1}A\left(Y'\nabla^\bot\Phi_{\ne}\cdot\nabla J\right)_{\ne}, AJ_{\ne} \right\ra\\
\label{CNLT}&&-\fr{1}{\al}\left\la\fr{\partial_t(m^\fr12)}{m^{\fr12}}\partial_X^{-1}A\Om_\ne, A\left(Y'\nabla^\bot\Phi_{\ne}\cdot\nabla\Om\right)_{\ne} \right\ra,\\
{\bf CCom}[B_0]_\Om&=&\fr{1}{\al }\left\la\fr{\partial_t(m^\fr12)}{m^{\fr12}}A\Om_{\ne}, \left[A\left(B^1_0 \Om_{\ne}\right)-B^1_0A\Om_{\ne}\right]\right\ra,\\
{\bf CCom}[B_0]_J&=&-\fr{1}{\al }\left\la\fr{\partial_t(m^\fr12)}{m^{\fr12}}\left[A\left(B^1_0 J_{\ne}\right)-B^1_0AJ_{\ne}\right], AJ_{\ne}\right\ra,\\
\mathbf{CComE}&=&\fr{1}{\al }\left\la\fr{\partial_t(m^\fr12)}{m^{\fr12}}A\Om_{\ne}, B^1_0A\Om_{\ne}\right\ra-\fr{1}{\al }\left\la\fr{\partial_t(m^\fr12)}{m^{\fr12}}\left(B^1_0AJ_{\ne}\right), AJ_{\ne}\right\ra,\\
\mathbf{CNLS}\nn&=&\fr{2}{\al}\left\la \fr{\partial_t(m^\fr12)}{m^{\fr12}}\partial_X^{-1}A\Om_{\ne}, A\left(\partial_{XY}^t\Phi(-\Om+2\partial_{XX}\Psi)\right)_{\ne}\right\ra\\
\label{CNLS}&&-\fr{2}{\al}\left\la \fr{\partial_t(m^\fr12)}{m^{\fr12}}\partial_X^{-1}A\Om_{\ne}, A\left(\partial_{XY}^t\Psi(-J+2\partial_{XX}\Phi)\right)_{\ne}\right\ra.
\eeq

\subsection{Linear errors} The linear errors $\mathbf{LE}$ can be bounded similar to \eqref{cro-dis}--\eqref{LSerr}:
\beq
\mathbf{LE}\nn&\les&\fr{\mu}{4}\left\|\nabla_LAJ\right\|^2_{L^2}+\fr{\mu}{4|\al|}\left\|\nabla_LAJ\right\|^2_{L^2}+\fr{5}{|\al|}\mathrm{CK}_{M^1,\Om}\\
&&+\fr{1}{|\al|}\sum_{\theta\in\{M^1, M^\mu\}}\sum_{g\in\{\Om, J\}}\mathrm{CK}_{\theta, g}+\fr{1}{2|\al|}\sum_{g\in\{\Om, J\}}\mathrm{CK}_{\lambda, g}+\fr{1}{2|\al|}\sum_{g\in\{\Om, J\}}\mathrm{CK}_{w, g}.
\eeq

\subsection{Dissipation errors}
In this section, we treat the dissipation error terms $\mathbf{DE}$ and $\mathbf{CDE}$. As mentioned in the introduction, the most challenging term arises when the coefficient $Y''$ is at high frequency. See \eqref{D-HL} below for the precise problematic term $\mathrm{DE}^{2;\mathrm{D}}_{2;\mathrm{HL}}$, where we need to control the derivative loss. 
\subsubsection{Dissipation errors  $\mathbf{DE}$}
Integrating by parts, we have
\beq\label{de-m}
\mathbf{DE}&=&-\mu\left\la \partial_Y^LAJ, A\left[\left((Y')^2-1\right)\partial_{Y}^LJ\right]\right\ra-\mu\left\la AJ, A\left(Y''\partial_Y^LJ\right)\right\ra=\mathrm{DE}_1+\mathrm{DE}_2.
\eeq
For $\mathrm{DE}_1$, using the product estimate \eqref{ap5}, the upper bound of $m_k(t,\eta)$ in \eqref{m1},  the algebra property of $A^R$ and the bootstrap hypotheses,  we have
\beq\label{e-DE1}
\int_1^t\mathrm{DE}_1dt'\nn&\les&\mu\|\partial_Y^LAJ\|_{L^2L^2}\left(\|(Y')^2-1\|_{L^\infty\mathcal{G}^{c\lm,1}}\|\partial_Y^LAJ\|_{L^2L^2}+\left\|A^R\left((Y')^2-1\right)\right\|_{L^\infty L^2}\|\partial_Y^LJ\|_{L^2\mathcal{G}^{c\lm,1}}\right)\\
&\les&\mu^\fr23\|\partial_Y^LAJ\|_{L^2L^2}^2\left(\left\|A^R\left(Y'-1\right)\right\|_{L^\infty L^2}+\left\|A^R\left(Y'-1\right)\right\|_{L^\infty L^2}^2\right)\les\mu^{-\fr13}\eps^3.
\eeq

Let us now focus on $\mathrm{DE}_2$. To this end, by \eqref{ref-Y''}, we further divide it into two parts:
\beno
\mathrm{DE}_2=-\mu\left\la AJ, A\left[\fr12\partial_Y(Y'-1)^2+\partial_Y(Y'-1)\right]\partial_Y^LJ\right\ra=\mathrm{DE}_2^{1}+\mathrm{DE}_2^{2}.
\eeno
First consider $\mathrm{DE}_2^{2}$, which can be spotted  by paraproduct decomposition
\beq
\mathrm{DE}^{2}_2\nn&=&-\mu\sum_{N\ge8}\int AJ A\left(\partial_Y(Y'-1)_{N}\partial_Y^LJ_{<\fr{N}{8}}\right)dXdY\\
\nn&&-\mu\sum_{N\in\mathbb{D}}\int AJ A\left(\partial_Y(Y'-1)_{<16N}\partial_Y^LJ_{N}\right)dXdY\\
&=&\mathrm{DE}^{2}_{2;\mathrm{HL}}+\mathrm{DE}^{2}_{2;\mathrm{LH}}.
\eeq
To bound $\mathrm{DE}^{2}_{2;\mathrm{HL}}$, we subdivide this integral according to whether or not $(k,\eta)$ and $(k,\xi)$ are resonant.
\beq\label{D-HL}
\nn \mathrm{DE}^{2}_{2;\mathrm{HL}}&=&-i\mu\sum_{N\ge8}\sum_{k\in\mathbb{Z}}\int_{\eta,\xi} (A\bar{\hat{J}})_k(\eta) \fr{A_k(\eta)}{A^R(\xi)}\xi \left(A^R\widehat{(Y'-1)}\right)(\xi)_{N}\widehat{\partial_Y^LJ}_k(\eta-\xi)_{<\fr{N}{8}}\\
\nn&&\times\left({\bf 1}_{t\in\mathbb{I}_{k,\eta}\cap\mathbb{I}_{k,\xi}}+{\bf 1}_{t\notin\mathbb{I}_{k,\eta}\cap\mathbb{I}_{k,\xi}}\right)d\xi d\eta\\
&=&\mathrm{DE}^{2;\mathrm{D}}_{2;\mathrm{HL}}+\mathrm{DE}^{2;*}_{2;\mathrm{HL}}.
\eeq 
Note that on the support of the integrand of  $\mathrm{DE}^{2}_{2;\mathrm{HL}}$, the frequency localizations \eqref{sptR1}--\eqref{sptR3} hold. As a result, one can use \eqref{AkR1} to deduce
\beq\label{e-ReD}
\nn \mathrm{DE}^{2;\mathrm{D}}_{2;\mathrm{HL}}&\les&\mu\sum_{N\ge8}\sum_{k\ne0}\int_{\eta,\xi} {\bf 1}_{t\in\mathbb{I}_{k,\eta}\cap\mathbb{I}_{k,\xi}}\left|\left(|\nabla_L|^{\fr12}A\hat{J}\right)_k(\eta)\right| \fr{|\xi|}{\left(|k|+|\eta-kt|\right)^\fr12} \left|\left(A^R\widehat{(Y'-1)}\right)(\xi)_{N}\right|\\
\nn&&\times e^{c\lm|k,\eta-\xi|^s}\left|\widehat{\partial_Y^LJ}_k(\eta-\xi)_{<\fr{N}{8}}\right|d\xi d\eta\\
\nn&\les&\mu\la t\ra\sum_{k\ne0}\int_{\eta,\xi} {\bf 1}_{t\in\mathbb{I}_{k,\eta}\cap\mathbb{I}_{k,\xi}}\left|\left(|\nabla_L|^\fr12A\hat{J}\right)_k(\eta)\right| \fr{|\xi|}{|k|}\sqrt{\fr{1+\left|t-\fr{\xi}{k}\right|}{1+\left|t-\fr{\eta}{k}\right|}} \fr{1}{\sqrt{1+\left|t-\fr{\xi}{k}\right|}}\left|\left(A^R\widehat{(Y'-1)}\right)(\xi)\right|\\
\nn&&\times e^{c\lm|k,\eta-\xi|^s}|k|^{\fr12}\left|\widehat{\nabla J}_k(\eta-\xi)\right|d\xi d\eta\\
\nn&\les&\mu\la t\ra^2\sum_{k\ne0}\int_{\eta,\xi} {\bf 1}_{t\in\mathbb{I}_{k,\eta}\cap\mathbb{I}_{k,\xi}}\left|\left(|\nabla_L|^{\fr12}A\hat{J}\right)_k(\eta)\right|  \sqrt{\fr{\partial_tw_R(\xi)}{w_R(\xi)}}\left|\left(A^R\widehat{(Y'-1)}\right)(\xi)\right|\\
\nn&&\times e^{c\lm|k,\eta-\xi|^s}\left|\widehat{\la\nabla\ra^2 J}_k(\eta-\xi)\right|d\xi d\eta\\
&\les&\mu\la t\ra^2\|J_{\ne}\|_{\mathcal{G}^{\lm,2}}\left\||\nabla_L|^{\fr12}AJ_{\ne}\right\|_{L^2}\left\|\sqrt{\fr{\partial_tw_R}{w_R}}A^R(Y'-1)\right\|_{L^2},
\eeq
where we have used the fact $t\approx\fr{|\xi|}{|k|}$ for $t\in\mathbb{I}_{k,\xi}$ and \eqref{pt_w}. By interpolation, it is easy to see that
\[
\left\||\nabla_L|^{\fr12}AJ_{\ne}\right\|_{L^2}\les\left\|\nabla_LAJ_{\ne}\right\|_{L^2}^{\fr12}\left\|AJ_{\ne}\right\|_{L^2}^{\fr12}.
\]
Combining this with \eqref{e-ReD}, the lossy estimate \eqref{loss3},  the upper bound of $m_k(t,\eta)$ in \eqref{m1} and the bootstrap hypotheses, we find that
\beq\label{e-ReD'}
\int_1^t\mathrm{DE}^{2;\mathrm{D}}_{2;\mathrm{HL}}dt'
\nn&\les&\left(\|\mathcal{A}F\|_{L^\infty L^2}+\mu^{-\fr13}\|A\Om\|_{L^\infty L^2}\right)\left\|\nabla_LAJ\right\|_{L^2L^2}^{\fr12}\left\|AJ_{\ne}\right\|_{L^2L^2}^{\fr12}\left\|\sqrt{\fr{\partial_tw_R}{w_R}}A^R(Y'-1)\right\|_{L^2L^2}\\
&\les&\eps^3\mu^{-1}.
\eeq

To bound $\mathrm{DE}^{2;*}_{2;\mathrm{HL}}$, we find that on the support of the integrand of $\mathrm{DE}^{2;*}_{2;\mathrm{HL}}$, there holds
\beno
\fr{{\bf 1}_{t\notin\mathbb{I}_{k,\eta}\cap\mathbb{I}_{k,\xi}}|\xi|}{|k|+|\eta-kt|}\les \la k, \eta-\xi\ra.
\eeno
This together with  \eqref{AkR2}  and \eqref{m1} implies that
\beq\label{e-Re*}
\nn \mathrm{DE}^{2;*}_{2;\mathrm{HL}}&\les&\mu\sum_{N\ge8}\sum_{k\in\mathbb{Z}}\int_{\eta,\xi} \left|\left(\nabla_LA\hat{J}\right)_k(\eta)\right|  \left|\left(A^R\widehat{(Y'-1)}\right)(\xi)_{N}\right|\\
\nn&&\times \la k, \eta-\xi\ra e^{c\lm|k,\eta-\xi|^s}\left|\widehat{\partial_Y^LJ}_k(\eta-\xi)_{<\fr{N}{8}}\right|d\xi d\eta\\
\nn&\les&\mu\|\nabla_LAJ\|_{L^2}\left\|A^R(Y'-1)\right\|_{L^2}\left\|\partial_Y^LJ\right\|_{\mathcal{G}^{\lm, 2}}\\
&\les&\mu^{-\fr13}\left(\mu\left\|\nabla_LAJ\right\|_{L^2}^2\right)\left\|A^R(Y'-1)\right\|_{L^2}.
\eeq

Finally, we  go to bound $\mathrm{DE}^{2}_{2;\mathrm{LH}}$:
\beno
 \mathrm{DE}^{2}_{2;\mathrm{LH}}=-i\mu\sum_{N\in\mathbb{D}}\sum_{k\in\mathbb{Z}}\int (A\bar{\hat{J}})_k(\eta) \fr{A_k(\eta)}{A_k(\xi)}(\eta-\xi)\widehat{(Y'-1)}(\eta-\xi)_{<16N}\left(A_k(\xi)\widehat{\partial_Y^LJ}_k(\xi)_{N}\right)d\xi d\eta.
\eeno
On the support of $\mathrm{DE}^{2}_{2;\mathrm{LH}}$, \eqref{spt<16N}--\eqref{AkAk} hold. Therefore,
\beq\label{e-Te}
\mathrm{DE}^{2}_{2;\mathrm{LH}}\nn&\les&\mu\sum_{k\in\mathbb{Z}}\int\left|(A{\hat{J}})_k(\eta)\right| e^{c\lm|\eta-\xi|^s}|\eta-\xi|\left|\widehat{(Y'-1)}(\eta-\xi)\right|\left|A\widehat{\partial_Y^LJ}_k(\xi)\right|d\xi d\eta\\
&\les&\mu\left(\|AJ_\ne\|_{L^2}+\|AJ_0\|_{L^2}\right)\left\|Y'-1\right\|_{\mathcal{G}^{c\lm, 2}}\left\|\partial_Y^LAJ\right\|_{L^2}.
\eeq
It follows from \eqref{e-Re*}, \eqref{e-Te}, the bootstrap hypotheses and \eqref{U0B0J0} that
\be\label{e-DE22}
\int_1^t\mathrm{DE}_{2;\mathrm{HL}}^{2;*}+\mathrm{DE}_{2;\mathrm{LH}}^{2}dt'\les\mu^{-\fr13}\eps^3.
\ee
The treatment of $\mathrm{DE}_2^1$ is analogous to \eqref{e-ReD'} and \eqref{e-DE22}. One subtle difference is that one should use the product estimate for $\sqrt{\fr{\partial_tw_R}{w_R}}A^R$ (see {\bf (3.39b)} of \cite{BM15}) to obtain
\be\label{pro-wA}
\left\|\sqrt{\fr{\partial_tw_R}{w_R}}A^R\left((Y'-1)^2\right)\right\|_{L^2}\les\|Y'-1\|_{\mathcal{G}^{\lm,2}}\left(\left\|\sqrt{\fr{\partial_tw_R}{w_R}}A^R(Y'-1)\right\|_{L^2}+\left\|\fr{|\partial_Y|^{\fr{s}{2}}}{\la t\ra^s}A^R(Y'-1)\right\|_{L^2}\right).
\ee
\subsubsection{Dissipation errors $\mathbf{CDE}$ from cross terms} Integrating by parts, similar to \eqref{de-m}, we have
\beq\label{de-c}
\nn\mathbf{CDE}&=&\fr{\mu}{\al }\left\la\fr{\partial_t(m^\fr12)}{m^{\fr12}} \partial_X^{-1}\partial_Y^LA\Om_{\ne}, A\left[\left((Y')^2-1\right)\partial_{Y}^LJ_{\ne}\right]\right\ra+\fr{\mu}{\al }\left\la\fr{\partial_t(m^\fr12)}{m^{\fr12}} \partial_X^{-1}A\Om_{\ne}, A\left(Y''\partial_Y^LJ_{\ne}\right)\right\ra\\
&=&\mathrm{CDE}_1+ \mathrm{CDE}_2.
\eeq
Recalling the definition of the multiplier $M_1$ in \eqref{M1}, thanks to \eqref{pt-m12} and \eqref{D1+D2}, and we find that
\be\label{635}
\left| {\bf 1}_{k\ne0}\fr{\partial_t(m^\fr12)}{m^{\fr12}}\right|\les\sqrt{\fr{\partial_t{M}^1}{M^1}},
\ee
and
\be\label{636}
\left| {\bf 1}_{k\ne0}\fr{\partial_t(m^\fr12)}{m^{\fr12}}\fr{|k,\eta-kt|}{|k|}\right|={\bf 1}_{k\ne0}\fr{|k|\left(t-\fr{\eta}{k}\right){\bf 1}_{D_1\cup D_2}}{\sqrt{k^2+(\eta-kt)^2}}\les\mu^{-\fr13}\sqrt{\fr{\partial_t{M}^1}{M^1}}.
\ee
Using \eqref{636}, similar to \eqref{e-DE1}, we arrive at
\beno
\mathrm{CDE}_1\les\fr{\mu^\fr23}{|\al|}\left\|\sqrt{\fr{\partial_t{M}^1}{M^1}}A\Om\right\|_{L^2}\left(\left\|(Y')^2-1\right\|_{\mathcal{G}^{c\lm,1}}\left\|\partial_Y^LAJ\right\|_{L^2}+\left\|A^R\left((Y')^2-1\right)\right\|_{L^2}\left\|\partial_Y^LJ\right\|_{\mathcal{G}^{c\lm,1}}\right).
\eeno
As for $\mathrm{CDE}_2$, in view of \eqref{pro-wA}, it is harmless to regard $Y''$ just as $\partial_Y(Y'-1)$. Then  by means of paraproduct decomposition, we use \eqref{636} when $Y''$ is at high frequency, and use \eqref{635} when $\partial_Y^LJ_\ne$ is at high frequency, similar to \eqref{e-ReD}, \eqref{e-Re*} and \eqref{e-Te}, we are led to
\beqno
\mathrm{CDE}_2\nn&\les& \mu^\fr23\la t\ra^2\|J_{\ne}\|_{\mathcal{G}^{\lm,2}}\left\|\sqrt{\fr{\partial_t{M}^1}{M^1}}A\Om\right\|_{L^2}\left\|\sqrt{\fr{\partial_tw_R}{w_R}}A^R(Y'-1)\right\|_{L^2}\\
\nn&&+\mu^\fr13\left\|\sqrt{\fr{\partial_t{M}^1}{M^1}}A\Om\right\|_{L^2}\left\|A^R(Y'-1)\right\|_{L^2}\left\|\partial_Y^LAJ\right\|_{L^2}\\
&&+\mu \left\|\sqrt{\fr{\partial_t{M}^1}{M^1}}A\Om\right\|_{L^2}\left\|Y'-1\right\|_{\mathcal{G}^{c\lm, 2}}\left\|\partial_Y^LAJ\right\|_{L^2}.
\eeqno
Using again \eqref{m1}, \eqref{loss3}, and the bootstrap hypotheses, we infer from the above two estimates that
\beq\label{e-CDE}
\nn\int_1^t{\bf CDE}dt'&\les&\mu^{-\fr13}\left(\|\mathcal{A}F\|_{L^\infty L^2}+\mu^{-\fr13}\|A\Om\|_{L^\infty L^2}\right)\left\|\sqrt{\fr{\partial_t{M}^1}{M^1}}A\Om\right\|_{L^2L^2}\left\|\sqrt{\fr{\partial_tw_R}{w_R}}A^R(Y'-1)\right\|_{L^2L^2}\\
&&+\mu^{-\fr16}\eps\left\|\sqrt{\fr{\partial_t{M}^1}{M^1}}A\Om\right\|_{L^2L^2}\left\|\partial_Y^LAJ\right\|_{L^2L^2}\les\mu^{-1}\eps^3.
\eeq

\subsection{Linear stretch errors}
To bound the linear stretch errors, we first establish the following lemma.
\begin{lem}\label{lem-nonres}
For $s\in(\fr12, 1)$, $t\ge1$, $l\ne0$ and $\xi\in\mathbb{R}$, there holds
\be\label{722}
\fr{{\bf1}_{t\notin\mathbb{I}_{l,\xi}}}{1+|t-\fr{\xi}{l}|}\les\left\la\fr{\xi}{lt}\right\ra^{-1}\min\left\{\fr{|l|^s}{t^s}, \fr{1}{t}+\fr{|\xi|^s}{t^{s+\fr12}}\right\}.
\ee
\end{lem}
\begin{proof}
The proof will be accomplished by considering the following three cases.\par
\noindent{\bf Case 1: $\fr{|\xi|}{|l|}\le\fr{t}{2}$}. Here we have $|\fr{\xi}{l}-t|\ge\fr{t}{2}$. Consequently,
\be\label{722a}
\fr{1}{1+|t-\fr{\xi}{l}|}\les\fr{1}{t}\les\left\la\fr{\xi}{lt}\right\ra^{-1}\fr{1}{t}.
\ee
{\bf Case 2: $\fr{|\xi|}{|l|}>2t$}. Now we have $|{\xi}-lt|\ge\fr{|\xi|}{2}$. Thus
\be\label{722b}
\fr{1}{1+|t-\fr{\xi}{l}|}=\fr{1}{1+|t-\fr{\xi}{l}|}\fr{|\xi|}{|lt|}\fr{|lt|}{|\xi|}=\fr{|\xi|}{|l|+|\xi-lt|}\fr{1}{t}\fr{|lt|}{|\xi|}\les \left\la\fr{\xi}{lt}\right\ra^{-1}\fr{1}{t}.
\ee
{\bf Case 3: $\fr{t}{2}\le\fr{|\xi|}{|l|}\le2t$}. If $\mathbb{I}_{l,\xi}\ne\emptyset$, then $t\notin\mathbb{I}_{l,\xi}$ implies that $|t-\fr{\xi}{l}|\gtr\fr{|\xi|}{|l|^2}$. Accordingly,
\be\label{648}
\fr{1}{1+|t-\fr{\xi}{l}|}\les\fr{|l|^2}{|\xi|}\les\left\la\fr{\xi}{lt}\right\ra^{-1}\fr{|l|}{t}.
\ee
On the other hand, the fact $\mathbb{I}_{l,\xi}\ne\emptyset$ implies that $|l|\les\sqrt{|\xi|}$. Combining this with the fact $t\le\fr{2|\xi|}{|l|}\le2|\xi|$, we infer from \eqref{648} that
\be\label{649}
\fr{1}{1+|t-\fr{\xi}{l}|}\les\left\la\fr{\xi}{lt}\right\ra^{-1}\fr{\sqrt{|\xi|}}{t}\left(\fr{2|\xi|}{t}\right)^{s-\fr12}\les\left\la\fr{\xi}{lt}\right\ra^{-1}\fr{|\xi|^s}{t^{s+\fr12}}.
\ee
If $\mathbb{I}_{l,\xi}=\emptyset$, then either $\sq{|\xi|}\les |l|$ or $l\xi<0$. For the case when $\sq{|\xi|}\les |l|$, we have $|l|\gtr\fr{|\xi|}{|l|}\gtr t$. Then
\be\label{650}
\fr{1}{1+|t-\fr{\xi}{l}|}\le1\les\left\la\fr{\xi}{lt}\right\ra^{-1}\fr{|l|}{t}.
\ee
On the other hand, under the restriction of this case, $\sq{|\xi|}\les |l|$ also implies that $t\les\fr{|\xi|}{|l|}\les\sqrt{|\xi|}$. Thus, instead of \eqref{650}, we have
\be\label{651}
\fr{1}{1+|t-\fr{\xi}{l}|}\le1\les\left(\fr{\sqrt{|\xi|}}{t}\right)^{2s}\les\left\la\fr{\xi}{lt}\right\ra^{-1}\fr{|\xi|^s}{t^{2s}}.
\ee
When $l\xi<0$, we find that
\be\label{652}
\fr{1}{1+|t-\fr{\xi}{l}|}=\fr{1}{1+t+|\fr{\xi}{l}|}\les\left\la\fr{\xi}{lt}\right\ra^{-1}\fr{1}{t}.
\ee

It follows from \eqref{722a}, \eqref{722b},\eqref{649}, \eqref{651} and \eqref{652} that
\be\label{e-nonres1}
\fr{{\bf1}_{t\notin\mathbb{I}_{l,\xi}}}{1+|t-\fr{\xi}{l}|}\les\left\la\fr{\xi}{lt}\right\ra^{-1} \left(\fr{1}{t}+\fr{|\xi|^s}{t^{s+\fr12}}\right).
\ee
Meanwhile, \eqref{722a}--\eqref{648}, \eqref{650} and \eqref{652}  show that
\be\label{e-nonres}
\fr{{\bf1}_{t\notin\mathbb{I}_{l,\xi}}}{1+|t-\fr{\xi}{l}|}\les\left\la\fr{\xi}{lt}\right\ra^{-1}\fr{|l|}{t}.
\ee
In addition, if $|l|\le t$, \eqref{e-nonres} implies that
\be\label{722'}
\fr{{\bf1}_{t\notin\mathbb{I}_{l,\xi}}}{1+|t-\fr{\xi}{l}|}\les\left\la\fr{\xi}{lt}\right\ra^{-1}\fr{|l|^s}{t^s}.
\ee
If $|l|>t$, one can see from \eqref{722a} and \eqref{722b} that \eqref{722'} still holds for the case when $\fr{|\xi|}{|l|}\le\fr{t}{2}$ or $\fr{|\xi|}{|l|}\ge2t$ since $t\ge1$. If $|l|>t$ and $\fr{t}{2}\le\fr{|\xi|}{|l|}\le2t$, then $1\les\left\la\fr{\xi}{lt}\right\ra^{-1}\fr{|l|^s}{t^s}$, and hence \eqref{722'} holds as well. Combining \eqref{e-nonres1} with \eqref{722'}  yields \eqref{722}. This completes the proof of Lemma \ref{lem-nonres}.
\end{proof}

Now we can treat $\mathbf{LSE}$ and $\mathbf{CLSE}$. To simplify the presentation, recalling the notations in \eqref{T12},  we write
\beno
\nn\mathbf{LSE}=2\left\la AJ, A\partial_{XY}^L\Dl_L^{-1}\left(\mathrm{T}^1+\mathrm{T^2} \right)\right\ra-2\left\la AJ, A\left((Y'-1)\partial_{XY}^L\Phi\right)\right\ra=\mathrm{LSE}_1+\mathrm{LSE}_2+\mathrm{LSE}_3,
\eeno
and
\beqno
\nn\mathbf{CLSE}&=&-\fr{2}{\al }\left\la\fr{\partial_t(m^\fr12)}{m^{\fr12}}\partial_X^{-1}A\Om_{\ne}, A\partial_{XY}^L\Dl_L^{-1}\left(\mathrm{T}^1+\mathrm{T2} \right)\right\ra+\fr{2}{\al }\left\la\fr{\partial_t(m^\fr12)}{m^{\fr12}}\partial_X^{-1}A\Om_{\ne}, A\left((Y'-1)\partial_{XY}^L\Phi_{\ne}\right)\right\ra\\
&=&\mathrm{CLSE}_1+\mathrm{CLSE}_2+\mathrm{CLSE}_3.
\eeqno
By virtue of Lemma \ref{lem-nonres}, recalling the definitions of $\mathcal{M}_1$ and $\mathcal{M}_2$ in \eqref{M12} (we still use the notation $\mathcal{M}_1$ if $s$ is replaced by $\fr{s}{2}+\fr14$ in \eqref{M12}), $\mathrm{LSE}_i, i=1, 2$, can be bounded as follows:
\beq\label{e-LSE12}
\mathrm{LSE}_i\nn&=&2\sum_{k\ne0}\int A\bar{\hat{J}}_k(\eta)A_k(\eta)\fr{k(\eta-kt)}{k^2+(\eta-kt)^2}\hat{\mathrm{T}}^i_k(\eta)d\eta\\
\nn&\les&\sum_{k\ne0}\int \left({\bf 1}_{t\in\mathbb{I}_{k,
\eta}}+{\bf 1}_{t\notin\mathbb{I}_{k,\eta}}\right)\left|A{\hat{J}}_k(\eta)\right|A_k(\eta)\fr{1}{1+\left|t-\fr{\eta}{k}\right|}\left|\hat{\mathrm{T}}^i_k(\eta)\right|d\eta\\
\nn&\les&\sum_{k\ne0}\int \left(\sqrt{\fr{\partial_tw_k(\eta)}{w_k(\eta)}}\left|\tl A{\hat{J}}_k(\eta)\right|\right)\left(\left\la \fr{\eta}{kt}\right\ra^{-1}\sqrt{\fr{\partial_tw_k(\eta)}{w_k(\eta)}}\tl A_k(\eta)\left|\hat{\mathrm{T}}^i_k(\eta)\right|\right)d\eta\\
\nn&&+\sum_{k\ne0}\int \left(\fr{|\eta|^\fr{s}{2}}{\la t\ra^{\fr{s}{2}+\fr14}}\left|A\hat{J}_k(\eta)\right|\right)\left(\left\la\fr{\eta}{kt}\right\ra^{-1}\fr{|\eta|^\fr{s}{2}}{\la t\ra^{\fr{s}{2}+\fr14}}A_k(\eta)\left|\hat{\mathrm{T}}^i_k(\eta)\right|\right)d\eta\\
\nn&&+\fr{1}{t}\sum_{k\ne0}\int \left|A\hat{J}_k(\eta)\right|\left(\left\la\fr{\eta}{kt}\right\ra^{-1}A_k(\eta)\left|\hat{\mathrm{T}}^i_k(\eta)\right|\right)d\eta\\
&\les&\left\|\sqrt{\fr{\partial_tw}{w}}\tl AJ\right\|_{L^2}\left\|\mathcal{M}_2\mathrm{T}^i\right\|_{L^2}+\fr{\left\||\nabla|^\fr{s}{2}AJ\right\|_{L^2}}{\la t\ra^{\fr{s}{2}+\fr14}}\left\|\mathcal{M}_1\mathrm{T}^i\right\|_{L^2}+\fr{\|AJ\|}{t}\left\|\left\la \fr{\partial_Y}{t\partial_X}\right\ra^{-1} A\mathrm{T}^i_{\ne}\right\|_{L^2}.
\eeq
It follows from \eqref{pre-ell2}, \eqref{pre-ell3}, \eqref{e-CCK} and the bootstrap hypotheses that
\beq\label{e-MT12}
\nn&&\int_1^t\sum_{i=1}^2\left(\left\|\mathcal{M}_1\mathrm{T}^i\right\|_{L^2}^2+\left\|\mathcal{M}_2\mathrm{T}^i\right\|_{L^2}^2\right)dt'\\
&\les&\eps^2\int_1^t\mathrm{CK}_{\lm,J}+\mathrm{CK}_{w,J}dt'+\mu^{-\fr23}\eps^2\int_1^t\mathrm{CK}^R_\lm+\mathrm{CK}^R_wdt'\les\mu^{-\fr23}\eps^4.
\eeq
On the other hand, by \eqref{pre-ell4}, \eqref{e-AT2}, \eqref{e-AT1} and the bootstrap hypotheses,
\beno
\sum_{i=1}^2\left\|\left\la\fr{\partial_Y}{t\partial_X}\right\ra^{-1}A\mathrm{T}^i_{\ne}\right\|_{L^2}\les\mu^{-\fr13}\eps\|AJ_{\ne}\|_{L^2}.
\eeno
Substituting the above two inequalities into \eqref{e-LSE12}, we find that
\beq
\sum_{i=1}^2\int_1^t\mathrm{LSE}_idt'\les\mu^{-\fr13}\eps^2\left(\int_1^t\mathrm{CK}_{\lm,J}+\mathrm{CK}_{w,J}dt'\right)^\fr12+\mu^{-\fr13}\eps^2\|AJ_{\ne}\|_{L^2L^2}\les\mu^{-\fr12}\eps^3.
\eeq
The errors $\mathrm{CLSE}_{i}$, $i=1, 2$ can be bounded in a similar manner. Indeed, in view of \eqref{635},  \eqref{722'} and \eqref{e-MT12}, we are led to
\beq\label{e-CLSEi}
\sum_{i=1}^2\int_1^t\mathrm{CLSE}_idt'\les\fr{1}{|\al|}\left\|\sqrt{\fr{\partial_tM^1}{M^1}}A\Om \right\|_{L^2L^2}\left(\left\|\mathcal{M}_1\mathrm{T}^i\right\|_{L^2L^2}+\left\|\mathcal{M}_2\mathrm{T}^i\right\|_{L^2L^2}\right)\les\mu^{-\fr13}\eps^3.
\eeq

To bound $\mathrm{LSE}_3$, by paraproduct decomposition, we write 
\beq
\mathrm{LSE}_3\nn&=&-2\sum_{N\in\mathbb{D}}\left\la AJ, A\left((Y'-1)_{<16N}\partial_{XY}^L\Phi_{N}\right)\right\ra-2\sum_{N\ge8}\left\la AJ, A\left((Y'-1)_N\partial_{XY}^L\Phi_{<\fr{N}{8}}\right)\right\ra\\
\nn&=&\mathrm{LSE}_3^{\mathrm{LH}}+\mathrm{LSE}_3^{\mathrm{HL}}.
\eeq
Like \eqref{e-CLSEi}, we use \eqref{722'} (together with \eqref{AkAk}) instead of \eqref{e-nonres1} to treat $\mathrm{LSE}_3^{\mathrm{LH}}$:
\beq
\mathrm{LSE}_3^{\mathrm{LH}}\nn&\les&\sum_{N\ge8}\sum_{k\ne0}\int_{\eta,\xi} \left|(A\hat{J})_k(\eta)\right|\fr{A_k(\eta)}{A_k(\xi)}\left|(\widehat{Y'-1})_{<16N}\right|\fr{{\bf1}_{t\in\mathbb{I}_{k,\xi}}+{\bf 1}_{t\notin\mathbb{I}_{k,\xi}}}{1+\left|t-\fr{\xi}{k} \right|}\left|A_k(\xi)\widehat{\Dl_L\Phi}_k(\xi)_N \right|d\xi d\eta\\
\nn&\les&\sum_{k\ne0}\int_{\eta,\xi}\left|(A\hat{J})_k(\eta)\right|e^{c\lm|\eta-\xi|^s}\left|(\widehat{Y'-1})\right| \left\la \fr{\xi}{kt}\right\ra^{-1}\sqrt{\fr{\partial_tw_k(\xi)}{w_k(\xi)}}\left|\tl A_k(\xi)\widehat{\Dl_L\Phi} _k(\xi)\right|d\xi d\eta\\
\nn&&+\sum_{k\ne0}\int_{\eta,\xi}\left|k^{\fr{s}{2}}(A\hat{J})_k(\eta)\right|e^{c\lm|\eta-\xi|^s}\left|(\widehat{Y'-1})\right| \left\la \fr{\xi}{kt}\right\ra^{-1}\fr{|k|^\fr{s}{2}}{\la t\ra^s}\left| A_k(\xi)\widehat{\Dl_L\Phi} _k(\xi)\right|d\xi d\eta\\
\nn&\les&\left\| |\partial_X|^{\fr{s}{2}} AJ_\ne\right\|_{L^2}\|Y'-1\|_{\mathcal{G}^{c\lm,1}}\left(\left\|\mathcal{M}_1\Dl_L\Phi\right\|_{L^2}+\left\|\mathcal{M}_2\Dl_L\Phi\right\|_{L^2}\right).
\eeq
In view of \eqref{pre-ell2}, \eqref{e-CCK} and the bootstrap hypotheses, similar to \eqref{e-MT12}, we have
\beq\label{e-MDlPhi}
\nn&&\int_1^t \left\|\mathcal{M}_1\Dl_L\Phi\right\|_{L^2}^2+\left\|\mathcal{M}_2\Dl_L\Phi\right\|_{L^2}^2dt'\\
&\les&\int_1^t\mathrm{CK}_{\lm,J}+\mathrm{CK}_{w,J}dt'+\mu^{-\fr23}\eps^2\int_1^t\mathrm{CK}^R_\lm+\mathrm{CK}^R_wdt'\les \left(1+\mu^{-\fr23}\eps^2\right)\eps^2.
\eeq
 The above inequalities then imply that (for $\eps\ll \mu^\fr13$)
\be\label{LSE-LH}
\int_1^t\mathrm{LSE}_3^{\mathrm{LH}}dt'\les\eps^2\|\nabla_LAJ\|_{L^2L^2}\les\mu^{-\fr12}\eps^3.
\ee
The treatment of $\mathrm{LSE}_3^{\mathrm{HL}}$ is analogous to \eqref{522} (the recovery of one derivative on the coefficient $Y'-1$ is not necessary):
\beq\label{LSE-HL}
\int_1^t\mathrm{LSE}_3^{\mathrm{HL}}dt'\nn&\les&\int_1^t\left\|A J_{\ne}\right\|_{L^2}\left\|A^R(Y'-1)\right\|_{L^2}\left\|\partial_{XY}^L\Phi\right\|_{\mathcal{G}^{c\lm, 1}}dt'\\
&&\les\mu^{-\fr13}\eps^2\int_1^t\fr{\|AJ_{\ne}\|_{L^2}}{\la t'\ra}dt'\les\mu^{-\fr12}\eps^3,
\eeq
where we have used the lossy estimate \eqref{loss2} and the upper bound of $m_k(t,\eta)$ in \eqref{m1}. As for $\mathrm{CLSE}_3$, just as what has been done in \eqref{e-CLSEi}, the extra derivative $|\partial_X|^{\fr{s}{2}}$ stems from \eqref{722'} can be absorbed by $\partial_X^{-1}$. Then using again \eqref{635}, following the treatment of $\mathrm{LSE}_3$ with $\left\|\sqrt{\fr{\partial_tM^1}{M^1}}A\Om \right\|_{L^2L^2}$ replacing  the role of $\|\nabla_LAJ\|_{L^2L^2}$ and $\|AJ_{\ne}\|_{L^2L^2}$ in \eqref{LSE-LH} and \eqref{LSE-HL}, respectively,  we arrive at
\beq
\int_1^t\mathrm{CLSE}_3dt'\nn&\les&\mu^{-\fr13}\eps^2 \left\| \sqrt{\fr{\partial_tM^1}{M^1}} A\Om\right\|_{L^2L^2}\les\mu^{-\fr13}\eps^3.
\eeq

\section{Estimates of  $\mathbf{Com}[B_0]_g$, $\mathbf{CCom}[B_0]_g$ and $\mathbf{CComE}$ with $g\in\{\Om, J\}$}\label{sec-zero-mode}
In this section, we deal with the interactions  between the of the zero mode $B_0^1$ and the non-zero modes $\partial_X\Om$ and $\partial_XJ$. To begin with, let us focus on ${\bf Com}[B_0]_g$, $g\in\{\Om, J\}$.
By paraproduct decomposition, 
\beq
\nn{\bf Com}[B_0]_{\Om}&=&\left\la AJ_\ne, A\left(B^1_0\partial_X\Om\right)-B^1_0\partial_XA\Om\right\ra\\
\nn&=&\sum_{N\ge8}\sum_{k\ne0}\int_{\eta,\xi} ik\left(\fr{A_k(\eta)}{A_k(\xi)}-1\right)\hat{B}^1_0(\eta-\xi)_{<\fr{N}{8}}A_k(\xi)\hat{\Om}_k(\xi)_NA\bar{\hat J}_k(\eta)d\xi d\eta\\
\nn&&+\sum_{N\ge 8}\sum_{k\ne0}\int_{\eta,\xi}ik\left({A_k(\eta)}-{A_k(\eta-\xi)}\right)\hat{B}^1_0(\xi)_N\hat{\Om}_k(\eta-\xi)_{<\fr{N}{8}}A\bar{\hat J}_k(\eta)d\xi d\eta\\
\nn&&+\sum_{N\in\mathbb{D}}\sum_{\fr{N}{8}\le N'\le8N}\sum_{k\ne0}\int_{\eta,\xi}ik\left({A_k(\eta)}-{A_k(\eta-\xi)}\right)\hat{B}^1_0(\xi)_N\hat{\Om}_k(\eta-\xi)_{N'}A\bar{\hat J}_k(\eta)d\xi d\eta\\
&=&{\bf Com}[B_0]_{\Om}^{\mathrm{LH}}+{\bf Com}[B_0]_{\Om}^{\mathrm{HL}}+{\bf Com}[B_0]_{\Om}^{\mathcal{R}}.
\eeq
We first treat ${\bf Com}[B_0]_{\Om}^{\mathrm{LH}}$, the main idea is to obtain a homogeneous derivative $\partial_Y$ from the commutator (instead of $\langle \partial_Y\rangle$) and to use the dissipation of $\partial_YB_0$, see \eqref{U0B0J0}. It requires a new estimate of $\fr{\mathcal{J}_k(\eta)}{\mathcal{J}_k(\xi)}-1$, see Lemma \ref{lem-com-J}. A direct calculation gives
\beq\label{com-Akk}
\fr{A_k(\eta)}{A_k(\xi)}-1\nn&=&\left(e^{\lm(|k,\eta|^s-|k,\xi|^s)}-1\right)+e^{\lm(|k,\eta|^s-|k,\xi|^s)}\left(\fr{\mathcal{J}_k(\eta)}{\mathcal{J}_k(\xi)}-1\right)\fr{\mathcal{M}_k(\eta)}{\mathcal{M}_k(\xi)}\fr{\la k,\eta\ra^\sigma}{\la k,\xi\ra^\sigma}\\
&&+e^{\lm(|k,\eta|^s-|k,\xi|^s)}\left(\fr{\mathcal{M}_k(\eta)}{\mathcal{M}_k(\xi)}-1\right)\fr{\la k,\eta\ra^\sigma}{\la k,\xi\ra^\sigma}+e^{\lm(|k,\eta|^s-|k,\xi|^s)}\left(\fr{\la k,\eta\ra^\sigma}{\la k,\xi\ra^\sigma}-1\right).
\eeq
On the support of the integrand of ${\bf Com}[B_0]_{\Om}^{\mathrm{LH}}$, the following frequency localization hold:
\beno
\fr{N}{2}\le |k,\xi|\le\fr{3N}{2},\quad |\eta-\xi|\le\fr{3N}{32}.
\eeno
Hence
\be\label{suppT0}
\big| |k,\eta|-|k,\xi| \big|\le|\eta-\xi|\le\fr{3}{16}|k,\xi|.
\ee
This enables us to  use the elementary inequality $|e^x-1|\le |x|e^{|x|}$, \eqref{ap1} and \eqref{ap3} to deduce that
\beq
\nn\left|e^{\lm(|k,\eta|^s-|k,\xi|^s)}-1\right|&\les&\lm \left| |k,\eta|^s-|k,\xi|^s\right|e^{\lm \big| |k,\eta|^s-|k,\xi|^s\big|}\\
\nn&\les&\lm\fr{|\eta-\xi|}{|k,\eta|^{1-s}+|k,\xi|^{1-s}}e^{c'\lm |\eta-\xi|^s},
\eeq
for some $c'\in(0,1)$.
Thanks to Lemmas \ref{lem-com-J} and \ref{lem-com-m}, we are led to 
\be
\nn\left|\fr{\mathcal{J}_k(\eta)}{\mathcal{J}_k(\xi)}-1\right|\fr{\mathcal{M}_k(\eta)}{\mathcal{M}_k(\xi)}\les|\eta-\xi|\la\eta-\xi\ra e^{20r|\eta-\xi|^\fr12}.
\ee
Next, using Lemmas \ref{lem-pr_eta} and \ref{lem-com-m}, we find that for $k\ne0$,
\beq
\nn\left|\fr{\mathcal{M}_k(\eta)}{\mathcal{M}_k(\xi)}-1\right|&=&\fr{\left|\left(M^1M^\mu m^\fr12\right)_k(\xi)-\left(M^1M^\mu m^\fr12\right)_k(\eta)\right|}{\left(M^1M^\mu m^\fr12\right)_k(\eta)}\\
\nn&\le&\fr{\left(M^1M^\mu\right)_k(\xi)}{\left(M^1M^\mu\right)_k(\eta)}\cdot\fr{\left|m^\fr12_k(\xi)-m^\fr12_k(\eta) \right|}{m^\fr12_k(\eta)}+\fr{\left|\left(M^1M^\mu\right)_k(\xi)-\left(M^1M^\mu\right)_k(\eta)\right|}{\left(M^1M^\mu\right)_k(\eta)}\\
\nn&\les&|\eta-\xi|\left|\fr{\partial_\eta m^\fr12_k(\zeta)}{m^\fr12_k(\zeta)}\right|\fr{m^\fr12_k(\zeta)}{m^\fr12_k(\eta)}+\fr{|\eta-\xi|}{|k|}\les|\eta-\xi|\fr{\la\eta-\xi\ra}{|k|},
\eeq
where $\zeta=\eta+\vartheta(\xi-\eta)$ for some $\vartheta\in(0,1)$. Moreover, using \eqref{suppT0} and the mean value theorem, it is easy to verify that
\be
\nn\left|\fr{\la k,\eta\ra^\sigma}{\la k,\xi\ra^\sigma}-1\right|\les\fr{|\eta-\xi|}{\la k,\xi\ra}.
\ee
Collecting the above estimates, and using \eqref{ap-5} and \eqref{ap-6}, one deduces that
\be
|k|\left|\fr{A_k(\eta)}{A_k(\xi)}-1\right|\les |k||\eta-\xi|e^{c|\eta-\xi|^s}, \quad \mathrm{for \ \ some\ \ } c\in(c',1).
\ee
Accordingly, in view of the bootstrap hypotheses and \eqref{U0B0J0},
\beq\label{ComOmLH}
\nn\int_1^t {\bf Com}[B_0]_{\Om}^{\mathrm{LH}}dt'&\les&\sum_{N\ge8}\sum_{k\ne0}\int_1^t\int_{\xi,\eta} |\eta-\xi|e^{c|\eta-\xi|^s}\left|\hat{B}^1_0(\eta-\xi)_{<\fr{N}{8}}\right| \left|A_k(\xi)\hat{\Om}_k(\xi)_N\right|\left|kA{\hat J}_k(\eta)\right|d\xi d\eta dt'\\
\nn&\les&\sum_{k\ne0}\int_1^t\int_{\xi,\eta} \left(e^{c|\eta-\xi|^s}\left|\partial_Y\hat{B}^1_0(\eta-\xi)\right|\right) \left|A\hat{\Om}_k(\xi)\right|\left|kA{\hat J}_k(\eta)\right|d\xi d\eta dt'\\
&\les&\mu^{-1}\left(\mu^\fr12\|\partial_YB^1_0\|_{L^2\mathcal{G}^{c\lm, 1}}\right)\left(\mu^\fr12\|\partial_XAJ\|_{L^2L^2}\right)\|A\Om\|_{L^\infty L^2}\les\mu^{-1}\eps^3.
\eeq

Now we turn to investigate ${\bf Com}[B_0]_{\Om}^{\mathrm{HL}}$.
On the support of the integrand of ${\bf Com}[B_0]_{\Om}^{\mathrm{HL}}$, the frequency localizations \eqref{sptR1}--\eqref{sptR3} hold. Then by virtue of Lemmas \ref{lem-wNR-ratio}, \ref{lem-J-ratio},\eqref{wRwNR} and  \eqref{ap3}, we find that
\beq
\fr{A_k(\eta)}{A_0(\xi)}\nn&=&\fr{e^{\lm|k,\eta|^s}\la k,\eta\ra^\sigma\mathcal{M}_k(\eta)}{e^{\lm|\xi|^s}\la \xi\ra^\sigma }\fr{\mathcal{J}_k(\eta)}{\mathcal{J}_0(\xi)}\left({\bf 1}_{t\in\mathbb{I}_{k,\eta}}+{\bf 1}_{t\notin\mathbb{I}_{k,\eta}}\right)\\
\nn&\les&e^{c\lm|k,\eta-\xi|^s}\left({\bf 1}_{t\in\mathbb{I}_{k,\eta}}\fr{w_{NR}(\xi)}{w_R(\eta)}e^{2r|\eta-\xi|^\fr12}+\fr{e^{2r|k|^\fr12}}{e^{2r\la\xi\ra^\fr12}}w_0(\xi)+e^{10r|k,\eta-\xi|^\fr12}\right)\\
\nn&\les&e^{c\lm|k,\eta-\xi|^s}e^{10r|k,\eta-\xi|^\fr12}\left(\fr{|\eta|}{|k|^2}\fr{{\bf 1}_{t\in\mathbb{I}_{k,\eta}}}{1+\left|\fr{\eta}{k}-t\right|}+1\right),
\eeq
for some $c\in(0,1)$. Note that \eqref{sptR3} implies that 
\[
|k|\left(\fr{|\eta|}{|k|^2}\fr{{\bf 1}_{t\in\mathbb{I}_{k,\eta}}}{1+\left|\fr{\eta}{k}-t\right|}+1\right)\les |\xi|.
\]
Combining the above two inequalities with \eqref{m1}, and using the bootstrap hypotheses and \eqref{U0B0J0}, we are led to 
\beq\label{ComOmHL}
\int_1^t{\bf Com}[B_0]_{\Om}^{\mathrm{HL}}dt'
\nn&\les&\sum_{N\ge8}\sum_{k\ne0}\int_1^t\int_{\eta,\xi}|\xi|\left|\hat{B}^1_0(\xi)_N\right|\left|A\hat{\Om}_k(\eta-\xi)_{<\fr{N}{8}}\right|\left|A{\hat J}_k(\eta)\right|d\xi d\eta dt'\\
\nn&&+\sum_{N\ge8}\sum_{k\ne0}\int_1^t\int_{\eta,\xi}|\xi|\left|A_0(\xi)\hat{B}^1_0(\xi)_N\right|e^{\lm|k,\eta-\xi|^s}\left|\hat{\Om}_k(\eta-\xi)_{<\fr{N}{8}}\right|\left|A{\hat J}_k(\eta)\right|d\xi d\eta dt'\\
\nn&\les&\|A J_\ne\|_{L^2L^2}\left(\left\|\partial_YB_0^1\right\|_{L^2H^1}\|A\Om\|_{L^\infty L^2}+\left\|A\partial_YB^1_0\right\|_{L^2L^2}\left\|\Om\right\|_{L^\infty\mathcal{G}^{\lm, 2}}\right)\\
&\les&\mu^{-\fr13}\|A J_\ne\|_{L^2L^2}\left\|A\partial_YB_0^1\right\|_{L^2L^2}\|A\Om\|_{L^\infty L^2}\les\mu^{-1}\eps^3.
\eeq

Next turn to ${\bf Com}[B_0]_{\Om}^{\mathcal{R}}$. On the support of the integrand of ${\bf Com}[B_0]_{\Om}^{\mathcal{R}}$, there hold
\beno
\fr{N}{2}\le |\xi|\le\fr{3N}{2},\quad \fr{N}{16}\le\fr{N'}{2}\le|k,\eta-\xi|\le\fr{3}{2}N'\le12N.
\eeno
This implies that
\beno
\fr{1}{24}\le\fr{|\xi|}{|k,\eta-\xi|}\le 24.
\eeno
Then we infer from \eqref{ap4} that
\beno
e^{\lm|k,\eta|^s}\les e^{c'\lm |\xi|^s}e^{c'\lm |k,\eta-\xi|^s},\quad\mathrm{for}\ \ \mathrm{some}\ \  c'\in(0,1).
\eeno
It follows from  the above two inequalities,  \eqref{J-ratio},  \eqref{com-m}, \eqref{ap-5} and \eqref{ap-6} that
\beq
\nn&&|k|\big(A_k(\eta-\xi)+A_k(\eta)\big)\\
\nn&\les& |\xi|A_k(\eta-\xi)+\left(e^{c'\lm |\xi|^s}\la\xi\ra^\sigma |\xi|\right) \left(e^{c'\lm|k,\eta-\xi|^s}\mathcal{M}_k(\eta-\xi)\mathcal{J}_k(\eta-\xi)\right)\fr{\mathcal{M}_k(\eta)}{\mathcal{M}_k(\eta-\xi)}\fr{\mathcal{J}_k(\eta)}{\mathcal{J}_k(\eta-\xi)}\\
\nn&\les&|\xi|A_k(\eta-\xi)+\left(e^{\lm |\xi|^s}\la\xi\ra^\sigma |\xi|\right) \left(e^{\lm|k,\eta-\xi|^s}\mathcal{M}_k(\eta-\xi)\mathcal{J}_k(\eta-\xi)\right).
\eeq
Thus, similar to \eqref{ComOmHL} (without resorting to \eqref{m1}), we obtain
\beq\label{ComOmR}
\int_1^t{\bf Com}[B_0]_{\Om}^{\mathcal{R}}dt'
\nn&\les&\sum_{N\in\mathbb{D}}\sum_{\fr{N}{8}\le N'\le8N}\sum_{k\ne0}\int_1^t\int_{\eta,\xi}|\xi|\left|\hat{B}^1_0(\xi)_N\right|\left|A\hat{\Om}_k(\eta-\xi)_{N'}\right|\left|A{\hat J}_k(\eta)\right|d\xi d\eta dt'\\
\nn&&+\sum_{N\in\mathbb{D}}\sum_{\fr{N}{8}\le N'\le8N}\sum_{k\ne0}\int_1^t\int_{\eta,\xi}\left(e^{\lm |\xi|^s}\la\xi\ra^\sigma |\xi|\left|\hat{B}^1_0(\xi)_N\right|\right)\\
\nn&&\qquad\qquad\qquad\qquad\times\left(e^{\lm|k,\eta-\xi|^s}\left|\mathcal{M}\mathcal{J}\hat{\Om}_k(\eta-\xi)_{N'}\right|\right)\left|A{\hat J}_k(\eta)\right|d\xi d\eta dt'\\
&\les&\|A J_\ne\|_{L^2L^2}\left\|A\partial_YB_0^1\right\|_{L^2L^2}\|A\Om\|_{L^\infty L^2} \les \mu^{-\fr23}\eps^3.
\eeq

The treatment of ${\bf Com}[B_0]_J$ is analogous and omitted. As for ${\bf CCom}[B_0]_\Om$ and ${\bf CCom}[B_0]_J$, thanks to \eqref{635}, similar to \eqref{ComOmLH}, \eqref{ComOmHL} and \eqref{ComOmR}, we arrive at
\beq
\nn&&\int_1^t{\bf CCom}[B_0]_\Om+{\bf CCom}[B_0]_Jdt'\\
\nn&\les&\fr{\mu^{-\fr13}}{|\al|}\left\|A\partial_YB^1_0 \right\|_{L^2L^2}\left(\left\|\sqrt{\fr{\partial_tM^1}{M^1}}A\Om\right\|_{L^2L^2}\|A\Om\|_{L^\infty L^2}+\left\|\sqrt{\fr{\partial_tM^1}{M^1}}AJ\right\|_{L^2L^2}\|AJ\|_{L^\infty L^2}\right)\\
&\les&\mu^{-\fr56}\eps^3.
\eeq

We are left to bound ${\bf CComE}$. To this end, using \eqref{pt-m12} and \eqref{D1+D2}, we find that
\beq\label{ptm-gain}
\fr{\partial_t(m)^\fr12}{m^\fr12}\nn&=&{\bf 1}_{D_1\cup D_2}\fr{k^2\left(t-\fr{\eta}{k}\right)}{k^2+(\eta-kt)^2}\les\mu^{-\fr13}\fr{k^2}{k^2+(\eta-kt)^2}\\
&\les&\mu^{-\fr13}\la\eta-\xi\ra\sqrt{\fr{\partial_tM^1_k(\eta)}{M^1_k(\eta)}}\sqrt{\fr{\partial_tM^1_k(\xi)}{M^1_k(\xi)}}.
\eeq
Combining this with the bootstrap hypotheses and \eqref{U0B0J0} yields
\beq
\int_1^t{\bf CComE}dt'\les\fr{\mu^{-\fr13}}{|\al|}\left\|B^1_0\right\|_{L^\infty H^2}\left(\left\|\sqrt{\fr{\partial_tM^1}{M^1}}A\Om\right\|_{L^2L^2}^2+\left\|\sqrt{\fr{\partial_tM^1}{M^1}}AJ\right\|_{L^2L^2}^2\right)\les \mu^{-\fr13}\eps^3.
\eeq

\section{Estimates of $\mathbf{Com}[V]_g$ and $\mathbf{Com}[B_{\ne}]_g$ with $g\in\{\Om, J\}$}\label{sec-NLT}
The commutators will be treated by paraproduct decomposition. For ${\bf Com}[V]_g$, $g\in\{\Om, J\}$, we write 
\be
{\bf Com}[V]_g=\fr{1}{2\pi}\sum_{N\ge8}T_N[V]_g+\fr{1}{2\pi}\sum_{N\ge8}R_N[V]_g+\fr{1}{2\pi}\mathcal{R}[V]_g,
\ee
with
\beq
T_N[V]_g\nn&=&2\pi\int Ag\left[A(V_{<\fr{N}{8}}\cdot\nabla g_N)-V_{<\fr{N}{8}}\cdot\nabla Ag_N\right]dXdY,\\
R_N[V]_g\nn&=&2\pi\int Ag\left[A(V_{N}\cdot\nabla g_{<\fr{N}{8}})-V_{N}\cdot\nabla Ag_{<\fr{N}{8}}\right]dXdY,\\
\mathcal{R}[V]_g\nn&=&2\pi\sum_{N\in\mathbb{D}}\sum_{\fr{N}{8}\le N'\le8N}\int Ag\left[A(V_{N}\cdot\nabla g_{N'})-V_{N}\cdot\nabla Ag_{N'}\right]dXdY.
\eeq
The other two commutators $\mathbf{Com}[B_{\ne}]_\Om$ and $\mathbf{Com}[B_{\ne}]_J$ can be decomposed in  the same manner.
\subsection{Transport}\label{sec-transport}
In this section, we deal with the contributions of the transport part of the nonlinearities ${\bf Com}[V]_g $, and ${\bf Com}[B_{\ne}]_g$ with $g\in\left\{\Om, J\right\}$. 
We would like to remark that the methods to treat the above four commutators are essentially the same. Let us just take $T_N[V]_g$ as an example and give the details below. To begin with, like \eqref{com-Akk}, we have
\beq
\fr{A_k(\eta)}{A_l(\xi)}-1\nn&=&\left(e^{\lm(|k,\eta|^s-|l,\xi|^s)}-1\right)+e^{\lm(|k,\eta|^s-|l,\xi|^s)}\left(\fr{\mathcal{J}_k(\eta)}{\mathcal{J}_l(\xi)}-1\right)\fr{\mathcal{M}_k(\eta)}{\mathcal{M}_l(\xi)}\fr{\la k,\eta\ra^\sigma}{\la l,\xi\ra^\sigma}\\
\nn&&+e^{\lm(|k,\eta|^s-|l,\xi|^s)}\left(\fr{\mathcal{M}_k(\eta)}{\mathcal{M}_l(\xi)}-1\right)\fr{\la k,\eta\ra^\sigma}{\la l,\xi\ra^\sigma}+e^{\lm(|k,\eta|^s-|l,\xi|^s)}\left(\fr{\la k,\eta\ra^\sigma}{\la l,\xi\ra^\sigma}-1\right).
\eeq
Then we write
\be
T_N[V]_g=T_{N,1}[V]_g+T_{N,2}[V]_g+T_{N,3}[V]_g+T_{N,4}[V]_g.
\ee
Note that on the support of the integrand of $T_N[V]_g$, the following frequency localization hold:
\beno
\fr{N}{2}\le|l,\xi|\le\fr{3}{2}N,\quad\mathrm{and}\quad |k-l,\eta-\xi|\le\fr{3N}{32}.
\eeno
Consequently,
\be\label{spt-T}
\big||k,\eta|-|l,\xi|\big|\le|k-l,\eta-\xi|\le\fr{3}{16}|l,\xi|.
\ee
Clearly, $T_{N,1}[V]_g$ and $T_{N,4}[V]_g$ can be treated in the same way as those corresponding terms in \cite{BM15}. We thus omit the details and conclude
\be\label{TN14}
\sum_{N\ge8}\left(T_{N,1}[V]_g+T_{N,4}[V]_g\right)
\les\left\|V\right\|_{\mathcal{G}^{\lm,2}}\left(\left\||\nabla|^\fr{s}{2}Ag\right\|_{L^2}^2+\left\|Ag\right\|_{L^2}^2\right).
\ee
As for $T_{N,2}[V]_g$, compared with the corresponding term in \cite{BM15}, the factor $\fr{\mathcal{M}_k(\eta)}{\mathcal{M}_l(\xi)}$ gives an extra  loss 
\be\label{M-ratio}
\fr{\mathcal{M}_k(\eta)}{\mathcal{M}_l(\xi)}\les\mu^{-\fr13},
\ee
due to \eqref{m1}. Combining  this with the proof in \cite{BM15}, one directly deduces that
\beq\label{TN2}
\sum_{N\ge8}T_{N,2}[V]_g\nn&\les&\mu^{-\fr13}\left\|V\right\|_{\mathcal{G}^{\lm,2}}\left(\left\||\nabla|^\fr{s}{2}Ag\right\|_{L^2}^2+\left\|Ag\right\|_{L^2}^2\right)+\mu^{-\fr13}t^2\left\|V_{\ne}\right\|_{\mathcal{G}^{\lm,2}}\left\|\sqrt{\fr{\partial_tw}{w}}\tl Ag\right\|_{L^2}^2\\
&&+\mu^{-\fr13}t^{2-2s}\left\|V\right\|_{\mathcal{G}^{\lm,2}}\left\||\nabla|^\fr{s}{2}Ag\right\|_{L^2}^2.
\eeq

We are left to bound $T_{N,3}[V]_g$. Since $M^1_k(t,\eta)$ and $M^\mu_k(t,\eta)$ are bounded from above and below by positive constants, one easily deduces that
\beq
\left|\fr{\mathcal{M}_k(t,\eta)}{\mathcal{M}_l(t,\xi)}-1\right|
\nn&\le&\fr{\left(M^1M^\mu\right)_l(t,\xi)\left|m^\fr12_l(t,\xi)-m^\fr12_k(t,\eta) \right|}{\left(M^1M^\mu m^\fr12\right)_k(t,\eta)}+\fr{\left|\left(M^1M^\mu\right)_l(t,\xi)-\left(M^1M^\mu\right)_k(t,\eta) \right|}{\left(M^1M^\mu \right)_k(t,\eta)}\\
\nn&\les&\left|\fr{m^\fr12_l(t,\xi)}{m^\fr12_k(t,\eta)}-1 \right|+\left|M^1_l(t,\xi)-M^1_k(t,\eta) \right|+\left|M^\mu_l(t,\xi)-M^\mu_k(t,\eta) \right|\\
&=&{\bf com}_{\sqrt{m}}+{\bf com}_1+{\bf com}_{\mu}.
\eeq
Then from Lemmas \ref{lem-com-sqm}, \ref{lem-com-mu},  Corollary \ref{coro-com-1}, \eqref{spt-T} and \eqref{ap3},  we are led to
\beq\label{TN3}
\sum_{N\ge8}T_{N,3}[V]_g
\nn&\les&\sum_{N\ge8}\sum_{k,l}\int_{\eta,\xi}|l,\xi|\left({\bf com}_{\sqrt{m}}+{\bf com}_1+{\bf com}_{\mu}\right)\left|A{\hat{g}}_k(\eta)\right|\left|A\hat{g}_l(\xi)_N\right|\\
\nn&&\times e^{c\lm |k-l,\eta-\xi|^s}|\hat{V}_{k-l}(\eta-\xi)_{<N/8}|  d\xi d\eta\\
\nn&\les&\mu^{-\fr13}\sum_{N\ge8}\sum_{k,l}\int_{\eta,\xi}\left|A{\hat{g}}_k(\eta)\right|\left|A\hat{g}_l(\xi)_N\right|\\
\nn&&\times e^{c\lm |k-l,\eta-\xi|^s}\la k-l,\eta-\xi\ra^3|\hat{V}_{k-l}(\eta-\xi)_{<\fr{N}{8}}|  d\xi d\eta \\
\nn&&+\mu^{-\fr13}t^{2-2s}\sum_{N\ge8}\sum_{k,l}\int_{\eta,\xi}\left||\eta|^\fr{s}{2}A{\hat{g}}_k(\eta)\right|\left||\xi|^\fr{s}{2}A\hat{g}_l(\xi)_N\right|\\
\nn&&\times e^{c\lm |k-l,\eta-\xi|^s}\la k-l,\eta-\xi\ra^2|\hat{V}_{k-l}(\eta-\xi)_{<\fr{N}{8}}|  d\xi d\eta \\
&\les&\mu^{-\fr13}\left\| V\right\|_{\mathcal{G}^{\lm,2}}\left(\left\|Ag\right\|_{L^2}^2+t^{2-2s}\left\||\nabla|^{\fr{s}{2}}Ag\right\|_{L^2}^2\right),
\eeq
where $c\in(0,1)$. Recalling the definition of $V$ in \eqref{OMJ}, using the algebra property \eqref{alge} of $\mathcal{G}^{\lm,2}$, the lossy estimate \eqref{loss1} and the upper bound of $m^\fr12_k(t,\eta)$ in \eqref{m1}, we have 
\beq\label{low-V}
\left\|V\right\|_{\mathcal{G}^{\lm,2}}\nn&\les&\left(1+\|Y'-1\|_{\mathcal{G}^{\lm,2}}\right)\|\Psi_{\ne}\|_{\mathcal{G}^{\lm, 3}}+\left\|\dot{Y}\right\|_{\mathcal{G}^{\lm, 2}}\les\fr{\|\Om\|_{G^{\lm,5}}}{\la t\ra^2}+\left\|\dot{Y}\right\|_{\mathcal{G}^{\lm, 2}}\\
&\les&\mu^{-\fr13}\la t\ra^{-2+\fr{K_D\mu^{-\fr23}\eps}{2}}\left(\|A\Om\|_{L^2}+\la t\ra^{2-\fr{K_D\mu^{-\fr23}\eps}{2}}\left\|\dot{Y}\right\|_{\mathcal{G}^{\lm, 2}}\right).
\eeq
In particular, 
\be\label{low-Vne}
\left\|V_{\ne}\right\|_{\mathcal{G}^{\lm,2}}\les\mu^{-\fr13}\la t\ra^{-2}\|A\Om\|_{L^2}.
\ee
Substituting \eqref{low-V} into \eqref{TN14}, \eqref{TN2} and  \eqref{TN3},  using \eqref{low-Vne} and the bootstrap hypotheses, one deduces that
\begin{align*}
\int_1^t\sum_{N\ge8}T_N[V]_gdt'\les\mu^{-\fr23}\eps\left(\|Ag\|_{L^\infty L^2}^2+\int_1^t\mathrm{CK}_{\lm,g}+\mathrm{CK}_{w,g}dt'\right)\les\mu^{-\fr23}\eps^3,
\end{align*}
as long as $\eps$ is so small that $2s-\fr{K_D\mu^{-\fr23}\eps}{2}\ge s+\fr12$.

It is worth pointing out that when dealing with $T_N[B_{\ne}]_\Om$ and $T_N[B_{\ne}]_J$, which are given by
\begin{align*}
T_N[B_{\ne}]_{\Om}=&2\pi\int AJ\left[A\left(Y'\nabla^\bot\Phi_{\ne}\right)_{<\fr{N}{8}}\cdot\nabla\Om_{N}-\left(Y'\nabla^\bot\Phi_{\ne}\right)_{<\fr{N}{8}}\cdot\nabla A\Om_N\right]dXdY,\\
T_N[B_{\ne}]_{J}=&2\pi\int A\Om\left[A\left(Y'\nabla^\bot\Phi_{\ne}\right)_{<\fr{N}{8}}\cdot\nabla J_N-\left(Y'\nabla^\bot\Phi_{\ne}\right)_{<\fr{N}{8}}\cdot\nabla AJ_N\right]dXdY,
\end{align*}
the low frequency estimate \eqref{low-V} should be replaced by the following
\be\label{low-Phi}
\left\|Y'\nabla^\bot\Phi_{\ne}\right\|_{\mathcal{G}^{\lm,2}}\les\mu^{-\fr13}\la t\ra^{-2}\|AJ\|_{L^2}.
\ee
Then similar to \eqref{TN14}, \eqref{TN2} and \eqref{TN3}, we conclude that
\beq
\nn&&\int_1^t\sum_{N\ge8}\left(T_N[B_{\ne}]_\Om+T_N[B_{\ne}]_J\right)dt'\\
\nn&\les&\mu^{-\fr23}\|AJ\|_{L^\infty L^2}\int_1^t\left(\la t'\ra^{-2}\|A\Om\|_{L^2}\|AJ\|_{L^2}+\left\|\sqrt{\fr{\partial_tw}{w}}\tl A\Om\right\|_{L^2}\left\|\sqrt{\fr{\partial_tw}{w}}\tl AJ\right\|_{L^2}\right)dt'\\
&&+\mu^{-\fr23}\|AJ\|_{L^\infty L^2}\int_1^t\left(\la t'\ra^{-2s}\left\||\nabla|^\fr{s}{2}A\Om\right\|_{L^2}\left\||\nabla|^\fr{s}{2}AJ\right\|_{L^2}\right)dt'\les\mu^{-\fr23}\eps^3.
\eeq

\subsection{Reaction }
Taking $R_N[V]_g$  for example, let us first divide it into four parts:
\beno
R_N[V]_g=R_N[V]_g^1+R_N[V]_g^{1,\eps}+R_N[V]_g^2+R_N[V]_g^3,
\eeno
where
\beq
R_N[V]_g^1\nn&=&2\pi\int AgA\left(\nabla^{\bot}\mathbb{P}_{\ne0}\Psi_{N}\cdot\nabla g_{<\fr{N}{8}}\right)dXdY,\\
R_N[V]_g^{1,\eps}\nn&=&2\pi\int AgA\left[\left((Y'-1)\nabla^{\bot}\mathbb{P}_{\ne0}\Psi\right)_{N}\cdot\nabla g_{<\fr{N}{8}}\right]dXdY,\\
R_N[V]_g^2\nn&=&2\pi\int AgA\left(\dot{Y}_{N}\partial_Y g_{<\fr{N}{8}}\right)dXdY,\\
R_N[V]_g^3\nn&=&-2\pi\int Ag\left(V_{N}\cdot\nabla Ag_{<\fr{N}{8}}\right)dXdY.
\eeq
Following the treatment of reaction in \cite{BM15} line by line, using  \eqref{M-ratio} and \eqref{m1}, we give the following estimates directly:
\beq\label{e-Rg1}
\nn\sum_{N\ge8}R_N[V]_g^1
&\les&\mu^{-\fr23}\left(\left\|\left\la \fr{\partial_Y}{t\partial_X}\right\ra^{-1}\sqrt{\fr{\partial_tw}{w}}\tl A\Dl_L\Psi_{\ne}\right\|_{L^2}+\left\|\left\la \fr{\partial_Y}{t\partial_X}\right\ra^{-1}\fr{|\nabla|^\fr{s}{2}}{\la t\ra^s}A\Dl_L\Psi_{\ne} \right\|_{L^2}\right)\\
&&\times\left(\left\|\sqrt{\fr{\partial_tw}{w}}\tl Ag\right\|_{L^2}+\fr{\left\||\nabla|^\fr{s}{2}Ag\right\|_{L^2}}{\la t\ra^s}\right)\left\|Ag\right\|_{L^2},
\eeq

\beqno
\sum_{N\ge8}R_N[V]_g^{1,\eps}\nn&\les&\|Y'-1\|_{\mathcal{G}^{\lm,2}}\cdot\Big[\mathrm{R.H.S.\ \ of}\  \eqref{e-Rg1}\Big]\\
&&+\mu^{-\fr13}\left(\left\|A^R(Y'-1)\right\|_{L^2}+\|Y'-1\|_{\mathcal{G}^{\lm,2}}\right)\|Ag\|_{L^2}^2\|\Psi_{\ne}\|_{\mathcal{G}^{\lm,2}},
\eeqno

\beqno
\nn\sum_{N\ge8}R_N[V]_g^2&\les&\mu^{-\fr13}\left\||\nabla|^{\fr{s}{2}}Ag\right\|_{L^2}\left\||\partial_Y|^{\fr{s}{2}}\fr{A}{\la \partial_Y\ra^s}\dot{Y}\right\|_{L^2}\|Ag\|_{L^2}\\
&&+\mu^{-\fr13}\left\|\sqrt{\fr{\partial_tw}{w}}\tl Ag\right\|_{L^2}\left(t^{1+s}\left\|\sqrt{\fr{\partial_tw}{w}}\fr{A}{\la\partial_Y\ra^s}\dot{Y}\right\|_{L^2}\right)\|Ag\|_{L^2},
\eeqno
and
\beno
\sum_{N\ge8}R_N[V]_g^3\les\left\| V\right\|_{H^3}\left\|Ag\right\|_{L^2}^2\les\mu^{-\fr13}\la t\ra^{-2+\fr{K_D\mu^{-\fr23}\eps}{2}}\left(\|A\Om\|_{L^2}+\la t\ra^{2-\fr{K_D\mu^{-\fr23}\eps}{2}}\left\|\dot{Y}\right\|_{\mathcal{G}^{\lm, 2}}\right)\left\|Ag\right\|_{L^2}^2.
\eeno
It follows from the above four estimates, \eqref{CKY1}, \eqref{loss1}, \eqref{pre-ell1}, \eqref{e-CCK} and the bootstrap hypotheses that 
\beq\label{e-RV}
\nn&&\int_1^t\sum_{N\ge8}R_N[V]_gdt'\\
\nn&\les&\mu^{-\fr23}\eps\left(\int_1^t\mathrm{CK}_{w,\Om}(t')+\mathrm{CK}_{\lm,\Om}(t')+\mu^{-\fr23}\eps^2\left(\mathrm{CK}_{\lm}^R(t')+\mathrm{CK}_{w}^R(t')\right)dt'\right)\\
\nn&&+\mu^{-\fr23}\eps\|A\Om\|_{L^\infty L^2}\|Ag\|_{L^\infty L^2}^2+\mu^{-\fr13}\eps\left(\int_1^t\mathrm{CK}_{w}^{Y,1}(t')+\mathrm{CK}_{\lm}^{Y,1}(t')dt'\right)+\mu^{-\fr13}\eps \|Ag\|_{L^\infty L^2}^2\\
&\les&\mu^{-\fr23}\eps^3.
\eeq

\begin{rem}
Compared with $R_N[V]_g$, the  treatment of $R_N[B_{\ne}]_{\Om}$ and $R_N[B_{\ne}]_J$, which are given by
\beqno
R_N[B_{\ne}]_{\Om}&=&2\pi\int AJ\left[A\left(Y'\nabla^\bot\Phi_{\ne}\right)_N\cdot\nabla\Om_{<\fr{N}{8}}-\left(Y'\nabla^\bot\Phi_{\ne}\right)_N\cdot\nabla A\Om_{<\fr{N}{8}}\right]dXdY,\\
R_N[B_{\ne}]_{J}&=&2\pi\int A\Om\left[A\left(Y'\nabla^\bot\Phi_{\ne}\right)_N\cdot\nabla J_{<\fr{N}{8}}-\left(Y'\nabla^\bot\Phi_{\ne}\right)_N\cdot\nabla AJ_{<\fr{N}{8}}\right]dXdY,
\eeqno
is slightly easier due to the absence of zero-mode in $Y'\nabla^\bot\Phi_{\ne}$. We just state the result here
\beq\label{e-RBOm}
\nn\int_1^t\sum_{N\ge8}R_N[B_{\ne}]_\Om dt'
\nn&\les&\mu^{-\fr23}\left(\left\|\left\la \fr{\partial_Y}{t\partial_X}\right\ra^{-1}\sqrt{\fr{\partial_tw}{w}}\tl A\Dl_L\Phi_{\ne}\right\|_{L^2L^2}+\left\|\left\la \fr{\partial_Y}{t\partial_X}\right\ra^{-1}\fr{|\nabla|^\fr{s}{2}}{\la t\ra^s}A\Dl_L\Phi_{\ne} \right\|_{L^2L^2}\right)\\
\nn&&\times\left(\left\|\sqrt{\fr{\partial_tw}{w}}\tl AJ\right\|_{L^2L^2}+\left\|\fr{|\nabla|^\fr{s}{2}AJ}{\la t'\ra^s}\right\|_{L^2L^2}\right)\left\|A\Om\right\|_{L^\infty L^2}\left(1+\|Y'-1\|_{L^\infty\mathcal{G}^{\lm,2}}\right)\\
\nn&&+\mu^{-\fr13}\left(\left\|A^R(Y'-1)\right\|_{L^\infty L^2}+\|Y'-1\|_{L^\infty\mathcal{G}^{\lm,2}}\right)\|A\Om\|_{L^\infty L^2}\|AJ\|_{L^\infty L^2}\|\Phi_{\ne}\|_{L^1\mathcal{G}^{\lm,2}}\\
&&+\mu^{-\fr13}\left\|Y'\nabla^\bot\Phi_{\ne}\right\|_{L^1H^3}\left\|A\Om\right\|_{L^\infty L^2}\|AJ\|_{L^\infty L^2}\les\mu^{-\fr23}\eps^3.
\eeq
As for  $\displaystyle\int_1^t\sum_{N\ge8}R_N[B_{\ne}]_Jdt'$, one can obtain the same estimate as above by simply swapping $\Om$ and $J$ in \eqref{e-RBOm}.
\end{rem}
\subsection{Remainder}\label{sec-rem}
The method used in \cite{BM15} to treat the remainder still works here. We omit the details and conclude that
\beno
\int_1^t\mathcal{R}[V]_g dt'\les\mu^{-\fr23}\int_1^t\la t'\ra^{-2+\fr{K_D\mu^{-\fr23}\eps}{2}}\left(\|A\Om\|_{L^2}+\la t'\ra^{2-\fr{K_D\mu^{-\fr23}\eps}{2}}\left\|\dot{Y}\right\|_{\mathcal{G}^{\lm, 2}}\right)\|Ag\|_{L^2}^2dt'\les\mu^{-\fr23}\eps^3,\quad g\in\{\Om, J\},
\eeno
and
\beno
\int_1^t\mathcal{R}[B_{\ne}]_\Om+\mathcal{R}[B_{\ne}]_Jdt'\les\mu^{-\fr23}\int_1^t\la t'\ra^{-2}\|AJ\|_{L^2}^2\|A\Om\|_{L^2}dt'\les\mu^{-\fr23}\eps^3.
\eeno

\section{Estimates of {\bf CNLT}}
Recalling \eqref{CNLT}, let us denote
\[
{\bf CNLT}=\mathrm{CNLT}_1+\mathrm{CNLT}_2+\mathrm{CNLT}_3+\mathrm{CNLT}_4,
\]
where
\beqno
\mathrm{CNLT}_1\nn&=&\fr{1}{\al}\left\la\fr{\partial_t(m^\fr12)}{m^{\fr12}}\partial_X^{-1}A\left(V\cdot\nabla\Om\right)_{\ne}, AJ_{\ne} \right\ra,\\
\mathrm{CNLT}_2\nn&=&\fr{1}{\al}\left\la\fr{\partial_t(m^\fr12)}{m^{\fr12}}\partial_X^{-1}A\Om_\ne, A\left(V\cdot\nabla J\right)_\ne \right\ra,\\
\mathrm{CNLT}_3&=&-\fr{1}{\al}\left\la\fr{\partial_t(m^\fr12)}{m^{\fr12}}\partial_X^{-1}A\left(Y'\nabla^\bot\Phi_{\ne}\cdot\nabla J\right)_{\ne}, AJ_{\ne} \right\ra,\\
\mathrm{CNLT}_4&=&-\fr{1}{\al}\left\la\fr{\partial_t(m^\fr12)}{m^{\fr12}}\partial_X^{-1}A\Om_\ne, A\left(Y'\nabla^\bot\Phi_{\ne}\cdot\nabla\Om\right)_{\ne} \right\ra.
\eeqno
In this section, we will not distinguish between $\Om$ and $J$, and only give the details of treatment of $\mathrm{CNLT}_1$. To this end, by paraproduct decomposition, we write
\be
\mathrm{CNLT}_1=\fr{1}{2\pi}\sum_{N\ge8} CT_N+\fr{1}{2\pi}\sum_{N\ge8}CR_N+\fr{1}{2\pi}C\mathcal{R},
\ee
with
\beqno
CT_N&=&\fr{2\pi}{\al }\left\la\fr{\partial_t(m^\fr12)}{m^{\fr12}}\partial_X^{-1}A\left(V_{<\fr{N}{8}}\cdot\nabla \Om_N\right)_{\ne}, AJ_{\ne}\right\ra,\\
CR_N&=&\fr{2\pi}{\al }\left\la\fr{\partial_t(m^\fr12)}{m^{\fr12}}\partial_X^{-1}A\left(V_N\cdot\nabla \Om_{<\fr{N}{8}}\right)_{\ne}, AJ_{\ne}\right\ra,\\
C\mathcal{R}&=&\fr{2\pi}{\al }\sum_{N\in\mathbb{D}}\sum_{\fr{N}{8}\le N'\le8N}\left\la\fr{\partial_t(m^\fr12)}{m^{\fr12}}\partial_X^{-1}A\left(V_N\cdot\nabla \Om_{N'}\right)_{\ne}, AJ_{\ne}\right\ra.
\eeqno
\subsection{Treatment of $CT_N$}\label{sec-CT}
Different from the estimates of $T_N[V]_\Om$ in Section \ref{sec-NLT}, instead of  using any possible gain from $\fr{A_k(\eta)}{A_l(\xi)}-1$, we will benefit from the multiplier $\fr{\partial_t(m^\fr12)}{m^{\fr12}}\partial_X^{-1}$. Let us start by dividing into two cases based on the relative size of $|l|$ and $|\xi|$:
\beq
\nn CT_N&=&\fr{1}{\al}\sum_{k\ne0}\int_{\eta,\xi}\left({\bf1}_{|l|>4|\xi|}+{\bf1}_{|l|\le4|\xi|}\right)\fr{\partial_t(m^\fr12_k)(\eta)}{km^\fr12_k(\eta)}\fr{A_k(\eta)}{A_l(\xi)}\\
\nn&&\quad\quad\quad\times\hat{V}_{k-l}(\eta-\xi)_{<\fr{N}{8}}\cdot(l,\xi)A_l(\xi)\hat{\Om}_l(\xi)_N (A\bar{\hat{J}})_k(\eta)d\xi d\eta\\
\nn&=&CT^X_N+CT^Y_N.
\eeq

We now treat $CT_N^X$. On the support of the integrand of $CT_N$, \eqref{spt-T} holds. Then in view of \eqref{m1} and \eqref{ap3}, one easily deduces that
\beq\label{switch-A}
\fr{A_k(\eta)}{A_l(\xi)}=e^{\lm(|k,\eta|^s-|l,\xi|^s)}\fr{\la k,\eta\ra^\sigma}{\la l,\xi\ra^\sigma}\fr{\mathcal{M}_k(\eta)}{\mathcal{M}_l(\xi)}\fr{\mathcal{J}_k(\eta)}{\mathcal{J}_l(\xi)}\les \mu^{-\fr13}e^{c\lm|k-l, \eta-\xi|^s}\fr{\mathcal{J}_k(\eta)}{\mathcal{J}_l(\xi)},
\eeq
for some $c\in(0,1)$. Furthermore, thanks to the cut-off ${\bf 1}_{|l|>4|\xi|}$, on the support of the integrand of $CT_N^X$, it is not difficulty to verify that
\beq\label{switch-J}
\fr{\mathcal{J}_k(\eta)}{\mathcal{J}_l(\xi)}\les e^{2r|k-l|^\fr12}.
\eeq
Noting that $0\le\fr{\partial_t(m^{\fr12})}{m^{\fr12}}\le\fr12$, it follows from \eqref{low-V},  \eqref{switch-A}, \eqref{switch-J}, \eqref{ap-5}, \eqref{ap-6} and the bootstrap hypotheses that
\beq
\nn \int_1^t\sum_{N\ge8}CT_N^Xdt'&\les&\fr{\mu^{-\fr13}}{|\al|}\sum_{N\ge8}\sum_{k\ne0}\int_1^t\int_{\eta,\xi}\fr{|l|}{|k|}e^{c\lm|k-l, \eta-\xi|^s} e^{2r|k-l|^\fr12}\left|V_{k-l}(\eta-\xi)_{<\fr{N}{8}}\right|\\
\nn&&\times\left|A_l(\xi)\hat{\Om}_l(\xi)_N\right|\left|(A\hat{J})_k(\eta)\right|d\xi d\eta dt'\\
&\les&\mu^{-\fr13}\left\|V\right\|_{L^1\mathcal{G}^{\lm, 2}}\left\|A\Om\right\|_{L^\infty L^2}\left\|AJ\right\|_{L^\infty L^2}\les\mu^{-\fr23}\eps^3.
\eeq

Next, we turn to bound $CT_N^Y$. We further split it as
\beq
CT_N^Y\nn&=&\fr{1}{\al}\sum_{k\ne0}\int_{\eta,\xi}{\bf1}_{|l|\le4|\xi|}\left(\chi^D+\chi^0+\chi^*\right)\fr{\partial_t(m^\fr12_k)(\eta)}{km^\fr12_k(\eta)}\fr{A_k(\eta)}{A_l(\xi)}\\
\nn&&\qquad\qquad\times\hat{V}_{k-l}(\eta-\xi)_{<\fr{N}{8}}\cdot(l,\xi)A_l(\xi)\hat{\Om}_l(\xi)_N(A\bar{\hat{J}})_k(\eta)d\xi d\eta\\
\nn&=&CT^{Y,D}_N+CT^{Y,0}_N+CT^{Y,*}_N,
\eeq
where $\chi^D={\bf1}_{t\in\mathbb{I}_{k,\eta}}{\bf1}_{k\ne l}$, $\chi^0={\bf1}_{t\in\mathbb{I}_{k,\eta}}{\bf1}_{k= l}$, and $\chi^*={\bf1}_{t\notin\mathbb{I}_{k,\eta}}$.

We first treat $CT_N^{Y,D}$. To begin with, we find that on the support of the integrand of $CT_N^{Y}$, the restrictions $|l|\le4|\xi|$ and $ |k-l,\eta-\xi|\le\fr{3}{16}|l,\xi|$ ensure that $|\eta|\approx|\xi|$. Then one deduces  from \eqref{wNR-ratio}, \eqref{pt_w}, \eqref{swi-ptw},  \eqref{wRwNR}  and  \eqref{ap1}  that
\beq\label{switch-J1}
{\bf 1}_{t\in\mathbb{I}_{k,\eta}}\fr{\mathcal{J}_k(\eta)}{\mathcal{J}_l(\xi)}\nn&\les& e^{2r|\eta-\xi|^\fr12}\fr{w_{NR}(\xi)}{w_{NR}(\eta)}\fr{w_{NR}(\eta)}{w_R(\eta)}{\bf 1}_{t\in\mathbb{I}_{k,\eta}}+e^{2r|k-l|^\fr12}{\bf 1}_{t\in\mathbb{I}_{k,\eta}}\\
\nn&\les&e^{3r|k-l, \eta-\xi|^\fr12}\fr{|\eta|}{k^2}\fr{{\bf 1}_{t\in\mathbb{I}_{k,\eta}}}{1+\left|t-\fr{\eta}{k}\right|}\\
&\les&\la \eta-\xi\ra e^{3r|k-l, \eta-\xi|^\fr12}\fr{|\eta|}{k^2}\sqrt{\fr{\partial_tw_k(\eta)}{w_k(\eta)}}\left(\sqrt{\fr{\partial_tw_l(\xi)}{w_l(\xi)}}+\fr{|\xi|^{\fr{s}{2}}}{\la t\ra^s}\right).
\eeq
Thus, if $|l|<|\xi|$, $A_l(\xi)$ can be bounded by $\tl A_l(\xi)$. Then by virtue of \eqref{low-Vne},
 \eqref{switch-A}, \eqref{switch-J1}, \eqref{ap-5}, \eqref{ap-6}, and  the fact ${\bf 1}_{t\in\mathbb{I}_{k,\eta}}\fr{|\eta|}{|k|}\approx t$, we arrive at
\beq\label{CTYD1}
\int_0^t\sum_{N\ge8}CT^{Y,D}_Ndt'\nn&\les&\fr{\mu^{-\fr13}}{|\al|}\sum_{N\ge8}\sum_{k\ne0,l\ne k}\int_1^t\int_{\eta,\xi}{\bf1}_{t\in\mathbb{I}_{k,\eta}}\left|\fr{\eta}{k}\right|^2\sqrt{\fr{\partial_tw_k(\eta)}{w_k(\eta)}}\left|(A\hat{J})_k(\eta)\right|\\
\nn&&\times \left(\sqrt{\fr{\partial_tw_l(\xi)}{w_l(\xi)}}+\fr{|\xi|^{\fr{s}{2}}}{\la t\ra^s}\right)\left|\tl A\hat{\Om}_l(\xi)_N\right|e^{\lm|k-l, \eta-\xi|^s}\left|\hat{V}_{k-l}(\eta-\xi)_{<\fr{N}{8}}\right|d\xi d\eta dt'\\
\nn&\les&\fr{\mu^{-\fr13}}{|\al|}\left\|t'^2V_{\ne}\right\|_{L^\infty\mathcal{G}^{\lm, 2}}\left\|\sqrt{\fr{\partial_tw}{w}}\tl AJ\right\|_{L^2L^2}\left(\left\|\sqrt{\fr{\partial_tw}{w}}\tl A\Om\right\|_{L^2L^2}+\left\|\fr{|\nabla|^\fr{s}{2}A\Om}{\la t'\ra^s}\right\|_{L^2L^2}\right)\\
&\les&\mu^{-\fr23}\eps^3.
\eeq
If $|l|\ge|\xi|$,  using the fact $|\eta|\approx|\xi|$, we obtain
\beno
{\bf 1}_{|l|\ge|\xi|}\fr{|\eta|}{k^2}|l,\xi|\les\fr{|l|^2}{|k|^2}\les\la k-l\ra^2.
\eeno
Combining this with \eqref{switch-A} and \eqref{switch-J1}, noting that $\fr{{\bf 1}_{t\in\mathbb{I}_{k,\eta}}}{1+\left|t-\fr{\eta}{k}\right|}\le1$, we have
\be
\int_1^t\sum_{N\ge8}CT^{Y,D}_Ndt'\les\fr{\mu^{-\fr13}}{|\al|}\left\|V_{\ne}\right\|_{L^1\mathcal{G}^{\lm, 2}}\left\| AJ\right\|_{L^\infty L^2}\left\| A\Om\right\|_{L^\infty L^2}\les\mu^{-\fr23}\eps^3.
\ee

As for $CT^{Y, 0}_N$, thanks to \eqref{J-ratio} and \eqref{com-m}, we find that \eqref{switch-A} reduces to
\be\label{switch-Akk}
\fr{A_k(\eta)}{A_k(\xi)}\les \la \eta-\xi\ra e^{c\lm|\eta-\xi|^s}e^{10r|\eta-\xi|^\fr12}.
\ee
This, together with \eqref{ptm-gain}, implies that
\beq
\int_1^t\sum_{N\ge8}CT^{Y,0}_Ndt'\nn&\les&\fr{\mu^{-\fr13}}{|\al|}\sum_{N\ge8}\sum_{k\ne0}\int_1^t\int_{\eta,\xi}\sqrt{\fr{\partial_tM^1_k(\eta)}{M^1_k(\eta)}}\left|(A\hat{J})_k(\eta)\right|\sqrt{\fr{\partial_tM^1_k(\xi)}{M^1_k(\xi)}}\left|(A\hat{\Om})_k(\xi)_N\right|\\
\nn&&\times {\bf1}_{t\in\mathbb{I}_{k,\eta}}\fr{|\eta|}{|k|}e^{\lm|\eta-\xi|^s}\left|\hat{\dot{Y}}(\eta-\xi)_{<\fr{N}{8}}\right|d\xi d\eta dt'\\
&\les&\fr{\mu^{-\fr13}}{|\al|}\left\|t'{\dot{Y}}\right\|_{L^\infty\mathcal{G}^{\lm, 1}}\left\|\sqrt{\fr{\partial_tM^1}{M^1}}AJ\right\|_{L^2L^2}\left\|\sqrt{\fr{\partial_tM^1}{M^1}}A\Om\right\|_{L^2L^2}\les\mu^{-\fr13}\eps^3.
\eeq

To bound $CT^{Y, *}_N$, note that
\be\label{gain1}
\left|\fr{\eta}{k}\right|\fr{{\bf1}_{k\ne0}{\bf1}_{t\notin\mathbb{I}_{k,\eta}}}{1+\left|\fr{\eta}{k}-t\right|}\les \sqrt{|\eta|}+1.
\ee
Then  using the fact $\left|\fr{\partial_t(m^\fr12)}{km^{\fr12}}\right|\les\fr{1}{|k|} \fr{1}{1+\left|\fr{\eta}{k}-t\right|}$, \eqref{J-ratio},  \eqref{switch-A},  \eqref{gain1} and the decay estimate \eqref{low-V}  of $V$ with low regularity, one deduces  that
\beq
\int_1^t\sum_{N\ge8}CT^{Y,*}_Ndt'\nn&\les&\fr{\mu^{-\fr13}}{|\al|}\sum_{N\ge8}\sum_{k\ne0}\int_1^t\int_{\eta,\xi}\left(|\eta|^\fr12+1\right)\left|(A\hat{J})_k(\eta)\right|\left|A_l(\xi)\hat{\Om}_l(\xi)_N\right|\\
\nn&&\times e^{\lm|k-l, \eta-\xi|^s}\left|\hat{V}_{k-l}(\eta-\xi)_{<\fr{N}{8}}\right|d\xi d\eta dt'\\
\nn&\les&\fr{\mu^{-\fr13}}{|\al|}\sum_{N\ge8}\sum_{k\ne0}\int_1^t\int_{\eta,\xi}|l,\xi|^\fr{s}{2}|k,\eta|^\fr{s}{2}\left|(A\hat{J})_k(\eta)\right|\left|A_l(\xi)\hat{\Om}_l(\xi)_N\right|\\
\nn&&\times e^{\lm|k-l, \eta-\xi|^s}\left|\hat{V}_{k-l}(\eta-\xi)_{<\fr{N}{8}}\right|d\xi d\eta dt'\\
&\les&\fr{\mu^{-\fr13}}{|\al|}\int_1^t\left\|V\right\|_{\mathcal{G}^{\lm, 2}}\left\||\nabla|^\fr{s}{2}AJ\right\|_{L^2}\left\||\nabla|^\fr{s}{2}A\Om\right\|_{L^2}dt'\les\mu^{-\fr23}\eps^3.
\eeq

\subsection{Treatment of $CR_N$ and $C\mathcal{R}$}\label{sec-CR}
One can see from \eqref{pt-m12} that on the Fourier side 
\be\label{up-pt_m}
\left|\mathcal{F}\left[\fr{\partial_t(m^{\fr12})}{m^\fr12}\partial^{-1}_X\right]\right|\les1.
\ee
Then the estimates of $R_N[V]_g$ and $\mathcal{R}[V]_g$ also apply to $CT_N$ and $C\mathcal{R}$, respectively. We state the results directly as follows:
\beqno
\int_1^t\sum_{N\ge8}CR_Ndt'&\les&\fr{\mu^{-\fr23}}{|\al|}\left(\left\|\left\la \fr{\partial_Y}{t\partial_X}\right\ra^{-1}\sqrt{\fr{\partial_tw}{w}}\tl A\Dl_L\Psi_{\ne}\right\|_{L^2L^2}+\left\|\left\la \fr{\partial_Y}{t\partial_X}\right\ra^{-1}\fr{|\nabla|^\fr{s}{2}}{\la t\ra^s}A\Dl_L\Psi_{\ne} \right\|_{L^2L^2}\right)\\
\nn&&\times\left(\left\|\sqrt{\fr{\partial_tw}{w}}\tl AJ\right\|_{L^2L^2}+\left\|\fr{|\nabla|^\fr{s}{2}AJ}{\la t'\ra^s}\right\|_{L^2L^2}\right)\left\|A\Om\right\|_{L^\infty L^2}\left(1+\|Y'-1\|_{L^\infty\mathcal{G}^{\lm,2}}\right)\\
&&+\fr{\mu^{-\fr13}}{|\al|}\left(\left\|A^R(Y'-1)\right\|_{L^\infty L^2}+\|Y'-1\|_{L^\infty \mathcal{G}^{\lm,2}}\right)\|A\Om\|_{L^\infty L^2}\|AJ\|_{L^\infty L^2}\|\Psi_{\ne}\|_{L^1\mathcal{G}^{\lm,2}}\\
&&+\fr{\mu^{-\fr13}}{|\al|}\left\|\sqrt{\fr{\partial_tw}{w}}\tl AJ\right\|_{L^2L^2}\left\|\la t'\ra^{1+s}\sqrt{\fr{\partial_tw}{w}}\fr{A}{\la\partial_Y\ra^s}\dot{Y}\right\|_{L^2L^2}\|A\Om\|_{L^\infty L^2}\\
&&+\fr{\mu^{-\fr13}}{|\al|}\left\|\fr{|\nabla|^{\fr{s}{2}}AJ}{\la t'\ra^s}\right\|_{L^2L^2}\left\|\la t'\ra^{s}|\partial_Y|^{\fr{s}{2}}\fr{A}{\la \partial_Y\ra^s}\dot{Y}\right\|_{L^2}\|A\Om\|_{L^\infty L^2}\les\mu^{-\fr23}\eps^3,
\eeqno
and
\beno
\int_1^tC\mathcal{R}dt'\les\fr{\mu^{-\fr23}}{|\al|}\int_1^t\la t'\ra^{-2+\fr{K_D\eps}{2}}\left(\|A\Om\|_{L^2}+\la t\ra^{2-\fr{K_D\eps}{2}}\left\|\dot{Y}\right\|_{\mathcal{G}^{\lm, 2}}\right)\|AJ\|_{L^2}\|A\Om\|_{L^2}dt'\les\mu^{-\fr23}\eps^3.
\eeno

\begin{rem}
Swapping $\Om$ with $J$ in Sections \ref{sec-CT} and \ref{sec-CR}, one gets the estimates of $\mathrm{CNLT}_2$. Compared with $\mathrm{CNLT}_1$ and $\mathrm{CNLT}_2$,   the treatments of $\mathrm{CNLT}_3$ and $\mathrm{CNLT}_4$ are very similar and slightly easier as long as $V$ is replaced by $Y'\nabla^\bot\Phi_\ne$. We omit them for the sake of brevity.
\end{rem}

\section{Estimates of ${\bf NLS}$ and ${\bf CNLS}$}\label{sec-eNLS}\label{sec-NLS}
Recalling \eqref{NLS} and \eqref{CNLS}, we write
\beno
{\bf NLS}=\mathrm{NLS}_1+\mathrm{NLS}_2,\quad\mathrm{and}\quad {\bf CNLS}=\mathrm{CNLS}_1+\mathrm{CNLS}_2,
\eeno
where
\beqno
\mathrm{NLS}_1&=&-2\left\la AJ, A\left(\partial_{XY}^t\Phi(-\Om+2\partial_{XX}\Psi)\right)\right\ra,\\
\mathrm{NLS}_2&=&2\left\la AJ, A\left(\partial_{XY}^t\Psi(-J+2\partial_{XX}\Phi)\right)\right\ra,\\
\mathrm{CNLS}_1&=&\fr{2}{\al}\left\la \fr{\partial_t(m^\fr12)}{m^{\fr12}}\partial_X^{-1}A\Om_{\ne}, A\left(\partial_{XY}^t\Phi(-\Om+2\partial_{XX}\Psi)\right)_{\ne}\right\ra,\\
\mathrm{CNLS}_2&=&-\fr{2}{\al}\left\la \fr{\partial_t(m^\fr12)}{m^{\fr12}}\partial_X^{-1}A\Om_{\ne}, A\left(\partial_{XY}^t\Psi(-J+2\partial_{XX}\Phi)\right)_{\ne}\right\ra.
\eeqno
Let us focus on the treatment of $\mathrm{NLS}_1$, which is decomposed with a paraproduct as follows
\beq
\mathrm{NLS}_1=\fr{1}{2\pi}\mathrm{NLS}_{1,\mathrm{HL}}+\fr{1}{2\pi}\mathrm{NLS}_{1,\mathrm{LH}}+\fr{1}{2\pi}\mathrm{NLS}_{1,\mathcal{R}},
\eeq
with
\beqno
\mathrm{NLS}_{1,\mathrm{HL}}&=&-4\pi\sum_{N\ge8}\left\la AJ, A\left(\partial_{XY}^t\Phi_N(-\Om_{<\fr{N}{8}}+2\partial_{XX}\Psi_{<\fr{N}{8}})\right)\right\ra,\\
\mathrm{NLS}_{1,\mathrm{LH}}&=&-4\pi\sum_{N\ge8}\left\la AJ, A\left(\partial_{XY}^t\Phi_{<\fr{N}{8}}(-\Om_{N}+2\partial_{XX}\Psi_{N})\right)\right\ra,\\
\mathrm{NLS}_{1,\mathcal{R}}&=&-4\pi\sum_{N\in\mathbb{D}}\sum_{\fr{N}{8}\le N'\le8N}\left\la AJ, A\left(\partial_{XY}^t\Phi_{N'}(-\Om_{N}+2\partial_{XX}\Psi_{N})\right)\right\ra.
\eeqno

\subsection{Treatment of $\mathrm{NLS}_{1,\mathrm{HL}}$} Noting that $\partial^t_{XY}\Phi=\partial_{XY}^L\Phi+(Y'-1)\partial_{XY}^L\Phi$, we split $\mathrm{NLS}_{1,\mathrm{HL}}$ into two parts:
\[
\mathrm{NLS}_{1,\mathrm{HL}}=\mathrm{NLS}_{1,\mathrm{HL}}^1+\mathrm{NLS}_{1,\mathrm{HL}}^{\eps},
\]
where
\beqno
\mathrm{NLS}_{1,\mathrm{HL}}^1&=&2\sum_{N\ge8}\sum_{(k,l)\in\mathbb{Z}\times\mathbb{Z^*}}\int_{\eta,\xi} A\bar{\hat{J}}_k(\eta)A_k(\eta)l(\xi-lt)\hat{\Phi}_l(\xi)_N\\
&&\quad\quad\quad\quad\quad\quad\quad\times\left(-\hat{\Om}_{k-l}(\eta-\xi)_{<\fr{N}{8}}+2\widehat{\partial_{XX}\Psi}_{k-l}(\eta-\xi)_{<\fr{N}{8}}\right) d\xi d\eta,\\
\mathrm{NLS}_{1,\mathrm{HL}}^\eps&=&-2\sum_{N\ge8}\sum_{(k,l)\in\mathbb{Z}\times\mathbb{Z^*}}\int_{\eta,\xi} A\bar{\hat{J}}_k(\eta)A_k(\eta)\mathcal{F}\left[(Y'-1)\partial^L_{XY}\Phi\right]_l(\xi)_N\\
&&\quad\quad\quad\quad\quad\quad\quad\times\left(-\hat{\Om}_{k-l}(\eta-\xi)_{<\fr{N}{8}}+2\widehat{\partial_{XX}\Psi}_{k-l}(\eta-\xi)_{<\fr{N}{8}}\right)d\xi d\eta .
\eeqno
\subsubsection{Main contribution} \label{sec-re-m} We first estimate the main contribution $\mathrm{NLS}_{1,\mathrm{HL}}^1$. Like \cite{BM15}, we divide $\mathrm{NLS}_{1,\mathrm{HL}}^1$ into four parts:
\beqno
\mathrm{NLS}_{1,\mathrm{HL}}^1&=&2\sum_{N\ge8}\sum_{(k,l)\in\mathbb{Z}\times\mathbb{Z^*}}\int_{\eta,\xi} \left({\bf 1}_{t\in \mathbb{I}_{k,\eta}\cap\mathbb{I}_{l,\xi}}+{\bf 1}_{t\in\mathbb{I}_{l,\xi}\backslash\mathbb{I}_{k,\eta} }+{\bf 1}_{t\in \mathbb{I}_{k,\eta}\backslash\mathbb{I}_{l,\xi}}+{\bf 1}_{t\notin \mathbb{I}_{k,\eta}\cup\mathbb{I}_{l,\xi}}\right)A\bar{\hat{J}}_k(\eta)\\
&&\quad\quad\quad\quad\quad\times A_k(\eta)l(\xi-lt)\hat{\Phi}_l(\xi)_N\left(-\hat{\Om}_{k-l}(\eta-\xi)_{<\fr{N}{8}}+2\widehat{\partial_{XX}\Psi}_{k-l}(\eta-\xi)_{<\fr{N}{8}}\right)d\xi d\eta \\
&=&\mathrm{NLS}_{1,\mathrm{HL}}^{1; \mathrm{R, R}}+\mathrm{NLS}_{1,\mathrm{HL}}^{1; \mathrm{NR, R}}+\mathrm{NLS}_{1,\mathrm{HL}}^{1; \mathrm{R, NR}}+\mathrm{NLS}_{1,\mathrm{HL}}^{1; \mathrm{NR, NR}}.
\eeqno
Note that for each frequency shell, \eqref{spt-T} holds.\par

{\em Treatment of $\mathrm{NLS}_{1,\mathrm{HL}}^{1; \mathrm{R, R}}$ and $\mathrm{NLS}_{1,\mathrm{HL}}^{1; \mathrm{NR, R}}$.}\par
\noindent For $t\in\mathbb{I}_{k,\eta}\cap\mathbb{I}_{l,\xi}$, if $k=l$, then \eqref{switch-Akk} holds. If $k\ne l$, we infer from  \eqref{wNR-ratio}, \eqref{switch-A}, \eqref{wRwNR} and (b) or (c) of Lemma \ref{1/3} that
\beq\label{swi-Akl1}
{\bf 1}_{k\ne l}{\bf 1}_{t\in\mathbb{I}_{k,\eta}\cap\mathbb{I}_{l,\xi}}\fr{A_k(\eta)}{A_l(\xi)}
\nn&\les&\mu^{-\fr13}e^{c\lm|k-l,\eta-\xi|^s}\left(e^{2r|\eta-\xi|^\fr12}\fr{w_R(\xi)}{w_{NR}(\xi)}\fr{w_{NR}(\xi)}{w_{NR}(\eta)}\fr{w_{NR}(\eta)}{w_{R}(\eta)}+e^{2r|k-l|^\fr12}\right){\bf 1}_{k\ne l}{\bf 1}_{t\in\mathbb{I}_{k,\eta}\cap\mathbb{I}_{l,\xi}}\\
\nn&\les&\mu^{-\fr13}e^{c\lm|k-l,\eta-\xi|^s}e^{3r|k-l,\eta-\xi|^\fr12}\fr{|\eta|}{k^2}\fr{{\bf 1}_{k\ne l}{\bf 1}_{t\in\mathbb{I}_{k,\eta}\cap\mathbb{I}_{l,\xi}}}{1+\left|\fr{\eta}{k}-t\right|}\\
&\les&\mu^{-\fr13}e^{c\lm|k-l,\eta-\xi|^s}e^{3r|k-l,\eta-\xi|^\fr12}\la\eta-\xi\ra.
\eeq
For $t\in\mathbb{I}_{l,\xi}\backslash\mathbb{I}_{k,\eta}$, different from the treatment of $R_N[V]^\Om$  in Section \ref{sec-NLT}, we do not need to make use of the gain of regularity from the exchange of   $\mathcal{J}_k(\eta)$ and $\mathcal{J}_l(\xi)$. From  \eqref{J-ratio}, \eqref{switch-A}, one easily deduces that 
\be\label{swi-Akl2}
{\bf 1}_{t\in\mathbb{I}_{l,\xi}\backslash\mathbb{I}_{k,\eta}}\fr{A_k(\eta)}{A_l(\xi)}
\les\mu^{-\fr13}e^{c\lm|k-l,\eta-\xi|^s}e^{10r|k-l,\eta-\xi|^\fr12}.
\ee
On the other hand, the fact  $t\in\mathbb{I}_{l,\xi}$  implies that $1\le|l|<\fr{\sqrt{|\xi|}}{2}$. Combining this with \eqref{spt-T}, we have
\[
|\eta-\xi|\le\fr{3}{16}|l,\xi|=\fr{3}{16}(|l|+|\xi|)\le\fr{3}{32}\sqrt{|\xi|}+\fr{3}{16}|\xi|\le\fr{9}{32}|\xi|.
\]
Consequently,
\be\label{eta=xi1}
\fr{23}{32}|\xi|\le|\eta|\le \fr{41}{32}|\xi|.
\ee
This enables us to use \eqref{pt_w} and \eqref{swi-ptw} to derive
\beq\label{dis-w}
{\bf 1}_{t\in\mathbb{I}_{l,\xi}}\fr{l(\xi-lt)}{l^2+(\xi-lt)^2}\les\fr{{\bf 1}_{t\in\mathbb{I}_{l,\xi}}}{1+\left|\fr{\xi}{l}-t\right|}
\les\sqrt{\fr{\partial_tw_l(\xi)}{w_l(\xi)}}\left(\sqrt{\fr{\partial_tw_k(\eta)}{w_k(\eta)}}+\fr{|\eta|^\fr{s}{2}}{\la t\ra^s}\right)\la\eta-\xi\ra.
\eeq
It follows from \eqref{switch-Akk}, \eqref{swi-Akl1}, \eqref{swi-Akl2} and \eqref{dis-w} that
\beq\label{10.6}
\nn&&\mathrm{NLS}_{1,\mathrm{HL}}^{1; \mathrm{R, R}}+\mathrm{NLS}_{1,\mathrm{HL}}^{1; \mathrm{NR, R}}\\
\nn&\les&\mu^{-\fr13}\sum_{N\ge8}\sum_{k,l}\int_{\eta,\xi} \left(\sqrt{\fr{\partial_tw_k(\eta)}{w_k(\eta)}}+\fr{|\eta|^\fr{s}{2}}{\la t\ra^s}\right)|A\hat{J}_k(\eta)|{\bf 1}_{t\in\mathbb{I}_{l,\xi}}\sqrt{\fr{\partial_tw_l(\xi)}{w_l(\xi)}}A_l(\xi)\left|\widehat{\Dl_L\Phi}_l(\xi)_N\right|\\
\nn&&\times e^{\lm|k-l,\eta-\xi|^s}\left(\left|\hat{\Om}_{k-l}(\eta-\xi)_{<N/8}\right|+\left|\widehat{\partial_{XX}\Psi}_{k-l}(\eta-\xi)_{<\fr{N}{8}}\right|\right)d\xi d\eta \\
\nn&\les&\mu^{-\fr13}\left(\left\|\sqrt{\fr{\partial_tw}{w}}\tl AJ\right\|_{L^2}+\fr{1}{\la t\ra^s}\left\||\nabla|^\fr{s}{2}AJ\right\|_{L^2}\right)\\
&&\times\left\|\left\la \fr{\partial_Y}{t\partial_X}\right\ra^{-1}\sqrt{\fr{\partial_tw}{w}}\tl A\Dl_L\Phi_{\ne} \right\|_{L^2}\left(\|\Om\|_{\mathcal{G}^{\lm,2}}+\|\partial_{XX}\Psi\|_{\mathcal{G}^{\lm,2}}\right),
\eeq
where we have used the fact 
\beqno
\nn|k|&\le& |l|+|k-l|\le\fr{{|\xi|}}{2}+|k-l|\le\fr{|\eta|}{2}+\fr{|\eta-\xi|}{2}+|k-l|\\
\nn&\le&\fr{|\eta|}{2}+\fr{3}{16}|l,\xi|\le\fr{|\eta|}{2}+\fr{3}{16}\cdot\fr32\cdot\fr{32}{23}|\eta|=\fr{41}{46} |\eta|,
\eeqno
due to \eqref{spt-T} and \eqref{eta=xi1} for $t\in\mathbb{I}_{l,\xi}$ to deduce that $A_k(\eta)\les\tl A_k(\eta)$ in the last inequality.  Then from \eqref{m1}, \eqref{loss1}, \eqref{pre-ell2}, \eqref{e-CCK}, \eqref{10.6} and the bootstrap hypotheses, we have
\beq\label{e-RRNRR}
\nn&&\int_1^t\mathrm{NLS}_{1,\mathrm{HL}}^{1; \mathrm{R, R}}+\mathrm{NLS}_{1,\mathrm{HL}}^{1; \mathrm{NR, R}}dt'\\
&\les&\mu^{-\fr23}\eps\left(\int_1^t\mathrm{CK}_{w,J}+\mathrm{CK}_{\lm,J}+\mu^{-\fr23}\eps^2\left(\mathrm{CK}_{w}^R+\mathrm{CK}_{\lm}^R\right)dt'\right)\les\mu^{-\fr23}\eps^3.
\eeq

{\em Treatment of $\mathrm{NLS}_{1,\mathrm{HL}}^{1; \mathrm{R, NR}}$ and $\mathrm{NLS}_{1,\mathrm{HL}}^{1; \mathrm{NR, NR}}$.}\par
\noindent For $\mathrm{NLS}_{1,\mathrm{HL}}^{1; \mathrm{R, NR}}$, we further split it into two parts:
\beqno
\mathrm{NLS}_{1,\mathrm{HL}}^{1; \mathrm{R, NR}}&=&2\sum_{N\ge8}\sum_{(k,l)\in\mathbb{Z}\times\mathbb{Z}_*}\int_{\eta,\xi}
\left({\bf 1}_{t\in\mathbb{I}_{k,\eta}\cap\mathbb{I}^c_{l,\xi}\cap\mathbb I_{k,\xi}}+{\bf 1}_{t\in\mathbb{I}_{k,\eta}\cap\mathbb{I}^c_{l,\xi}\cap\mathbb I_{k,\xi}^c}\right)A\bar{\hat{J}}_k(\eta)\\
&&\quad\quad\quad\times A_k(\eta)l(\xi-lt)\hat{\Phi}_l(\xi)_N\left(-\hat{\Om}_{k-l}(\eta-\xi)_{<\fr{N}{8}}+2\widehat{\partial_{XX}\Psi}_{k-l}(\eta-\xi)_{<\fr{N}{8}}\right)d\xi d\eta \\
&=&\mathrm{NLS}_{1,\mathrm{HL}}^{1; \mathrm{R, NR;D}}+\mathrm{NLS}_{1,\mathrm{HL}}^{1; \mathrm{R, NR;*}}.
\eeqno
Before proceeding any further, we claim that on the support of the integrand of $\mathrm{NLS}_{1,\mathrm{HL}}^{1; \mathrm{R, NR}}$, there holds $|\eta|\approx |\xi|$. As a matter of fact, combining the fact $t\in\mathbb{I}_{k,\eta}$ (which implies that $1\le|k|<\fr{\sqrt{|\eta|}}{2}$) with \eqref{spt-T} yields
\[
|l|\le |k|+|k-l|\le\fr{\sqrt{|\eta|}}{2}+\fr{3}{16}\left(|l|+|\xi|\right)\le\fr{1}{2}|k,\eta|+\fr{3}{16}\left(|l|+|\xi|\right).
\]
Thus, using again \eqref{spt-T}, we find that
\beqno
|l|\le\fr{8}{13}|k,\eta|+\fr{3}{13}|\xi|
\le\fr{8}{13}\cdot\fr{19}{16}|l,\xi|+\fr{3}{13}|\xi|
=\fr{19}{26}|l|+\fr{25}{26}|\xi|.
\eeqno
Consequently, $|l|\le \fr{25}{7}|\xi|$, and hence 
\be\label{eta=xi2}
|\xi|\approx |l,\xi|\approx |k,\eta|\approx |\eta|,
\ee
so the claim is true.

To deal with $\mathrm{NLS}_{1,\mathrm{HL}}^{1; \mathrm{R, NR;D}}$,  note first that \eqref{spt-T} and
 the restriction $t\in\mathbb{I}_{k,\xi}$ further imply that
\beno
|l|\le |k-l|+|k|\le\fr{3}{16}|l,\xi|+\fr{\sqrt{|\xi|}}{2}=\fr{3}{16}|l|+\fr{11}{16}|\xi|,
\eeno
and hence
\be\label{l<xi}
|l|\le \fr{11}{13}|\xi|.
\ee 
On the other hand, by virtue of \eqref{switch-A} and \eqref{switch-J1}, we have
\be\label{rnr-key}
\fr{{\bf 1}_{t\in\mathbb{I}_{k,\eta}\cap\mathbb{I}_{l,\xi}^c\cap\mathbb{I}_{k,\xi}}}{1+\left|\fr{\xi}{l}-t\right|}\fr{A_k(\eta)}{A_l(\xi)}\les \mu^{-\fr13}e^{c\lm|k-l,\eta-\xi|^s}e^{3r|k-l,\eta-\xi|^\fr12}\fr{{\bf 1}_{t\in\mathbb{I}_{k,\eta}\cap\mathbb{I}_{l,\xi}^c\cap\mathbb{I}_{k,\xi}}}{1+\left|\fr{\xi}{l}-t\right|}\fr{|\eta|}{|k|^2\left(1+\left|t-\fr{\eta}{k}\right|\right)}.
\ee
Next, we have to bound $\fr{{\bf 1}_{t\in\mathbb{I}_{k,\eta}\cap\mathbb{I}_{l,\xi}^c\cap\mathbb{I}_{k,\xi}}}{1+\left|\fr{\xi}{l}-t\right|}\fr{|\eta|}{|k|^2\left(1+\left|t-\fr{\eta}{k}\right|\right)}$.\par
\noindent{\bf Case 1: $\mathbb{I}_{l,\xi}\neq\emptyset$.} Now $t\in\mathbb{I}^c_{l,\xi}$ implies that 
\be\label{lowb}
\left|t-\fr{\xi}{l}\right|\ge \fr12\left(\fr{|\xi|}{|l|}-\fr{|\xi|}{|l|+1}\right)=\fr{|\xi|}{2|l|(|l|+1)}\ge\fr{|\xi|}{4|l|^2}.
\ee
Then we infer from \eqref{pt_w},\eqref{eta=xi2} and \eqref{lowb} that
\beq\label{720}
\fr{{\bf 1}_{t\in\mathbb{I}_{k,\eta}\cap\mathbb{I}_{l,\xi}^c\cap\mathbb{I}_{k,\xi}}}{1+\left|\fr{\xi}{l}-t\right|}\fr{|\eta|}{|k|^2\left(1+\left|t-\fr{\eta}{k}\right|\right)}\nn
\nn&\les&\la k-l\ra^2\fr{{\bf 1}_{t\in\mathbb{I}_{k,\eta}}}{\sqrt{1+|t-\fr{\eta}{k}|}}\fr{{\bf 1}_{t\in\mathbb{I}_{k,\xi}}}{\sqrt{1+|t-\fr{\xi}{k}|}}\fr{\sqrt{1+|t-\fr{\xi}{k}|}}{\sqrt{1+|t-\fr{\eta}{k}|}}\\
&\les&\la k-l,\eta-\xi\ra^3\sqrt{\fr{\partial_tw_k(\eta)}{w_k(\eta)}}\sqrt{\fr{\partial_tw_l(\xi)}{w_l(\xi)}}{\bf 1}_{t\in\mathbb{I}_{k,\xi}}.
\eeq
\noindent{\bf Case 2: $\mathbb{I}_{l,\xi}=\emptyset$.} In this case, either $\sqrt{|\xi|}\les|l|$ or $l\xi<0$ holds. If $\sqrt{|\xi|}\les|l|$,  then clearly \eqref{720} still holds. If $l\xi<0$, 
\beq\label{720'}
\fr{{\bf 1}_{t\in\mathbb{I}_{k,\eta}\cap\mathbb{I}_{l,\xi}^c\cap\mathbb{I}_{k,\xi}}}{1+\left|\fr{\xi}{l}-t\right|}\fr{|\eta|}{|k|^2\left(1+\left|t-\fr{\eta}{k}\right|\right)}
\nn&\les&\la k-l\ra\fr{{\bf 1}_{t\in\mathbb{I}_{k,\eta}}}{\sqrt{1+|t-\fr{\eta}{k}|}}\fr{{\bf 1}_{t\in\mathbb{I}_{k,\xi}}}{\sqrt{1+|t-\fr{\xi}{k}|}}\fr{\sqrt{1+|t-\fr{\xi}{k}|}}{\sqrt{1+|t-\fr{\eta}{k}|}}\\
&\les&\la k-l,\eta-\xi\ra^2\sqrt{\fr{\partial_tw_k(\eta)}{w_k(\eta)}}\sqrt{\fr{\partial_tw_l(\xi)}{w_l(\xi)}}{\bf 1}_{t\in\mathbb{I}_{k,\xi}}.
\eeq
On the other hand,
\be\label{720''}
{\bf 1}_{t\in\mathbb{I}_{k,\xi}}\approx{\bf 1}_{t\in\mathbb{I}_{k,\xi}}\left(1+\fr{|\xi|}{|kt|}\fr{|k|}{|l|}\right)\left\la\fr{|\xi|}{|lt|}\right\ra^{-1}\les\la k-l\ra\left\la\fr{|\xi|}{|lt|}\right\ra^{-1}.
\ee
It follows from \eqref{720}, \eqref{720'} and \eqref{720''} that
\be\label{RNR-key}
\fr{{\bf1}_{t\in\mathbb{I}_{k,\eta}\cap\mathbb{I}^c_{l,\xi}\cap\mathbb I_{k,\xi}}}{1+|\fr{\xi}{l}-t|}\fr{|\eta|}{|k|^2\left(1+\left|t-\fr{\eta}{k}\right|\right)}\les\sqrt{\fr{\partial_t{w_k(\eta)}}{w_k(\eta)}}\left(\left\la \fr{|\xi|}{|lt|} \right\ra^{-1}\sqrt{\fr{\partial_t{w_l(\xi)}}{w_l(\xi)}}\right)\la k-l,\eta-\xi\ra^4.
\ee
In view of \eqref{l<xi}, \eqref{rnr-key} and \eqref{RNR-key}, similar to \eqref{e-RRNRR}, we are led to
\beq
\nn\int_1^t\mathrm{NLS}_{1,\mathrm{HL}}^{1; \mathrm{R, NR;D}}dt'
&\les&\mu^{-\fr13}\int_1^t\left\|\sqrt{\fr{\partial_tw}{w}}\tl AJ\right\|_{L^2}\left\|\left\la \fr{\partial_Y}{t\partial_X}\right\ra^{-1}\sqrt{\fr{\partial_tw}{w}}\tl A\Dl_L\Phi_{\ne} \right\|_{L^2}\left(\left\| \Om\right\|_{\mathcal{G}^{\lm,2}}+\|\partial_{XX}\Psi\|_{\mathcal{G}^{\lm,2}}\right)dt'\\
&\les&\mu^{-\fr23}\eps^3.
\eeq

Next we go to bound $\mathrm{NLS}_{1,\mathrm{HL}}^{1; \mathrm{R, NR;*}}$ and $\mathrm{NLS}_{1,\mathrm{HL}}^{1; \mathrm{NR, NR}}$. Noting that \eqref{eta=xi2} holds on the support of the integrand of $\mathrm{NLS}_{1,\mathrm{HL}}^{1; \mathrm{R, NR;*}}$,  thanks to \eqref{J-ratio}, for $t\in\mathbb{I}_{k,\eta}\cap\mathbb{I}^c_{l,\xi}\cap\mathbb I_{k,\xi}^c$ or $t\notin \mathbb{I}_{k,\eta}\cup\mathbb{I}_{l,\xi}$, \eqref{switch-A} reduces to
\be\label{AkAl-ratio}
\fr{A_k(\eta)}{A_l(\xi)}\les\mu^{-\fr13}e^{c\lm|k-l,\eta-\xi|^s}e^{10r|k-l,\eta-\xi|^\fr12}.
\ee
To proceed, we claim that on the support of the integrand of $\mathrm{NLS}_{1,\mathrm{HL}}^{1; \mathrm{R, NR;*}}$ and $\mathrm{NLS}_{1,\mathrm{HL}}^{1; \mathrm{NR, NR}}$, there holds
\be\label{Reaction-key}
\fr{{\bf 1}_{l\ne0}{\bf1}_{t\notin\mathbb{I}_{l,\xi}}}{1+|t-\fr{\xi}{l}|}\les\left\la\fr{\xi}{lt}\right\ra^{-1}\left(\fr{1}{t}+\fr{|\eta|^{\fr{s}{2}}|\xi|^{\fr{s}{2}}}{t^{s+\fr12}}\right).
\ee
In fact, for $\mathrm{NLS}_{1,\mathrm{HL}}^{1; \mathrm{R, NR;*}}$, \eqref{Reaction-key} follows from \eqref{722} immediately owing to \eqref{eta=xi2}. For $\mathrm{NLS}_{1,\mathrm{HL}}^{1; \mathrm{NR, NR}}$, if $|l|\le2|\xi|$, using \eqref{spt-T}, one deduces that
\[
|\eta-\xi|\le\fr{3}{16}|l,\xi|\le\fr{9}{16}|\xi|.
\]
Consequently, $|\eta|\approx|\xi|$, and thus \eqref{Reaction-key} holds. If $|l|>2|\xi|$, then $\fr{|\xi|}{|l|}\le\fr{t}{2}$. Thus we can use \eqref{722a} to deduce \eqref{Reaction-key}. So the claim above is true. 
Then using \eqref{AkAl-ratio} and \eqref{Reaction-key}, we find that
\beq\label{need-ti1}
\nn&&\mathrm{NLS}_{1,\mathrm{HL}}^{1; \mathrm{R, NR;*}}+\mathrm{NLS}_{1,\mathrm{HL}}^{1; \mathrm{NR, NR}}\\
\nn&\les&\mu^{-\fr13}\sum_{(k,l)\in\mathbb{Z}\times\mathbb{Z^*}}\int_{\eta,\xi} \left|A\hat{J}_k(\eta)\right|\left\la\fr{\xi}{lt}\right\ra^{-1}\left(\fr{1}{t}+\fr{|\eta|^{\fr{s}{2}}|\xi|^{\fr{s}{2}}}{t^{s+\fr12}}\right)A_l(\xi)\left|\widehat{\Dl_L\Phi}_l(\xi)\right|\\
\nn&&\quad\times e^{\lm|k-l,\eta-\xi|^s}\left(\left|\hat{\Om}_{k-l}(\eta-\xi)\right|+\left|\widehat{\partial_{XX}\Psi}_{k-l}(\eta-\xi)\right|\right) d\xi d\eta\\
\nn&\les&\mu^{-\fr13}\fr{\left\| |\nabla|^{\fr{s}{2}}AJ\right\|_{L^2}}{t^{\fr{s}{2}+\fr14}}\left\|\left\la \fr{\partial_Y}{t\partial_X}\right\ra^{-1}\fr{|\nabla|^\fr{s}{2}}{\la t\ra^{\fr{s}{2}+\fr14}}A\Dl_L\Phi_{\ne} \right\|_{L^2} \left(\left\| \Om\right\|_{\mathcal{G}^{\lm,2}}+\left\|\partial_{XX}\Psi\right\|_{\mathcal{G}^{\lm,2}}\right)\\
&&+\mu^{-\fr13}\fr{\left\| AJ\right\|_{L^2}}{t}\left\|\left\la \fr{\partial_Y}{t\partial_X}\right\ra^{-1}A\Dl_L\Phi_{\ne} \right\|_{L^2} \left(\left\|\Om\right\|_{\mathcal{G}^{\lm,2}}+\left\|\partial_{XX}\Psi\right\|_{\mathcal{G}^{\lm,2}}\right).
\eeq
Combining this with \eqref{m1}, \eqref{pre-ell2}, Remark \ref{rem-s}, \eqref{e-CCK}, \eqref{pre-ell4} and the bootstrap hypotheses yields
\beq\label{e-RNRNRNR}
\nn&&\int_1^t \mathrm{NLS}_{1,\mathrm{HL}}^{1; \mathrm{R, NR;*}}+\mathrm{NLS}_{1,\mathrm{HL}}^{1; \mathrm{NR, NR}} dt'\\
\nn&\les&\mu^{-\fr23}\eps\left(\int_1^t\mathrm{CK}_{\lm,J}+\mathrm{CK}_{w,J}+\mu^{-\fr23}\eps^2\left(\mathrm{CK}_{\lm}^R+\mathrm{CK}_{w}^R\right)dt'\right)\\
&&+\mu^{-\fr23}\|A\Om\|_{L^\infty L^2}\|AJ\|_{L^\infty L^2}\|AJ_{\ne}\|_{L^2L^2}\les\mu^{-\fr56}\eps^3.
\eeq
\subsubsection{Correction} Now let us treat $\mathrm{NLS}_{1,\mathrm{HL}}^\eps$. To begin with, we expand $(Y'-1)\partial_{XY}^L\Phi$ with a paraproduct only in $Y$:
\beq
\mathrm{NLS}_{1,\mathrm{HL}}^\eps\nn&=&-\fr{1}{\pi}\sum_{M\ge8}\sum_{N\ge8}\sum_{(k,l)\in\mathbb{Z}\times\mathbb{Z^*}}\int_{\eta,\xi,\xi'} A\bar{\hat{J}}_k(\eta)A_k(\eta)\rho_N(l,\xi)\left[(\widehat{Y'-1})(\xi-\xi')\right]_{<\fr{M}{8}}\\
\nn&&\quad\quad\times\widehat{\partial^L_{XY}\Phi}_l(\xi')_M\left(-\hat{\Om}_{k-l}(\eta-\xi)_{<\fr{N}{8}}+2\widehat{\partial_{XX}\Psi}_{k-l}(\eta-\xi)_{<\fr{N}{8}}\right)d\xi'd\xi d\eta\\
\nn&&-\fr{1}{\pi}\sum_{M\ge8}\sum_{N\ge8}\sum_{(k,l)\in\mathbb{Z}\times\mathbb{Z^*}}\int_{\eta,\xi,\xi'} A\bar{\hat{J}}_k(\eta)A_k(\eta)\rho_N(l,\xi)\left[(\widehat{Y'-1})(\xi-\xi')\right]_{M}\\
\nn&&\quad\quad\times\widehat{\partial^L_{XY}\Phi}_l(\xi')_{<\fr{M}{8}}\left(-\hat{\Om}_{k-l}(\eta-\xi)_{<\fr{N}{8}}+2\widehat{\partial_{XX}\Psi}_{k-l}(\eta-\xi)_{<\fr{N}{8}}\right)d\xi'd\xi d\eta\\
\nn&&-\fr{1}{\pi}\sum_{M\in\mathbb{D}}\sum_{\fr{M}{8}\le M'\le 8M}\sum_{N\ge8}\sum_{(k,l)\in\mathbb{Z}\times\mathbb{Z^*}}\int_{\eta,\xi,\xi'} A\bar{\hat{J}}_k(\eta)A_k(\eta)\rho_N(l,\xi)\left[(\widehat{Y'-1})(\xi-\xi')\right]_{M'}\\
\nn&&\quad\quad\times\widehat{\partial^L_{XY}\Phi}_l(\xi')_{M}\left(-\hat{\Om}_{k-l}(\eta-\xi)_{<\fr{N}{8}}+2\widehat{\partial_{XX}\Psi}_{k-l}(\eta-\xi)_{<\fr{N}{8}}\right)d\xi'd\xi d\eta\\
\nn&=&\mathrm{NLS}_{1,\mathrm{HL}}^{\eps;\mathrm{LH}}+\mathrm{NLS}_{1,\mathrm{HL}}^{\eps;\mathrm{HL}}+\mathrm{NLS}_{1,\mathrm{HL}}^{\eps;\mathcal{R}}.
\eeq

For $\mathrm{NLS}_{1,\mathrm{HL}}^{\eps;\mathrm{LH}}$, $|l,\xi'|\approx|l,\xi|$ and $\widehat{\partial^L_{XY}\Phi}_l(\xi')_M$ is the true high frequency quantity. Then $\mathrm{NLS}_{1,\mathrm{HL}}^{\eps;\mathrm{LH}}$ can be treated as $\mathrm{NLS}_{1,\mathrm{HL}}$ with $(l,\xi')$ playing the role of $(l,\xi)$. We omit the details and conclude that the estimates in Section \ref{sec-re-m} are still valid, except with an additional factor $\|Y'-1\|_{\mathcal{G}^{\lm,2}}$.
 
The treatment of $\mathrm{NLS}_{1,\mathrm{HL}}^{\eps;\mathrm{HL}}$ is more subtle:  $l$ could be large relative to $|\xi-\xi'|$ and hence $\widehat{\partial^L_{XY}\Phi}_l(\xi')_{<\fr{M}{8}}$ is again a high frequency quantity. We thus consider two cases
\beq
\nn\mathrm{NLS}_{1,\mathrm{HL}}^{\eps;\mathrm{HL}}&=&-\fr{1}{\pi}\sum_{M\ge8}\sum_{N\ge8}\sum_{(k,l)\in\mathbb{Z}\times\mathbb{Z^*}}\int_{\eta,\xi,\xi'}\left({\bf 1}_{|\xi|\le4|l|}+{\bf 1}_{|\xi|>4|l|}\right) A\bar{\hat{J}}_k(\eta)A_k(\eta)\\
\nn&&\times \rho_N(l,\xi)\left[(\widehat{Y'-1})(\xi-\xi')\right]_{M}\widehat{\partial^L_{XY}\Phi}_l(\xi')_{<\fr{M}{8}}\\
\nn&&\times \left(-\hat{\Om}_{k-l}(\eta-\xi)_{<\fr{N}{8}}+2\widehat{\partial_{XX}\Psi}_{k-l}(\eta-\xi)_{<\fr{N}{8}}\right)d\xi'd\xi d\eta\\
\nn&=&\mathrm{NLS}_{1,\mathrm{HL}}^{\eps;\mathrm{HL};X}+\mathrm{NLS}_{1,\mathrm{HL}}^{\eps;\mathrm{HL};Y}.
\eeq

For $\mathrm{NLS}_{1,\mathrm{HL}}^{\eps;\mathrm{HL};X}$, on the support of the integrand (noting that $|\xi|\le4|l|$), there hold
\be\label{spt-X}
\big||k,\eta|-|l,\xi| \big|\le |k-l,\eta-\xi|\le\fr{3}{16}|l,\xi|\le\fr{15}{16}|l|,\quad\mathrm{and}\quad |\xi'|\le\fr{3}{16}|\xi-\xi'|.
\ee
This enables us to use \eqref{ap3} and \eqref{ap4} (considering two cases $|\xi|\le\fr{|l|}{4}$ and $\fr{|l|}{4}\le|\xi|\le4|l|$ separately) to deduce that, for some $c\in(0,1)$
\beno
e^{\lm|k,\eta|^s}\le e^{\lm|l,\xi'|^s+c\lm|\xi-\xi'|^s+c\lm|k-l,\eta-\xi|^s}.
\eeno
Furthermore, together with the restriction $|\xi|\le4|l|$, we infer from \eqref{spt-X} that
\beno
|l|\ge\fr{|\xi|}{4}\ge\fr{1}{4}\cdot\fr{13}{16}|\xi-\xi'|\ge\fr{13}{12}|\xi'|.
\eeno
This implies that on the support of the integrand $(l,\xi')$ is non-resonant.  
If the worst scenario $t\in \mathbb{I}_{k,\eta}\backslash\mathbb{I}_{l,\xi'}$ happen, the exchange from $\mathcal{J}_k(\eta)$ to $\mathcal{J}_l(\xi')$ may produce a bad factor $\fr{|\eta|}{k^2\left(1+\left|\fr{\eta}{k}-t\right|\right)}$ that is harmless. Indeed, using \eqref{spt-X} and the restriction $|\xi|\le4|l|$, we find that
\beno
\fr{|\eta|}{k^2\left(1+\left|\fr{\eta}{k}-t\right|\right)}\les\fr{|l|}{|k|}\les\la k-l\ra.
\eeno
On the other hand, noting that $\left|\widehat{\partial^L_{XY}\Phi}_l(\xi')_{<\fr{M}{8}}\right|\le \fr{{\bf1}_{l\ne0}}{1+\left|\fr{\xi'}{l}-t\right|}\left|\widehat{\Dl_{L}\Phi}_l(\xi')_{<\fr{M}{8}}\right|$, thanks to the fact $|\xi'|\le\fr{12}{13}|l|\le\fr{12}{13}|l|t$, similar to \eqref{722a}, we have
\be\label{X-domi}
\fr{{\bf1}_{l\ne0}}{1+\left|\fr{\xi'}{l}-t\right|}\les\left\la\fr{\xi'}{lt}\right\ra^{-1}\fr{1}{t}.
\ee
Therefore,
\beq
\mathrm{NLS}_{1,\mathrm{HL}}^{\eps;\mathrm{HL};X}\nn&\les&\fr{\mu^{-\fr13}}{t}\sum_{M\ge8}\sum_{N\ge8}\sum_{(k,l)\in\mathbb{Z}\times\mathbb{Z^*}}\int_{\eta,\xi,\xi'}{\bf 1}_{|\xi|\le4|l|} \left|A\hat{J}_k(\eta)\right|\rho_N(l,\xi)\\
\nn&&\times e^{\lm|\xi-\xi'|^s}\left|(\widehat{Y'-1})(\xi-\xi')_{M}\right|\left\la\fr{\xi'}{lt}\right\ra^{-1}A_l(\xi')\left|\widehat{\Dl_{L}\Phi}_l(\xi')_{<\fr{M}{8}}\right|\\
\nn&&\times e^{\lm|k-l,\eta-\xi|^s}\left|-\hat{\Om}_{k-l}(\eta-\xi)_{<\fr{N}{8}}+2\widehat{\partial_{XX}\Psi}_{k-l}(\eta-\xi)_{<\fr{N}{8}}\right|d\xi'd\xi d\eta\\
\nn&\les&\fr{\mu^{-\fr13}}{t}\|AJ\|_{L^2}\|Y'-1\|_{\mathcal{G}^{\lm,2}}\left\|\left\la \fr{\partial_Y}{t\partial_X}\right\ra^{-1}A\Dl_L\Phi_{\ne} \right\|_{L^2}\left(\left\| \Om\right\|_{\mathcal{G}^{\lm,2}}+\left\|\partial_{XX}\Psi\right\|_{\mathcal{G}^{\lm,2}}\right).
\eeq
Then similar to \eqref{e-RNRNRNR}, one deduces that
\be
\int_1^t\mathrm{NLS}_{1,\mathrm{HL}}^{\eps;\mathrm{HL};X}dt'\les\mu^{-\fr23}\|Y'-1\|_{L^\infty \mathcal{G}^{\lm,2}}\|A\Om\|_{L^\infty L^2}\|AJ\|_{L^\infty L^2}\|AJ_{\ne}\|_{L^2 L^2}\les\mu^{-\fr56}\eps^4.
\ee

For $\mathrm{NLS}_{1,\mathrm{HL}}^{\eps;\mathrm{HL};Y}$, on the support of the integrand,
\beno
\big||\xi-\xi'|-|l,\xi| \big|\le|l|+|\xi'|\le\fr{|\xi|}{4}+|\xi'|\le\fr{|\xi-\xi'|}{4}+\fr{5|\xi'|}{4}\le\fr{31}{64}|\xi-\xi'|.
\eeno
Combining this with \eqref{spt-T} (using \eqref{ap3} twice) yields
\be\label{e-exp}
e^{\lm|k,\eta|^s}\le e^{\lm|\xi-\xi'|^s+c\lm|l,\xi'|^s+c\lm|k-l,\eta-\xi|^s},\quad\mathrm{for \ \ some} \quad c\in(0,1).
\ee
In addition, on the support of the integrand, we have
\beno
\left(\fr{13}{16}\right)^2|\xi-\xi'|\le\fr{13}{16}|\xi|\le\fr{13}{16}|l,\xi|\le|k,\eta|\le\fr{19}{16}|l,\xi|\le\fr{5}{4}\cdot\fr{19}{16}|\xi|\le\fr{5}{4}\cdot\left(\fr{19}{16}\right)^2|\xi-\xi'|,
\eeno
and
\beno
|k|\le|k-l|+|l|\le\fr{3}{16}|l,\xi|+|l|\le\fr{3}{16}|\xi|+\fr{19}{16}|l|\le\fr{31}{64}|\xi|\le\fr{31}{64}\cdot\fr{19}{16}|\xi-\xi'|=\fr{589}{1024}|\xi-\xi'|.
\eeno
Nothing that $\left(\fr{13}{16}\right)^2>\fr{589}{1024}$,  the above two inequalities  imply $|\eta|\approx|\xi-\xi'|$. Then in view of \eqref{e-exp}, similar to \eqref{AkR1} and \eqref{AkR2}, we are led to
\beno
\fr{A_k(\eta)}{A^R(\xi-\xi')}\les e^{c\lm|l,\xi'|^s+c\lm|k-l,\eta-\xi|^s},
\eeno
where the constant $c\in(0,1)$ may be slightly larger than the one in \eqref{e-exp}. This means that all derivatives landing on $Y'-1$ and $\rho_N(l,\xi)\widehat{\partial^L_{XY}\Phi}_l(\xi')_{<\fr{M}{8}}$ is actually  a low frequency quantity.  Accordingly,
\beq
\nn\mathrm{NLS}_{1,\mathrm{HL}}^{\eps;\mathrm{HL};Y}&\les&\sum_{M\ge8}\sum_{N\ge8}\sum_{(k,l)\in\mathbb{Z}\times\mathbb{Z^*}}\int_{\eta,\xi,\xi'} \left|A\hat{J}_k(\eta)\right|\rho_N(l,\xi)\\
\nn&&\times A^R(\xi-\xi')\left|(\widehat{Y'-1})(\xi-\xi')_{M}\right|e^{c\lm|l,\xi'|^s}\left|\widehat{\partial^L_{XY}\Phi}_l(\xi')_{<\fr{M}{8}}\right|\\
\nn&&\times e^{c\lm|k-l,\eta-\xi|^s}\left|-\hat{\Om}_{k-l}(\eta-\xi)_{<\fr{N}{8}}+2\widehat{\partial_{XX}\Psi}_{k-l}(\eta-\xi)_{<\fr{N}{8}}\right|d\xi'd\xi d\eta\\
\nn&\les&\left\|AJ\right\|_{L^2}\left\|A^R(Y'-1)\right\|_{L^2}\left(t\|\Phi_\ne\|_{\mathcal{G}^{\lm,2}}\right)\left(\|\Om\|_{\mathcal{G}^{\lm,2}}+\left\|\partial_{XX}\Psi\right\|_{\mathcal{G}^{\lm,2}}\right),
\eeq
where we have used $\left|\widehat{\partial_{XY}^L\Phi}_l(\xi') \right|\les t|l,\xi'|^2\left|\hat{\Phi}_l(\xi') \right|$ in the last line above. Then by \eqref{m1}, \eqref{loss2} and the bootstrap hypotheses, 
\be\label{e-HLY}
\int_1^t\mathrm{NLS}_{1,\mathrm{HL}}^{\eps;\mathrm{HL};Y}dt'\les\mu^{-\fr23}\|A\Om\|_{L^\infty L^2}\|AJ\|_{L^\infty L^2}\left\|A^R(Y'-1)\right\|_{L^\infty L^2}\|AJ_{\ne}\|_{L^2 L^2}\les\mu^{-\fr56}\eps^4.
\ee

Now we turn to the remainder term $\mathrm{NLS}_{1,\mathrm{HL}}^{\eps;\mathcal{R}}$, which is also 
divided into two parts:
\beq
\mathrm{NLS}_{1,\mathrm{HL}}^{\eps;\mathcal{R}}\nn&=&-\fr{1}{\pi}\sum_{M\in\mathbb{D}}\sum_{\fr{M}{8}\le M'\le 8M}\sum_{N\ge8}\sum_{(k,l)\in\mathbb{Z}\times\mathbb{Z^*}}\int_{\eta,\xi,\xi'} A\bar{\hat{J}}_k(\eta)A_k(\eta)\\
\nn&&\times \rho_N(l,\xi)\left({\bf 1}_{|l|>100|\xi'|}+{\bf 1}_{|l|}\le100|\xi'|\right)\left[(\widehat{Y'-1})(\xi-\xi')\right]_{M'}\\
\nn&&\times\widehat{\partial^L_{XY}\Phi}_l(\xi')_{M}\left(-\hat{\Om}_{k-l}(\eta-\xi)_{<\fr{N}{8}}+2\widehat{\partial_{XX}\Psi}_{k-l}(\eta-\xi)_{<\fr{N}{8}}\right)d\xi'd\xi d\eta\\
\nn&=&\mathrm{NLS}_{1,\mathrm{HL}}^{\eps;\mathcal{R};X}+\mathrm{NLS}_{1,\mathrm{HL}}^{\eps;\mathcal{R};Y}.
\eeq
For $\mathrm{NLS}_{1,\mathrm{HL}}^{\eps;\mathcal{R};X}$, on the support of the integrand we have
\beno
\big| |l,\xi|-|l,\xi'|\big|\le |\xi-\xi'|\le24|\xi'|\le\fr{24}{100}|l,\xi'|.
\eeno
Furthermore, the restriction $|l|>100|\xi'|$ and \eqref{spt-T} imply
\beno
|\xi|\le|\xi-\xi'|+|\xi'|\le25|\xi'|\le\fr{|l|}{4},\quad |k-l|\le\fr{3}{16}|l,\xi|\le\fr{3}{16}\cdot\fr{5}{4}|l|=\fr{15}{64}|l|,
\eeno
and hence
\beno
|k|\ge \fr{49}{64} |l|,\quad          |\eta|\le|\xi|+|\eta-\xi|\le\fr{|l|}{4}+\fr{3}{16}|l,\xi|\le\fr{31}{64}|l|\le\fr{31}{49}|k|.
\eeno
Thus, both $(l,\xi')$ and $(k,\eta)$ are non-resonant. Then combining  \eqref{J-ratio} and \eqref{ap3} twice yields
\beno
\fr{A_k(\eta)}{A_l(\xi')}\les\mu^{-\fr13}e^{c\lm|\xi-\xi'|^s+c\lm|k-l,\eta-\xi|^s}.
\eeno
In addition, $|l|\ge100|\xi'|$ implies that \eqref{X-domi} still holds. Therefore, $\mathrm{NLS}_{1,\mathrm{HL}}^{\eps;\mathcal{R};X}$ can be bounded in the same manner of $\mathrm{NLS}_{1,\mathrm{HL}}^{\eps;\mathrm{HL};X}$, and we omit the details.

As for $\mathrm{NLS}_{1,\mathrm{HL}}^{\eps;\mathcal{R};Y}$, on the support of the integrand we have
\beno
|l,\xi'|\le101|\xi'|\le24\cdot 101|\xi-\xi'|\le24^2\cdot101|\xi'|\le24^2\cdot101|l,\xi'|,
\eeno
and hence $|l,\xi'|\approx|\xi'|\approx|\xi-\xi'|$. Then owing to \eqref{ap3} and \eqref{ap4}, for some $c\in(0,1)$,
\beno
e^{\lm|k,\eta|^s}\le e^{c\lm|l,\xi'|^s+c\lm|\xi-\xi'|^s+c\lm|k-l,\eta-\xi|^s}.
\eeno
On the other hand, $\mathcal{J}_k(\eta)$ can be bounded straightforwardly by using Lemma \ref{lem-gw}:
\beno
\mathcal{J}_k(\eta)\les e^{4r|k,\eta|^{\fr12}}\les e^{4r|k-l,\eta-\xi|^{\fr12}+4r|l,\xi'|^{\fr12}+4r|\xi-\xi'|^{\fr12}}.
\eeno
It follows from the above two estimates, \eqref{ap-5} and \eqref{ap-6} (noting that $\mathcal{M}_k(\eta)\les1$) that
\beq
\mathrm{NLS}_{1,\mathrm{HL}}^{\eps;\mathcal{R};Y}\nn&\les&\sum_{M\in\mathbb{D}}\sum_{\fr{M}{8}\le M'\le 8M}\sum_{N\ge8}\sum_{(k,l)\in\mathbb{Z}\times\mathbb{Z^*}}\int_{\eta,\xi,\xi'} \left|A\hat{J}_k(\eta)\right|\rho_N(l,\xi)\\
\nn&&\times e^{\lm|\xi-\xi'|^{s}}\left|(\widehat{Y'-1})(\xi-\xi')_{M'}\right|e^{c\lm|l,\xi'|^s+4r|l,\xi'|^\fr12}\left(t|l,\xi'|^2|\hat{\Phi}_l(\xi')_{M}|\right)\\
\nn&&\times e^{\lm|k-l,\eta-\xi|^s}\left|-\hat{\Om}_{k-l}(\eta-\xi)_{<\fr{N}{8}}+2\widehat{\partial_{XX}\Psi}_{k-l}(\eta-\xi)_{<\fr{N}{8}}\right|d\xi'd\xi d\eta\\
\nn&\les&\left\|AJ\right\|_{L^2}\|Y'-1\|_{\mathcal{G}^{\lm, 2}}\left(t\|\Phi_\ne\|_{\mathcal{G}^{\lm,2}}\right)\left(\|\Om\|_{\mathcal{G}^{\lm,2}}+\left\|\partial_{XX}\Psi\right\|_{\mathcal{G}^{\lm,2}}\right).
\eeq
Then similar to \eqref{e-HLY}, we have
\be\label{e-RY}
\int_1^t\mathrm{NLS}_{1,\mathrm{HL}}^{\eps;\mathcal{R};Y}dt'\les\mu^{-\fr56}\eps^4.
\ee
\subsection{ Treatment of $\mathrm{NLS}_{1,\mathrm{LH}}$} First of all, let us denote
\be
\mathrm{NLS}_{1,\mathrm{LH}}=\mathrm{NLS}_{1,\mathrm{LH}}^1+\mathrm{NLS}_{1,\mathrm{LH}}^2,
\ee
where
\beqno
\mathrm{NLS}_{1,\mathrm{LH}}^1&=&4\pi\sum_{N\ge8}\left\la AJ, A\left(\partial_{XY}^t\Phi_{<\fr{N}{8}}\Om_{N}\right)\right\ra,\\
\mathrm{NLS}_{1,\mathrm{LH}}^2&=&-8\pi\sum_{N\ge8}\left\la AJ, A\left(\partial_{XY}^t\Phi_{<\fr{N}{8}}\partial_{XX}\Psi_{N}\right)\right\ra.
\eeqno
We would like to remark that $\mathrm{NLS}_{1,\mathrm{LH}}^2$ should be regarded as an analogue of $\mathrm{NLS}_{1,\mathrm{HL}}^1$. Owing to the fact
\beno
\fr{{\bf 1}_{l\ne0}l^2}{l^2+(\xi-lt)^2}\les\fr{{\bf 1}_{l\ne0}}{\left(1+\left|\fr{\xi}{l}-t\right|\right)^2},
\eeno
the high frequency term $\partial_{XX}\Psi_{N}$ in $\mathrm{NLS}_{1,\mathrm{LH}}^2$ possesses better decay than the corresponding term $\partial_{XY}^L\Phi_{N}$ in $\mathrm{NLS}_{1,\mathrm{HL}}^1$. 

Let us focus on the estimates of $\mathrm{NLS}_{1,\mathrm{LH}}^1$ next. It is natural to divide it into the following two parts:
\beq
\nn \mathrm{NLS}_{1,\mathrm{LH}}^1&=&2\sum_{N\ge8}\sum_{\substack{(k,l)\in\mathbb{Z}^2, \\k\ne l}}\int_{\eta,\xi} \left({\bf 1}_{t\in\mathbb{I}_{k,\eta}\cap\mathbb{I}_{k,\xi}}+{\bf 1}_{t\notin\mathbb{I}_{k,\eta}\cap\mathbb{I}_{k,\xi}}\right)A\bar{\hat{J}}_k(\eta)A_k(\eta)\hat{\Om}_l(\xi)_N \\
\nn&&\quad\quad\quad\quad\quad\quad\times\mathcal{F}\left[Y'\partial^L_{XY}\Phi\right]_{k-l}(\eta-\xi)_{<\fr{N}{8}}d\xi d\eta\\
&=&\mathrm{NLS}_{1,\mathrm{LH}}^{1;\mathrm{D}}+\mathrm{NLS}_{1,\mathrm{LH}}^{1;*}.
\eeq
Similar to\eqref{switch-A} and \eqref{switch-J1}, we have
\beq
\nn{\bf 1}_{t\in\mathbb{I}_{k,\eta}\cap\mathbb{I}_{k,\xi}}\fr{A_k(\eta)}{A_l(\xi)}&\les&\mu^{-\fr13}e^{c\lm|k-l, \eta-\xi|^s}e^{3r|k-l, \eta-\xi|^\fr12}\fr{|\eta|}{k^2}\fr{{\bf 1}_{t\in\mathbb{I}_{k,\eta}}}{\sqrt{1+\left|\fr{\eta}{k}-t\right|}}\fr{{\bf 1}_{t\in\mathbb{I}_{k,\xi}}}{\sqrt{1+\left|\fr{\xi}{k}-t\right|}}\fr{\sqrt{1+\left|\fr{\xi}{k}-t\right|}}{\sqrt{1+\left|\fr{\eta}{k}-t\right|}}\\
\nn&\les&\mu^{-\fr13}e^{c\lm|k-l, \eta-\xi|^s} e^{3r|k-l, \eta-\xi|^\fr12}\fr{|\eta|}{k^2}\sqrt{\fr{\partial_tw_k(\eta)}{w_k(\eta)}}\sqrt{\fr{\partial_tw_l(\xi)}{w_l(\xi)}}\la \eta-\xi\ra^\fr12.
\eeq
On the other hand, it is easy to verify that $A_k(\eta)\les \tl A_k(\eta)$ and $A_l(\eta)\les \tl A_l(\eta)$ on the support of the integrand of $\mathrm{NLS}_{1,\mathrm{LH}}^{1;\mathrm{D}}$. Consequently,
\beqno
\nn \mathrm{NLS}_{1,\mathrm{LH}}^{1;\mathrm{D}}&\les&\mu^{-\fr13}\sum_{N\ge8}\sum_{\substack{(k,l)\in\mathbb{Z}^2,\\ k\ne l}}\int_{\eta,\xi}\fr{|\eta|}{k^2}\sqrt{\fr{\partial_tw_k(\eta)}{w_k(\eta)}}\left|\tl A\hat{J}_k(\eta)\right|\sqrt{\fr{\partial_tw_l(\xi)}{w_l(\xi)}}\tl A_l(\xi)|\hat{\Om}_l(\xi)_N|\\
\nn&&\times e^{\lm|k-l,\eta-\xi|^s}\left|\mathcal{F}\left[Y'\partial^L_{XY}\Phi\right]_{k-l}(\eta-\xi)_{<\fr{N}{8}}\right| d\xi d\eta \\
&\les&\mu^{-\fr13}t\left\|Y'\partial_{XY}^L\Phi\right\|_{\mathcal{G}^{\lm,2}}\left\|\sqrt{\fr{\partial_tw}{w}}\tl AJ\right\|_{L^2}\left\|\sqrt{\fr{\partial_tw}{w}}\tl A\Om\right\|_{L^2}.
\eeqno
The algebra property \eqref{alge} of $\mathcal{G}^{\lm,2}$ shows that
\be\label{prtXYPhi}
\left\|Y'\partial_{XY}^L\Phi\right\|_{\mathcal{G}^{\lm,2}}\les\left(1+\left\|Y'-1\right\|_{\mathcal{G}^{\lm,2}}\right)\left\|\partial_{XY}^L\Phi\right\|_{\mathcal{G}^{\lm,2}}\les t\left(1+\left\|Y'-1\right\|_{\mathcal{G}^{\lm,2}}\right)\left\|\Phi_\ne\right\|_{\mathcal{G}^{\lm,4}}.
\ee
Combining the above two estimates with \eqref{m1}, \eqref{loss2} and the bootstrap hypotheses yields
\be
\int_1^t\mathrm{NLS}_{1,\mathrm{LH}}^{1;\mathrm{D}}dt'\les\mu^{-\fr23}\|AJ\|_{L^\infty L^2}\left\|\sqrt{\fr{\partial_tw}{w}}\tl AJ\right\|_{L^2L^2}\left\|\sqrt{\fr{\partial_tw}{w}}\tl A\Om\right\|_{L^2L^2}\les\mu^{-\fr23}\eps^3.
\ee

Now we turn to bound $\mathrm{NLS}_{1,\mathrm{LH}}^{1;*}$. Note that
\beno
{\bf 1}_{t\notin\mathbb{I}_{k,\eta}\cap\mathbb{I}_{k,\xi}}={\bf 1}_{t\in\mathbb{I}_{k,\eta}^c}+{\bf 1}_{t\in\mathbb{I}_{k,\xi}^c\cap\mathbb{I}_{k,\eta} }=\left({\bf 1}_{t\in\mathbb{I}_{k,\eta}^c}+{\bf 1}_{t\in\mathbb{I}_{k,\xi}^c\cap\mathbb{I}_{k,\eta} }{\bf 1}_{|l|\le4|\xi|}\right)+{\bf 1}_{t\in\mathbb{I}_{k,\xi}^c\cap\mathbb{I}_{k,\eta} }{\bf 1}_{|l|>4|\xi|}.
\eeno
If $t\notin\mathbb{I}_{k,\eta}$ or ($t\in\mathbb{I}_{k,\eta}\backslash\mathbb{I}_{k,\xi}$ and $|l|\le4|\xi|$, which implies $|\eta|\approx|\xi|$), then \eqref{J-ratio} applies. If $t\in\mathbb{I}_{k,\eta}\backslash\mathbb{I}_{k,\xi}$ and $|l|>4|\xi|$, then \eqref{switch-J} holds. We thus conclude that \eqref{AkAl-ratio} holds for $t\notin\mathbb{I}_{k,\eta}\cap\mathbb{I}_{k,\xi}$. Accordingly, using again \eqref{m1}, \eqref{loss2} and the bootstrap hypotheses, we find that
\beq\label{need-ti2}
\nn \int_1^t\mathrm{NLS}_{1,\mathrm{LH}}^{1;*}dt'&\les&\mu^{-\fr13}\sum_{N\ge8}\sum_{\substack{(k,l)\in\mathbb{Z}^2,\\ k\ne l}}\int_1^t\int_{\eta,\xi} \left|A\hat{J}_k(\eta)\right| A_l(\xi)|\hat{\Om}_l(\xi)_N| \\
\nn&&\times e^{\lm|k-l,\eta-\xi|^s}\left|\mathcal{F}\left[Y'\partial^L_{XY}\Phi\right]_{k-l}(\eta-\xi)_{<\fr{N}{8}}\right|d\xi d\eta dt'\\
&\les&\mu^{-\fr13}\left(1+\left\|Y'-1\right\|_{L^\infty\mathcal{G}^{\lm,2}}\right)\left\|t'\Phi_\ne\right\|_{L^1\mathcal{G}^{\lm,4}}\left\|AJ\right\|_{L^\infty L^2}\left\| A\Om\right\|_{L^\infty L^2}\les\mu^{-\fr56}\eps^3.
\eeq
\subsection{Treatment of $\mathrm{NLS}_{1,\mathcal{R}}$}  To treat $\mathrm{NLS}_{1,\mathcal{R}}$, there is no need to split $\partial_{XY}^t\Phi$ into the main part $\partial_{XY}^L\Phi$  and the correction $(Y'-1)\partial_{XY}^L\Phi$. In fact,  the treatments of $\mathrm{NLS}_{1,\mathcal{R}}$, the remainder in Section \ref{sec-rem} and $\mathrm{NLS}_{1,\mathrm{HL}}^{\eps;\mathcal{R};Y}$ share a basic similarity. We omit the details and conclude that
\beq
\nn\int_1^t\mathrm{NLS}_{1,\mathcal{R}}dt'&\les&\|AJ\|_{L^\infty L^2}\left(\|\Om\|_{L^\infty\mathcal{G}^{\lm,2}}+\|\partial_{XX}\Psi\|_{L^\infty\mathcal{G}^{\lm,2}}\right)\left\|Y'\partial_{XY}^L\Phi\right\|_{L^1\mathcal{G}^{\lm,2}}\\
&\les&\mu^{-\fr23}\eps\|A\Om\|_{L^\infty L^2}\|AJ_{\ne}\|_{L^2L^2}\les\mu^{-\fr56}\eps^3,
\eeq
where we have used \eqref{m1}, \eqref{loss1}, \eqref{loss2} and  \eqref{prtXYPhi} above.
\begin{rem}
Given the same structure of $\mathrm{NLS}_1$ and $\mathrm{NLS}_2$, one can deal with $\mathrm{NLS}_2$ by following the way of treatment of $\mathrm{NLS}_1$ above. Unlike $\mathrm{NLS}_1$, even though the elliptic estimates (see \eqref{loss1} and Remark \ref{rem-preci-Psi}) hold for $\Psi_{\ne}$, the quantity  $\partial_{XY}^t\Psi$ without zero mode in $\mathrm{NLS}_2$  cannot provide any extra time integrability which is needed in \eqref{need-ti1} and \eqref{need-ti2}, for instance. Fortunately, the two $J$ in $\mathrm{NLS}_2$ are impossible at zero-mode simultaneously. Hence, it is vital to use  $\left\|AJ_{\ne}\right\|_{L^2}$ to compensate  the time integrability.

In addition, thanks to the presence of the multiplier $\fr{\partial_t(m^\fr12)}{m^{\fr12}}\partial_X^{-1}$, all the estimates for $\mathbf{NLS}$ apply to ${\bf CNLS}$. In fact,  apart from \eqref{up-pt_m}, we also have
\beno
\left|\mathcal{F}\left[\fr{\partial_t(m^{\fr12})}{m^\fr12}\partial^{-1}_X\right]\right|\les\sqrt{\fr{\partial_tM^1}{M^1}}.
\eeno
This enables us to  use $\left\| \sqrt{\fr{\partial_tM^1}{M^1}}A\Om\right\|_{L^2}$ to play the role of $\left\|AJ_{\ne}\right\|_{L^2}$ in $\mathbf{NLS}$.
\end{rem}

\section{Estimates of $b^1_0$ and $u_0^1$ in $L^2$}\label{sec-ub0} 
In this section, we will bound the $L^2$ norm of $b_0^1$ and $u_0^1$ in the original coordinate $(t, x, y)$. One can also regard it as the energy estimate of $B_0^1$ in weighted $L^2$ space, with solution-dependent weight $\frac{1}{\sqrt{Y'}}$. The reason for this is to eliminate the slow-decay dissipation error term $Y''\partial_YB_0^1$. More precisely, although, in the energy estimate of $B_0^1$, we do not need to worry about the derivative loss problem. However, the time integrability of the normal $L^2$ inner product $\left\langle Y''\partial_YB_0^1, B_0^1\right\rangle$ requires some weak decay of $B_0^1$, which requires more estimates that we want to avoid. So we estimate them in $(t, x, y)$ coordinates.

The hypotheses \eqref{coor1} ensures that $Y'\approx1$ for $\eps$ sufficiently small, which in turn implies that the $L^2$ norms of any given square-integrable function are equivalent under the invertible coordinate transform $y\rightarrow Y$. On this basis, taking the $L^2$ inner product of  the second equation of \eqref{ub01} with $b^1_0$, we have
\beq
\fr12\fr{d}{dt}\|b^1_0\|^2_{L^2}+\mu\|\partial_yb_0^1\|_{L^2}^2\nn&=&-\int_{\mathbb{R}}\left(\nabla^\bot\psi_\ne\cdot\nabla b^1_\ne\right)_0b^1_0dy+\int_{\mathbb{R}}\left(\nabla^\bot\phi_\ne\cdot\nabla u^1_\ne\right)_0b^1_0dy\\
\nn&\le&\left(\left\|\left(\nabla^\bot\psi_\ne\cdot\nabla b^1_\ne\right)_0\right\|_{L^2}+\left\|\left(\nabla^\bot\phi_\ne\cdot\nabla u^1_\ne\right)_0\right\|_{L^2}\right)\|b^1_0\|_{L^2}\\
\nn&\approx&\left(\left\|\left(Y'\nabla^\bot\Psi_\ne\cdot\nabla B^1_\ne\right)_0\right\|_{L^2}+\left\|\left(Y'\nabla^\bot\Phi_\ne\cdot\nabla U^1_\ne\right)_0\right\|_{L^2}\right)\|b^1_0\|_{L^2}\\
\nn&\les&\left(1+\|Y'-1\|_{L^\infty}\right)\left(\left\|\nabla^\bot\Psi_\ne\cdot\nabla B^1_\ne\right\|_{L^2}+\left\|\nabla^\bot\Phi_\ne\cdot\nabla U^1_\ne\right\|_{L^2}\right)\|b^1_0\|_{L^2}\\
\nn&\les&\left(\|\nabla B^1_\ne\|_{L^\infty}\|\nabla^\bot\Psi_\ne\|_{L^2}+\|\nabla U^1_\ne\|_{L^\infty}\|\nabla^\bot\Phi_\ne\|_{L^2}\right)\|b^1_0\|_{L^2}.
\eeq
Recalling that $B^1=-Y'\partial_Y^L\Phi$, and $U^1=-Y'\partial_Y^L\Psi$, thanks to \eqref{m1}, using the lossy estimates \eqref{loss1} and \eqref{loss2},  we find that
\beq\label{B0-NL}
\nn&&\|\nabla B^1_\ne\|_{L^\infty}\|\nabla^\bot\Psi_\ne\|_{L^2}+\|\nabla U^1_\ne\|_{L^\infty}\|\nabla^\bot\Phi_\ne\|_{L^2}\\
&\les&\left(1+\|Y'-1\|_{H^3}\right)\left(\fr{\|J_{\ne}\|_{H^6}}{\la t\ra}\cdot\fr{\|\Om\|_{H^3}}{\la t\ra^2}+\fr{\|\Om\|_{H^6}}{\la t\ra}\cdot\fr{\|J_{\ne}\|_{H^3}}{\la t\ra^2}\right)\les\fr{1}{\la t\ra}\|A\Om\|_{L^2}\|AJ_{\ne}\|_{L^2}.
\eeq
It follows from the above two inequalities and the bootstrap hypotheses that
\beq
\fr12\|b^1_0(t)\|^2_{L^2}+\mu\|\partial_yb_0^1\|_{L^2L^2}^2\nn&\le&\fr12\|b^1_0(1)\|^2_{L^2}+C\|A\Om\|_{L^\infty L^2}\|AJ_{\ne}\|_{L^2L^2}\|b^1_0\|_{L^\infty L^2}\\
\nn&\le&\fr12\|b^1_0(1)\|^2_{L^2}+C\mu^{-\fr16}\eps^3.
\eeq
Similar arguments applying to the first  equation of \eqref{ub01} yields
\beq
\fr12\|u_0^1(t)\|_{L^2}^2\nn&\le&\fr12\|u_0^1(1)\|_{L^2}^2+C\int_1^t\fr{\|\Om\|_{H^3}}{\la t'\ra^2}\fr{\|\Om\|_{H^6}}{\la t'\ra}+\fr{\|J\|_{H^3}}{\la t'\ra^2}\fr{\|J\|_{H^6}}{\la t'\ra}dt'\|u_0^1\|_{L^\infty L^2}\\
\nn&\les&\fr12\|u_0^1(1)\|_{L^2}^2+C\mu^{-\fr13}\left(\|A\Om\|_{L^\infty L^2}^2+\|AJ\|_{L^\infty L^2}^2\right)\|u_0^1\|_{L^\infty L^2}\\
\nn&\les&\fr12\|u_0^1(1)\|_{L^2}^2+C\mu^{-\fr13}\eps^3.
\eeq
The above two estimates are sufficient to improve \eqref{hy-ub0}.

\section{Low norm energy estimates of $F$}\label{sec-F}
In this section, we improve the bootstrap hypothesis \eqref{hy-F}.
From \eqref{e-F}, integrating by parts and using \eqref{chain}, we get the energy identity:
\beq
\nn&&\fr12\left\|\mathcal{A}F(t)\right\|_{L^2}^2+\mu\left\|\nabla_L\mathcal{A}F\right\|_{L^2L^2}^2-\int_1^t\dot{\lm}(t')\left\||\nabla|^{\fr{s}{2}}\mathcal{A}F\right\|^2_{L^2}dt'+\left\|\sqrt{\fr{\partial_tM^{\mu}}{M^\mu}}\mathcal{A}F\right\|^2_{L^2L^2}\\
\nn&=&\fr12\left\|\mathcal{A}F(1)\right\|_{L^2}^2+\al^2\int_1^t\left\la\partial_{XX}\mathcal{A}J,\mathcal{A}F\right\ra dt'-4\mu\int_1^t\left\la \partial_{XY}^L\mathcal{A}J,\mathcal{A}F\right\ra dt'\\
\nn&&-4\mu\int_1^t\left\la\mathcal{A}\left[(Y'-1)\partial^L_{XY}J\right],\mathcal{A}F\right\ra dt'-\mu\int_1^t\left\la \partial_Y^L\mathcal{A}F, \mathcal{A}\left[\left((Y')^2-1\right)\partial_{Y}^LF\right]\right\ra dt'\\
&&-\mu\int_1^t\left\la \mathcal{A}F, \mathcal{A}\left(Y''\partial_Y^LF\right)\right\ra dt'+\int_1^t\left\la\mathcal{A}\mathrm{NL}[F], \mathcal{A}F\right\ra dt'.
\eeq
The linear errors can be bounded as follows. Using \eqref{m1} and the enhanced dissipation of $J_{\ne}$ and $F_{\ne}$, one easily deduces that there exists a constant $C_*$ depending on $\al$, such that
\begin{align*}
\al^2\int_1^t\left\la\partial_{XX}\mathcal{A}J,\mathcal{A}F\right\ra dt'
\le&\al^2\|\partial_{XX}\mathcal{A}J\|_{L^2L^2}\|\mathcal{A}F\|_{L^2L^2}\les \mu^{-\fr13}\left\|AJ_{\ne}\right\|_{L^2L^2}\left\|\mathcal{A}F_{\ne}\right\|_{L^2L^2}\\
\le& \mu^{-\fr13}C_*\left(\mu^{-\fr16}\eps\right)\left(\mu^{-\fr16}C_0\eps\mu^{-\fr23}\right)=\fr{C_*}{C_0}\left(C_0\eps\mu^{-\fr23}\right)^2,
\end{align*}
which is consistent for $C_0$ sufficiently large. The dissipation of $J$ is involved in the other linear error:
\be
-4\mu\int_1^t\left\la \partial_{XY}^L\mathcal{A}J,\mathcal{A}F\right\ra dt'\les\left(\mu^{\fr12}\left\|\partial_Y^L{A}J\right\|_{L^2L^2}\right)\left(\mu^\fr16\|\mathcal{A}F_{\ne}\|_{L^2L^2}\right)\les C_0\eps^2\mu^{-\fr23}.
\ee
Similarly, in view of \eqref{m1} and the algebra property \eqref{alge}, we have
\beq
-4\mu\int_1^t\left\la\mathcal{A}\left[(Y'-1)\partial^L_{XY}J\right],\mathcal{A}F\right\ra dt'
\nn&\les&\mu\left\|(Y'-1)\partial_{XY}^LJ\right\|_{L^2\mathcal{G}^{\lm, \sigma-6}}\|\mathcal{A}F_{\ne}\|_{L^2L^2}\\
&\les&\mu^{\fr23}\left\|Y'-1\right\|_{L^\infty\mathcal{G}^{\lm,\sigma-6}}\left\|\partial_Y^L{A}J\right\|_{L^2L^2}\|\mathcal{A}F_{\ne}\|_{L^2L^2}\les C_0\eps^3\mu^{-\fr23}.
\eeq

Next, consider the dissipation errors. In fact, using  \eqref{ref-Y'}, \eqref{ref-Y''},  and the algebra property \eqref{alge} again, we are led to
\beq
\nn&&-\mu\int_1^t\left\la \partial_Y^L\mathcal{A}F, \mathcal{A}\left[\left((Y')^2-1\right)\partial_{Y}^LF\right]\right\ra dt'-\mu\int_1^t\left\la \mathcal{A}F, \mathcal{A}\left(Y''\partial_Y^LF\right)\right\ra dt'\\
\nn&\les&\mu\left\|\partial_Y^L\mathcal{A}F\right\|_{L^2L^2}\left(\|Y'-1\|^2_{L^\infty\mathcal{G}^{\lm,\sigma-6}}+\|Y'-1\|_{L^\infty\mathcal{G}^{\lm,\sigma-6}}\right)\left\|\partial_Y^LF\right\|_{L^2\mathcal{G}^{\lm,\sigma-6}}\\
\nn&&+\mu\left\|\mathcal{A}F\right\|_{L^2L^2}\left(\|Y'-1\|^2_{L^\infty\mathcal{G}^{\lm,\sigma-5}}+\|Y'-1\|_{L^\infty\mathcal{G}^{\lm,\sigma-5}}\right)\left\|\partial_Y^LF\right\|_{L^2\mathcal{G}^{\sigma-6}}\\
\nn&\les&\eps \left(\mu^{\fr12}\left\|\partial_Y^L\mathcal{A}F\right\|_{L^2L^2}+\mu^{\fr12}\|F_0\|_{L^2\mathcal{G}^{\lm,\sigma-6}}+\mu^\fr12\|\mathcal{A}F_{\ne}\|_{L^2L^2}\right)\left(\mu^{\fr12}\left\|\partial_Y^L\mathcal{A}F\right\|_{L^2L^2}\right).
\eeq
To bound $\|F_0\|_{\mathcal{G}^{\sigma-6}}$, recalling that $F_0=\mu Y'\partial_Y\left(Y'\partial_YJ_0\right)=\mu \left(Y'\partial_YY'\partial_YJ_0+(Y')^2\partial_{YY}J_0\right)$, we have
\beq\label{F0L2}
\|F_0\|_{L^2\mathcal{G}^{\lm,\sigma-6}}\nn&\les&\mu\left(\|Y'-1\|^2_{L^\infty \mathcal{G}^{\lm,\sigma-5}}+\|Y'-1\|_{L^\infty \mathcal{G}^{\lm,\sigma-5}}\right)\|\partial_YJ_0\|_{L^2\mathcal{G}^{\lm,\sigma-6}}\\
\nn&&+\mu\left(\|Y'-1\|^2_{L^\infty \mathcal{G}^{\lm,\sigma-6}}+\|Y'-1\|_{L^\infty \mathcal{G}^{\lm,\sigma-6}}+1\right)\|\partial_YJ_0\|_{L^2\mathcal{G}^{\lm,\sigma-5}}\\
&\les&\mu\|\partial_YJ_0\|_{L^2\mathcal{G}^{\lm,\sigma-5}}\les\mu^\fr12\eps.
\eeq
It follows that
\be
-\mu\int_1^t\left\la \partial_Y^L\mathcal{A}F, \mathcal{A}\left[\left((Y')^2-1\right)\partial_{Y}^LF\right]\right\ra dt'-\mu\int_1^t\left\la \mathcal{A}F, \mathcal{A}\left(Y''\partial_Y^LF\right)\right\ra dt'\les \eps\left(C_0\eps\mu^{-\fr23}\right)^2.
\ee

Now we turn to bound $\displaystyle\int_1^t\left\la\mathcal{A}\mathrm{NL}[F], \mathcal{A}F\right\ra dt'$. Recalling the definition of $\mathrm{NL}[F]$ in \eqref{NLF}, let us estimate $\displaystyle\int_1^t\left\la\mathcal{A}\mathrm{NL}[F], \mathcal{A}F\right\ra dt'$ term by term. We will frequently use the algebra property of $\mathcal{G}^{\lm,\tl\sigma}$ for all $\tl\sigma\ge0$ (see Lemma {\bf A.3} of \cite{BM15} for more details) without notice.\par
\noindent{\em Estimate of $\displaystyle-\int_1^t\la\mathcal{A}(V\cdot\nabla F), \mathcal{A}F\ra dt'$.} To this end, one deduces from the third equation of \eqref{OMJ} that
\beq
\nn\|V\cdot\nabla F\|_{\mathcal{G}^{\lm, \sigma-6}}&\les&\|Y'\nabla^\bot\Psi_\ne\cdot\nabla F\|_{\mathcal{G}^{\lm,\sigma-6}}+\|\dot{Y}\partial_YF\|_{\mathcal{G}^{\lm, \sigma-6}}\\
\nn&\les&\left(1+\|Y'-1\|_{\mathcal{G}^{\lm,\sigma-6}}\right)\|\nabla^\bot\Psi_\ne\|_{\mathcal{G}^{\lm,\sigma-6}}\|\nabla F\|_{\mathcal{G}^{\lm,\sigma-6}}+\|\dot{Y}\|_{\mathcal{G}^{\lm,\sigma-6}}\|\partial_YF\|_{\mathcal{G}^{\lm,\sigma-6}}.
\eeq
On the other hand, recalling that $F=\mu\left(\Dl_LJ+\left((Y')^2-1\right)\partial_{YY}^LJ+Y''\partial_Y^LJ\right)+(\al+B^1_0)\partial_X\Om$,  in view of \eqref{ref-Y'} and \eqref{ref-Y''} and using the fact $|k,\eta-kt|\le\la t\ra|k,\eta|$, we  have
\beq
\nn\|\nabla F\|_{\mathcal{G}^{\lm,\sigma-6}}&\les&\mu\la t\ra\left\|\nabla_LJ\right\|_{\mathcal{G}^{\lm,\sigma-4}}+\mu\la t\ra \left(\|Y'-1\|^2_{\mathcal{G}^{\lm,\sigma-5}}+\|Y'-1\|_{\mathcal{G}^{\lm,\sigma-5}}\right)\left\|\partial_Y^LJ\right\|_{\mathcal{G}^{\lm,\sigma-4}}\\
\nn&&+\mu \left(\|Y'-1\|^2_{\mathcal{G}^{\lm,\sigma-4}}+\|Y'-1\|_{\mathcal{G}^{\lm,\sigma-4}}\right)\left\|\partial_Y^LJ\right\|_{\mathcal{G}^{\lm,\sigma-5}}+\left(1+\|B^1_0\|_{\mathcal{G}^{\lm,\sigma-5}}\right)\|\Om\|_{\mathcal{G}^{\lm,\sigma-4}}\\
\nn&\les&\mu\la t\ra\|\nabla_LJ\|_{\mathcal{G}^{\lm,\sigma-4}}+\|\Om\|_{\mathcal{G}^{\lm,\sigma-4}}.
\eeq
Combining the above two inequalities with the lossy estimate \eqref{loss1}  gives
\[
\|V\cdot\nabla F\|_{\mathcal{G}^{\lm,\sigma-6}}
\les\mu\left(\fr{\left\|\Om\right\|_{\mathcal{G}^{\lm,\sigma-3}}}{\la t \ra}+\la t\ra\|\dot{Y}\|_{\mathcal{G}^{\lm,\sigma-6}}\right){\|\nabla_LJ\|_{\mathcal{G}^{\lm,\sigma-4}}}+\left(\fr{\left\|\Om\right\|_{\mathcal{G}^{\lm,\sigma-3}}}{\la t\ra^2}+\|\dot{Y}\|_{\mathcal{G}^{\lm,\sigma-6}}\right)\|\Om\|_{\mathcal{G}^{\lm,\sigma-4}}.
\]
Then using \eqref{m1} twice, and by the hypotheses, for sufficiently small $\eps$ such that $K_D\mu^{-\fr23}\eps<1$, there holds
\beq
\nn&&-\int_1^t\la\mathcal{A}(V\cdot\nabla F), \mathcal{A}F\ra dt'\\
\nn&\les&\|V\cdot\nabla F\|_{L^1\mathcal{G}^{\lm,\sigma-6}}\|\mathcal{A}F\|_{L^\infty L^2}\\
\nn&\les&\mu^{-\fr23}\left(\|A\Om\|_{L^\infty L^2}+\la t\ra^{2-\fr{K_D\mu^{-\fr23}\eps}{2}}\|\dot{Y}\|_{L^\infty \mathcal{G}^{\lm,\sigma-6}}\right)\left(\mu^\fr12\|\nabla_LAJ\|_{L^2L^2}+\|A\Om\|_{L^\infty L^2}\right)\|\mathcal{A}F\|_{L^\infty L^2}\\
\nn&\les&\mu^{-\fr23}\eps^2\left(C_0\eps\mu^{-\fr23}\right)\les \eps\left(C_0\eps\mu^{-\fr23}\right)^2.
\eeq
\noindent{\em Estimate of $\displaystyle\int_1^t\left\la \mathcal{A}\left(B^1_0\partial_{XX}J\right), \mathcal{A}F\right\ra dt'$.} Clearly, it follows from \eqref{m1}, \eqref{U0B0J0} and the bootstrap hypotheses that
\beq
\nn\int_1^t\left\la \mathcal{A}\left(B^1_0\partial_{XX}J\right), \mathcal{A}F\right\ra dt'&=&\al \int_0^1\left\la\mathcal{A} \left(B^1_0\partial_{XX}J\right), \mathcal{A}F_{\ne} \right\ra dt'
\les\|B^1_0J_\ne\|_{L^2\mathcal{G}^{\lm,\sigma-4}}\|\mathcal{A}F_\ne\|_{L^2L^2}\\
\nn&\les& \mu^{-\fr23}\|B^1_0\|_{L^\infty\mathcal{G}^{\lm,\sigma-4}}\left(\mu^\fr{1}{6}\|AJ_{\ne}\|_{L^2L^2}\right)\left(\mu^\fr{1}{6}\|\mathcal{A}F_{\ne}\|_{L^2L^2}\right)\les\eps\left(C_0\eps\mu^{-\fr23}\right)^2.
\eeq
\noindent{\em Estimate of $\displaystyle\mu\int_1^t\left\la \mathcal{A}\left[Y'\partial_Y\left(Y'\partial_YB^1_0\right)\partial_X\Om\right], \mathcal{A}F\right\ra dt'$.} Noting that $Y'\partial_Y\left(Y'\partial_YB^1_0\right)=(Y')^2\partial_{YY}B^1_0+Y''\partial_YB^1_0$, we thus infer from \eqref{m1}, \eqref{U0B0J0}, \eqref{ref-Y'}, \eqref{ref-Y''}  and the bootstrap hypotheses that
\beq
\nn&&\mu\int_1^t\left\la \mathcal{A}\left[Y'\partial_Y\left(Y'\partial_YB^1_0\right)\partial_X\Om\right], \mathcal{A}F\right\ra dt'\\
\nn&\les&\mu\left\|Y'\partial_Y\left(Y'\partial_YB^1_0\right)\right\|_{L^2\mathcal{G}^{\lm, \sigma-6}}\|\Om\|_{L^\infty
\mathcal{G}^{\lm, \sigma-9}}\|\mathcal{A}F_{\ne}\|_{L^2L^2}\\
\nn&\les&\mu\Bigg[\left(1+\|Y'-1\|^2_{L^\infty \mathcal{G}^{\lm,\sigma-6}}+\|Y'-1\|_{L^\infty \mathcal{G}^{\lm, \sigma-6}}\right)\|\partial_YB^1_0\|_{L^2\mathcal{G}^{\lm,\sigma-5}}\\
\nn&&+\left(\|Y'-1\|^2_{L^\infty \mathcal{G}^{\lm, \sigma-5}}+\|Y'-1\|_{L^\infty \mathcal{G}^{\lm, \sigma-5}}\right)\|\partial_YB^1_0\|_{L^2\mathcal{G}^{\lm,\sigma-6}}\Bigg]\|\Om\|_{L^\infty
\mathcal{G}^{\lm, \sigma-5}}\|\mathcal{A}F_{\ne}\|_{L^2L^2}\\
\nn&\les&\left(\mu^\fr12\|\partial_YB^1_0\|_{L^2\mathcal{G}^{\sigma-5}}\right)\|A\Om\|_{L^\infty
L^2}\left(\mu^\fr16\|\mathcal{A}F_{\ne}\|_{L^2L^2}\right)\les\eps^3\mu^{-\fr23}.
\eeq

\noindent {\em Estimate of $\displaystyle-\int_1^t\left\la\mathcal{A}\left(Y'\left(\nabla^\bot\Psi_{\ne}\cdot\nabla B^1_{\ne}\right)_0\partial_X\Om\right),\mathcal{A}F\right\ra dt'$.} Recalling that $B^1_{\ne}=-Y'\partial_Y^L\Phi_\ne$, using the lossy estimates \eqref{loss1} and \eqref{loss2} and the fact  $\la t\ra^{-1}\les m^{-\fr12}_k(t,\eta)$ implied in \eqref{m1}, similar to \eqref{B0-NL}, we are led to
\beq
\nn&&-\int_1^t\left\la\mathcal{A}\left(Y'\left(\nabla^\bot\Psi_{\ne}\cdot\nabla B^1_{\ne}\right)_0\partial_X\Om\right),\mathcal{A}F\right\ra dt'\\
\nn&\les&\int_1^t\left\|Y'\left(\nabla^\bot\Psi_{\ne}\cdot\nabla B^1_{\ne}\right)_0\partial_X\Om\right\|_{\mathcal{G}^{\lm,\sigma-6}}\|\mathcal{A}F_{\ne}\|_{L^2}dt'\\
\nn&\les&\left(1+\|Y'-1\|_{L^\infty \mathcal{G}^{\lm,\sigma-5}}\right)^2\int^t_1\fr{\|\Om\|_{\mathcal{G}^{\lm,\sigma-3}}}{\la t'\ra^2}\fr{\|J_{\ne}\|_{\mathcal{G}^{\lm, \sigma-2}}}{\la t'\ra}\|\Om\|_{\mathcal{G}^{\lm,\sigma-5}}\|\mathcal{A}F_\ne\|_{ L^2}dt'\\
&\les&\|A\Om\|_{L^\infty L^2}^2\|AJ_{\ne}\|_{L^2L^2}\|\mathcal{A}F_{\ne}\|_{L^2 L^2}\les\eps^4\mu^{-1}.
\eeq
Obviously,  $\displaystyle\int_1^t\left\la\mathcal{A}\left(Y'\left(\nabla^\bot\Phi_{\ne}\cdot\nabla U^1_{\ne}\right)_0\partial_X\Om\right),\mathcal{A}F\right\ra dt'$ can be treated in the same way.\par

 \noindent{\em Estimate of $\displaystyle-\int_1^t\left\la\mathcal{A}\left([B_0^1, Y'\nabla^\bot\Psi_\ne\cdot\nabla]\partial_X\Om\right), \mathcal{A}F\right\ra dt'$.} We do not use any possible gain from the commutator, and write
 $-[B_0^1, Y'\nabla^\bot\Psi_\ne\cdot\nabla]\partial_X\Om=Y'\nabla^\bot\Psi_\ne\cdot\nabla\left(B^1_0\partial_X\Om\right)-B^1_0Y'\nabla^\bot\Psi_\ne\cdot\nabla\partial_X\Om$. Then we infer from \eqref{m1},  \eqref{loss1}  and \eqref{F0L2} that
\beq
\nn&&\int_1^t\left\la \mathcal{A}\left(Y'\nabla^\bot\Psi_\ne\cdot\nabla\left(B^1_0\partial_X\Om\right)\right),\mathcal{A}F\right\ra dt'\\
\nn&\les&\int_1^t\left\|Y'\nabla^\bot\Psi_\ne\cdot\nabla\left(B^1_0\partial_X\Om\right)\right\|_{\mathcal{G}^{\lm,\sigma-6}}\|\mathcal{A}F\|_{L^2}dt'\\
\nn&\les&\left(1+\|Y'-1\|_{L^\infty \mathcal{G}^{\lm,\sigma-6}}\right)\|B^1_0\|_{L^\infty\mathcal{G}^{\lm,\sigma-5}}\int_1^t\fr{\|\Om\|_{\mathcal{G}^{\sigma-3}}}{\la t'\ra^2}\|\Om\|_{\mathcal{G}^{\lm,\sigma-4}}\|\mathcal{A}F\|_{L^2}dt'\\
&\les&\mu^{-\fr13}\|B^1_0\|_{L^\infty\mathcal{G}^{\lm,\sigma-5}}\|A\Om\|_{L^\infty L^2}^2\|\mathcal{A}F\|_{L^2 L^2}\les\eps^4\mu^{-\fr76},
\eeq
and $\displaystyle\int_0^t\left\la \mathcal{A}\left(B^1_0Y'\nabla^\bot\Psi_\ne\cdot\nabla\partial_X\Om\right), \mathcal{A}F\right\ra dt'$ can be treated similarly.

 \noindent{\em Estimate of $\displaystyle\int_1^t\left\la\mathcal{A}\left((\al+B^1_0)Y'\nabla^\bot\partial_X\Psi\cdot\nabla\Om\right), \mathcal{A}F\right\ra dt'$.} Similar to the above inequality, we have
\beq
\nn&&\int_1^t\left\la\mathcal{A}\left((\al+B^1_0)Y'\nabla^\bot\partial_X\Psi\cdot\nabla\Om\right), \mathcal{A}F\right\ra dt'\\
\nn&\les&\left(1+\|B^1_0\|_{L^\infty\mathcal{G}^{\sigma-6}}\right)\left(1+\|Y'-1\|_{L^\infty \mathcal{G}^{\sigma-6}}\right)\int_1^t\fr{\|\Om\|_{\mathcal{G}^{\lm,\sigma-2}}}{\la t'\ra^2}\|\Om\|_{\mathcal{G}^{\lm,\sigma-5}}\|\mathcal{A}F\|_{L^2}dt'\\
&\les&\mu^{-\fr13}\|A\Om\|_{L^\infty L^2}^2\|\mathcal{A}F\|_{L^2L^2}\les\eps^3\mu^{-\fr76}.
\eeq

\noindent{\em Estimate of $\displaystyle\int_1^t\left\la\mathcal{A}\left((\al+B^1_0)\partial_X\left(Y'\nabla^\bot\Phi\cdot\nabla J\right)\right),\mathcal{A}F\right\ra dt'$.} Note first that 
\[
Y'\nabla^\bot\Phi\cdot\nabla J=Y'\nabla^\bot\Phi_{\ne}\cdot\nabla J+B^1_0\partial_XJ,
\]
and the term with $\Phi_{\ne}$ involved can be treated as above. Then by \eqref{U0B0J0} and the hypotheses,
\begin{align*}
&\int_1^t\left\la\mathcal{A}\left((\al+B^1_0)\partial_X\left(Y'\nabla^\bot\Phi\cdot\nabla J\right)\right),\mathcal{A}F\right\ra dt'\\
&=\int_1^t\left\la\mathcal{A}\left((\al+B^1_0)\partial_X(Y'\nabla^\bot\Phi_\ne\cdot\nabla J)\right),\mathcal{A}F_{\ne}\right\ra dt'
+\int_1^t\left\la\mathcal{A}\left((\al+B^1_0)(B^1_0\partial_{XX} J)\right),\mathcal{A}F_{\ne}\right\ra dt'\\
&\les\|AJ\|_{L^\infty L^2}\|{A}J_{\ne}\|_{L^2L^2}\|\mathcal{A}F_{\ne}\|_{L^2L^2}+\mu^{-\fr13}\|B^1_0\|_{L^\infty\mathcal{G}^{\lm, \sigma-6}}\|AJ_\ne\|_{L^2L^2}\|\mathcal{A}F_\ne\|_{L^2L^2}
\les\eps^3\mu^{-\fr43}.
\end{align*}

\noindent {\em Estimate of $\displaystyle\mu\int_1^t \left\la\mathcal{A}\left(Y'\nabla^\bot\Om\cdot\nabla J\right),\mathcal{A}F \right\ra dt'$.} Splitting $J$ into $J_{\ne}$ and $J_0$, using \eqref{m1} and \eqref{F0L2}, we arrive at
\beq
\nn&&\mu\int_1^t \left\la\mathcal{A}\left(Y'\nabla^\bot\Om\cdot\nabla J\right),\mathcal{A}F \right\ra dt'\\
\nn&=&\mu\int_1^t\left\la\mathcal{A}\left(Y'\nabla^\bot\Om\cdot\nabla J_\ne\right), \mathcal{A}F\right\ra dt'+\mu\int_1^t\left\la\mathcal{A}\left(Y'\partial_X\Om\partial_Y J_0\right), \mathcal{A}F\right\ra dt'\\
\nn&\les&\mu\left(1+\|Y'-1\|_{L^\infty\mathcal{G}^{\lm, \sigma-6}}\right)\|\Om\|_{L^\infty\mathcal{G}^{\lm, \sigma-5}}\|J_{\ne}\|_{L^2\mathcal{G}^{\lm, \sigma-5}}\|\mathcal{A}F\|_{L^2L^2}\\
\nn&&+\mu\left(1+\|Y'-1\|_{L^\infty\mathcal{G}^{\lm,\sigma-6}}\right)\|\Om\|_{L^\infty\mathcal{G}^{\lm,\sigma-5}}\|\partial_YJ_{0}\|_{L^2\mathcal{G}^{\sigma-6}}\|\mathcal{A}F\|_{L^2L^2}\\
\nn&\les&\mu^{\fr16}\|A\Om\|_{L^\infty L^2}\left(\mu^{\fr16}\|AJ_{\ne}\|_{L^2L^2}+\mu^\fr12\|A\partial_YJ_0\|_{L^2L^2}\right)\|\mathcal{A}F\|_{L^2L^2}\les\eps^3\mu^{-\fr23}.
\eeq

\noindent {\em Estimate of  $\displaystyle -2\mu\int_1^t\left\la\mathcal{A}\left(Y'\nabla^\bot\partial_X\Psi\cdot\nabla\partial_XJ\right), \mathcal{A}F \right\ra dt'$.} It follows from \eqref{m1}, \eqref{loss1}, \eqref{F0L2} and the bootstrap hypotheses that
\beq
\nn&&-2\mu\int_1^t\left\la\mathcal{A}\left(Y'\nabla^\bot\partial_X\Psi\cdot\nabla\partial_XJ\right), \mathcal{A}F \right\ra dt'\\
\nn&\les&\mu\left(1+\|Y'-1\|_{L^\infty\mathcal{G}^{\lm, \sigma-6}}\right)\int_1^t\fr{\|\Om\|_{\mathcal{G}^{\lm,\sigma-2}}}{\la t'\ra^2} \|J_\ne\|_{\mathcal{G}^{\lm,\sigma-4}}\|\mathcal{A}F\|_{ L^2}dt'\\
\nn&\les&\mu\|A\Om\|_{L^\infty L^2} \|AJ_{\ne}\|_{L^2 L^2}\|\mathcal{A}F\|_{L^2 L^2}\les\eps^3.
\eeq

\noindent {\em Estimate of $\displaystyle-2\mu\int_1^t\left\la\mathcal{A}\left(Y'\nabla^\bot(Y'\partial_Y^L\Psi)\cdot\nabla(Y'\partial_Y^LJ)\right), \mathcal{A}F\right\ra dt'$.} Recalling that $U^1_0=-Y'\partial_Y\Psi_0$, splitting $\Psi$ into $\Psi_{\ne}$ and $\Psi_0$, thanks to the fact $|\eta-kt|\les\la t\ra|k,\eta|$, and using \eqref{m1}, \eqref{U0B0J0}, \eqref{loss1}, \eqref{F0L2} and the bootstrap hypotheses gives
\beq
\nn&&-2\mu\int_1^t\left\la\mathcal{A}\left(Y'\nabla^\bot(Y'\partial_Y^L\Psi)\cdot\nabla(Y'\partial_Y^LJ)\right), \mathcal{A}F\right\ra dt'\\
\nn&=&-2\mu\int_1^t\left\la\mathcal{A}\left(Y'\nabla^\bot(Y'\partial_Y^L\Psi_\ne)\cdot\nabla(Y'\partial_Y^LJ)\right),\mathcal{A}F\right\ra dt'-2\mu\int_1^t\left\la\mathcal{A}\left((Y')^2\partial_YU^1_0\partial_X\partial_Y^LJ\right), \mathcal{A}F_{\ne} \right\ra dt'\\
\nn&\les&\mu\left(1+\|Y'-1\|_{L^\infty\mathcal{G}^{\lm, \sigma-5}}\right)^3\int_1^t\fr{\|\Om\|_{\mathcal{G}^{\lm,\sigma-2}}}{\la t'\ra}\|\partial_Y^LJ\|_{\mathcal{G}^{\lm,\sigma-5}}\|\mathcal{A}F\|_{L^2}dt'\\
\nn&&+\mu\left(1+\|Y'-1\|_{L^\infty\mathcal{G}^{\lm,\sigma-6}}+\|Y'-1\|^2_{L^\infty\mathcal{G}^{\sigma-6}}\right)\|U^1_0\|_{L^\infty\mathcal{G}^{\sigma-5}}\|\partial_Y^LJ\|_{L^2\mathcal{G}^{\sigma-5}}\|\mathcal{A}F_{\ne}\|_{L^2L^2}\\
\nn&\les&\mu^\fr23\left(\|A\Om\|_{L^\infty L^2}+\|U^1_0\|_{L^\infty\mathcal{G}^{\lm, \sigma-5}}\right)\|\partial_Y^LAJ\|_{L^2L^2}\|\mathcal{A}F\|_{L^2 L^2}
\les\eps^3\mu^{-\fr23}.
\eeq

\noindent {\em Estimate of $\displaystyle\mu\int_1^t\left\la\mathcal{A}\left(\Dl_t\left(Y'\nabla^\bot\Phi_\ne\cdot\nabla\Om\right)\right),\mathcal{A}F\right\ra dt'$.} In view of \eqref{Dlt}, integrating by parts, and using the reformulations of $(Y')^2-1$ and $Y''$ in \eqref{ref-Y'} and \eqref{ref-Y''}, the lossy estimate \eqref{loss2}, the upper bound of $m_k(t,\eta)$ in \eqref{m1}, and the boundedness of $\|F_0\|_{L^2\mathcal{G}^{\lm,\sigma-10}}$ in \eqref{F0L2}, we find that
\beq\label{e-NL-3}
\nn&&\mu\int_1^t\left\la\mathcal{A}\left(\Dl_t\left(Y'\nabla^\bot\Phi_\ne\cdot\nabla\Om\right)\right),\mathcal{A}F\right\ra dt'\\
\nn&=&-\mu\int_1^t\left\la\mathcal{A}\left(\nabla_L\left(Y'\nabla^\bot\Phi_\ne\cdot\nabla\Om\right)\right), \nabla_L\mathcal{A}F\right\ra dt'\\
\nn&&-\mu\int_1^t\left\la\mathcal{A}\left(\left((Y')^2-1\right)\partial_{Y}^L\left(Y'\nabla^\bot\Phi_\ne\cdot\nabla\Om\right)\right), \partial_{Y}^L\mathcal{A}F\right\ra dt'\\
\nn&&-\mu\int_1^t\left\la\mathcal{A}\left(Y''\partial_{Y}^L\left(Y'\nabla^\bot\Phi_\ne\cdot\nabla\Om\right)\right),\mathcal{A}F\right\ra dt'\\
\nn&\les&\mu\left\|\la t'\ra\|Y'\nabla^\bot\Phi_\ne\cdot\nabla\Om\|_{\mathcal{G}^{\lm,\sigma-5}}\right\|_{L^2}\|\nabla_L\mathcal{A}F\|_{L^2L^2}\\
\nn&&+\mu\left(\|Y'-1\|^2_{L^\infty \mathcal{G}^{\lm, \sigma-6}}+\|Y'-1\|_{L^\infty \mathcal{G}^{\lm,\sigma-6}}\right)\left\|\la t'\ra\|Y'\nabla^\bot\Phi_\ne\cdot\nabla\Om\|_{\mathcal{G}^{\lm,\sigma-5}}\right\|_{L^2}\|\partial_Y^L\mathcal{A}F\|_{L^2L^2}\\
\nn&&+\mu\left(\|Y'-1\|^2_{L^\infty \mathcal{G}^{\lm,\sigma-5}}+\|Y'-1\|_{L^\infty \mathcal{G}^{\lm,\sigma-5}}\right)\left\|\la t'\ra\|Y'\nabla^\bot\Phi_\ne\cdot\nabla\Om\|_{\mathcal{G}^{\lm,\sigma-5}}\right\|_{L^2}\|\mathcal{A}F\|_{L^2L^2}\\
\nn&\les&\mu\left(\|Y'-1\|^2_{L^\infty \mathcal{G}^{\lm,\sigma-5}}+\|Y'-1\|_{L^\infty \mathcal{G}^{\lm,\sigma-5}}+1\right)\left(1+\|Y'-1\|_{L^\infty\mathcal{G}^{\lm,\sigma-5}}\right)\\
\nn&&\times\|m^{-\fr12}J_{\ne}\|_{L^2 \mathcal{G}^{\lm,\sigma-2}}\|\Om\|_{L^\infty\mathcal{G}^{\lm,\sigma-4}}\left(\|\nabla_L\mathcal{A}F\|_{L^2L^2}+\|\mathcal{A}F\|_{L^2L^2}\right)\\
&\les&\mu^{\fr23}\|AJ_{\ne}\|_{L^2L^2}\|A\Om\|_{L^\infty L^2}\left(\|\nabla_L\mathcal{A}F\|_{L^2L^2}+\|\mathcal{A}F\|_{L^2L^2}\right)\les\eps^3\mu^{-\fr23}.
\eeq

\noindent{\em Estimate of $\displaystyle 2\mu\int_1^t\left\la\mathcal{A}\Dl_t\left(\partial_{XY}^t\Psi(-J+2\partial_{XX}\Phi)\right), \mathcal{A}F\right\ra dt'$.}
Splitting $J$ into $J_{\ne}$ and $J_{0}$, similar to \eqref{e-NL-3}, one deduces that
\beq
\nn&&2\mu\int_1^t\left\la\mathcal{A}\Dl_t\left(\partial_{XY}^t\Psi(-J+2\partial_{XX}\Phi)\right), \mathcal{A}F\right\ra dt'\\
\nn&\les&\mu\left\|\la t'\ra\|\partial_{XY}^t\Psi\left(-J+2\partial_{XX}\Phi\right)\|_{\mathcal{G}^{\lm, \sigma-5}}\right\|_{L^2}\left(\|\nabla_L\mathcal{A}F\|_{L^2L^2}+\|\mathcal{A}F\|_{L^2L^2}\right)\\
\nn&\les&\mu\|\Om\|_{L^\infty\mathcal{G}^{\lm, \sigma-1}}\left(\|J\|_{L^2\mathcal{G}^{\lm,\sigma-5}}+\|J_\ne\|_{L^2\mathcal{G}^{\lm,\sigma-1}}\right)\left(\|\nabla_L\mathcal{A}F\|_{L^2L^2}+\|\mathcal{A}F\|_{L^2L^2}\right)\\
\nn&\les&\mu^{-\fr13}\|A\Om\|_{L^\infty L^2}\left(\mu^\fr16\|AJ_\ne\|_{L^2L^2}+\mu^\fr12\|J_0\|_{L^2\mathcal{G}^{\lm,\sigma-5}}\right)\left(\mu^\fr12\|\nabla_L\mathcal{A}F\|_{L^2L^2}+\mu^\fr16\|\mathcal{A}F\|_{L^2L^2}\right)\les\eps^3\mu^{-1},
\eeq
where we have used \eqref{U0B0J0} in the last line above. The treatment of  the analogous nonlinear term $2\mu\int_1^t\left\la\mathcal{A}\left(\Dl_t\left(\partial_{XY}^t\Phi\left(-\Om+2\partial_{XX}\Psi\right)\right)\right),\mathcal{A}F\right\ra dt'$ is slightly simpler since $\partial_{XY}^t\Phi$ contributes a $J_{\ne}$ in low norm (via the lossy estimate \eqref{loss2}) which is $L^2$  integrable in time.

Collecting the above estimates and choosing $C_0$ sufficiently large, we conclude that \eqref{hy-F} can be improved.

\section{Coordinate system controls}\label{sec-coor}
The coordinate system \eqref{nl-coor} we choose in this paper is the same as that in \cite{BM15}. Up to the nonlinear terms involving the magnetic  potential $\Phi$, the system \eqref{sys-coor} of $(Y'-1, \dot{Y}, H)$   possesses the same structure of $(v'-1, [\partial_tv], \bar{h})$ in \cite{BM15}. Accordingly,  the desired bounds for $(Y'-1, \dot{Y}, H)$ can be obtained by following the strategy of estimating $(v'-1, [\partial_tv], \bar{h})$ in \cite{BM15}. Nevertheless, given that $m_k(t,\eta)$ is not uniform bounded in $\mu$ from above, and that the multiplier $A$ applying to the non-zero modes of the unknowns contains $m^{-\fr12}$, we sketch the estimates of $(Y'-1, \dot{Y}, H)$ below to indicate the possible losses caused by the usage of $m^{-\fr12}$.  
\subsection{ Improvement of \eqref{coor1}}
Note that $Y'-1$ in this paper solves the same equation as that of $v'-1$ in \cite{BM15}, and the multiplier $m^{-\fr12}$ is not involved at all, one can improve \eqref{coor1} in the same way as the estimates of $v'-1$ in \cite{BM15}, we omit the details to avoid unnecessary repetition. It is worth mentioning that the linear coupling between $Y'-1$ and  $H$ in the first equation of \eqref{sys-coor} explains the appearance of the large constant $K_Y$ in \eqref{coor1}.
\subsection{Improvement of \eqref{coor2}.}
From the third equation of \eqref{sys-coor}, we infer that
\beq\label{evo-H}
\nn&&\fr{d}{dt}\left(\la t\ra^{2+2s}\left\|\fr{A}{\la \partial_Y\ra^s}H \right\|^2_{L^2}\right)+\mathrm{CK}_{\lm}^{Y,2}+\mathrm{CK}_{w}^{Y,2}\\
\nn&=&-(2-2s)t\la t\ra^{2s}\left\|\fr{A}{\la\partial_Y\ra^s}H\right\|_{L^2}^2-2\la t\ra^{2+2s}\int\fr{A}{\la\partial_Y\ra^s}H\fr{A}{\la\partial_Y\ra^s}\left(\dot{Y}\partial_YH\right)dY\\
\nn&&+2t^{-1}\la t\ra^{2+2s}\int\fr{A}{\la\partial_Y\ra^s}H\fr{A}{\la\partial_Y\ra^s}\left(Y'\nabla^\bot\Psi_{\ne}\cdot\nabla\Om_{\ne}\right)_0dY\\
\nn&&-2t^{-1}\la t\ra^{2+2s}\int\fr{A}{\la\partial_Y\ra^s}H\fr{A}{\la\partial_Y\ra^s}\left(Y'\nabla^\bot\Phi_{\ne}\cdot\nabla J_{\ne}\right)_0dY\\
&=&-\mathrm{CK}^{Y,2}_{\mathrm{L}}+\mathcal{T}^H+\mathrm{F}_1+\mathrm{F}_2.
\eeq
In addition to the presence of the multiplier $m^{-\fr12}$ in $A$ mentioned above, comparing with the evolution of  $\la t\ra^{2+2s}\left\|\fr{A}{\la \partial_Y\ra^s}\bar{h} \right\|^2_{L^2}$ in \cite{BM15} (see ({\bf 8.10}) of \cite{BM15}), we have an extra term $\mathrm{F}_2$.  Even though the enhanced dissipation of $J_{\ne}$ in $\mathrm{F}_2$ is available herein, we prefer to use the boundedness of $\|AJ\|_{L^2}$ to avoid the extra loss in $\mu$. Then essentially  $\mathrm{F}_2$ can be treated in a manner similar to the treatment of $\mathrm{F}_1$ in \cite{BM15}.  Since the higher order contribution of $\mathrm{F}_2$ results in an additional power of $\eps$ at worst, let us focus on the leading order contribution of $\mathrm{F}_2$:
\beq
\nn\mathrm{F}_2^0&=&-2t^{-1}\la t\ra^{2+2s}\int\fr{A}{\la\partial_Y\ra^s}H\fr{A}{\la\partial_Y\ra^s}\left(\nabla^\bot\Phi_{\ne}\cdot\nabla J_{\ne}\right)_0dY\\
\nn&=&-\fr{1}{\pi}\sum_{M\ge8}\sum_{k\ne0}t^{-1}\la t\ra^{2+2s}\int_{\eta,\xi}\fr{A_0(\eta)}{\la\eta\ra^s}\bar{\hat{H}}(\eta)\fr{A_0(\eta)}{\la\eta\ra^s}\widehat{\nabla^\bot\Phi}_{-k}(\eta-\xi)_{<\fr{M}{8}}\cdot\widehat{\nabla J}_{k}(\xi)_Md\xi d\eta\\
\nn&&-\fr{1}{\pi}\sum_{M\ge8}\sum_{k\ne0}t^{-1}\la t\ra^{2+2s}\int_{\eta,\xi}\fr{A_0(\eta)}{\la\eta\ra^s}\bar{\hat{H}}(\eta)\fr{A_0(\eta)}{\la\eta\ra^s}\widehat{\nabla^\bot\Phi}_{k}(\xi)_{M}\cdot\widehat{\nabla J}_{-k}(\eta-\xi)_{<\fr{M}{8}}d\xi d\eta\\
\nn&&-\fr{1}{\pi}\sum_{M\in\mathbb{D}}\sum_{\fr{M}{8}\le M'\le8M}\sum_{k\ne0}t^{-1}\la t\ra^{2+2s}\int_{\eta,\xi}\fr{A_0(\eta)}{\la\eta\ra^s}\bar{\hat{H}}(\eta)\fr{A_0(\eta)}{\la\eta\ra^s}\widehat{\nabla^\bot\Phi}_{k}(\xi)_{M}\cdot\widehat{\nabla J}_{-k}(\eta-\xi)_{M'}d\xi d\eta\\
\nn&=&\mathrm{F}_{2,\mathrm{LH}}^0+\mathrm{F}_{2,\mathrm{HL}}^0+\mathrm{F}_{2,\mathcal{R}}^0,
\eeq
where the paraproduct is performed only in $Y$ variable. 

Consider $\mathrm{F}_{2,\mathrm{HL}}^0$ first. As in \cite{BM15}, subdivide further as follows:
\beq
\mathrm{F}_{2,\mathrm{HL}}^0\nn&=&-\fr{1}{\pi}\sum_{M\ge8}\sum_{k\ne0}t^{-1}\la t\ra^{2+2s}\int_{\eta,\xi}\left[{\bf 1}_{t\notin\mathbb{I}_{k,\xi}}\left({\bf 1}_{10|\eta|\ge t}+{\bf 1}_{10|\eta|<t}\right)+{\bf 1}_{t\in\mathbb{I}_{k,\xi}}\right]\fr{A_0(\eta)}{\la\eta\ra^s}\bar{\hat{H}}(\eta)\fr{A_0(\eta)}{\la\eta\ra^s}\\
\nn&&\times\widehat{\nabla^\bot\Phi}_{k}(\xi)_{M}\cdot\widehat{\nabla J}_{-k}(\eta-\xi)_{<\fr{M}{8}}d\xi d\eta\\
\nn&=&\mathrm{F}_{2,\mathrm{HL}}^{0;\mathrm{NR},\mathrm{S}}+\mathrm{F}_{2,\mathrm{HL}}^{0;\mathrm{NR},\mathrm{L}}+\mathrm{F}_{2,\mathrm{HL}}^{0, \mathrm{R}}.
\eeq
Note that in view of \eqref{wNR-ratio}, on the support of the integrand of $\mathrm{F}_{2,\mathrm{HL}}^{0;\mathrm{NR},\mathrm{S}}$ and $\mathrm{F}_{2,\mathrm{HL}}^{0;\mathrm{NR},\mathrm{L}}$,  there holds
\be\label{A0k}
{\bf 1}_{t\notin\mathbb{I}_{k,\xi}}\fr{A_0(\eta)}{A_k(\xi)}\les\mu^{-\fr13}e^{c\lm|\eta-\xi|^s},\quad\mathrm{for\ \ some} \quad c\in(0,1).
\ee
Then for the short time case, combining \eqref{A0k} with the fact $|k,\xi|\les|k,\xi|^{1-\fr{s}{2}}|k|^\fr{s}{2}|\eta|^{\fr{s}{2}}$ yields
\beq
\mathrm{F}_{2,\mathrm{HL}}^{0;\mathrm{NR},\mathrm{S}}\nn&\les& \mu^{-\fr13}\sum_{M\ge8}\sum_{k\ne0}t^{1+s}\int_{\xi,\eta}\fr{|\eta|^{\fr{s}{2}}A_0(\eta)}{\la\eta\ra^s}|\hat{H}(\eta)|{\bf 1}_{t\notin\mathbb{I}_{k,\xi}}|k,\xi|^{1-\fr{s}{2}}\left|A\widehat{\Phi}_{k}(\xi)_{M}\right|\\
\nn&&\times |k|^{\fr{s}{2}}e^{c\lm |\eta-\xi|^s}\left|\widehat{\nabla J}_{-k}(\eta-\xi)_{<\fr{M}{8}}\right|d\xi d\eta\\
\nn&\les&\mu^{-\fr13}t^{1+s}\|J\|_{\mathcal{G}^{\lm,2}}\left\||\partial_Y|^\fr{s}{2}\fr{A}{\la \partial_Y\ra^s}H\right\|_{L^2}\left\||\nabla|^{1-\fr{s}{2}}\chi^{\mathrm{NR}}A\Phi_{\ne}\right\|_{L^2},
\eeq
where $\widehat{(\chi^{\mathrm{NR}}g)}_k(\eta)={\bf 1}_{t\notin\mathbb{I}_{k,\eta}}\hat{g}_k(\eta)$.
Thanks to ({\bf 6.11a}) of \cite{BM15}, we have
\[
\la t\ra^s\left\||\nabla|^{1-\fr{s}{2}}\chi^{\mathrm{NR}}A\Phi_{\ne}\right\|_{L^2}\les\left\| \left\la\fr{\partial_Y}{t\partial_X}\right\ra^{-1}\fr{|\nabla|^\fr{s}{2}}{\la t\ra^s}A\Dl_L\Phi_{\ne} \right\|_{L^2}=\|\mathcal{M}_1\Dl_L\Phi\|_{L^2}.
\]
It follows the above two inequalities and \eqref{m1} that
\be\label{e-0S}
\mathrm{F}_{2,\mathrm{HL}}^{0;\mathrm{NR},\mathrm{S}} \les\mu^{-\fr23}\|AJ\|_{ L^2}\left(t\left\||\partial_Y|^\fr{s}{2}\fr{A}{\la \partial_Y\ra^s}H\right\|_{L^2}\right)\|\mathcal{M}_1\Dl_L\Phi\|_{L^2}.
\ee
As for the long time case, on the support of the integrand, it is easy to verify that $t\les|\xi-kt|$ and $\left\la\fr{\xi}{kt}\right\ra^{-1}\approx1$. These, together with \eqref{A0k} and the fact $|k,\xi|\les|k,\xi|^{\fr{s}{2}}|k|^{\fr{3s}{2}}|\eta|^{\fr{3s}{2}}$ imply that
\beq\label{e-0L}
\mathrm{F}_{2,\mathrm{HL}}^{0;\mathrm{NR},\mathrm{L}}\nn&\les& \mu^{-\fr13}\sum_{M\ge8}\sum_{k\ne0}t^{3s-1}\int_{\xi,\eta}\left\la\fr{\xi}{tk}\right\ra^{-1}\fr{|k,\xi|^{\fr{s}{2}}}{\la t\ra^s}\left(k^2+(\xi-kt)^2\right)\left|A\widehat{\Phi}_{k}(\xi)_{M}\right|\\
\nn&&\times |k|^{\fr{3s}{2}}e^{c\lm |\eta-\xi|^s}\left|\widehat{\nabla J}_{-k}(\eta-\xi)_{<\fr{M}{8}}\right|\fr{|\eta|^{\fr{s}{2}}A_0(\eta)}{\la\eta\ra^s}|\hat{H}(\eta)|d\xi d\eta\\
&\les&\mu^{-\fr23}\|AJ\|_{L^2}\left(t^{3s-1}\left\||\partial_Y|^\fr{s}{2}\fr{A}{\la \partial_Y\ra^s}H\right\|_{L^2}\right)\left\|\mathcal{M}_1 \Dl_L\Phi \right\|_{L^2}.
\eeq
Next turn to consider $\mathrm{F}_{2,\mathrm{HL}}^{0, \mathrm{R}}$. On the support of the integrand of $\mathrm{F}_{2,\mathrm{HL}}^{0, \mathrm{R}}$, in view of \eqref{wNR-ratio}, \eqref{wRwNR}, \eqref{ap2}, \eqref{ap3},   and  \eqref{ap-5}, we find that there exists a positive constant $c\in(0,1)$, such that
\beq\label{A0kR}
{\bf 1}_{t\in\mathbb{I}_{k,\xi}}\fr{A_0(\eta)}{A_k(\xi)}
\nn&\les&\mu^{-\fr13}e^{\lm\big||\eta|^s-|\xi|^s\big|}\left(e^{2r|\eta-
\xi|^{\fr12}}\fr{w_{NR}(\xi)}{w_{NR}(\eta)}+w_{NR}(\xi)\right)\fr{w_R(\xi)}{w_{NR}(\xi)}{\bf 1}_{t\in\mathbb{I}_{k,\xi}}\\
&\les&\mu^{-\fr13}e^{c\lm |\eta-\xi|^s}\fr{|k|^2}{|\xi|}\left(1+\left|t-\fr{\xi}{k}\right|\right){\bf 1}_{t\in\mathbb{I}_{k,\xi}}
=\mu^{-\fr13}e^{c\lm |\eta-\xi|^s}\fr{|k|^2+|\xi-kt|^2}{|\xi|}\fr{{\bf 1}_{t\in\mathbb{I}_{k,\xi}}}{1+\left|t-\fr{\xi}{k}\right|}.
\eeq
Combining this and the fact $|\eta|\approx|\xi|\approx |k|t$ with \eqref{pt_w} and \eqref{swi-ptw},  we are led to
\beq\label{e-0R}
\mathrm{F}_{2,\mathrm{HL}}^{0;\mathrm{R}}\nn&\les& \mu^{-\fr13}\sum_{M\ge8}\sum_{k\ne0}t^{1+s}\int_{\xi,\eta}\left(\sqrt{\fr{\partial_tw_0(\eta)}{w_0(\eta)}}+\fr{|\eta|^\fr{s}{2}}{\la t\ra^s}\right)\fr{A_0(\eta)}{\la\eta\ra^s}|\hat{H}(\eta)|\\
\nn&&\times \left\la \fr{\xi}{tk}\right\ra^{-1}\sqrt{\fr{\partial_tw_k(\xi)}{w_k(\xi)}}\left|\tl A\widehat{\Dl_L\Phi}_{k}(\xi)_{M}\right|
 e^{c\lm |\eta-\xi|^s}\left|\widehat{\nabla J}_{-k}(\eta-\xi)_{<\fr{M}{8}}\right|d\xi d\eta\\
&\les&\mu^{-\fr23}\|AJ\|_{L^2}\left(t^{1+s}\left\|\left(\sqrt{\fr{\partial_tw}{w}}+\fr{|\partial_Y|^\fr{s}{2}}{\la t\ra^s}\right)\fr{A}{\la \partial_Y\ra^s}H\right\|_{L^2}\right)\left\|\mathcal{M}_2\Dl_L\Phi\right\|_{L^2}.
\eeq
Consequently,  thanks to \eqref{pre-ell2}, \eqref{e-CCK} and the bootstrap hypotheses,  we infer from \eqref{e-0S}, \eqref{e-0L} and  \eqref{e-0R} that 
\beq\label{e-0HL}
\nn&&\int_1^t\mathrm{F}_{2,\mathrm{HL}}^{0;\mathrm{NR},\mathrm{S}}+\mathrm{F}_{2,\mathrm{HL}}^{0;\mathrm{NR},\mathrm{L}}+\mathrm{F}_{2,\mathrm{HL}}^{0;\mathrm{R}} dt'\\
\nn&\les&\mu^{-\fr23}\eps\int_1^t\left(\mathrm{CK}_{\lm}^{Y,2}(t')+\mathrm{CK}_{w}^{Y,2}(t')\right)+\Big[\mathrm{CK}_{\lm,J}(t')+\mathrm{CK}_{w,J}(t')\\
&&+\mu^{-\fr23}\eps^2\left(\mathrm{CK}_{\lm}^R(t')+\mathrm{CK}_{w}^R(t')\right)\Big]dt'
\les\mu^{-\fr23}\eps^3.
\eeq

To bound $\mathrm{F}_{2,\mathrm{LH}}^0$, note that on the support of the integrand $\fr{|k,\xi|}{\la\eta\ra^s}\les|k||\xi|^{1-s}\les |k||\eta|^{\fr{s}{2}}|\xi|^{\fr{s}{2}}$ holds due to $s>\fr12$ and $1\les|\eta|\approx|\xi|$. Then using \eqref{A0k}, \eqref{A0kR} (actually \eqref{A0kR} implies ${\bf 1}_{t\in\mathbb{I}_{k,\xi}}\fr{A_0(\eta)}{A_k(\xi)}\les \mu^{-\fr13}e^{c\lm |\eta-\xi|^s}$), the lossy estimate \eqref{loss2},  the upper bound of $m_k(t,\eta)$ in \eqref{m1}, and the bootstrap hypotheses, one deduces that
\beq\label{e-0LH}
\int_1^t\mathrm{F}_{2,\mathrm{LH}}^{0}dt'\nn&\les& \mu^{-\fr13}\sum_{M\ge8}\sum_{k\ne0}\int_1^tt'^{1+2s}\int_{\xi,\eta}\fr{|\eta|^{\fr{s}{2}}A_0(\eta)}{\la\eta\ra^s}|\hat{H}(\eta)|e^{c\lm |\eta-\xi|^s}|k|\left|\nabla^\bot\widehat{\Phi}_{-k}(\eta-\xi)_{<\fr{M}{8}}\right|\\
\nn&&\times |\xi|^\fr{s}{2}\left|A\hat{ J}_{k}(\xi)_{M}\right|d\xi d\eta dt'\\
\nn&\les&\mu^{-\fr23}\|AJ\|_{L^\infty L^2}\int_1^t\left(t'^{1+s-\tl q}\left\||\partial_Y|^\fr{s}{2}\fr{A}{\la \partial_Y\ra^s}H\right\|_{L^2}\right)\left(t'^{\tl q+s-2}\left\||\nabla|^{\fr{s}{2}}AJ\right\|_{L^2}\right)dt'\\
&\les&\mu^{-\fr23}\eps\int_1^t\mathrm{CK}_{\lm}^{Y,2}(t')+\mathrm{CK}_{\lm,J}(t')dt'\les\mu^{-\fr23}\eps^3,
\eeq
for $s$ close to $\fr12$, say, $\fr12<s<\fr23$, such that $\fr12<\tl q<1-\fr{s}{2}$.

As for the remainder $\mathrm{F}_{2,\mathcal{R}}^0$, one can see from \eqref{ap4} that there are enough regularity gaps both for $\Phi$ and $J$ on the support of the integrand. We state the result straightly:
\beq\label{e-0Re}
\mathrm{F}_{2,\mathcal{R}}^0\nn&\les& \left(\mu^{-\fr23}\|AJ\|_{L^2}\right)^\fr12t^{\fr12+s}\left\|\fr{A}{\la \partial_Y\ra^s}H\right\|_{L^2}\cdot \left(\mu^{-\fr23}\|AJ\|_{L^2}\right)^\fr12t^{s-\fr32}\|AJ\|_{L^2}\\
&\les& (\mu^{-\fr23}\eps)t^{1+2s}\left\|\fr{A}{\la \partial_Y\ra^s}H\right\|_{L^2}^2+(\mu^{-\fr23}\eps)t^{2s-3}\eps^2.
\eeq
The first term is absorbed by $\mathrm{CK}_{\mathrm{L}}^{Y,2}$ in \eqref{evo-H} provided $\mu^{-\fr23}\eps$ is sufficiently small, and the latter is time integrable since $s<1$.

The above estimates for $\mathrm{F}_2$ apply to $\mathrm{F}_1$ with $J$ and $\Phi$ replaced by $\Om$ and $\Psi$, respectively. Finally, it is worth pointing out that all the quantities in $\mathcal{T}^H$ are $X$ independent and the multiplier $m_k(t,\eta)$ is not involved, in view of \eqref{coor3}, we can adopt the estimate for $\mathcal{T}^H$ in \cite{BM15} as follows:
\be\label{e-0T}
\mathcal{T}^H\les \eps\mathrm{CK}_{\lm}^{Y,1}+ \eps\mathrm{CK}_{\lm}^{Y,2}+\eps\la t\ra^{2s+\fr{K_D\mu^{-\fr23}\eps}{2}}\left\|\fr{A}{\la\partial_Y\ra^s}H\right\|_{L^2}^2.
\ee
The last two terms can be absorbed by $\mathrm{CK}^{Y,2}_{\lm}$ and $\mathrm{CK}^{Y,2}_{\mathrm{L}}$, respectively. 

Integrating \eqref{evo-H} with respect to the time variable over the interval $[0,t]$, using \eqref{CKY1} and \eqref{e-0HL}--\eqref{e-0T}, we complete the improvement of \eqref{coor2} provided $\mu^{-\fr23}\eps\ll1$.

\subsection{ Improvement of \eqref{coor3}} We sketch the estimates of $\la t\ra^{4-K_D\mu^{-\fr23}\eps}\|\dot{Y}(t)\|^2_{\mathcal{G}^{\lm(t),\sigma-6}}$ to show the effect of the multiplier $m^{-\fr12}$  explicitly. From the second equation of \eqref{sys-coor}, we have
\beq\label{evo-dotY}
\nn&&\fr{d}{dt}\left(\la t\ra^{4-K_D\mu^{-\fr23}\eps}\left\|\dot{Y}\right\|_{\mathcal{G}^{\lm(t),\sigma-6}}^2\right)+2\la t\ra^{4-K_D\mu^{-\fr23}\eps}\left(-\dot{\lm}(t)\right)\left\||\partial_Y|^{\fr{s}{2}}\dot{Y}\right\|_{\mathcal{G}^{\lm(t),\sigma-6}}^2\\
\nn&&+\left(\fr{4\la t\ra^{4-K_D\mu^{-\fr23}\eps}}{t}-(4-K_D\mu^{-\fr23}\eps)t\la t\ra^{2-K_D\mu^{-\fr23}\eps}\right)\left\|\dot{Y}\right\|_{\mathcal{G}^{\lm(t),\sigma-6}}^2\\
\nn&=&-2\la t\ra^{4-K_D\mu^{-\fr23}\eps}\int A^S\dot{Y}A^S\left(\dot{Y}\partial_Y\dot{Y}\right)dY-\fr{2\la t\ra^{4-K_D\mu^{-\fr23}\eps}}{t}\int A^S\dot{Y}A^S\left(Y'\nabla^\bot\Psi_{\ne}\cdot\nabla U^1_{\ne}\right)_0dY\\
\nn&&+\fr{2\la t\ra^{4-K_D\mu^{-\fr23}\eps}}{t}\int A^S\dot{Y}A^S\left(Y'\nabla^\bot\Phi_{\ne}\cdot\nabla B^1_{\ne}\right)_0dY\\
&=&\mathrm{Q}_1+\mathrm{Q}_2+\mathrm{Q}_3,
\eeq
where $A^S(t,\eta)=e^{\lm(t)|\eta|^s}\la\eta\ra^{\sigma-6}$. $\mathrm{Q}_1$ has been treated in \cite{BM15}:
\beno
\mathrm{Q}_1\le\fr{K_D\eps}{2}t\la t\ra^{2-K_D\mu^{-\fr23}\eps-s}\left\|\dot{Y}\right\|_{\mathcal{G}^{\lm(t),\sigma-6}}^2,
\eeno
where $K_D$ is the maximum of the constant appearing in this term and the others below.

The treatments of $\mathrm{Q}_2$ and $\mathrm{Q}_3$ can be obtained by slightly modifying the estimates of $\mathrm{V}_3$ in \cite{BM15}. In fact,  recalling that $B^1=-Y'\partial_Y^L\Phi$, then by virtue of the fact $|\eta-kt|\le\la t\ra |k,\eta|$, and  the lossy estimate \eqref{loss2}, the upper bound of $m_k(t,\eta)$ and the bootstrap hypotheses,
\beno
\left\|\nabla B^1_{\ne}\right\|_{\mathcal{G}^{\lm(t),\sigma-6}}\les\min\left\{\mu^{-\fr13}, \la t\ra\right\}\left(1+\|Y'-1\|_{\mathcal{G}^{\lm(t),\sigma-5}}\right)\fr{\|AJ\|_{L^2}}{\la t\ra}\les\fr{\min\left\{\mu^{-\fr13}, \la t\ra\right\}\eps}{\la t\ra}.
\eeno
Therefore, using  the lossy estimate \eqref{loss2} and the upper bound of $m_k(t,\eta)$ again yields
\beqno
\mathrm{Q}_3\nn&\les&\fr{\la t\ra^{4-K_D\mu^{-\fr23}\eps}}{t}\left\|\dot{Y} \right\|_{\mathcal{G}^{\lm,\sigma-6}}\|\nabla^\bot\Phi_{\ne}\|_{\mathcal{G}^{\lm,\sigma-6}}\|\nabla B^1_{\ne}\|_{\mathcal{G}^{\lm,\sigma-6}}\left(1+\|Y'-1\|_{\mathcal{G}^{\lm,\sigma-6}}\right)\\
&\les&\la t\ra^{-K_D\mu^{-\fr23}\eps}\mu^{-\fr23}\eps^2\left\|\dot{Y} \right\|_{\mathcal{G}^{\lm,\sigma-6}}\le\fr{K_D\mu^{-\fr23}\eps}{4}t\la t\ra^{2-K_D\mu^{-\fr23}\eps}\left\|\dot{Y} \right\|_{\mathcal{G}^{\lm,\sigma-6}}^2+C\mu^{-\fr23}\eps^3\la t\ra^{-3-K_D\mu^{-\fr23}\eps}.
\eeqno
Similarly,
\beno
\mathrm{Q}_2\le\fr{K_D\mu^{-\fr23}\eps}{4}t\la t\ra^{2-K_D\mu^{-\fr23}\eps}\left\|\dot{Y} \right\|_{\mathcal{G}^{\lm,\sigma-6}}^2+C\mu^{-\fr23}\eps^3\la t\ra^{-3-K_D\mu^{-\fr23}\eps}.
\eeno
Substituting the estimates for $\mathrm{Q}_1, \mathrm{Q}_2$ and $\mathrm{Q}_3$ above into \eqref{evo-dotY}, absorbing all the three terms involving $K_D$,  integrating with respect to the time variable over $[1, t]$, and choosing $\eps$ so small that $\mu^{-\fr23}\eps\ll1$, we complete the improvement of \eqref{coor3}.

\begin{appendix}

\section{Construction of $w_{k}(t,\eta)$}\label{subsec-con-w}
 To begin with, let us recall the definitions of {\em critical intervals} and {\em resonant intervals} originated in \cite{BM15}. Denote by $E(\sqrt{|\eta|})$ the the integer part of $\sqrt{|\eta|}$. For $(k,\eta)\in\mathbb{Z}\times\mathbb{R}$, define the {\em critical intervals}
\be\label{cri_int}
{\rm I}_{k, \eta}:=\begin{cases}
[t_{|k|,\eta}, t_{|k|-1,\eta}],\quad \mathrm{if}\ \ |\eta|>1, k\eta>0,\ \mathrm{and}\ 1\le |k|\le E(\sqrt{|\eta|}),\\
\emptyset, \quad\quad\quad\quad\ \ \ \ \ \mathrm{otherwise},
\end{cases}
\ee
with 
\beno
t_{k,\eta}=\fr12\left(\fr{|\eta|}{|k|}+\fr{|\eta|}{|k|+1}\right)=\fr{(2|k|+1)|\eta|}{2|k|\left(|k|+1\right)}, \quad t_{0,\eta}=2|\eta|.
\eeno
The {\em resonant intervals} are defined by
\[
{\mathbb{I}}_{k, \eta}:=\begin{cases}
{\rm I}_{k,\eta},\ \ \mathrm{if}\ \ 2\sqrt{|\eta|}\le t_{|k|,\eta},\\
\emptyset, \ \ \ \ \ \mathrm{otherwise}.
\end{cases}
\]
The following well-separation property of critical times is useful in this paper.
\begin{lem}[From \cite{BM15}]\label{1/3}
Let $\xi, \eta$ be such that there exists some $C\ge1$ with $\fr{1}{C}|\xi|\le|\eta|\le C|\xi|$ and let $k, n$ be such that $t\in\mathrm{I}_{k,\eta}\cap\mathrm{I}_{n,\xi}$. Then $k\approx n$, and at least one of the following holds:
\begin{enumerate}
\item[(a)] k=n;
\item[(b)] $\left|t-\fr{\eta}{k}\right|\ge\fr{1}{10C}\fr{|\eta|}{k^2}$ and $\left|t-\fr{\xi}{n}\right|\ge\fr{1}{10C}\fr{|\xi|}{n^2}$;
\item[(c)] $|\eta-\xi|\gt_C\fr{|\eta|}{|n|}$.
\end{enumerate}
\end{lem}

Now we give a brief review of the constructions of $w_{NR}(t,\eta)$ and $w_R(t,\eta)$ in \cite{BM15}. Let $\kappa$ and $\mathrm{C}$ be two positive constants (See Proposition 3.1 of \cite{BM15}) satisfying $\kappa<\fr12$ and $\fr32<1+2\mathrm{C}\kappa<10$. For  $|\eta|\le1$, set
\be
w_{NR}(t,\eta)=w_{R}(t,\eta)=1.
\ee
For $\eta>1$, $w_{NR}(t,\eta)$  will be constructed backward in time. First, define
\be\label{def-w0}
w_{NR}(t,\eta)=w_{R}(t,\eta)=1, \quad\mathrm{if}\quad t>t_{0,\eta}=2\eta.
\ee
Next, for $l\in\left\{ 1, 2, 3, \cdots, E(\sqrt{\eta})\right\}$, define
\be\label{def-w}
\begin{cases}
w_{NR}(t,\eta)=\left(\fr{l^2}{\eta}\left[1+b_{l,\eta}\left|t-\fr{\eta}{l}\right|\right]\right)^{\mathrm{C}\kappa}w_{NR}(t_{l-1},\eta),\quad\mathrm{if}\quad t\in\mathrm{I}^\mathrm{R}_{l,\eta}=\left[\fr{\eta}{l}, t_{l-1,\eta}\right],\\[3mm]
w_{NR}(t,\eta)=\left(1+a_{l,\eta}\left|t-\fr{\eta}{l}\right|\right)^{-1-\mathrm{C}\kappa}w_{NR}\left(\fr{\eta}{l},\eta\right),\ \ \ \quad\mathrm{if}\quad t\in\mathrm{I}^\mathrm{L}_{l,\eta}=\left[ t_{l,\eta},\fr{\eta}{l}\right],
\end{cases}
\ee
where 
\be\label{def-ab}
b_{l,\eta}=
\begin{cases}
\fr{2(l+1)}{l}\left(1-\fr{l^2}{\eta}\right), \quad\mathrm{for}\quad l\ge2,\\
1-\fr{1}{\eta},\quad\qquad\ \ \ \ \ \mathrm{for}\quad l=1,
\end{cases}
\quad a_{l,\eta}=\fr{2(l+1)}{l}\left(1-\fr{l^2}{\eta}\right).
\ee
It is worth pointing out that the choice of $b_{l,\eta}$ and $a_{l,\eta}$ ensures that
\be\label{eq-b}
\fr{l^2}{\eta}\left(1+b_{l,\eta}\left|t_{l-1,\eta}-\fr{\eta}{l}\right|\right)=1,
\ee
and
\be\label{eq-a}
\fr{l^2}{\eta}\left(1+a_{l,\eta}\left|t_{l,\eta}-\fr{\eta}{l}\right|\right)=1.
\ee
Finally, define
\be\label{wNR-st}
w_{NR}(t,\eta)=w_{NR}(t_{E(\sqrt{\eta}),\eta},\eta),\quad \mathrm{if}\quad t\in\left[0,t_{E(\sqrt{\eta}),\eta}\right].
\ee
If $\eta<-1$, set  $w_{NR}(t,\eta)=w_{NR}(t,-\eta)$.  The construction of $w_{NR}(t,\eta)$ is completed. As for $w_{R}(t,\eta)$, on each interval $\mathrm{I}_{l,\eta}$, define
\be\label{wRwNR}
\begin{cases}
w_{R}(t,\eta)=\fr{l^2}{\eta}\left(1+b_{l,\eta}\left|t-\fr{\eta}{l}\right|\right)w_{NR}(t,\eta),\ \quad\mathrm{if}\quad t\in\mathrm{I}^\mathrm{R}_{l,\eta}=\left[\fr{\eta}{l}, t_{l-1,\eta}\right],\\[3mm]
w_{R}(t,\eta)=\fr{l^2}{\eta}\left(1+a_{l,\eta}\left|t-\fr{\eta}{l}\right|\right)w_{NR}\left(t,\eta\right), \quad\mathrm{if}\quad t\in\mathrm{I}^\mathrm{L}_{l,\eta}=\left[ t_{l,\eta},\fr{\eta}{l}\right].
\end{cases}
\ee
Moreover, like \eqref{wNR-st}, set 
\be
w_{R}(t,\eta)=w_{R}(t_{E(\sqrt{|\eta|}),\eta},\eta),\quad \mathrm{if}\quad t\in\left[0,t_{E(\sqrt{|\eta|}),\eta}\right].
\ee
Note that, thanks to the choice of $b_{l,\eta}$ and $a_{l,\eta}$, there holds $w_R(t_{l,\eta},\eta)=w_{NR}(t_{l,\eta},\eta)$. Then
\be\label{w=w}
w_{NR}(t,\eta)=w_{R}(t,\eta), \quad\mathrm{if}\quad t\in\left[0,t_{E(\sqrt{|\eta|}),\eta}\right]\cup\left[2|\eta|,+\infty\right).
\ee
\begin{rem}
One can see from \eqref{eq-b}, \eqref{eq-a}, \eqref{wRwNR} and \eqref{w=w} that
\be\label{R<NR}
w_{R}(t,\eta)\le w_{NR}(t,\eta),\quad\mathrm{for\ \ all}\quad t\ge0.
\ee
\end{rem}

Now the weight $w_k(t,\eta)$ can be defined as follows:
\be\label{w_keta}
w_k(t,\eta)=
\begin{cases}
w_{R}(t,\eta), \ \ \quad\mathrm{if}\quad t\in\mathrm{I}_{k,\eta},\\
w_{NR}(t,\eta),\quad\mathrm{if}\quad t\notin\mathrm{I}_{k,\eta}.
\end{cases}
\ee
In other words, $w_k(t,\eta)=w_{NR}(t,\eta)$  unless $|\eta|>1$, $1\le|k|\le E(\sqrt{|\eta|})$, $k\eta>0$, and $t\in\mathrm{I}_{k,\eta}$. In particular, $w_0(t,\eta)=w_{NR}(t,\eta)$.

\section{Properties of the multipliers $w$, $\mathcal{J}$, $m$, $M^1$ and $M^\mu$}
In this section, we give some useful properties of the multipliers $w$, $\mathcal{J}$, $m$, $M^1$ and $M^\mu$. The first five lemmas are directly borrowed from \cite{BM15}.
\begin{lem}[Growth  of $w$, from \cite{BM15}]\label{lem-gw}
 Assume that $|\eta|>1$, and $r=4(1+\mathrm{C}\kappa)$, where $\mathrm{C}$ and $\kappa$ are the positive constants appearing in \eqref{def-w}. Then we have
\be
\fr{1}{w_k(0,\eta)}=\fr{1}{w_k(t_{E(\sqrt{|\eta|}),\eta},\eta)}\sim\fr{1}{|\eta|^\fr{r}{8}}e^{\fr{r}{2}\sqrt{|\eta|}},
\ee
where $\sim$ is in the sense of asymptotic expansion as $|\eta|\rightarrow\infty$.
\end{lem}

\begin{lem}[From \cite{BM15}]\label{lem-wNR-ratio}
For all $t, \eta, \xi$, there holds
\be\label{wNR-ratio}
\fr{w_{NR}(t,\xi)}{w_{NR}(t,\eta)}\les e^{r|\eta-\xi|^\fr12}.
\ee
\end{lem}

\begin{lem}[From \cite{BM15}]\label{lem-J-ratio}
If any one of the following holds: {\em $(t\notin\mathbb{I}_{k,\eta})$} or {\em $(k=l)$} or {\em $(t\in\mathbb{I}_{k,\eta}, t\notin\mathbb{I}_{k,\xi}\ and \ |\xi|\approx|\eta|$)}, then we have 
\be\label{J-ratio}
\fr{\mathcal{J}_k(\eta)}{\mathcal{J}_l(\xi)}\les e^{10r|k-l, \eta-\xi|^\fr12}.
\ee
\end{lem}

\begin{lem}[From \cite{BM15}]
For $t\in\mathrm{I}_{k,\eta}$ and $t>2\sqrt{|\eta|}$, we have the following with $\tau=t-\fr{\eta}{k}$:
\be\label{pt_w}
\fr{\partial_tw_{NR}(t,\eta)}{w_{NR}(t,\eta)}\approx\fr{1}{1+|\tau|}\approx\fr{\partial_tw_R(t,\eta)}{w_R(t,\eta)}.
\ee
\end{lem}

\begin{lem}[From \cite{BM15}]
For all $t\ge1$, and $k, l, \eta, \xi$, such that for some $C\ge1$, $\fr{1}{C}|\xi|\le|\eta|\le C|\xi|$, there holds
\be\label{swi-ptw}
\sqrt{\fr{\partial_tw_k(t,\eta)}{w_k(t,\eta)}}\les \left(\sqrt{\fr{\partial_tw_l(t,\xi)}{w_l(t,\xi)}}+\fr{|\xi|^{\fr{s}{2}}}{\la t\ra^s}\right)\la \eta-\xi\ra.
\ee
\end{lem}
The next three lemmas play important roles to deal with the nonlinear terms $B_0^1\partial_X\Om$ and $B_0^1\partial_XJ$ in Section \ref{sec-zero-mode}. The proof of Lemma \ref{lem-com-J} is very delicate, and  the key point is that  the homogeneous factor $|\eta-\xi|$ on the right hand side of \eqref{gain0} enables us to bound $B_0^1$ in terms of $\partial_YB_0^1$ when $B_0^1$ is at low frequency.
\begin{lem}\label{lem-com-J}
Let $\mathcal{J}$ be given in \eqref{def-J}. Then there holds
\be\label{gain0}
\left|\fr{\mathcal{J}_k(\eta)}{\mathcal{J}_k(\xi)}-1\right|\les |\eta-\xi|e^{20r|\eta-\xi|^\fr12}.
\ee
\end{lem}
\begin{proof}
If $|\xi|\ge |\eta|$, thanks to \eqref{J-ratio}, we find that
\beno
\left|\fr{\mathcal{J}_k(\eta)}{\mathcal{J}_k(\xi)}-1\right|=\fr{\mathcal{J}_k(\eta)}{\mathcal{J}_k(\xi)}\left|\fr{\mathcal{J}_k(\xi)}{\mathcal{J}_k(\eta)}-1\right|\les e^{10r|\eta-\xi|^\fr12}\left|\fr{\mathcal{J}_k(\xi)}{\mathcal{J}_k(\eta)}-1\right|.
\eeno
Exchanging the place of $\eta$ and $\xi$, we may assume without loss of generality that $|\xi|<|\eta|$ and prove instead of \eqref{gain0} that
\be\label{gain1'}
\left|\fr{\mathcal{J}_k(\eta)}{\mathcal{J}_k(\xi)}-1\right|\les |\eta-\xi|e^{10r|\eta-\xi|^\fr12}.
\ee
If $|\eta-\xi|\ge\fr{1}{100}$, \eqref{gain1'} is a consequence of \eqref{J-ratio}. In the following,  we assume that  $|\xi|<|\eta|$ and $|\eta-\xi|\le\fr{1}{100}$. Note that by virtue of the elementary inequalities $|e^x-1|\le|x|e^{|x|}$, \eqref{ap1} and \eqref{ap2}, there holds
\beq\label{e420}
\fr{\mathcal{J}_k(\eta)}{\mathcal{J}_k(\xi)}-1
\nn&=&\fr{\fr{e^{2r\la\eta\ra^\fr12}-e^{2r\la\xi\ra^\fr12}}{w_k(\eta)}}{\fr{e^{2r\la\xi\ra^\fr12}}{w_k(\xi)}+e^{2r|k|^\fr12}}+\fr{e^{2r\la\xi\ra^\fr12}\left(\fr{1}{w_k(\eta)}-\fr{1}{w_k(\xi)}\right)}{\fr{e^{2r\la\xi\ra^\fr12}}{w_k(\xi)}+e^{2r|k|^\fr12}}\\
&\les&\fr{|\eta-\xi|}{\la\eta\ra^\fr12+\la\xi\ra^\fr12}e^{2r|\eta-\xi|^\fr12}\fr{w_k(\xi)}{w_k(\eta)}+\left|\fr{w_k(\xi)}{w_k(\eta)}-1\right|.
\eeq
It suffices to focus on the estimates of $\left|\fr{w_k(\xi)}{w_k(\eta)}-1\right|$.

If  $|\eta|\le1$, and $|\xi|\le1$,  then $w_k(\eta)=w_k(\xi)=1$, hence \eqref{gain1'} follows immediately.

If  $|\eta|>1\ge|\xi|$, we only need to consider the case $\eta \xi>0$, otherwise $|\eta-\xi|=|\eta|+|\xi|>1$, then \eqref{gain1'} is a consequence of \eqref{J-ratio}. Assume, W. O. L. G., that $\eta>1\ge\xi>0$. To bound the second term on the right-hand side of \eqref{e420}, we infer from $|\eta-\xi|\le\fr{1}{100}$ and $|\xi|\le1$ that $|\eta|\le |\eta-\xi|+|\xi|<2$. This implies that $E(\sqrt{\eta})=1$. Consequently, 
if $k=1$ and $t\in \mathrm{I}_{1,\eta}=\left[\fr34\eta,2\eta\right]$, in view of \eqref{def-w} and \eqref{wRwNR},
\be
w_k(t,\eta)=w_{R}(t,\eta)
=\begin{cases}
\left(\fr{1}{\eta}\left[1+(1-\fr{1}{\eta})(t-\eta)\right]\right)^{\mathrm{C}\kappa+1},&\mathrm{if}\ \ \eta\le t\le2\eta;\\[2mm]
\left(1+4(1-\fr{1}{\eta})(\eta-t)\right)^{-\mathrm{C}\kappa}\left(\fr{1}{\eta}\right)^{\mathrm{C}\kappa+1},&\mathrm{if}\ \ \fr34\eta\le t\le\eta.
\end{cases}
\ee
For $\eta\le t\le2\eta$, using the fact $0<\xi\le1$, we have
\be\label{e429}
\fr{1}{\eta}\left[1+(1-\fr{1}{\eta})(t-\eta)\right]=1+\fr{1}{\eta}(1-\fr{1}{\eta})(t-2\eta),
\ee
and
\be\label{e430}
\fr{1}{\eta}(1-\fr{1}{\eta})|t-2\eta|\les\fr{\eta-1}{\eta}\le\fr{\eta-\xi}{\eta}.
\ee
Then  by using the mean value theorem, there holds
\beq
\nn \left|\fr{1}{w_{k}(\eta)}-1\right|\les \fr{\left|\left(\fr{1}{\eta}\left[1+(1-\fr{1}{\eta})(t-\eta)\right]\right)^{\mathrm{C}\kappa+1}-1\right|}{\left(\fr{1}{\eta}\left[1+(1-\fr{1}{\eta})(t-\eta)\right]\right)^{\mathrm{C}\kappa+1}}\les \eta^{\mathrm{C}\kappa+1}\fr{\eta-\xi}{\eta}\les|\eta-\xi|.
\eeq
Similarly, for $\fr34\eta\le t\le\eta$, $4(1-\fr{1}{\eta})|\eta-t|\le\eta-1\le\eta-\xi $, thus
\beq
\nn \left|\fr{1}{w_{k}(\eta)}-1\right|&\les& {\left|\left(1+4(1-\fr{1}{\eta})(\eta-t)\right)^{\mathrm{C}\kappa}{\eta}^{\mathrm{C}\kappa+1}-1\right|}\\
\nn&\les& {\left|\left(1+4(1-\fr{1}{\eta})(\eta-t)\right)^{\mathrm{C}\kappa}-1\right|{\eta}^{\mathrm{C}\kappa+1}+\left|{\eta}^{\mathrm{C}\kappa+1}-1\right|}\\
&\les& |\eta-\xi|.
\eeq
If $k\ne1$ or $t\notin \mathrm{I}_{1,\eta}$, $w_k(t,\eta)=w_{NR}(t,\eta)$, and  $\left|\fr{w_k(\xi)}{w_k(\eta)}-1\right|=\left|\fr{1}{w_{NR}(\eta)}-1\right|$ can be treated similarly as above.

We are left to treat the trickiest case when $|\eta|>1$ and $|\xi|>1$. Again, we only need to investigate the case $\eta\xi>0$.  Assume W. L. O.G., that $\eta>\xi>1$, and $\eta-\xi<\fr{1}{100}$.\par
\noindent{\bf Case 1: $t\in\mathrm{I}_{n,\eta}\cap\mathrm{I}_{n,\xi}$ for some $1\le n\le\min\left\{E(\sqrt{\eta}),E(\sqrt{\xi})\right\}$.}\par
 {\em Case 1.1: $n=k$.} Assume first $t\in\mathrm{I}^\mathrm{R}_{k,\eta}\cap\mathrm{I}^\mathrm{R}_{k,\xi}$. 
Now we infer from \eqref{def-w}, \eqref{wRwNR} and \eqref{w_keta} that
\beq\label{421}
\fr{w_k(\xi)}{w_k(\eta)}-1
\nn&=&\left(\fr{\eta}{\xi}\right)^{c(k-1)+\mathrm{C}\kappa+1}\left(\fr{1+b_{k,\xi}\left|t-\fr{\xi}{k}\right|}{1+b_{k,\eta}\left|t-\fr{\eta}{k}\right|}\right)^{\mathrm{C}\kappa+1}-1\\
&=&\left(\fr{1+b_{k,\xi}\left|t-\fr{\xi}{k}\right|}{1+b_{k,\eta}\left|t-\fr{\eta}{k}\right|}\right)^{\mathrm{C}\kappa+1}\left[\left(\fr{\eta}{\xi}\right)^{c(k-1)+\mathrm{C}\kappa+1}-1\right]+\left[\left(\fr{1+b_{k,\xi}\left|t-\fr{\xi}{k}\right|}{1+b_{k,\eta}\left|t-\fr{\eta}{k}\right|}\right)^{\mathrm{C}\kappa+1}-1\right],
\eeq
where $c=1+2\mathrm{C}\kappa$. The second term on the  right hand of \eqref{421} can be treated as follows:
\beq
\nn\left|\left(\fr{1+b_{k,\xi}\left|t-\fr{\xi}{k}\right|}{1+b_{k,\eta}\left|t-\fr{\eta}{k}\right|}\right)^{\mathrm{C}\kappa+1}-1\right|&\le&\left|\left(1+\fr{b_{k,\xi}\big|t-\fr{\xi}{k}\big|-b_{k,\eta}\big|t-\fr{\eta}{k}\big|}{1+b_{k,\eta}\left|t-\fr{\eta}{k}\right|}\right)^{\mathrm{C}\kappa+1}-1\right|\\
\nn&\le& b_{k,\xi}\left|\fr{\eta-\xi}{k}\right|+|b_{k,\eta}-b_{k,\xi}|\left|t-\fr{\eta}{k}\right|\\
\nn&\les&\left|\fr{\eta-\xi}{k}\right|+k^2\left|\fr{1}{\eta}-\fr{1}{\xi}\right|\fr{|\eta|}{k^2}
\les\fr{|\eta-\xi|}{|k|}+\fr{|\eta-\xi|}{|\xi|}.
\eeq
Combining this with the elementary inequalities $1+x\le e^x$, and $|e^x-1|\le|x|e^{|x|}$, and using the fact $1\le k\le\sqrt{\xi}$, we find that
\beq\label{433}
\nn&&\left(\fr{1+b_{k,\xi}\left|t-\fr{\xi}{k}\right|}{1+b_{k,\eta}\left|t-\fr{\eta}{k}\right|}\right)^{\mathrm{C}\kappa+1}\left[\left(\fr{\eta}{\xi}\right)^{c(k-1)+\mathrm{C}\kappa+1}-1\right]\\
\nn&\les& \left(1+\fr{\eta-\xi}{\xi}\right)^{c(k-1)+\mathrm{C}\kappa+1}-1\les e^{\fr{|\eta-\xi|}{|\xi|}\left(c(k-1)+\mathrm{C}\kappa+1\right)}-1\\
&\les&\fr{|\eta-\xi|}{|\xi|}\left(c(k-1)+\mathrm{C}\kappa+1\right)e^{\fr{|\eta-\xi|}{|\xi|}\left(c(k-1)+\mathrm{C}\kappa+1\right)}\les\fr{|\eta-\xi|}{\sqrt{|\xi|}}.
\eeq

The case when $t\in\mathrm{I}^\mathrm{L}_{k,\eta}\cap\mathrm{I}^\mathrm{L}_{k,\xi}$ can be treated in the same way. Next let us consider the case when $t\in\mathrm{I}^\mathrm{L}_{k,\eta}\cap\mathrm{I}^\mathrm{R}_{k,\xi}$, and write
\be\label{wNReta}
w_{NR}(\eta)=\left(\fr{1^2}{\eta}\right)^c\left(\fr{2^2}{\eta}\right)^c\cdots\left(\fr{(k-1)^2}{\eta}\right)^c\left(\fr{k^2}{\eta}\right)^{\mathrm{C}\kappa}\left(1+a_{k,\eta}\left|t-\fr{\eta}{k}\right|\right)^{-1-\mathrm{C}\kappa},
\ee
and
\be\label{wNRxi}
w_{NR}(\xi)=\left(\fr{1^2}{\xi}\right)^c\left(\fr{2^2}{\xi}\right)^c\cdots\left(\fr{(k-1)^2}{\xi}\right)^c\left(\fr{k^2}{\xi}\left[1+b_{k,\xi}\left|t-\fr{\xi}{k}\right|\right]\right)^{\mathrm{C}\kappa}.
\ee
Thus,
\beq\label{434}
\fr{w_k(\xi)}{w_k(\eta)}-1=\fr{w_R(\xi)}{w_R(\eta)}-1\nn&=&\fr{w_R(\xi)}{w_{NR}(\xi)}\fr{w_{NR}(\xi)}{w_{NR}(\eta)}\fr{w_{NR}(\eta)}{w_R(\eta)}-1\\
\nn&=&\left(\fr{\eta}{\xi}\right)^{c(k-1)+\mathrm{C}\kappa+1}\left(1+b_{k,\xi}\left|t-\fr{\xi}{k}\right|\right)^{\mathrm{C}\kappa+1}\left(1+a_{k,\eta}\left|t-\fr{\eta}{k}\right|\right)^{\mathrm{C}\kappa}-1\\
\nn&=&\left(1+b_{k,\xi}\left|t-\fr{\xi}{k}\right|\right)^{\mathrm{C}\kappa+1}\left(1+a_{k,\eta}\left|t-\fr{\eta}{k}\right|\right)^{\mathrm{C}\kappa}\left[\left(\fr{\eta}{\xi}\right)^{c(k-1)+\mathrm{C}\kappa+1}-1\right]\\
&&+\left[\left(1+b_{k,\xi}\left|t-\fr{\xi}{k}\right|\right)^{\mathrm{C}\kappa+1}\left(1+a_{k,\eta}\left|t-\fr{\eta}{k}\right|\right)^{\mathrm{C}\kappa}-1\right].
\eeq
Note that $t\in\mathrm{I}^\mathrm{L}_{k,\eta}\cap\mathrm{I}^\mathrm{R}_{k,\xi}$ implies $\max\left\{\left|t-\fr{\xi}{k}\right|,\left|t-\fr{\eta}{k}\right|\right\}\les\fr{|\eta-\xi|}{|k|}$. Then
\beq\label{435}
\nn&&\left(1+b_{k,\xi}\left|t-\fr{\xi}{k}\right|\right)^{\mathrm{C}\kappa+1}\left(1+a_{k,\eta}\left|t-\fr{\eta}{k}\right|\right)^{\mathrm{C}\kappa}-1\\
\nn&=&\left(1+b_{k,\xi}\left|t-\fr{\xi}{k}\right|\right)^{\mathrm{C}\kappa+1}\left[\left(1+a_{k,\eta}\left|t-\fr{\eta}{k}\right|\right)^{\mathrm{C}\kappa}-1\right]+\left[\left(1+b_{k,\eta}\left|t-\fr{\eta}{k}\right|\right)^{\mathrm{C}\kappa}-1\right]\\
&\les&\fr{|\eta-\xi|}{|k|}.
\eeq
The treatment of the first term on the right-hand side of \eqref{434}  is similar to \eqref{433} and is thus omitted.

{\em Case 1.2: $n\ne k$}. Now $\displaystyle\fr{w_k(\xi)}{w_k(\eta)}-1=\fr{w_{NR}(\xi)}{w_{NR}(\eta)}-1$, which can be treated in the same manner as the case when $n=k$.

\noindent{\bf Case 2: $t\in \mathrm{I}_{n,\eta}\cap \mathrm{I}_{l,\xi}, n\ne l$ for some $1\le n\le E(\sqrt{\eta})$, and $1\le l\le E(\sqrt{\xi})$.} Recalling the definition of critical intervals \eqref{cri_int}, under the restrictions $\eta>\xi>0$ and $\eta-\xi<\fr{1}{100}$, it is not difficult to verify that the condition $t\in \mathrm{I}_{n,\eta}\cap \mathrm{I}_{l,\xi}, n\ne l$ actually implies  $l=n-1$. 
 Furthermore, either  $t\in \mathrm{I}_{n-1,\xi}\cap \mathrm{I}^\mathrm{L}_{n,\eta}$ or $t\in \mathrm{I}^\mathrm{R}_{n-1,\xi}\cap \mathrm{I}_{n,\eta}$ can not happen due to $\eta-\xi<\fr{1}{100}$. Therefore, we only need to focus on the case when $t\in \mathrm{I}_{n-1,\xi}^\mathrm{L}\cap\mathrm{I}_{n,\eta}^\mathrm{R}$.
 
In fact, for $t\in \mathrm{I}_{n-1,\xi}^\mathrm{L}\cap\mathrm{I}_{n,\eta}^\mathrm{R}$, 
similar to \eqref{wNReta} and \eqref{wNRxi}, $w_{NR}(t,\xi)$ and $w_{NR}(t,\eta)$ can be given explicitly as follows
\beno 
w_{NR}(t,\xi)=\left(1+a_{n-1,\xi}\left|t-\fr{\xi}{n-1}\right|\right)^{-1-\mathrm{C}\kappa}\left(\fr{(n-1)^2}{\xi}\right)^{\mathrm{C}\kappa}\left(\fr{(n-2)^2}{\xi}\right)^c\cdots\left(\fr{2^2}{\xi}\right)^c\left(\fr{1^2}{\xi}\right)^c,
\eeno
and
\beno
w_{NR}(t,\eta)=\left(\fr{n^2}{\eta}\left[1+b_{n,\eta}\left|t-\fr{\eta}{n}\right|\right]\right)^{\mathrm{C}\kappa}\left(\fr{(n-1)^2}{\eta}\right)^{c}\left(\fr{(n-2)^2}{\eta}\right)^c\cdots\left(\fr{2^2}{\eta}\right)^c\left(\fr{1^2}{\eta}\right)^c.
\eeno
Then (recalling that $c=1+2\mathrm{C}\kappa$)
\be\label{436}
\fr{w_{NR}(t,\xi)}{w_{NR}(t,\eta)}
=\left(\fr{(n-1)^2}{\xi}\left[1+a_{n-1,\xi}\left|t-\fr{\xi}{n-1}\right|\right]\right)^{-1-\mathrm{C}\kappa}\left(\fr{n^2}{\eta}\left[1+b_{n,\eta}\left|t-\fr{\eta}{n}\right|\right]\right)^{-\mathrm{C}\kappa}\left(\fr{\eta}{\xi}\right)^{(n-1)c}.
\ee
From \eqref{eq-a} and \eqref{eq-b}, we find that
\[
\fr{(n-1)^2}{\xi}\left[1+a_{n-1,\xi}\left|t-\fr{\xi}{n-1}\right|\right]=1+\fr{(n-1)^2}{\xi}a_{n-1,\xi}\left(t_{n-1,\xi}-t\right),
\]
and
\[
\fr{n^2}{\eta}\left[1+b_{n,\eta}\left|t-\fr{\eta}{n}\right|\right]
=1+\fr{n^2}{\eta}b_{n,\eta}\left(t-t_{n-1,\eta}\right).
\]
Note that
\[
\max\left\{\left|t_{n-1,\xi}-t\right|,\left|t-t_{n-1,\eta}\right|\right\}\le t_{n-1,\eta}-t_{n-1,\xi}=\fr{(2n-1)}{2(n-1)n}(\eta-\xi).
\]
Then
\be\label{438}
\max\left\{\fr{(n-1)^2}{\xi}a_{n-1,\xi}\left|t_{n-1,\xi}-t\right|,\fr{n^2}{\eta}b_{n,\eta}\left|t-t_{n-1,\eta}\right|\right\}\les \fr{|\eta-\xi|}{\sqrt{\xi}}.
\ee
Consequently,
\be\label{439}
\min\left\{1+\fr{(n-1)^2}{\xi}a_{n-1,\xi}\left(t_{n-1,\xi}-t\right), 1+\fr{n^2}{\eta}b_{n,\eta}\left(t-t_{n-1,\eta}\right)\right\}\ge\fr12.
\ee
To simplify the presentation, let us denote 
\[
G_1=1+\fr{(n-1)^2}{\xi}a_{n-1,\xi}\left(t_{n-1,\xi}-t\right), \quad G_2=1+\fr{n^2}{\eta}b_{n,\eta}\left(t-t_{n-1,\eta}\right).
\]
If $n=k$, by virtue of \eqref{436} and \eqref{wRwNR}, one deduces that
 \beq
 \nn\fr{w_k(\xi)}{w_k(\eta)}-1&=& \fr{w_{NR}(\xi)}{w_{R}(\eta)}-1= \fr{w_{NR}(\xi)}{w_{NR}(\eta)}\fr{w_{NR}(\eta)}{w_{R}(\eta)}-1=(G_1G_2)^{-1-\mathrm{C}\kappa}\left(\fr{\eta}{\xi}\right)^{(n-1)c}-1\\
\nn &=&(G_1G_2)^{-1-\mathrm{C}\kappa}\left[\left(\fr{\eta}{\xi}\right)^{(n-1)c}-1\right]+(G_1G_2)^{-1-\mathrm{C}\kappa}\left[1-(G_1G_2)^{1+\mathrm{C}\kappa}\right].
 \eeq
Thanks to \eqref{438} and \eqref{439}, we have 
\[
 (G_1G_2)^{-1-\mathrm{C}\kappa}\left|\left(\fr{\eta}{\xi}\right)^{(n-1)c}-1\right|\les\left(1+\fr{\eta-\xi}{\xi}\right)^{(n-1)c}-1\les\fr{|\eta-\xi|}{\sqrt{\xi}},
 \]
and
 \[
(G_1G_2)^{-1-\mathrm{C}\kappa}\left[1-(G_1G_2)^{1+\mathrm{C}\kappa}\right]
 \les\left|1-G_1^{1+\mathrm{C}\kappa}\right|+\left|1-G_2^{1+\mathrm{C}\kappa}\right|\les \fr{|\eta-\xi|}{\sqrt{\xi}}.
 \]
 If $n-1=k$ or ($n\ne k,  n-1\ne k$), $\displaystyle\fr{w_k(\xi)}{w_k(\eta)}-1$ can be bounded similarly as above, and we omit the details for brevity.

\noindent{\bf Case 3: $t\le \min\left\{t_{E(\sqrt{|\xi|}), \xi}, t_{E(\sqrt{|\eta|}), \eta}\right\}$.} Now the restriction on $t$ implies that, for all $k, \eta, \xi$, there hold
\beno
w_k(t, \eta)=w_{NR}(t, \eta)=w_{NR}(0,\eta),\quad\mathrm{and}\quad w_l(t, \xi)=w_{NR}(t, \eta)=w_{NR}(0,\xi).
\eeno
To bound $\left|\fr{w_{NR}(0,\xi)}{w_{NR}(0,\eta)}-1 \right|$, it is worth pointing out  that  the assumptions $\eta>\xi>1$, and $\eta-\xi<\fr{1}{100}$ imply that 
\be\label{Exieta}
E(\sqrt{\xi})\le E(\sqrt{\eta})\le E(\sqrt{\xi})+1.
\ee
By the definition of $w_{NR}(0,\cdot)$, see \eqref{def-w0}, \eqref{def-w} and \eqref{wNR-st} in the appendix, we have
\be\label{wNR0}
w_{NR}(0,\eta)=\left(\fr{\left(E(\sqrt{|\eta|})!\right)^2}{|\eta|^{E(\sqrt{|\eta|})}}\right)^c,\quad w_{NR}(0,\xi)=\left(\fr{\left(E(\sqrt{|\xi|})!\right)^2}{|\xi|^{E(\sqrt{|\xi|})}}\right)^c.
\ee
Accordingly, for the case $E(\sqrt{\eta})=E(\sqrt{\xi})$,  similar to \eqref{433}, one deduces that
\be\label{252}
\left| \fr{w_{NR}(0,\xi)}{w_{NR}(0,\eta)}-1\right|
=\left(1+\fr{\eta-\xi}{\xi}\right)^{cE(\sqrt{\xi})}-1\le e^{\fr{cE(\sqrt{|\xi|})}{|\xi|}|\eta-\xi|}-1\les\fr{|\eta-\xi|}{\sqrt{|\xi|}}.
\ee
If $E(\sqrt{\eta})=E(\sqrt{\xi})+1$, then $\sqrt{|\xi|}<E(\sqrt{|\xi|})+1=E(\sqrt{|\eta|})\le\sqrt{|\eta|}$. Then Using again \eqref{wNR0}, similar to \eqref{252}, we have
\beq\label{e465}
\left| \fr{w_{NR}(0,\xi)}{w_{NR}(0,\eta)}-1\right|\nn&=&\left|\left(\fr{|\eta|}{|\xi|}\right)^{cE(\sqrt{|\xi|})}\left(\fr{|\eta|}{E(\sqrt{|\eta|})^2}\right)^c-1\right|\\
\nn&\le&\left(\fr{|\eta|}{E(\sqrt{|\eta|})^2}\right)^c\left|\left(\fr{|\eta|}{|\xi|}\right)^{cE(\sqrt{|\xi|})}-1\right|+\left|\left(\fr{|\eta|}{E(\sqrt{|\eta|})^2}\right)^c-1\right|\\
&\les&\left|\left(\fr{|\eta|}{|\xi|}\right)^{cE(\sqrt{|\xi|})}-1\right|+\left[\left(\fr{|\eta|}{|\xi|}\right)^c-1\right]\les\fr{|\eta-\xi|}{\sqrt{|\xi|}}+\fr{|\eta-\xi|}{|\xi|}.
\eeq
{\bf Case 4: $t_{E(\sqrt{|\xi|}), \xi}\le t\le t_{E(\sqrt{|\eta|}), \eta}$.} Here we have
$
w_k(t,\eta)=w_{NR}(t_{E(\sqrt{\eta}),\eta},\eta)=\left(\fr{(E(\sqrt{\eta})!)^2)}{\eta^{E(\sqrt{\eta})}}\right)^c,
$
and in the mean while there exists an integer $n\le E(\sqrt{\xi})$ such that $t\in \mathrm{I}_{n,\xi}$. Therefore, there holds
\[
t_{n,\xi}\le t\le t_{E(\sqrt{|\eta|}), \eta}, \quad\mathrm{i. e.}\quad \fr{\xi}{n+\fr{1}{2+\fr{1}{n}}}\le\fr{\eta}{E(\sqrt{\eta})+\fr{1}{2+\fr{1}{E(\sqrt{\eta})}}}.
\]
Consequently, in view of the facts $\eta-\xi\le\fr{1}{100}$ and $1<\xi<\eta$, the above inequality gives
\beno
E(\sqrt{\eta})-n<\fr{\eta-\xi}{\xi}\left(n+\fr{1}{2+\fr{1}{n}}\right)+\fr{1}{2+\fr{1}{n}}<\fr{\eta-\xi}{{\xi}}\left(\sqrt{\xi}+\fr12\right)+\fr12<1.
\eeno
Combining this with \eqref{Exieta} yields
\[
E(\sqrt{\xi})=n=E(\sqrt{\eta}),\quad\mathrm{and}\quad t\in\mathrm{I}_{E(\sqrt{\xi}),\xi}.
\]
Furthermore,  under the conditions $t_{E(\sqrt{|\xi|}), \xi}\le t\le t_{E(\sqrt{|\eta|}), \eta}$, $\eta-\xi\le\fr{1}{100}$ and $1<\xi<\eta$, $t\in\mathrm{I}^\mathrm{R}_{E(\sqrt{\xi}),\xi}$ cannot happen. 
Thus we only need to focus on the case when $t\in\mathrm{I}^\mathrm{L}_{E(\sqrt{\xi}),\xi}$. In fact, similar to \eqref{wNReta}, we have
\beno 
w_{NR}(t,\xi)=\left(1+a_{E(\sqrt{\eta}),\,\xi}\left|t-\fr{\xi}{E(\sqrt{\xi})}\right|\right)^{-1-\mathrm{C}\kappa}\left(\fr{E(\sqrt{\xi})^2}{\xi}\right)^{\mathrm{C}\kappa}\left(\fr{(E(\sqrt{\xi})-1)^2}{\xi}\right)^c\cdots\left(\fr{2^2}{\xi}\right)^c\left(\fr{1^2}{\xi}\right)^c,
\eeno
for $t\in\mathrm{I}^\mathrm{L}_{E(\sqrt{\xi}),\xi}$.
On the other hand, noting that 
\[
\left|t_{E(\sqrt{\xi}), \xi}-t\right|\le t_{E(\sqrt{\eta}), \eta}-t_{E(\sqrt{\xi}), \xi}=\fr{2E(\sqrt{\eta})+1}{2(E(\sqrt{\eta})+1)E(\sqrt{\eta})}(\eta-\xi)\les\fr{\eta-\xi}{\sqrt{\xi}},
\]
we find that the inequalities involving $a_{n-1,\xi}$ in \eqref{438} and \eqref{439} hold with $n-1$ replaced by $E(\sqrt{\xi})$. 
Then $\left|\fr{w_k(t,\xi)}{w_k(t,\eta)}-1\right|$  can be treated in the same manner as that in {\bf Case2}. We omit the details and conclude that
\be\label{con-4}
\left|\fr{w_k(t,\xi)}{w_k(t,\eta)}-1\right|\les\fr{|\eta-\xi|}{\sqrt{\xi}}.
\ee
\noindent{\bf Case 5: $t_{E(\sqrt{|\eta|}), \eta}\le t\le t_{E(\sqrt{|\xi|}), \xi}$.} Here we have
\be\label{e466}
w_k(t,\xi)=w_{NR}(t_{E(\sqrt{\xi}),\xi},\xi)=\left(\fr{(E(\sqrt{\xi})!)^2)}{\xi^{E(\sqrt{\xi})}}\right)^c.
\ee
On the other hand, recalling \eqref{Exieta}, we must have $E(\sqrt{\eta})=E(\sqrt{\xi})+1$, otherwise $t_{E(\sqrt{|\eta|}), \eta}<t_{E(\sqrt{|\xi|}), \xi}$ cannot happen under the condition $\eta>\xi>1$. Consequently, $t\in\mathrm{I}_{E(\sqrt{\eta}),\eta}$. If $t\in \mathrm{I}^\mathrm{L}_{E(\sqrt{\eta}),\eta}$,   similar to \eqref{wNReta}, using the fact $E(\sqrt{\eta})=E(\sqrt{\xi})+1$, we have
\be\label{e446}
w_{NR}(t,\eta)=\left(\fr{E(\sqrt{\eta})^2}{\eta}\left[1+a_{E(\sqrt{\eta}),\eta}\left|t-\fr{\eta}{E(\sqrt{\eta})}\right|\right]\right)^{-1-\mathrm{C}\kappa}\left(\fr{E(\sqrt{\eta})^2}{\eta}\right)^{c}\left(\fr{(E(\sqrt{\xi})!)^2)}{\eta^{E(\sqrt{\xi})}}\right)^c.
\ee
Taking $l=E(\sqrt{\eta})$ in \eqref{eq-a}  gives
\be\label{e447}
\fr{E(\sqrt{\eta})^2}{\eta}\left[1+a_{E(\sqrt{\eta}),\eta}\left|t-\fr{\eta}{E(\sqrt{\eta})}\right|\right]=1+\fr{E(\sqrt{\eta})^2}{\eta}a_{E(\sqrt{\eta}),\eta}\left(t_{E(\sqrt{\eta}),\eta}-t\right).
\ee
Recalling the definition of $a_{E(\sqrt{\eta}),\eta}$ in \eqref{def-ab}, and nothing that the restriction $t\in \mathrm{I}^\mathrm{L}_{E(\sqrt{\eta}),\eta}$ implies that
\beno
\left|t_{E(\sqrt{\eta}),\eta}-t \right|\le\fr{\eta}{E(\sqrt{\eta})}-\fr{\eta}{E(\sqrt{\eta})+\fr{1}{2+\fr{1}{E(\sqrt{\eta})}}}=\fr{\eta}{2\left(E(\sqrt{\eta})+1\right)E(\sqrt{\eta})},
\eeno
we are led to
\be\label{e471}
\fr{E(\sqrt{\eta})^2}{\eta}a_{E(\sqrt{\eta}),\eta}\left|t-t_{E(\sqrt{\eta})-1,\eta}\right|
\le\fr{E(\sqrt{\eta})^2}{\eta}\fr{a_{E(\sqrt{\eta}),\eta}\eta}{2\left(E(\sqrt{\eta})+1\right)E(\sqrt{\eta})}=\fr{\eta-E(\sqrt{\eta})^2}{\eta}\le\fr{\eta-\xi}{\eta},
\ee
where we have used the fact $\sqrt{\xi}<E(\sqrt{\xi})+1=E(\sqrt{\eta})$ in the last inequality. From \eqref{e466}--\eqref{e471}, similar to   \eqref{435} and \eqref{e465}, we find that \eqref{con-4} holds.
If $t\in \mathrm{I}^\mathrm{R}_{E(\sqrt{\eta}),\eta}$,  the treatment $\left|\fr{w_k(\xi)}{w_k(\eta)}-1\right|$ of  is analogous (actually easier) and omitted.\par
\noindent{\bf Case 6: $2\xi\le t\le 2\eta$.}
Note that $\eta-\xi<\fr{1}{100}$ and $\eta>\xi>1$ imply that 
\be\label{xieta}
\fr{99}{100}\eta<\xi<\eta. 
\ee
Consequently, for $t\in[2\xi, 2\eta]$, we actually have $t\in \mathrm{I}^\mathrm{R}_{1,\eta}$, and hence 
\beno
w_{NR}(t,\eta)=\left(\fr{1}{\eta}\left[1+(1-\fr{1}{\eta})(t-{\eta})\right]\right)^{\mathrm{C}\kappa},\quad
\mathrm{and}\quad
w_{R}(t,\eta)=\left(\fr{1}{\eta}\left[1+(1-\fr{1}{\eta})(t-{\eta})\right]\right)^{\mathrm{C}\kappa+1}.
\eeno
Furthermore, we infer from \eqref{xieta} that
\[
\fr{1}{\eta}\left[1+(1-\fr{1}{\eta})(t-{\eta})\right]\ge\fr{1}{\eta}\left[1+(1-\fr{1}{\eta})(2\xi-{\eta})\right]>\fr{49}{50}.
\]
On the other hand, it is easy to see that $\fr{1}{\eta}(1-\fr{1}{\eta})|t-2\eta|\le\fr{2(\eta-\xi)}{\eta}$ due to the fact $2\xi\le t\le2\eta$. Therefore, if $k=1$, thanks to \eqref{e429}, we find that
\beq
\nn\left|\fr{w_k(\xi)}{w_k(\eta)}-1\right|=\left|\fr{1}{w_R(t,\eta)}-1\right|&=&\fr{\left|\left(1-\fr{1}{\eta}(1-\fr{1}{\eta})(2\eta-t)\right)^{\mathrm{C}\kappa+1}-1\right|}{\left(\fr{1}{\eta}\left[1+(1-\fr{1}{\eta})(t-\eta)\right]\right)^{\mathrm{C}\kappa+1}}\\
\nn&\les &\left[1-\left(1-\fr{1}{\eta}(1-\fr{1}{\eta})(2\eta-t)\right)^{\mathrm{C}\kappa+1}\right]\les \fr{|\eta-\xi|}{|\eta|}.
\eeq
The case $k\ne1$ can be treated similarly.

\noindent{\bf Case 7: $t>2\eta$.} Now $w_k(t, \eta)=w_k(t,\xi)=1$, there is nothing to prove.

The proof of Lemma \ref{lem-com-J} is completed.
\end{proof}

\begin{lem}\label{lem-pr_eta}
For $k\ne0$, there holds
\be\label{pr_eta-m}
\left|\fr{\partial_\eta M^1_k(\eta)}{M^1_k(\eta)}\right|+\left|\fr{\partial_\eta M^\mu_k(\eta)}{M^\mu_k(\eta)}\right|+\left|\fr{\partial_\eta m_k(\eta)}{m_k(\eta)}\right|\les \fr{1}{|k|}.
\ee
\end{lem}
\begin{proof}
The estimates for $M^1$ and $M^\mu$ can be found in \cite{BVW18}.
For the multiplier $m$, if $k\ne 0, -2\mu^{-\frac{1}{3}}<\frac{\eta}{k} < 0$,  and $0<t<\frac{\eta}{k}+2\mu^{-\frac{1}{3}}$:
\beq
\left| \frac{\partial_{\eta}m_k(\eta)}{m_k(\eta)}\right|\nn&=&
\left| \frac{2\eta}{k^2+\eta^2}-\frac{2\left( \eta-kt\right) }{ k^2+\left( \eta-kt\right) ^2}\right| \\
\nn&\leq & \frac{2|\eta|}{k^2+\eta^2} + \frac{2| \eta-kt|}{k^2+\left( \eta-kt\right) ^2} \lesssim \frac{1}{\left| k\right| }. 
\eeq
If $k\ne 0$, $-2\mu^{-\frac{1}{3}}<\frac{\eta}{k} < 0$, and $t>\frac{\eta}{k}+2\mu^{-\frac{1}{3}}$:
\begin{align*}
\begin{split}
\left|\frac{\partial_{\eta}m_k(\eta)}{m_k(\eta)}\right|=\frac{2|\eta|}{k^2+\eta^2} \lesssim\frac{1}{\left| k\right| }. 
\end{split}
\end{align*}
If $k\ne 0$, $\frac{\eta}{k} > 0,\frac{\eta}{k}<t<\frac{\eta}{k}+2\mu^{-\frac{1}{3}}$:
\begin{align*}
\begin{split}
\left| \frac{\partial_{\eta}m_k(\eta)}{m_k(\eta)}\right|=\frac{2|\eta-kt| }{ k^2+\left( \eta-kt\right) ^2}  \lesssim \frac{1}{\left| k\right| }. 
\end{split}
\end{align*}
Thus, \eqref{pr_eta-m} holds. This completes the proof of Lemma \ref{lem-pr_eta}.
\end{proof}
The following lemma is a 2D version of L{\footnotesize EMMA} A.1 in \cite{BGM17}, the proof is thus omitted. 
\begin{lem}\label{lem-com-m}
The Fourier multiplier $m$ satisfies
\be\label{com-m}
 m_k(t, \eta)\les \la\eta-\xi\ra^{2} m_k(t,\xi).
\ee
\end{lem}

The next two lemmas and Corollary \ref{coro-com-1} will be used in the treatment of the transport nonlinearities in Section \ref{sec-transport}.
\begin{lem}\label{lem-com-sqm}
Let $s\in(0,1)$, and $t\ge1$, then
\be\label{com-sqm}
\left|\fr{m^\fr12_l(t,\xi)}{m^\fr12_k(t,\eta)}-1\right||l,\xi|\les\mu^{-\fr13}\la k-l,\eta-\xi\ra^3+\mu^{-\fr13}\la k-l,\eta-\xi\ra^2|\eta|^{\fr{s}{2}}|\xi|^{\fr{s}{2}}t^{2-2s}.
\ee
\end{lem}
\begin{proof} To simplify the presentation, let us denote 
\[
{\bf com}_{\sqrt{m}}=\left|\fr{m^\fr12_l(t,\xi)}{m^\fr12_k(t,\eta)}-1\right|.
\]
{\bf Case 1: $l=0, m_l(t,\xi)=1$.} If $k=0$, or ($k\ne0, \fr{\eta}{k}<-2\mu^\fr13$) or ($k\ne0, t<\fr{\eta}{k}$), then $m_k(t,\eta)=1$, and we have nothing to prove. On the other hand, to bound $|l,\xi|=|\xi|$, it suffices 
to consider the case 
\be\label{etaxi}
\eta\xi>0\quad\mathrm{and}\quad\fr12|\xi|\le|\eta|\le2|\xi|,
\ee
otherwise, there holds $|\xi|\les |\eta-\xi|$.

{\em Case 1.1: $k\ne0, -2\mu^\fr13<\fr{\eta}{k}<0$, and $t<\fr{\eta}{k}+2\mu^{-\fr13}$, $m_k(t,\eta)=\fr{1+\left(\fr{\eta}{k}-t\right)^2}{1+\left(\fr{\eta}{k}\right)^2}$.} 
\be\label{472}
{\bf com}_{\sqrt{m}}|l,\xi|=\fr{\left|\sqrt{{1+\left(\fr{\eta}{k}-t\right)^2}}-\sqrt{{1+\left(\fr{\eta}{k}\right)^2}}\right|}{\sqrt{1+\left(\fr{\eta}{k}-t\right)^2}}|\xi|\les\fr{t\fr{|\eta|}{|k|}|k|}{1+t+|\fr{\eta}{k}|}\les\mu^{-\fr13}|k-l|.
\ee
where we have used the fact that $\fr{|\eta|}{|k|}=-\fr{\eta}{k}<2\mu^{-\fr13}$.

{\em Case 1.2: $k\ne0, -2\mu^\fr13<\fr{\eta}{k}<0$, and $t>\fr{\eta}{k}+2\mu^{-\fr13}$, $m_k(t,\eta)=\fr{1+\left(2\mu^{-\fr13}\right)^2}{1+\left(\fr{\eta}{k}\right)^2}$.} 
Similar to \eqref{472}, we have
\be\label{473}
{\bf com}_{\sqrt{m}}|l,\xi|=\fr{\left|\sqrt{{1+\left(2\mu^{-\fr13}\right)^2}}-\sqrt{{1+\left(\fr{\eta}{k}\right)^2}}\right|}{\sqrt{1+\left(2\mu^{-\fr13}\right)^2}}|\xi|\les\fr{\left(2\mu^{-\fr13}+\fr{\eta}{k}\right)\fr{|\eta|}{|k|}|k|}{1+2\mu^{-\fr13}}\les\mu^{-\fr13}|k-l|.
\ee

{\em Case 1.3: $k\ne0$, and $0<\fr{\eta}{k}<t$, $m_k(t,\eta)=1+\left(\fr{\eta}{k}-t\right)^2$ or  $m_k(t,\eta)=1+\left(2\mu^{-\fr13}\right)^2$.} 
Now it is easy to see that ${\bf com}_{\sqrt{m}}\les1$. For long time $t\ge\fr12\min\left\{\sqrt{|\eta|}, \sqrt{|\xi|}\right\}$, in view of \eqref{etaxi}, we find that
\be\label{lt1}
{\bf com}_{\sqrt{m}}|l,\xi|\les|\eta|\les |\eta|^\fr{s}{2}|\xi|^\fr{s}{2}t^{2-2s}.
\ee
For short time $t<\fr12\min\left\{\sqrt{|\eta|}, \sqrt{|\xi|}\right\}$, combing this with the fact $0<\fr{\eta}{k}<t$ yields that
\be\label{etakt}
\fr{|\eta|}{|kt|}\le1\quad\mathrm{and} \sqrt{|\eta|}\le\fr{|k|}{2}.
\ee
Consequently,
\be\label{st1}
{\bf com}_{\sqrt{m}}|l,\xi|\les|\eta|\les\left(1+t^2\fr{|\eta|^2}{|kt|^2}\right)|k|\les\left(1+t^2\fr{|\eta|^{2s}}{|kt|^{2s}}\right)|k|\les\left(1+|\eta|^\fr{s}{2}|\xi|^\fr{s}{2}t^{2-2s}\right)|k-l|.
\ee
{\bf Case 2: ($l\ne0, \fr{\xi}{l}<-2\mu^{-\fr13})$ or ($l\ne0, t<\fr{\xi}{l}), m_l(t,\xi)=1$.}  By using the definition of the multiplier $m$, one deduces that
\beq
{\bf com}_{\sqrt{m}}\nn&\les&
\begin{cases}
\fr{t}{1+t+\fr{|\eta|}{|k|}},\qquad\  \mathrm{if}\ \ k\ne0, -2\mu^\fr13<\fr{\eta}{k}<0, \mathrm{and}\ \ t<\fr{\eta}{k}+2\mu^{-\fr13};\\[3mm]
\fr{2\mu^{-\fr13}+\fr{\eta}{k}}{1+2\mu^{-\fr13}}, \ \ \ \, \quad \mathrm{if}\ \ k\ne0, -2\mu^\fr13<\fr{\eta}{k}<0, \mathrm{and}\ \ t>\fr{\eta}{k}+2\mu^{-\fr13};\\[3mm]
\fr{t-\fr{\eta}{k}}{1+|t-\fr{\eta}{k}|}, \qquad\  \mathrm{if}\ \ k\ne0, \fr{\eta}{k}>0, \mathrm{and}\ \ \fr{\eta}{k}<t<\fr{\eta}{k}+2\mu^{-\fr13};\\[3mm]
\fr{2\mu^{-\fr13}}{1+2\mu^{-\fr13}}, \qquad   \mathrm{if}\ \ k\ne0, \fr{\eta}{k}>0, \mathrm{and}\ \ t>\fr{\eta}{k}+2\mu^{-\fr13}
\end{cases}
\\
\nn&\les&
\begin{cases}
\fr{\left|\fr{\eta}{k}-\fr{\xi}{l}\right|}{1+t+\fr{|\eta|}{|k|}}, \qquad\ \mathrm{if}\ \ k\ne0, -2\mu^\fr13<\fr{\eta}{k}<0, \mathrm{and}\ \ t<\fr{\eta}{k}+2\mu^{-\fr13};\\[3mm]
\fr{\left|\fr{\eta}{k}-\fr{\xi}{l}\right|}{1+2\mu^{-\fr13}}, \quad\quad \mathrm{if}\ \ k\ne0, -2\mu^\fr13<\fr{\eta}{k}<0, \mathrm{and}\ \ t>\fr{\eta}{k}+2\mu^{-\fr13};\\[3mm]
\left|\fr{\eta}{k}-\fr{\xi}{l}\right|, \quad\ \ \, \ \mathrm{if}\ \ k\ne0, \fr{\eta}{k}>0, \mathrm{and}\ \ t>\fr{\eta}{k}.
\end{cases}
\eeq
Then if $|\xi|\le |l|$,
\be\label{477}
{\bf com}_{\sqrt{m}}|l,\xi|\les\left|\fr{\eta}{k}-\fr{\xi}{l}\right||l|\les|\eta-\xi|\fr{|l|}{|k|}+|k-l|\fr{|l|}{|k|}\fr{|\xi|}{|l|}\les|k-l,\eta-\xi|\fr{|l|}{|k|}.
\ee
If  $|l|\le|\xi|$,  $k\ne0, -2\mu^\fr13<\fr{\eta}{k}<0, \mathrm{and}\ \ t<\fr{\eta}{k}+2\mu^{-\fr13}$, using again \eqref{etaxi}, we are led to
\beq\label{478}
{\bf com}_{\sqrt{m}}|l,\xi|\nn&\les&\fr{|\eta-\xi|\fr{|l|}{|k|}\fr{|\xi|}{|l|}+|k-l|\fr{|l|}{|k|}\fr{|\xi|^2}{|l|^2}}{1+\fr{|\eta|}{|k|}}\les|k-l,\eta-\xi|\fr{|l|}{|k|}\fr{\fr{|\xi|^2}{|l|^2}}{1+\fr{|\eta|}{|k|}}\\
&\les&|k-l,\eta-\xi|\fr{|k|}{|l|}\fr{|\eta|}{|k|}\les\mu^{-\fr13}|k-l,\eta-\xi|\fr{|k|}{|l|}.
\eeq
The sub-case $|l|\le|\xi|$,  $k\ne0, -2\mu^\fr13<\fr{\eta}{k}<0, \mathrm{and}\ \ t>\fr{\eta}{k}+2\mu^{-\fr13}$ can be treated in the same way. If $|l|\le|\xi|$,  $k\ne0, \mathrm{and}\ \ t>\fr{\eta}{k}>0$, then ${\bf com}_{\sqrt{m}}|l,\xi|$ can be bounded by distinguishing long time and short time as in {\bf Case 1}.
In fact,  in this sub-case, for the short time $t\le\fr12\min\left\{\sqrt{|\eta|}, \sqrt{|\xi|}\right\}$, \eqref{etakt} still holds. Then in view of \eqref{etaxi},  following the computations in \eqref{lt1}, \eqref{st1} and \eqref{478}, we find that
\beq\label{479}
{\bf com}_{\sqrt{m}}|l,\xi|\nn&\les&|\xi|{\bf 1}_{t\ge\fr12\min\left\{\sqrt{|\eta|}, \sqrt{|\xi|}\right\}}+|\xi|\left|\fr{\eta}{k}-\fr{\xi}{l}\right|{\bf 1}_{t\le\fr12\min\left\{\sqrt{|\eta|}, \sqrt{|\xi|}\right\}}\\
\nn&\les&|\eta|{\bf 1}_{t\ge\fr12\min\left\{\sqrt{|\eta|}, \sqrt{|\xi|}\right\}}+|k-l,\eta-\xi|\fr{|l|}{|k|}{\fr{|\xi|^2}{|l|^2}}{\bf 1}_{t\le\fr12\min\left\{\sqrt{|\eta|}, \sqrt{|\xi|}\right\}}\\
\nn&\les&|\eta|^\fr{s}{2}|\xi|^\fr{s}{2}t^{2-2s}+|k-l,\eta-\xi|\fr{|k|}{|l|}{\fr{|\eta|^2}{|k|^2}}{\bf 1}_{t\le\fr12\min\left\{\sqrt{|\eta|}, \sqrt{|\xi|}\right\}}\\
&\les&\left(1+|k-l,\eta-\xi|\fr{|k|}{|l|}\right)|\eta|^{\fr{s}{2}}|\xi|^{\fr{s}{2}}t^{2-2s}.
\eeq
\par
\noindent{\bf  Case 3: $l\ne0, -2\mu^{-\fr13}<\fr{\xi}{l}<0, t<\fr{\xi}{l}+2\mu^{-\fr13}, m_l(t,\xi)=\fr{1+(\fr{\xi}{l}-t)^2}{1+\left(\fr{\xi}{l}\right)^2}$.}

{\em Case 3.1: $(k=0)$, or $(k\ne0, \fr{\eta}{k}<-2\mu^{-\fr13})$, or $(k\ne0, t<\fr{\eta}{k})$, $m_k(t,\eta)=1$.} It is easy to see that 
\be\label{480}
{\bf com}_{\sqrt{m}}\les\fr{t}{1+\fr{|\xi|}{|l|}}\le
\begin{cases}
\fr{2\mu^{-\fr13}}{1+\fr{|\xi|}{|l|}},\quad\mathrm{if}\quad k=0;\\[3mm]
\fr{\left|\fr{\xi}{l}-\fr{\eta}{k}\right|}{1+\fr{|\xi|}{|l|}}, \quad\mathrm{if}\quad k\ne0.
\end{cases}
\ee
Thus, if $k=0$, 
\be\label{4800}
{\bf com}_{\sqrt{m}}|l,\xi|\les\fr{\mu^{-\fr13}\fr{|\xi|}{|l|}|l|}{1+\fr{|\xi|}{|l|}}+\mu^{-\fr13}|l|\les\mu^{-\fr13}|k-l|.
\ee
If $k\ne0$, and $|\xi|\le|l|$,
\[
{\bf com}_{\sqrt{m}}|l,\xi|\les\left|\fr{\eta}{k}-\fr{\xi}{l}\right||l|\le|k-l,\eta-\xi|\fr{|l|}{|k|}.
\]
If $k\ne0$, and $|\xi|\ge|l|$, then we have
\be
{\bf com}_{\sqrt{m}}|l,\xi|\les\fr{|k-l,\eta-\xi|\fr{|l|}{|k|}\fr{|\xi|^2}{|l|^2}}{1+\fr{|\xi|}{|l|}}\les\mu^{-\fr13}|k-l,\eta-\xi|\fr{|l|}{|k|}.
\ee

{\em Case 3.2: $k\ne0, -2\mu^{-\fr13}<\fr{\eta}{k}<0$.} If $t<\fr{\eta}{k}+2\mu^{-\fr13}$, $m_k(t,\eta)=\fr{1+(\fr{\eta}{t}-t)^2}{1+(\fr{\eta}{k})^2}$.
\beq\label{481}
{\bf com}_{\sqrt{m}}\nn&=&\fr{\left|\sqrt{\left(1+(\fr{\eta}{k})^2\right)\left(1+(\fr{\xi}{l}-t)^2\right)}-\sqrt{\left(1+(\fr{\xi}{l})^2\right)\left(1+(\fr{\eta}{k}-t)^2\right)}\right|}{\sqrt{\left(1+(\fr{\xi}{l})^2\right)\left(1+(\fr{\eta}{k}-t)^2\right)}}\\
\nn&\le&\fr{\left|(\fr{\eta}{k}-t)^2-(\fr{\xi}{l}-t)^2\right|\left(1+(\fr{\eta}{k})^2\right)+\left|(\fr{\xi}{l})^2-(\fr{\eta}{k})^2\right|\left(1+(\fr{\eta}{k}-t)^2\right)}{\left(\sqrt{\left(1+(\fr{\eta}{k})^2\right)\left(1+(\fr{\xi}{l}-t)^2\right)}+\sqrt{\left(1+(\fr{\xi}{l})^2\right)\left(1+(\fr{\eta}{k}-t)^2\right)}\right)\sqrt{\left(1+(\fr{\xi}{l})^2\right)\left(1+(\fr{\eta}{k}-t)^2\right)}}\\
\nn&\le&\left|\fr{\eta}{k}-\fr{\xi}{l}\right|\left(\fr{1}{\sqrt{1+(\fr{\xi}{l}-t)^2}}\fr{\sqrt{1+(\fr{\eta}{k})^2}}{\sqrt{1+(\fr{\xi}{l})^2}}+\fr{1}{\sqrt{1+(\fr{\eta}{k}-t)^2}}\fr{\sqrt{1+(\fr{\eta}{k})^2}}{\sqrt{1+(\fr{\xi}{l})^2}}+\fr{\fr{|\eta|}{|k|}+\fr{|\xi|}{|l|}}{{1+(\fr{\xi}{l})^2}}\right)\\
&\les&\fr{\left|\fr{\eta}{k}-\fr{\xi}{l}\right|}{1+\fr{|\xi|}{|l|}}\left(1+\fr{1+\fr{|\eta|}{|k|}}{1+\fr{|\xi|}{|l|}}\right),
\eeq
where we  have used the fact that both $\fr{\eta}{k}$ and $\fr{\xi}{l}$ are negative in the last inequality above. If $t>\fr{\eta}{k}+2\mu^{-\fr13}$, $m_k(t,\eta)=\fr{1+(2\mu^{-\fr13})^2}{1+(\fr{\eta}{k})^2}$, by means of similar computations as those in \eqref{481}, and using the fact $0<2\mu^{-\fr13}+\fr{\xi}{l}-t\le\fr{\xi}{l}-\fr{\eta}{k}$, it is not difficult to verify that we still have
\[
{\bf com}_{\sqrt{m}}\les\fr{\left|\fr{\eta}{k}-\fr{\xi}{l}\right|}{1+\fr{|\xi|}{|l|}}\left(1+\fr{1+\fr{|\eta|}{|k|}}{1+\fr{|\xi|}{|l|}}\right).
\]
Nothing that
\[
\fr{1+\fr{|\eta|}{|k|}}{1+\fr{|\xi|}{|l|}}\le1+\fr{\left|\fr{\eta}{k}-\fr{\xi}{l}\right|}{1+\fr{|\xi|}{|l|}}\les\la k-l, \eta-\xi\ra,
\]
then similar to \eqref{477} and \eqref{478}, we are led to
\beq
{\bf com}_{\sqrt{m}}|l,\xi|\nn&\les&\la k-l,\eta-\xi\ra^2\fr{|l|}{|k|}{\bf 1}_{|\xi|\le|l|}+\mu^{-\fr13}\la k-l,\eta-\xi\ra^2\fr{|l|}{|k|}{\bf 1}_{|\xi|\ge|l|}\\
&\les&\mu^{-\fr13}\la k-l, \eta-\xi\ra^3.
\eeq

{\em Case 3.3: $k\ne0, 0<\fr{\eta}{k}<t<\fr{\eta}{k}+2\mu^{-\fr13}, m_k(t,\eta)=1+(\fr{\eta}{k}-t)^2$.} (Note that $t>\fr{\eta}{k}+2\mu^{-\fr13}$ with $\fr{\eta}{k}>0$ can never happen under the condition $t<\fr{\xi}{l}+2\mu^{-\fr13}$ with $\fr{\xi}{l}<0$.) Direct computations yield
\beq\label{483}
{\bf com}_{\sqrt{m}}\nn&\le&\fr{\left|\fr{\eta}{k}-\fr{\xi}{l}\right|\left(|\fr{\eta}{k}-t|+|\fr{\xi}{l}-t|\right)}{\left(\sqrt{1+(\fr{\xi}{l})^2}\sqrt{1+(\fr{\eta}{k}-t)^2}+\sqrt{1+(\fr{\xi}{l}-t)^2}\right)\sqrt{1+(\fr{\xi}{l})^2}\sqrt{1+(\fr{\eta}{k}-t)^2}}\\
\nn&&+\fr{\left|\fr{\xi}{l}\right|\sqrt{1+(\fr{\eta}{k}-t)^2}}{\sqrt{1+(\fr{\xi}{l})^2}\sqrt{1+(\fr{\eta}{k}-t)^2}+\sqrt{1+(\fr{\xi}{l}-t)^2}}\\
&\les&\fr{\left|\fr{\eta}{k}-\fr{\xi}{l}\right|}{1+\fr{|\xi|}{|l|}}+\fr{\fr{|\xi|}{|l|}}{1+\fr{|\xi|}{|l|}}.
\eeq
It is worth pointing out that $\fr{\xi}{l}<0$ and $\fr{\eta}{k}>0$ implies that $(\eta\xi)\cdot(kl)<0$. If $|\xi|\le|l|$, it suffices to consider the case $\fr{|l|}{2}\le|k|\le 2|l|$ and $kl>0$ which implies that $\eta\xi<0$. Combining this with \eqref{483} gives
\beq\label{486}
{\bf com}_{\sqrt{m}}|l,\xi|\les\left(\left|\fr{\eta}{k}-\fr{\xi}{l}\right|+\fr{|\xi|}{|l|}\right)|l|\les |k-l,\eta-\xi|\fr{|l|}{|k|}+|\xi|\les |k-l,\eta-\xi|.
\eeq
If $|l|\le|\xi|$, we only need to consider the case when $\eta$ and $\xi$ satisfy \eqref{etaxi}, which in turn implies that $kl<0$. Then ${\bf com}_{\sqrt{m}}|l,\xi|$ can be bounded by distinguishing short time and long time. For short time $t\le\fr12\min\left\{\sqrt{|\eta|}, \sqrt{|\xi|} \right\}$, owing to the fact $0<\fr{\eta}{k}<t$, \eqref{etakt} holds. Then in view of \eqref{m1}, we obtain
\be\label{487}
{\bf com}_{\sqrt{m}}|l,\xi|\les\mu^{-\fr13}|\eta|\les\mu^{-\fr13}|k|^2\les\mu^{-\fr13}|k-l|^2.
\ee
For long time $t\ge\fr12\min\left\{\sqrt{|\eta|}, \sqrt{|\xi|} \right\}$, similar to \eqref{lt1}, we have
\be\label{lt2}
{\bf com}_{\sqrt{m}}|l,\xi|\les\mu^{-\fr13}|\eta|\les \mu^{-\fr13}|\eta|^\fr{s}{2}|\xi|^\fr{s}{2}t^{2-2s}.
\ee
{\bf Case 4: $l\ne0, -2\mu^{-\fr13}<\fr{\xi}{l}<0, t>\fr{\xi}{l}+2\mu^{-\fr13}, m_l(t,\xi)=\fr{1+(2\mu^{-\fr13})^2}{1+\left(\fr{\xi}{l}\right)^2}$.} In this case, ${\bf com}_{\sqrt{m}}|l,\xi|$ can be treated by following the arguments in {\bf Case 3} line by line, and we omit the details.

\noindent
{\bf Case 5: $l\ne0, 0<\fr{\xi}{l}<t<\fr{\xi}{l}+2\mu^{-\fr13}, m_l(t,\xi)=1+(\fr{\xi}{l}-t)^2$.}

{\em Case 5.1: $(k=0)$, or $(k\ne0, \fr{\eta}{k}<-2\mu^{-\fr13})$, or $(k\ne0, t<\fr{\eta}{k})$, $m_k(t,\eta)=1$.} Similar to \eqref{480}, one deduces that
\[
{\bf com}_{\sqrt{m}}\les{t-\fr{\xi}{l}}\le
\begin{cases}
2\mu^{-\fr13},\ \ \, \quad\mathrm{if}\quad k=0;\\[3mm]
{\left|\fr{\xi}{l}-\fr{\eta}{k}\right|}, \quad\mathrm{if}\quad k\ne0.
\end{cases}
\]
If $k=0$, different from \eqref{4800}, we will bound ${\bf com}_{\sqrt{m}}|l,\xi|$ by splitting into short time and long time. Indeed, for $t\le\fr12\min\left\{\sqrt{|\eta|}, \sqrt{|\xi|}\right\}$, the fact $0<\fr{\xi}{l}<t$ implies that \eqref{etakt} holds with $\eta$ and $k$ replaced by $\xi$ and $l$, respectively. Then under the condition \eqref{etaxi}, similar to \eqref{lt1} and \eqref{st1}, we arrive at
\beq
{\bf com}_{\sqrt{m}}|l,\xi|\nn&\les&\mu^{-\fr13}\left[|l|+|\xi|\left({\bf 1}_{|\xi|\le|l|}+{\bf 1}_{|\xi|\ge|l|}\right)\right]\\
\nn&\les&\mu^{-\fr13}\left(|l|+|l|\left(\fr{|\xi|}{|l|}\right)^2{\bf 1}_{t\le\fr12\min\left\{\sqrt{|\eta|}, \sqrt{|\xi|}\right\}}+|\xi|{\bf 1}_{t\ge\fr12\min\left\{\sqrt{|\eta|}, \sqrt{|\xi|}\right\}}\right)\\
&\les&\mu^{-\fr13}\la k-l\ra|\eta|^\fr{s}{2}|\xi|^\fr{s}{2}t^{2-2s}.
\eeq
If $k\ne0$, and $|\xi|\le |l|$, then ${\bf com}_{\sqrt{m}}|l,\xi|$ can be bounded in the same manner as \eqref{477}. If $k\ne0$, and $|l|\le|\xi|$, we directly use \eqref{m1} to deal with the long time case $t\ge\fr12\min\left\{\sqrt{|\eta|}, \sqrt{|\xi|}\right\}$, and make use of the gain from $\left|\fr{\eta}{k}-\fr{\xi}{k} \right|$ to deal with the short time case like \eqref{479}:
\be\label{490}
{\bf com}_{\sqrt{m}}|l,\xi|\les\left(\mu^{-\fr13}+|k-l,\eta-\xi|\fr{|l|}{|k|}\right)|\eta|^{\fr{s}{2}}|\xi|^{\fr{s}{2}}t^{2-2s}.
\ee

{\em Case 5.2: $k\ne0, -2\mu^{-\fr13}<\fr{\eta}{k}<0$. } The treatment of ${\bf com}_{\sqrt{m}}|l,\xi|$ is essentially the same as the analogous in {\em Case 3.3}, we only give the sketches here. If $t<\fr{\eta}{k}+2\mu^{-\fr13}, m_k(t,\eta)=\fr{1+(\fr{\eta}{k}-t)^2}{1+(\fr{\eta}{k})^2}$, then
\beq
{\bf com}_{\sqrt{m}}\nn&\le&\fr{\left|\fr{\eta}{k}-\fr{\xi}{l}\right|\left(\left|\fr{\xi}{l}-t\right|+\left|\fr{\eta}{k}-t\right|\right)+\left(\fr{\eta}{k}\right)^2\left(1+(\fr{\xi}{l}-t)^2\right)}{\left(\sqrt{1+(\fr{\eta}{k}-t)^2}+\sqrt{1+(\fr{\eta}{k})^2}\sqrt{1+(\fr{\xi}{l}-t)^2}\right)\sqrt{1+(\fr{\eta}{k}-t)^2}}\\
\nn&\les&\left|\fr{\eta}{k}-\fr{\xi}{l}\right|+\fr{|\eta|}{|k|}\fr{\sqrt{1+(\fr{\xi}{l}-t)^2}}{\sqrt{1+(\fr{\eta}{k}-t)^2}}
\les\left|\fr{\eta}{k}-\fr{\xi}{l}\right|+\fr{|\eta|}{|k|}\left(1+\left|\fr{\eta}{k}-\fr{\xi}{l}\right|\right).
\eeq
The fact  that $\fr{\xi}{l}$ and $\fr{\eta}{k}$ are of opposite sign plays important role to deal with ${\bf com}_{\sqrt{m}}|l,\xi|$. Similar to \eqref{487}--\eqref{lt2}, we are led to
\beq
{\bf com}_{\sqrt{m}}|l,\xi|\nn&\les&\mu^{-\fr13}|\xi|{\bf 1}_{|l|\le|\xi|}+\mu^{-\fr13}|l|{\bf 1}_{|l|\ge|\xi|}{\bf 1}_{kl<0}+\left[\left(1+\fr{|\eta|}{|k|}\right)\left|\fr{\eta}{k}-\fr{\xi}{l}\right|+\fr{|\eta|}{|k|}\right]|l|{\bf 1}_{|l|\ge|\xi|}{\bf 1}_{kl>0}\\
&\les&\mu^{-\fr13}\left(|k-l|^2+|\eta|^{\fr{s}{2}}|\xi|^{\fr{s}{2}}t^{2-2s}+|k-l,\eta-\xi|\right)+\la\eta-\xi\ra|k-l,\eta-\xi|\fr{|l|}{|k|}.
\eeq
If $t>\fr{\eta}{k}+2\mu^{-\fr13}, m_k(t,\eta)=\fr{1+(2\mu^{-\fr13})^2}{1+(\fr{\eta}{k})^2}$, the treatment of ${\bf com}_{\sqrt{m}}|l,\xi|$ is almost the same as above (actually easier), and is hence omitted.

{\em Case 5.3:$k\ne0, 0<\fr{\eta}{k}<t$.} In this case, whether $t<\fr{\eta}{k}+2\mu^{-\fr13}$ or not, it is easy to verify that
\[
{\bf com}_{\sqrt{m}}\les \left|\fr{\eta}{k}-\fr{\xi}{l}\right|.
\]
Then in view of \eqref{477} and \eqref{490}, we infer that
\be
{\bf com}_{\sqrt{m}}|l,\xi|\les|k-l,\eta-\xi|\fr{|k|}{|l|}+\left(\mu^{-\fr13}+|k-l,\eta-\xi|\fr{|l|}{|k|}\right)|\eta|^{\fr{s}{2}}|\xi|^{\fr{s}{2}}t^{2-2s}.
\ee
{\bf Case 6: $l\ne0, \fr{\xi}{l}>0, t>\fr{\xi}{l}+2\mu^{-\fr13}, m_l(t,\xi)=1+(2\mu^{-\fr13})^2$}. In this case, the sub-case ($k\ne0, -2\mu^{-\fr13}<\fr{\eta}{k}<0$, and $t<\fr{\eta}{k}+2\mu^{-\fr13}$) can not happen. All the other sub-cases can be treated by following the way in {\bf Case 5}, the details are omitted.

To sum up, we conclude that \eqref{com-sqm} holds. The proof of Lemma \ref{lem-com-sqm} is completed.
\end{proof}

\begin{lem}\label{lem-com-mu}
Let $s\in(0,1)$, and $t\ge1$, then there holds
\be\label{com-mu}
\left|M^{\mu}_k(t,\eta)-M^{\mu}_l(t,\xi)\right||l,\xi|\les\mu^{-\fr13}|k-l,\eta-\xi|+\la k-l,\eta-\xi\ra^2|\eta|^\fr{s}{2}|\xi|^\fr{s}{2}t^{2-2s}.
\ee
\end{lem}
\begin{proof}
{\bf Case 1: $k\ne0$ and $l\ne0$.} By the mean value theorem, we obtain
\beq\label{494}
\left|M^{\mu}_k(t,\eta)-M^{\mu}_l(t,\xi)\right|
\nn&\les&\int_0^t\left|\fr{\mu^{\fr13}}{1+\mu^{\fr23}(\fr{\eta}{k}-t')^2}-\fr{\mu^{\fr13}}{1+\mu^{\fr23}(\fr{\xi}{l}-t')^2}\right|dt'\\
&\les&\mu^{\fr23}\left|\fr{\eta}{k}-\fr{\xi}{l}\right| \left(\int_0^t\fr{dt'}{\left(1+\mu^{\fr23}(\fr{\eta}{k}-t')^2\right)\sqrt{1+\mu^{\fr23}(\fr{\xi}{l}-t')^2}}\right.\\
\nn&&\left.+\int_0^t\fr{dt'}{\sqrt{1+\mu^{\fr23}(\fr{\eta}{k}-t')^2}\left(1+\mu^{\fr23}(\fr{\xi}{l}-t')^2\right)}\right).
\eeq
If $|\xi|\le|l|$, then we infer from \eqref{494} that
\be
\left|M^{\mu}_k(t,\eta)-M^{\mu}_l(t,\xi)\right||l,\xi|\les\mu^\fr13\left|\fr{\eta}{k}-\fr{\xi}{l}\right||l|\les\mu^\fr13|k-l,\eta-\xi|\fr{|l|}{|k|}.
\ee
If $|l|\le|\xi|$, it suffices to investigate $\left|M^{\mu}_k(t,\eta)-M^{\mu}_l(t,\xi)\right||l,\xi|$ under the condition \eqref{etaxi}. In the following, we only consider the short-time case $t\le\fr12\min\left\{\sqrt{|\eta|}, \sqrt{|\xi|}\right\}$, since the long time case can be treated like \eqref{lt1}.

{\em Case 1.1: $|\eta|\le2|k|t$ or $|\xi|\le2|l|t$.} In fact, the restrictions $|\eta|\le2|k|t$ and $t\le\fr12\min\left\{\sqrt{|\eta|}, \sqrt{|\xi|}\right\}$ imply that a variant of \eqref{etakt} holds. Then similar to \eqref{st1}, using again \eqref{479}, we are led to
\be\label{496}
\left|M^{\mu}_k(t,\eta)-M^{\mu}_l(t,\xi)\right||l,\xi|\les\mu^\fr13\left|\fr{\eta}{k}-\fr{\xi}{l}\right||\xi|\les\mu^{\fr13}|k-l,\eta-\xi|\fr{|k|}{|l|}|\eta|^\fr{s}{2}|\xi|^\fr{s}{2}t^{2-2s}.
\ee
The treatment of the sub-case $|\xi|\le2|l|t$ is the same way (actually easier) and is thus omitted. 

{\em Case 1.2: $|\eta|>2|k|t$ and $|\xi|>2|l|t$.} Noting that $0<t'<t$, we immediately have
\be\label{497}
\left|\fr{\eta}{k}-t' \right|\ge\max\left\{ \fr{|\eta|}{2|k|}, t\right\}\quad\mathrm{and}\quad\left|\fr{\xi}{l}-t' \right|\ge\max\left\{ \fr{|\xi|}{2|l|}, t\right\}.
\ee
Consequently,
\be
\int_0^t\fr{\mu^\fr23dt'}{\left(1+\mu^{\fr23}(\fr{\eta}{k}-t')^2\right)\sqrt{1+\mu^{\fr23}(\fr{\xi}{l}-t')^2}}+\int_0^t\fr{\mu^\fr23dt'}{\sqrt{1+\mu^{\fr23}(\fr{\eta}{k}-t')^2}\left(1+\mu^{\fr23}(\fr{\xi}{l}-t')^2\right)}\les\fr{\mu^{-\fr13}}{\fr{|\eta|}{|k|}\cdot\fr{|\xi|}{|l|}}.
\ee
Combining this with \eqref{494} yields
\be
\left|M^{\mu}_k(t,\eta)-M^{\mu}_l(t,\xi)\right||l,\xi|\les\mu^{-\fr13}\fr{\left|\fr{\eta}{k}-\fr{\xi}{l}\right||\xi|}{\fr{|\eta|}{|k|}\cdot\fr{|\xi|}{|l|}}\les\mu^{-\fr13}|k-l,\eta-\xi|.
\ee
{\bf Case 2: $k\ne0, l=0$ or $k=0, l\ne0$.} We only treat the sub-case when $k=0, l\ne0$ under the restriction  \eqref{etaxi} for brevity. By
the definition of $M^{\mu}$, and using the elementary inequality $|e^x-1|\le|x|e^{|x|}$, one deduces that
\be\label{4100}
\left|M^{\mu}_k(t,\eta)-M^{\mu}_l(t,\xi)\right||l,\xi|\les |l|+\int_0^t\fr{\mu^\fr13dt'}{1+\mu^\fr23(\fr{\xi}{l}-t')^2}|\xi|.
\ee

{\em Case 2.1: $|\xi|\le2|l|t$.} Different from \eqref{496},  from \eqref{4100} now we directly have
\be
\left|M^{\mu}_k(t,\eta)-M^{\mu}_l(t,\xi)\right||l,\xi|\les|k-l|+|\eta|^{\fr{s}{2}}|\xi|^{\fr{s}{2}}|l|^{1-s}t^{1-s}\les|k-l|\left(1+|\eta|^{\fr{s}{2}}|\xi|^{\fr{s}{2}}t^{1-s}\right).
\ee

{\em Case 2.2: $|\xi|>2|l|t$.} It follows from  \eqref{4100} and the lower bound of $\left|\fr{\xi}{l}-t' \right|$ in \eqref{497} that
\be
\left|M^{\mu}_k(t,\eta)-M^{\mu}_l(t,\xi)\right||l,\xi|\les|k-l|+\mu^{-\fr13}|k-l|.
\ee
Then \eqref{com-mu} follows immediately. This completes the proof of Lemma \ref{lem-com-mu}.
\end{proof}
Taking $\mu=1$ in \eqref{com-mu}, we obtain the following commutator estimate for $M^1$.
\begin{coro}\label{coro-com-1}
Under the conditions of Lemma \ref{lem-com-mu}, there holds
\be\label{com-1}
\left|M^{1}_k(t,\eta)-M^{1}_l(t,\xi)\right||l,\xi|\les|k-l,\eta-\xi|+\la k-l,\eta-\xi\ra^2|\eta|^\fr{s}{2}|\xi|^\fr{s}{2}t^{2-2s}.
\ee
\end{coro}

\section{Littlewood-Paley decomposition and paraproducts}\label{subsec-LP}
The nonlinear estimates in this paper rely on proper frequency splitting, which can be fulfilled by using the classical Littlewood-Paley theory. In the following, we define the Littlewood-Paley decomposition only in the $Y$ variable. Let $\varphi\in C_0^\infty(\mathbb{R})$ be such that
$\varphi(\xi)=1$ for $|\xi|\le\fr12$ and $\varphi(\xi)=0$ for $|\xi|\ge\fr34$ and define $\rho(\xi)=\varphi\left(\fr{\xi}{2}\right)-\varphi(\xi)$, supported in the range $\xi\in\left(\fr12,\fr32\right)$. Then we have the partition of unity
\[
1=\varphi(\xi)+\sum_{M\in 2^{\mathbb{N}}}\rho_M(\xi),
\]
where $\rho_M(\xi)=\rho\left(\fr{\xi}{M}\right)$. For $f\in L^2(\mathbb{R})$, let us define
\begin{align*}
f_M&=\rho_M(|\partial_Y|)f,\quad
f_{\fr12}=\varphi(|\partial_Y|)f,\\
f_{<M}&=f_{\fr12}+\sum_{K\in2^{\mathbb{N}}: K<M}f_{K}=\varphi\left(\fr{|\partial_Y|}{M}\right)f.
\end{align*}
Then the Littlewood-Paley decomposition is defined by
\[
f=f_{\fr12}+\sum_{M\in2^{\mathbb{N}}}f_M.
\]
We use the notation
\[
f_{\sim M}=\sum_{K\in\mathbb{D}:\fr{M}{C}\le K\le CM}f_K,
\]
for some constant $C$ independent of $M$. The Littlewood-Paley decomposition in the $(X, Y)$ variables can be defined analogously( with $M$ replaced by $N$ to show the distinction).

Now we are in a position to define the paraproduct decomposition, introduced by Bony \cite{B81}. Given suitable functions $f,g$, the product of $fg$ can be decomposed as follows
\beqno
fg&=&T_fg+T_gf+\mathcal{R}(f,g)\\
&=&\sum_{N\ge8}f_{<\fr{N}{8}}g_N+\sum_{N\ge8}g_{<\fr{N}{8}}f_N+\sum_{N\in\mathbb{D}}\sum_{\fr{N}{8}\le N'\le8N}g_{N'}f_N\\
&=&\sum_{N\ge8}f_{<\fr{N}{8}}g_N+\sum_{N\in\mathbb{D}}g_{<16N}f_N,
\eeqno
where all the sums are understood to run over $\mathbb{D}$.

\section{Elementary inequalities and Gevrey spaces}
In this section, we list some useful lemmas related to the Gevrey norms. 
\begin{lem}
Let $0<s<1$ and $x, y\ge0$.
\begin{enumerate}
\item If $x+y>0$, then
\be\label{ap1}
\left|x^s-y^s\right|\les\fr{1}{x^{1-s}+y^{1-s}}|x-y|.
\ee
In particular,
\be\label{ap2}
\left|\la x\ra^\fr12-\la y\ra^\fr12\right|\le |x-y|^\fr12.
\ee

\item If $|x-y|\le\fr{x}{K}$ for some $K>1$, then
\be\label{ap3}
\left| x^s- y^s\right|\le \fr{s}{(K-1)^{1-s}}|x-y|^s.
\ee
\item If $\fr{y}{K}\le x\le Ky$ for some $K\ge1$, then
\be\label{ap4}
|x+y|^s\le\left(\fr{K}{ 1+K}\right)^{1-s}\left(x^s+ y^s\right).
\ee

\item For all $x\ge0$, $\al>\beta\ge0$, $C, \delta>0$,
\be\label{ap-5}
e^{Cx^\beta}\le e^{C\left(\fr{C}{\dl}\right)^{\fr{\beta}{\al-\beta}}}e^{\dl x^\al}.
\ee

\item For all $x\ge0, \al, \sigma, \dl>0$,
\be\label{ap-6}
\la x\ra^\sigma\les\fr{e^{\dl x^\al}}{\dl^{\fr{\sigma}{\al}}}.
\ee
\end{enumerate}

\end{lem}

\begin{lem}[Product lemma]\label{lem-product}
Let $f=f(X,Y), g=g(Y)$. Then there exists a constant $c\in(0,1)$, such that
\be\label{ap5}
\|A(fg)\|_{L^2}\les\|Af\|_{L^2}\|g\|_{\mathcal{G}^{c\lm,\fr12+}}+\|f\|_{\mathcal{G}^{c\lm,\fr12+}}\|A^Rg\|_{L^2}.
\ee
Similarly, for all $0<s<1$, $\sigma\ge0$, there holds
\be\label{product}
\left\|fg\right\|_{\mathcal{G}^{\lm,\sigma;s}}\les \|f\|_{\mathcal{G}^{\lm,\sigma;s}}\|g\|_{\mathcal{G}^{c\lm,\fr12+;s}}+\|g\|_{\mathcal{G}^{\lm,\sigma;s}}\|f\|_{\mathcal{G}^{c\lm,\fr12+;s}}.
\ee
In particular, for $\sigma>\fr12$, $\mathcal{G}^{\lm,\sigma;s}$ has the algebra property:
\be\label{alge}
\left\|fg\right\|_{\mathcal{G}^{\lm,\sigma;s}}\les \|f\|_{\mathcal{G}^{\lm,\sigma;s}}\|g\|_{\mathcal{G}^{\lm,\sigma;s}}.
\ee
\end{lem}
\begin{proof} 
The inequality \eqref{ap5} can be obtained in a similar manner as the proof of \eqref{pre-ell4}.
For any $h\in L^2$,
\beq
\la A(fg), h\ra\nn&=&\sum_{N\ge8}\sum_{k\in\mathbb{Z}}\int A_k(\eta)\hat{f}_k(\eta-\xi)_{<\fr{N}{8}}\hat{g}(\xi)_{N}\bar{\hat{h}}_k(\eta) d\xi d\eta\\
\nn&&+\sum_{N\in\mathbb{D}}\sum_{k\in\mathbb{Z}}\int A_k(\eta)\hat{f}_k(\xi)_N\hat{g}(\eta-\xi)_{<{16N}}\bar{\hat{h}}_k(\eta)d\xi d\eta\\
\nn&=&I_{\mathrm{LH}}+I_{\mathrm{HL}}.
\eeq
On the support of the integrand of $I_{\mathrm{LH}}$, \eqref{sptR1}--\eqref{sptR3} hold.
Then from \eqref{AkR1} and \eqref{AkR2}, we find that
\beno
\fr{A_k(\eta)}{A^R(\xi)}\les e^{c\lm|k,\eta-\xi|^s},\quad \mathrm{for} \ \ \mathrm{some}\ \ c\in(0,1).
\eeno
Thus,
\beq\label{ap10}
I_{\mathrm{LH}}\nn&\les&\sum_{k\in\mathbb{Z}}\int \left|e^{c\lm|k,\eta-\xi|^s}\hat{f}_k(\eta-\xi)\right|\left|A^R(\xi)\hat{g}(\xi)\right||{\hat{h}}_k(\eta)| d\xi d\eta\\
&\les&\|f\|_{\mathcal{G}^{c\lm, \fr12+}}\|A^Rg\|_{L^2}\|h\|_{L^2}.
\eeq
On the other hand, on the support of the integrand of $I_{\mathrm{HL}}$, \eqref{spt<16N}--\eqref{AkAk} hold. 
 Accordingly,  using Young's inequality, we are led to
\beq\label{ap8}
I_{\mathrm{HL}}\nn&=&\sum_{N\in\mathbb{D}}\sum_{k\in\mathbb{Z}}\int A_k(\eta)\hat{f}_k(\xi)_N\hat{g}(\eta-\xi)_{<16N}\bar{\hat{h}}_k(\eta)d\xi d\eta\\
\nn&\les&\sum_{k\in\mathbb{Z}}\int \left|A_k(\xi)\hat{f}_k(\xi)\right|\left|e^{c\lm|\eta-\xi|^s}\hat{g}(\eta-\xi)\right||\hat{h}_k(\eta)|d\xi d\eta\\
&\les&\|Af\|_{L^2}\|g\|_{\mathcal{G}^{c\lm,\fr12+}}\|h\|_{L^2}.
\eeq
It follows from \eqref{ap10}  and \eqref{ap8} that \eqref{ap5} holds. The proof of \eqref{product} and \eqref{alge} is similar (actually easier), and is thus omitted. We complete the proof of Lemma \ref{lem-product}.
\end{proof}

\end{appendix}

\bigbreak
\noindent{\bf Acknowledgments}
  R.  Zi is partially  supported by NSF of China under  Grants 12222105.

\end{document}